%% file: matrixrecursion.tex
\RequirePackage{rotating}

\documentclass[a4paper]{amsart}

\usepackage{standalone}

\usepackage{etex}
\usepackage{amsfonts, amsmath, amssymb, pifont}
\usepackage{enumitem}
\usepackage{upref}
\usepackage[all]{xy}
\usepackage[bookmarksdepth=2]{hyperref}
\usepackage{graphicx}
\usepackage{tabularx}
\usepackage{subcaption}
\usepackage{rotating}
\usepackage{mathtools}
\usepackage{calc}

\usepackage{todonotes}

\usepackage{tikz}
\usetikzlibrary{matrix, arrows.meta, intersections, calc, decorations.pathreplacing} 
\input{matrixrecursiontikzstyles.tex}

\usepackage{tikz-cd}

\usepackage{glossaries}
\usepackage{glossary-mcols}

\renewcommand*{\glossarymark}[1]{}
\glsdisablehyper
\makenoidxglossaries

\newglossary[flg]{fake}{fls}{flo}{Fake Entries}

\theoremstyle{plain}
\newtheorem{thm}{Theorem}[section]
\newtheorem*{thm*}{Theorem}
\newtheorem{thmroman}{Theorem}[section]

\newtheorem{lem}[thm]{Lemma}

\newtheorem*{conj*}{Conjecture}
\newtheorem{cor}[thm]{Corollary}
\newtheorem*{cor*}{Corollary}
\newtheorem{prop}[thm]{Proposition}
\newtheorem*{obs*}{Observation}

\theoremstyle{definition}
\newtheorem{defn}[thm]{Definition}
\newtheorem{exam}[thm]{Example}
\newtheorem{notn}[thm]{Notation}
\newtheorem{constr}[thm]{Construction}
\theoremstyle{remark}
\newtheorem{rem}[thm]{Remark}

\numberwithin{equation}{section}

\input{matrixrecursion-symbolmacros.tex}
\input{matrixrecursion-symbollist.tex}

\begin{document}

\title[Matrix recursion for Soergel bimodules]{Matrix recursion for positive characteristic diagrammatic Soergel bimodules for affine Weyl groups}
\author{Amit Hazi}
\date{27 January 2025}
\address{School of Mathematics\\
University of Leeds\\
Leeds\\
LS2 9JT\\
United Kingdom}
\email{A.Hazi@leeds.ac.uk}


\maketitle

\begin{abstract} 
Let $\W$ be an affine Weyl group, and let $\Bbbk$ be a field of characteristic $p>0$. The diagrammatic Hecke category $\mathcal{D}$ for $\W$ over $\Bbbk$ is a categorification of the Hecke algebra for $\W$ with rich connections to modular representation theory. We explicitly construct a functor from $\mathcal{D}$ to a matrix category which categorifies a recursive representation $\xi: \mathbb{Z}\W \rightarrow M_{p^r}(\mathbb{Z}\W)$, where $r$ is the rank of the underlying finite root system. This functor gives a method for understanding diagrammatic Soergel bimodules in terms of other diagrammatic Soergel bimodules which are ``smaller'' by a factor of $p$. It also explains the presence of self-similarity in the $p$-canonical basis, which has been observed in small examples. By decategorifying we obtain a new lower bound on the $p$-canonical basis, which corresponds to new lower bounds on the characters of the indecomposable tilting modules by the recent $p$-canonical tilting character formula due to Achar--Makisumi--Riche--Williamson. 

\end{abstract}

\section*{Introduction}

\subsection*{Diagrammatic Soergel bimodules}

The Hecke category provides the universal setting for studying the interaction between Kazhdan--Lusztig theory and Lie theory.
This category lifts the combinatorics of Coxeter groups to the level of a monoidal category, with analogues of the classical Coxeter relations as well as deeper `relations between relations'.
Moreover, diagrammatic presentations of this category have allowed for direct comparisons between representation theory over fields of positive characteristic versus characteristic $0$.
Unifying these two worlds will undoubtedly be the core of Lie theory in the 21st century.

More precisely, let $(\W,\Ss)$ be a Coxeter system, and let $\hecke$ be the Hecke algebra associated to $\W$. 
The \defnemph{diagrammatic Hecke category $\dgrm$} of $\W$ over a field $\field$ is an entirely algebraic construction of a categorification of $\hecke$ \cite{ew-soergelcalc}. 
In other words, $\dgrm$ is a $\field$-linear additive graded monoidal category defined by a diagrammatic presentation, whose split Grothendieck ring $[\dgrm]$ (generated by isomorphism classes of indecomposable objects) is isomorphic to $\hecke$. 
For each $x \in \W$ there is an indecomposable object $B_x$ labeled by $x$, and all indecomposable objects are of this form up to grade shift. The set $\{[B_x] : x \in \W\}$ is a basis of the split Grothendieck ring $[\dgrm]$ and thus corresponds to a basis of $\hecke$. 

When $\W$ is finite crystallographic and $\field$ is of characteristic $0$, $\dgrm$ is equivalent to the category of Soergel bimodules; for this reason objects in $\dgrm$ are called \defnemph{(diagrammatic) Soergel bimodules}, and $\dgrm$ is sometimes called the \defnemph{diagrammatic category of Soergel bimodules}. 
The diagrammatic category $\dgrm$ provides a suitable replacement for the category of Soergel bimodules in situations where Soergel bimodules are less well behaved, such as when $\W$ is infinite or $\field$ has positive characteristic. 
For $\W$ crystallographic, the \defnemph{$p$-canonical basis} $\{\prescript{p}{}{b}_x : x \in \W\}$ for $\W$ is the basis for $\hecke$ corresponding to the isomorphism classes of indecomposable objects in $\dgrm$ defined over a field $\field$ of characteristic $p>0$ \cite{jensen-williamson}. 
Equality of the $p$-canonical basis and the Kazhdan--Lusztig basis $\{b_x : x \in \W\}$ is associated with several classical conjectures in modular representation theory, the most important of which is Lusztig's conjecture on the characters of the simple modules for a reductive algebraic group \cite{lusztig-chevalley,soergel-ontherel,fiebig}. 
In particular, in the case of finite $\W$ it is now known that $b_x=\prescript{p}{}{b}_x$ for all $x \in \W$ only when $p$ is extremely large, i.e.~exponential in the rank of $\W$ \cite{williamson-schubertcalctorexpl}. 
Finding efficient ways to compute the $p$-canonical basis for smaller $p$ is an important open problem. 

\subsection*{Main results}

In this paper we focus on the diagrammatic Hecke category for an affine Weyl group over a field of characteristic $p$.
Our main theorem roughly states that this category can be decomposed into fractally smaller copies of itself, scaled down by a factor of $p$. 
Characteristic $p$-fractal behavior, driven by the Frobenius endomorphism, has long been a centerpiece of modular Lie theory --- our decomposition theorem provides the first realization of this $p$-fractal behavior within the Hecke category. 
This provides the first categorical setting in which to investigate the generational philosophy of Lusztig--Williamson and the billiards conjecture of Lusztig--Williamson--Jensen \cite{lw-tiltinggens,lw-billiards,ltjensen-billiardscorrection}.



In more detail, let $\Phif$ be an irreducible finite root system with finite Weyl group $\Wf$ and let $\W=\Wf \ltimes \ZZ\Phif$ be the affine Weyl group generated by $\Wf$ and translations $t(\alpha) : \lambda \mapsto \lambda+\alpha$ for $\alpha \in \ZZ\Phif$. 
For $p$ a prime we write $\Wp=\W \ltimes p\ZZ\Phif$ for the $p$-affine subgroup of $\W$ generated by $\Wf$ and $t(p\alpha)$ for $\alpha \in \ZZ\Phif$. 
The homomorphism 
\begin{align*}
F : \W & \longrightarrow \W  \\
w & \longmapsto \textover[l]{$w$}{$t(p\alpha)$} \qquad \text{for $w \in \Wf$,} \\ 
t(\alpha) & \longmapsto t(p\alpha) \qquad \text{for $\alpha \in \ZZ\Phif$,} 
\end{align*} 
which fixes $\Wf$ and scales translations by a factor of $p$, induces an isomorphism $\W \iso \Wp$. 
Consider the group ring $\ZZ\W$ as a $(\ZZ\Wp,\ZZ\W)$-bimodule. 
As a left $\ZZ\Wp$-module $\ZZ\W$ is free, with a basis $\Wpcosets \subset \W$ of minimal length representatives for the right cosets $\Wp \backslash \W$. 
The right $\W$-action on $\ZZ\W$ induces a faithful homomorphism 
\begin{equation*}
\ZZ\W \longrightarrow \End_{\ZZ\Wp}(\ZZ\W) \xrightarrow{\sim} M_{|\Wpcosets|}(\ZZ\Wp) \xrightarrow{\sim} M_{|\Wpcosets|}(\ZZ\W) \text{.} 
\end{equation*}
Here the first isomorphism is obtained by writing $\ZZ\Wp$-endomorphisms of $\ZZ\W$ as matrices with respect to the basis $\Wpcosets$. 
The second isomorphism is just $F^{-1}$ applied to the entries of the matrices. 
Following \cite{selfsimilargroups} we call the composition $\xi:\ZZ\W \rightarrow M_{|\Wpcosets|}(\ZZ\W)$ the \defnemph{matrix recursion representation} of $\ZZ\W$. 
As far as we know this representation does not lift easily to Hecke algebras.

Our most important result is a categorification of the matrix recursion representation. 
More precisely, suppose $\field$ is a field of characteristic $p>2$ (or $p>3$ if $\Phif$ is of type $G_2$). 
The construction of the diagrammatic category uses the data of a \defnemph{$\field$-realization} of $\W$, i.e.~a reflection representation of $\W$ over $\field$. 
We consider two different $\field$-realizations of $\W$: the \defnemph{universal realization} of $\W$ over $\field$ with respect to the affine simple roots (see Definition~\ref{defn:univrlz}) and the \defnemph{$F$-twist} of this realization (see Definition~\ref{defn:Ftwist}). 
Let $\dgrm$ and $\dgrmF$ respectively denote the diagrammatic Hecke categories constructed from these two realizations of $\W$. 
The categorification of $\xi$ is as follows (see also Theorem~\ref{thm:linkagefunctor}). 

\begin{thmroman} \label{thm:intro-matrix-recursion}
There is a faithful monoidal functor 
\begin{equation*}
\pr: \dgrmdeg \longrightarrow \Mat_{|\Wpcosets|}(\hat{R} \otimes \dgrm^{F,{\rm deg}}) \text{,} 
\end{equation*}
which induces the matrix recursion representation on Grothendieck rings.
\end{thmroman}

We call $\pr$ the \defnemph{matrix recursion functor}.
The source of $\pr$ is the ``degrading'' of $\dgrm$, i.e.~the ungraded category obtained by forgetting the grading on $\Hom$-spaces in $\dgrm$. 
Its split Grothendieck ring $[\dgrmdeg]$ is canonically isomorphic to $\hecke/(v-1) \iso \ZZ\W$, the classical limit of $\hecke$. 
The target of $\pr$ is a certain ``matrix category'' constructed as follows. 
(This construction is almost certainly not new, but we were unable to find a specific reference for it in the literature.) 
The objects and morphisms in the matrix category $\Mat_{|\Wpcosets|}(\hat{R} \otimes \dgrm^{F,{\rm deg}})$ are just $|\Wpcosets| \times |\Wpcosets|$ matrices whose entries are objects and morphisms respectively in the category $\hat{R} \otimes \dgrm^{F,{\rm deg}}$. 
Composition of matrices of morphisms is just entrywise composition (see \eqref{eq:hdmrdmatrix}), while the tensor product in the matrix category is given by the categorical analogue of the ordinary matrix product (see \eqref{eq:prodmatrix}). 
Here $\hat{R} \otimes \dgrm^{F,{\rm deg}}$ is a scalar extension of $\dgrm^{F,{\rm deg}}$ (the degrading of $\dgrmF$) with respect to a certain complete discrete valuation ring $\hat{R}$ (see Definition~\ref{defn:Rhat}).


The matrix recursion representation $\xi$ characterizes the elements of $\W$ in terms of ``smaller'' elements in the following sense. 
The affine Weyl group $\W$ acts via Euclidean isometries on the root lattice $\ZZ\Phif$. 
For any $x \in \W$, the vector $x(0)$ is contained in the closure of the alcove associated to $x$.
Therefore we can think of the Euclidean norm $|x(0)|$ as a rough measure of the length or complexity of $x$. 
For each $x \in \W$ and each non-zero entry $y$ in the matrix $\xi(x)$, the Euclidean norms of $x(0)$ and $y(0)$ satisfy the inequality 
\begin{equation*}
|y(0)|\leq p^{-1}|x(0)|+2p^{-1} \max_{w \in \Wpcosets} |w(0)| \leq p^{-1}|x(0)|+K(\Phif) \text{,} 
\end{equation*}
where $K(\Phif)$ is a positive constant which only depends on the root system $\Phif$, and not on $x$, $y$, or $p$.
The functor behaves similarly, with direct summands $B_y$ of non-zero entries of $\pr(B_x)$ satisfying the same inequality. 
Since smaller Soergel bimodules are easier to understand than larger ones, the functor is a useful tool for computations in $\dgrm$. 
For example, by reading off the first row of the $\pr(B_x)$ and decategorifying we obtain the following new lower bound on the $p$-canonical basis (see also Corollary~\ref{cor:sbim-highordlink}). 
For $b \in \hecke$ write $\prescript{}{v=1}{b}$ for the image of $b$ in $\ZZ\W$. 

\begin{thmroman} \label{thm:intro-positivity}
Let $x \in \W$. Then
\begin{equation*}
\prescript{p}{v=1}{b}_x \in \sum_{\substack{y \in \W \\ w \in \Wpcosets\\ F(y)w \leq x}} \ZZ_{\geq 0} F(\prescript{p}{v=1}{b}_y) w \text{.}
\end{equation*}
\end{thmroman}

This lower bound can be applied recursively to each of the $p$-canonical basis elements on the right-hand side to obtain an even better bound. 
This feature is especially notable because there is no similar recursion for the functor $\pr$.
The main obstruction to functorial recursion is that the categories $\dgrm$ and $\dgrmF$ are not precisely equivalent, even though their indecomposable objects both give the $p$-canonical basis.

All of the algebraic and categorical constructions which we use to prove the main results are graded or filtered in a natural way. 
For example, in \S\ref{sec:heckeast-heckepast} we introduce an algebra $\heckeast$ which is a $\laur$-module extension of $\ZZ\W$, and an $(\heckep,\heckeast)$-bimodule $\heckepast$ which is a $\laur$-module extension of the $(\ZZ\Wp,\ZZ\W)$-bimodule $\ZZ\W$. 
The right action of $\heckeast$ on $\heckepast$ gives an extension $\heckeast \rightarrow M_{|\Wpcosets|}(\hecke)$ of the matrix recursion representation.
This extends to our categorifications of these structures through the notion of a \defnemph{valuation} on a category (see Definition~\ref{defn:valmod}), which gives rise to filtered $\Hom$-spaces. 
Interestingly, the valuation structure on our new categories does not agree with the usual grading on $\dgrm$.
This is the reason why both of the main results above are inherently ungraded.
We hope to study the extra graded structure in our constructions in future work.

\subsection*{Tilting modules and Andersen's observation}

We will comment on the connection between our results and modular representation theory of reductive groups. 
In this setting, our work builds a framework through which to view Andersen's conjecture on tilting characters.
Let $G$ be a semisimple, simply connected algebraic group with root system $\Phif$ over an algebraically closed field of characteristic $p\geq h$, where $h$ is the Coxeter number of $\Phif$. 
The representation theory of $G$ is controlled by the \emph{dual} affine Weyl group, so we now take $\W=\Wf \ltimes \ZZ\Phi_{\rm f}^\vee$, with corresponding $p$-affine subgroup $\Wp=\Wf \ltimes p\ZZ\Phi_{\rm f}^\vee$ and minimal length right coset representatives $\Wpcosets$. 
We write $\cdot_p$ for the $p$-dilated dot action, i.e.~for $w \in \W$ and $\lambda$ a weight, 
\begin{equation*}
w \cdot_p \lambda=pw(p^{-1}\lambda+p^{-1}\rho)-\rho
\end{equation*}
where $\rho$ is the half-root sum. 
We also write $C_p$ for the dominant weights in the interior of the fundamental ($\rho$-shifted) $p$-alcove. 
We call $w \in \W$ dominant if $w \cdot_p \lambda$ is dominant for any $\lambda \in C_p$. 

A \defnemph{tilting module} is a $G$-module with a filtration by Weyl modules and a filtration by dual Weyl modules. 
For each dominant weight $\lambda$ there is (up to isomorphism) exactly one indecomposable tilting module $T(\lambda)$ with highest weight $\lambda$ \cite{ringel,donkin-ontilting}. 
When $p \geq 2h-2$ there is a formula for the character of the simple $G$-modules in terms of the characters of the indecomposable tilting modules. 
In the analogous setting of quantum groups at an $l$th root of unity, Soergel proved a character formula for the indecomposable tilting modules $T_l(\lambda)$ in terms of antispherical Kazhdan--Lusztig polynomials \cite{soergel-KL,soergel-kacmoody}. 
By a lifting argument, Andersen observed in \cite[\S 4.2]{andersen-sumformula} that quantum tilting characters form a lower bound for modular tilting characters in the following sense.

\begin{obs*}[Andersen]
Let $\lambda \in C_p$. For any $r \in \NN$ and any dominant $x \in \W$,
\begin{equation*}
\charctr T(x \cdot_p \lambda) \in 
\sum_{y \leq x} \Nneg \charctr T_{p^r}(y \cdot_p \lambda) \text{.}
\end{equation*}
\end{obs*}

Here $\charctr M$ denotes the formal character of $M$, i.e.~$\charctr M=\sum_\lambda (\dim M_\lambda) e^{\lambda}$.
 
Andersen's observation is notable for being one of the few results on tilting characters in the classical modular representation theory of algebraic groups (i.e.~using methods similar to those in \cite{jantzen}) which is valid for all weights. 
It was used extensively in \cite{jg-jensen} to calculate characters of indecomposable tilting modules for $G=\SL_3$. 
Andersen also made the following conjecture related to the $r=1$ case \cite[Remark~3.6(i)]{andersen-filtrationstilting}. 

\begin{conj*}[Andersen]
Suppose Lusztig's conjecture holds for $G$ in characteristic $p$. 
For $\lambda \in C_p$ and $x \in \Wpcosets$, we have $\charctr T(x \cdot_p \lambda)=\charctr T_p(x \cdot_p \lambda)$.
\end{conj*}

This is essentially the strongest possible conjecture of this nature; in particular Lusztig's conjecture is necessary by \cite[Theorem~5.1]{soergel-KL}. 
Note that weights of the form $x \cdot_p \lambda$ for $\lambda \in C_p$ and $x \in \Wpcosets$ are just the dominant weights in the fundamental $p^2$-alcove which lie in the principal block. 

We can combine Andersen's observation and Andersen's conjecture to obtain a special case of a recursive lower bound on tilting characters. 
Even this special case appears to be new, although there are implicit references to it (and to the limits of this approach) in \cite[\S 6]{jg-jensen}.

For $\lambda$ a dominant weight let $\chi(\lambda)$ denote the character of the Weyl module of highest weight $\lambda$ (given by the Weyl character formula).
For
\begin{equation*}
\chi=\sum_{\substack{x \in \W\\ \lambda \in C_p}} a_{x,\lambda} \chi(x \cdot_p \lambda)
\end{equation*}
a character and $w \in \Wpcosets$, write 
\begin{equation*}
F^w(\chi)=\sum_{\substack{x \in \W\\ \lambda \in C_p}} a_{x,\lambda} \chi(F(x)w \cdot_p \lambda) \text{.} 
\end{equation*}

\begin{thm*}[Recursive tilting character lower bound, \textit{\`{a} la} Andersen]~\\
Assume Andersen's conjecture holds. For any $\lambda \in C_p$ and $x,x' \in \Wpcosets$, we have
\begin{equation*}
\charctr T(F(x')x \cdot_p \lambda) \in \sum_{\substack{y \in \W \\ w \in \Wpcosets\\ F(y)w \leq F(x')x}} \Nneg F^w(\charctr T(y \cdot_p \lambda)) \text{.}
\end{equation*}
\end{thm*}

See \S \ref{sec:tiltingchars} for a derivation of this result.
Here we note that weights of the form $F(x')x \cdot_p \lambda$ for $x,x' \in \Wpcosets$ are just the regular dominant weights in the fundamental $p^3$-alcove. 
By contrast, our lower bound on the $p$-canonical basis immediately gives a far stronger result using the $p$-canonical tilting character formula from \cite[Theorem~7.6]{amrw-koszultilting}.


\begin{thmroman} \label{thm:tiltcharbound}
For any $\lambda \in C_p$ and $x \in \W$ dominant we have
\begin{equation*}
\charctr T(x \cdot_p \lambda) \in \sum_{\substack{y \in \W\\ w \in \Wpcosets\\ F(y)w \leq x}} \Nneg F^w(\charctr T(y \cdot_p \lambda)) \text{.}
\end{equation*}
\end{thmroman}

As with the $p$-canonical basis lower bound, we can apply this tilting character lower bound recursively as there are no restrictions on $x$. 

The $p$-canonical tilting character formula in \cite{amrw-koszultilting} is the ``combinatorial shadow'' of an equivalence between the antispherical quotient category of $\dgrm$ and the full subcategory of tilting modules lying in some regular block \cite{riche-williamson,bezrukavnikovriche,ciappara}. 
We believe that this equivalence could lead to a more directly representation-theoretic interpretation of the functor $\pr$.

\subsection*{Reflection subgroups}

Some of the proofs in earlier versions of this paper were initially very indirect, relying on complicated relations and extensive use of localization matrices.
In subsequent versions we have focused on formulating more explicit constructions (e.g.~for the image of the braid under the Frobenius functor) and improving the exposition and notation. 
This has had the added benefit of making future generalizations of our work more apparent: since this paper first appeared, there has been a proliferation of similar work decomposing various non-diagrammatic versions of the Hecke category for a Coxeter group $\W$ in terms of a reflection subgroup \cite{mcdonnell,jensen-patimo,williamson-reflsubgroups}. 
In particular, Williamson's 2019 lectures at Current Developments in Mathematics at Harvard \cite{williamson-reflsubgroups} were directly inspired by our work. 
Williamson showed that the geometric incarnation of our decomposition appears to be a consequence of Smith--Treumann theory (see also \cite{rw-smithtreumann}). 
He also proved the analogue of Theorems~\ref{thm:intro-matrix-recursion} and \ref{thm:intro-positivity} for Abe's bimodule-theoretic version of the Hecke category and a ``good'' reflection subgroup.\footnote{Here ``good'' simply means that an analogue of Lemma~\ref{lem:rlzcoefs} holds.}
 
It is likely that the diagrammatic results in this paper can be generalized in a similar way to arbitrary Coxeter groups. 
We have decided to only focus on the affine Weyl group in this paper for two reasons.
Firstly, the affine Weyl group setting is of particular interest due to the aforementioned applications to modular representation theory. 
Secondly, writing our results in complete generality would likely increase the complexity of our notation to unmanageable levels; in some sense the case of $\Wp \leq \W$ is one of the simplest non-trivial examples of a reflection subgroup.

\subsection*{Acknowledgments} 
The author would like to thank Chris Bowman, Ben Elias, Thorge Jensen, Paul Martin, Stuart Martin, Alison Parker, Leonardo Patimo, and Geordie Williamson for their comments on earlier versions of this paper. 
The author is particularly grateful to the referees for their detailed corrections and suggestions. 
In the course of writing this paper the author received financial support from the the Royal Commission for the Exhibition of 1851, the EPSRC grant EP/V00090X/1, and the EPSRC Programme Grant EP/W007509/1.

\section{Affine Weyl groups}

We give an overview of the relationship between affine Cartan matrices, affine Weyl groups, and affine reflection groups.
Many of the ideas here should be familiar to readers with a thorough understanding of affine root systems (e.g.~from \cite{macdonald} or \cite{kac}). 
We will only assume knowledge of finite root systems and general Coxeter groups in what follows.

\subsection{Affine Cartan matrices and realizations}

\label{sec:affWnotn}

Let $\gls*{A}=(a_{st})_{s,t \in \Ss}$ be an indecomposable affine Cartan matrix.
We will use the following facts about such matrices, which can be found in \cite[Chapter~4]{kac}.
The rank of $A$ is $|\Ss|-1$, and there is a unique column vector $(c_s)_{s \in \Ss}$ with non-negative integer entries, at least one of which is equal to $1$, such that
\begin{equation}
\sum_{t \in \Ss} a_{st} c_t=0
\end{equation}
for all $s \in \Ss$.
In addition, every proper submatrix of $A$ is positive definite.
We will also need the fact that $A$ is symmetrizable. 
In other words, there are positive real numbers $(n_s)_{s \in \Ss}$ such that $n_s a_{st}=n_t a_{ts}$ for all $s,t \in \Ss$.

The \defnemph{affine Weyl group} $\glsuseri*{W}$ corresponding to $A$ is the free group generated by $\Ss$ subject to the relations
\begin{gather}
s^2=1 \qquad \text{for all $s \in \Ss$,} \\
\underbrace{sts \dotsm}_{m_{st}} = \underbrace{tst \dotsm}_{m_{st}} \qquad \text{for all distinct $s,t \in \Ss$ with $m_{st}\neq \infty$,}
\end{gather}
where
\begin{equation}
m_{st}=\begin{cases}
2 & \text{if $a_{st}a_{ts}=0$,} \\
3 & \text{if $a_{st}a_{ts}=1$,} \\
4 & \text{if $a_{st}a_{ts}=2$,} \\
6 & \text{if $a_{st}a_{ts}=3$,} \\
\infty & \text{if $a_{st}a_{ts}=4$.}
\end{cases}
\end{equation}
As our notation suggests we interpret $\glsuseri*{Ss}$ primarily as the set of simple generators in the Coxeter system $(\W,\Ss)$, with the rows and columns of the Cartan matrix indexed by these generators.

We recall some Coxeter group terminology and set up some related notation. 
We call $t \in \W$ a \defnemph{reflection} if $t$ is conjugate to a simple generator (i.e.~if $t=wsw^{-1}$ for some $s \in \Ss$ and $w \in \W$.).
An \defnemph{expression} in $\Ss$ is a finite sequence $\glsuseri*{exprx}=\expr{s_1 s_2 \dotsc s_m}$ of Coxeter generators of $\W$ written using an underline. 
We denote the set of all expressions in $\Ss$ by $\glsuseri*{exprSs}$. 
We use the non-underlined counterpart of an expression to denote the product of these generators in $\W$, i.e.~$x=s_1 s_2 \dotsm s_m \in \W$. 
We write $\len(\expr{x})=m$ for the \defnemph{length} of $\expr{x}$. 
As $\W$ is a Coxeter group, it is equipped with a length function $\glsuseri*{len}:W \rightarrow \ZZ_{\geq 0}$ and a closely related partial order $\leq$ on $\W$ called the \defnemph{Bruhat order}. 
We call $\expr{x}$ a \defnemph{reduced expression} (or \defnemph{rex}) for $x$ when $\len(\expr{x})=\len(x)$. 
We note that $\expr{\Ss}$ also includes the degenerate case of the \defnemph{empty expression} $\emptexpr$, i.e.~an expression of length $0$, which is the unique rex for the identity element in $\W$. 
We will use the adjective ``empty'' in a similar way in later definitions involving finite sequences.



We are primarily interested in representations of $\W$ called \defnemph{realizations}, where the action of each $s \in \Ss$ behaves like a reflection (see \cite[Definition 3.1]{ew-soergelcalc} and \cite[\S A.2]{elias-dihedralcathedral}). 
We will later construct the diagrammatic category $\dgrm$ of Soergel bimodules from a realization.


\begin{defn}
Let $\field$ be an integral domain. A \defnemph{realization} of the affine Weyl group $(\W,\Ss)$ over $\field$ consists of a free, finite rank $\field$-module $V$ along with subsets $\{\glsuseri*{as} : s \in \Ss\} \subset V$ and $\{a_s^\vee : s \in \Ss\} \subset V^\ast=\Hom_\field(V,\field)$ such that
\begin{enumerate}[label=(\roman*)]
\item $\langle a_s^\vee,a_s \rangle=2$ for all $s \in \Ss$;

\item the assignment 
\begin{equation*}
s(b)=b-\langle a_s^\vee, b \rangle a_s
\end{equation*} 
for all $s \in \Ss$ and $b \in V$ defines a representation of $\W$ on $V$. 
\item the technical condition \cite[Definition~5.1(iii)]{2colJW} is satisfied.
\end{enumerate}
\end{defn}

When clear from context we will often write a realization $(V,\{a_s\},\{a_s^\vee\})$ simply as $V$. 
We call the matrix $(\langle a_s^\vee,a_t \rangle)_{s,t \in \Ss}$ the \defnemph{Cartan matrix} of $V$. 
One useful fact about realizations which we will use frequently is the following.
For any $w \in \W$ and $s \in \Ss$, the reflection $t=wsw^{-1}$ acts on any realization $V$ via the following formula:
\begin{equation} \label{eq:nonsimplereflections}
\begin{aligned}
t(b)& =wsw^{-1}(b) \\
& =w(w^{-1}(b)-\langle a_s^\vee, w^{-1}(b)\rangle a_s) \\
& =b-\langle a_s^\vee, w^{-1}(b) \rangle w(a_s) \\
& =b-\langle w(a_s^\vee), b \rangle w(a_s) \text{.}
\end{aligned}
\end{equation}

\begin{defn}
Suppose $(U,\{a_s\},\{a_s^\vee\})$ and $(V,\{b_s\},\{b_s^\vee\})$ are two realizations of $(\W,\Ss)$ over $\field$. 
We call a $\field$-linear map $\phi:U \rightarrow V$ a \defnemph{morphism of realizations} if $\phi$ is a homomorphism of $\W$-representations and $\phi(a_s)=b_s$ for all $s \in \Ss$.
\end{defn}

Realizations of $(\W,\Ss)$ over $\field$ naturally form a category using this notion of morphism.
As we will see later, the construction of the diagrammatic category $\dgrm$ from a realization is functorial. 
All of the realizations we will encounter in this paper are variants of the following.

\begin{defn} \label{defn:univrlz}
The \defnemph{universal realization} $(\gls*{Va},\{a_s\},\{a_s^\vee\})$ of $(\W,\Ss)$ with respect to $A=(a_{st})_{s,t \in \Ss}$ is defined as follows. Let $V_{A}$ be a free $\ZZ$-module with basis $\{a_s:s \in \Ss\}$ and define $\{a_s^\vee\} \subset V_{A}^\ast$ by
\begin{equation}
\langle a_s^\vee, a_t \rangle=a_{st}  \qquad \text{for all $s,t \in \Ss$.} \label{eq:univrealiz}
\end{equation}
\end{defn}

We can verify that $V_{A}$ really is a representation of $\W$ by checking this on the finite dihedral parabolic subgroups (see e.g.~\cite[(A.2)]{elias-dihedralcathedral}).
By definition the Cartan matrix of $V_{A}$ is just $A$.
Many of the constructions below work for any realization of $(\W,\Ss)$ with the same Cartan matrix.
We are primarily interested in $V_{A}$ because it has the additional property that the set $\{a_s\}$ is linearly independent in $V$. 

\begin{rem} \label{rem:realizations} \hfill
\begin{enumerate}[label=(\arabic*)]
\item Suppose $\phi:U \rightarrow V$ is a morphism between the realizations $(U,\{a_s\},\{a_s^\vee\})$ and $(V,\{b_s\},\{b_s^\vee\})$. 
It is straightforward to check that $\phi^\ast(b_s^\vee)=a_s^\vee$ for all $s \in \Ss$, so $\phi^\ast:V^\ast \rightarrow U^\ast$ is a morphism between the dual realizations $(V^\ast,\{b_s^\vee\},\{b_s\})$ and $(U^\ast,\{a_s^\vee\},\{a_s\})$.

\item \label{item:morphism-realizations} Morphisms of realizations are quite restrictive. 
Suppose there exists a morphism between two realizations $U$ and $V$.
Then $U$ and $V$ must have the same Cartan matrix.
This condition is necessary but not sufficient for there to be a morphism between $U$ and $V$.

\item The universal realization of $(\W,\Ss)$ has the following universal property: for any realization $V$ of $(\W,\Ss)$ over $\ZZ$ with Cartan matrix $A$, there is a unique morphism of realizations $V_{A} \rightarrow V$. 
For $\field$ an arbitrary integral domain, the extension by scalars $\field \otimes V_{A}$ has the same properties for realizations defined over $\field$.
More generally, the notion of a universal realization makes sense for arbitrary Coxeter groups, not just affine Weyl groups.
\end{enumerate}
\end{rem}

The universal realization $V_{A}$ is closely related to an affine reflection representation of $\W$.
Write $V_{A,\RR}=\RR \otimes V_{A}$ and set 
\begin{equation}
\gls*{c}=\sum_{s \in \Ss} c_s a_s \in V_{A} \text{.}
\end{equation}
Let $E$ be the affine hyperplane
\begin{equation}
\gls*{E}=\left\{b \in V_{A,\RR}^\ast : \langle b, c \rangle =1\right\} \subset V_{A,\RR}^\ast \text{.}
\end{equation}

\begin{thm} \label{thm:affrefl}
The affine space $E$ has the structure of an affine inner product space (i.e.~a Euclidean space).
The affine Weyl group $\W$ acts on $E$ by Euclidean isometries, with each $s \in \Ss$ corresponding to reflection in the affine hyperplane
\begin{equation*}
E \cap (\RR a_s)^\circ=\{b \in E : \langle b, a_s \rangle=0\} \text{.}
\end{equation*}
\end{thm}

\begin{proof}
See \cite[Chapter~6.5]{humphreys} for the case when $A$ is the Cartan matrix of an extended Dynkin diagram, or Appendix~\ref{app:affrefl} for the general case (i.e.~for $A$ of twisted affine type).
\end{proof}

We will frequently use the following standard terminology regarding the affine reflection representation of $\W$, some of which can be found in e.g.~\cite[Chapter~4]{humphreys}. 
The \defnemph{fundamental alcove} is the open subset
\begin{equation}
\glsuseri*{C}=\{b \in E : \langle b, a_s \rangle > 0 \text{ for all $s \in \Ss$}\} \subset E \text{.}
\end{equation}
Its closure in $E$ is a simplex of dimension $|\Ss|-1$.
The images of $C$ under the action of $\W$ are called \defnemph{alcoves}. 
Let $\alcs$ be the set of all alcoves in $E$. 
The group $\W$ acts simply transitively on $\alcs$, which implies that the $\W$-action on $E$ (and thus on $V_{A}$) is faithful.
The action on $\alcs$ also gives a bijection $\W \xrightarrow{\sim} \alcs$ mapping $w \mapsto w(C)$.
We are often interested in the \emph{right} regular $\W$-action defined by this bijection, i.e.~we set $(w(C))x=wx(C)$ for all $w,x \in \W$.
A \defnemph{wall} of an alcove is a (codimension $1$) face of its closure.
Two alcoves $C_1,C_2$ are \defnemph{adjacent} if they share a wall.
It is easy to show that this is equivalent to $C_2=C_1 s$ for some $s \in \Ss$.
We call the wall between $C_1$ and $C_2$ the \defnemph{$s$-wall} of $C_1$ or of $C_2$.

There is an alternative description of the length function $\len : \W \rightarrow \Nneg$ in terms of alcoves.
For any $w \in \W$, $\len(w)$ equals the number of reflection hyperplanes which separate $C$ and $w(C)$ \cite[Theorem~4.5]{humphreys}.
We can rephrase this in terms of the realization as follows. 
Let $\Phi=\{w(a_s) : w \in \W,\ s \in \Ss\}$ be the \defnemph{roots} of the realization $V_{A,\RR}$.
Then $\len(w)$ equals the number of roots $a \in \Phi$ such that
\begin{equation*}
\langle C, a \rangle >0 \qquad \text{but} \qquad \langle w(C), a \rangle <0 \text{.}
\end{equation*}
This suggests the following geometrically intuitive way to compare the lengths of two elements (see \cite[Proposition~5.7]{humphreys}).
Given $w \in \W$ and $s \in \Ss$, $\len(ws) > \len(w)$ if and only if $C$ and $w(C)$ lie on the same side of the reflection hyperplane
\begin{equation*}
E \cap (\RR w(a_s))^\circ=\{b \in E : \langle b, w(a_s) \rangle=0\} \text{.}
\end{equation*}
In other words, 
\begin{equation}
\len(ws) > \len(w) \Longleftrightarrow \langle C, w(a_s) \rangle > 0 \text{.} \label{eq:lengthposrootcriterion}
\end{equation}

It is also important for us to understand the normal subgroup $L \lhd \W$ consisting of elements which act on $E$ as translations. 
Let
\begin{equation}
Z_{\RR}=\{\lambda \in V_{A,\RR}^\ast : \langle \lambda, c\rangle=0\} \text{,}
\end{equation}
a hyperplane in $V_{A,\RR}^\ast$.
Every translation of $E$ is of the form
\begin{align*}
t(\lambda) : E&  \longrightarrow E \\
b & \longmapsto b+\lambda
\end{align*}
for some $\lambda \in Z_{\RR}$.
Write $\Lambda$ for the set of all $\lambda \in Z_{\RR}$ such that $t(\lambda) \in L$. 
In particular, note that we distinguish between the subgroup $L \leq \W$ of translations and the lattice $\Lambda \subset Z_{\RR}$, even though they are isomorphic as groups, with $t(\Lambda)=L$.
The next result describes the relationship between $\Lambda$ and $\{a_s^\vee : s \in \Ss\}$.

\begin{prop} \label{prop:translation-index}
The lattice $\Lambda$ is a sublattice of 
\begin{equation*}
Z=\sum_{s \in \Ss} \ZZ a_s^\vee \text{.}
\end{equation*} 
The index $[Z: \Lambda]$ is either a power of $2$ or $3$.
\end{prop}

\begin{proof}
When $A$ is a Cartan matrix of extended Dynkin type, we have $\Lambda=Z$ essentially by definition and by the construction in \cite[Chapter~6.5]{humphreys}.
For the general case see Appendix~\ref{app:affrefl} (in particular Remark~\ref{rem:affrootsystems}).
\end{proof}

Finally we explain how to obtain translations from the Coxeter generators $\Ss$.

\begin{prop} \label{prop:tildes-defn}
There exists some $\glsuseri*{tildes} \in \Ss$ such that $c_{\tilde{s}}=1$ and $\W$ is generated by $\Ss \setminus \{\tilde{s}\}$ and $L$.
For any such $\tilde{s}$, write $\glsuseri*{Ssf}=\Ss \setminus \{\tilde{s}\}$ and let $\Wf \leq \W$ be the subgroup generated by $\Ssf$.
\begin{enumerate}[label={\rm (\roman*)}]
\item There is a unique reflection $\gls*{sh} \in \Wf$ such that $\tilde{s} \sh$ is a translation.

\item Let
\begin{align*}
a_{\rm h}& =\sum_{s \in \Ssf} c_s a_s \in V_A \text{,} &
a_{\rm h}^\vee & =-a_{\tilde{s}}^\vee \in V_A^\ast \text{.}
\end{align*}
Then for all $b \in V_{A}$ we have
\begin{equation*}
\sh(b)=b-\langle a_{\rm h}^\vee, b \rangle a_{\rm h} \text{.}
\end{equation*}
\end{enumerate}
\end{prop}

\begin{proof}
Again, see \cite{humphreys} for the proof in extended Dynkin type, and Appendix~\ref{app:affrefl} for the general case.
\end{proof}

We fix a choice of such an $\tilde{s}$ once and for all, and take $\Ssf$ and $\Wf$ as above.
We call the generators in $\Ssf$ the \defnemph{finite generators} and $\tilde{s}$ the \defnemph{affine generator} of $\W$.
The subgroup $\Wf$ is finite, as it is isomorphic to the Weyl group of a (finite) root system with Cartan matrix $(a_{st})_{s,t \in \Ssf}$.
Thus $\Wf$ intersects $L$ trivially, so we have $\W=\Wf \ltimes L$.
Also $E$ has a unique $\Wf$-fixed point $o$, defined by
\begin{equation*}
\langle o, a_s \rangle=\delta_{\tilde{s},s} \text{.}
\end{equation*}
As an affine space, $E$ does not have a preferred choice of origin, but a useful convention in what follows is to set the origin of $E$ to be $o$.

Another consequence of Proposition~\ref{prop:tildes-defn} is the following computation, which will be useful in the next section. 
For $m \in \NN$, a straightforward induction argument gives
\begin{equation}
\begin{aligned}
(\tilde{s} s_{\rm h})^m(a_{\tilde{s}})& =2ma_{\rm h}+(2m+1)a_{\tilde{s}} \text{,} \label{eq:dihedralcalc1} \\
(s_{\rm h} \tilde{s})^m(a_{\rm h})& =(2m+1)a_{\rm h}+2ma_{\tilde{s}} \text{,}  \\
(\tilde{s} s_{\rm h})^{m-1} \tilde{s}(a_{\rm h})& =(2m-1)a_{\rm h}+2ma_{\tilde{s}} \text{,} \\
(s_{\rm h} \tilde{s})^{m-1} s_{\rm h}(a_{\tilde{s}})& =2ma_{\rm h}+(2m-1)a_{\tilde{s}} \text{.}
\end{aligned}
\end{equation}


%

\subsection{The Frobenius map}
\label{sec:frobmap}

Let $l$ be an odd positive integer coprime to $3$. 
(See Remark~\ref{rem:chrstcassump} for more detail regarding necessary assumptions on $l$.
In later sections we will restrict to the case where $l$ is prime.)
Recall that $\W = \Wf \ltimes L$ acts via affine reflections on $E$, and that translations in $L$ have the form $t(\lambda)$ for some $\lambda \in \Lambda$.

\begin{defn}
The \defnemph{Frobenius map} $\gls*{F}$ on $\W$ is defined to be the group homomorphism mapping
\begin{equation*}
\begin{aligned}
F:\W & \longrightarrow \W \\
w & \longmapsto w \\ 
t(\lambda) & \longmapsto t(l\lambda) 
\end{aligned}
\end{equation*}
for all $w \in \Wf$ and $\lambda \in \Lambda$. 
The image of $F$ is called the \defnemph{$l$-affine Weyl group} and is denoted by $\glsuseri*{Wl}$.
\end{defn}

The Frobenius map is well defined and injective because it corresponds to conjugation by the dilation map 
\begin{equation*}
o+\lambda \longmapsto o+l\lambda
\end{equation*}
on $E$, where $\lambda \in Z_{\RR}$.
As $F$ is injective it induces an isomorphism $\W \xrightarrow{\sim} \Wl$, so we can transfer the constructions in \S\ref{sec:affWnotn} to $\Wl$. 
Thus $\Wl$ is a Coxeter group with Coxeter generators $\glsuseri*{Ssl}=\Ssf \cup \{\glsuseri*{tildesl}\}$, where
\begin{equation}
\tilde{s}_l=F(\tilde{s})=F(\tilde{s} \sh) F(\sh)=(\tilde{s}\sh)^l \sh=\underbrace{\tilde{s}\sh\tilde{s}\sh \dotsm \tilde{s}}_{\text{$2l-1$ terms}} \text{.} \label{eq:tildesp}
\end{equation}
In particular the isomorphism $\W \xrightarrow{\sim} \Wl$ induced by $F$ is an isomorphism of Coxeter groups.
We write $\len_l : \Wl \rightarrow \Nneg$ for the length function on $\Wl$ in terms of its Coxeter generators $\Ssl$, and $\Sslexpr$ for the set of expressions in $\Ssl$.

Let
\begin{align}
\glsuseri*{atildesl}& =(\tilde{s}\sh)^{(l-1)/2}(a_{\tilde{s}})=(l-1)a_{\rm h}+la_{\tilde{s}} \in V_{A} \text{,} \\
a_{\tilde{s}_l}^\vee& =(\tilde{s}\sh)^{(l-1)/2}(a_{\tilde{s}}^\vee)=(l-1)a_{\rm h}^\vee+la_{\tilde{s}}^\vee \in V_{A}^\ast \text{.}
\end{align}
(Here we have used \eqref{eq:dihedralcalc1} and its dual variant.)
Since $\tilde{s}_l$ is a conjugate of $\tilde{s}$, by \eqref{eq:nonsimplereflections} it acts on $V_{A}$ via the formula
\begin{equation*}
\tilde{s}_l(b)=b-\langle a_{\tilde{s}_l}^\vee, b\rangle a_{\tilde{s}_l} \text{.}
\end{equation*}
We can use $a_{\tilde{s}_l}$ and $a_{\tilde{s}_l}^\vee$ to define a new realization.

\begin{defn} \label{defn:Ftwist}
Let $\glsuseri*{VAF}$ denote the \defnemph{$F$-twist} of the $\W$-representation $V_{A}$; in other words, as a $\ZZ$-module $V_{A}^F$ is $V_{A}$, but the $\W$-action is given by $w \cdot_F v=F(w)v$ for all $w \in \W$ and $v \in V_{A}^F$. 
We call the realization $(V_{A}^F,\{a_{F(s)}\},\{a_{F(s)}^\vee\})$ the \defnemph{$F$-twist} of the universal realization $(V_{A},\{a_s\},\{a_s^\vee\})$. 
\end{defn}

\begin{lem}
The Cartan matrix of the realization $V_{A}^F$ is $A$.
\end{lem}

\begin{proof}
We have 
\begin{equation*}
\langle a_{\tilde{s}_l}^\vee, a_{\tilde{s}_l} \rangle=2(l-1)^2+2(-2)l(l-1)+2l^2=2 \text{,}
\end{equation*}
while for any $s \in \Ssf$ we have
\begin{equation*}
\langle a_s^\vee, a_{\tilde{s}_l} \rangle=(l-1)\langle a_s^\vee, a_{\rm h} \rangle+l\langle a_s^\vee,a_{\tilde{s}} \rangle=-(l-1)a_{s\tilde{s}}+la_{s\tilde{s}}=a_{s\tilde{s}}
\end{equation*}
and
\begin{equation*}
\langle a_{\tilde{s}_l}^\vee, a_s \rangle=(l-1)\langle a_{\rm h}^\vee, a_s \rangle+l\langle a_{\tilde{s}}^\vee, a_s \rangle=-(l-1)a_{\tilde{s}s}+la_{\tilde{s}s}=a_{\tilde{s}s} \text{.} \qedhere
\end{equation*} 
\end{proof}

The Frobenius map also gives an isomorphism $\Wl \iso \Wf \ltimes lL$ of affine reflection groups. 
We define the \defnemph{fundamental $l$-alcove} to be the open subset
\begin{equation}
\glsuseri*{C}_l=\{b \in E : \langle b, a_s \rangle > 0 \text{ for all $s \in \Ssl$}\} \subset E \text{,}
\end{equation}
and call the images of $C_l$ under the action of $\Wl$ \defnemph{$l$-alcoves}.
Let $\glsuseri*{alcsl}$ be the set of all $p$-alcoves in $E$.
As with $\W$, the group $\Wl$ acts simply transitively on $\alcsl$, giving a bijection $\Wl \xrightarrow{\sim} \alcsl$ and a right $\Wl$-action on $\alcsl$, and we may identify each wall of an $l$-alcove with some $s \in \Ssl$.
Note that
\begin{equation}
\langle C, a_{\tilde{s}_l} \rangle=\langle C, (l-1)c+a_{\tilde{s}}\rangle>l-1 \geq 0 \text{,} \label{eq:atildesppositive}
\end{equation}
which in particular means that $\langle C, a_s \rangle > 0$ for all $s \in \Ssl$ and thus $C \subseteq C_l$.

\begin{exam}
Suppose $\W$ is of type $\widetilde{A_2}$ and $l=3$.
The alcove geometry for $\W$ is given in Figure~\ref{fig:lalcgeom}.
Here the small triangles are ordinary alcoves, and the large triangles are $l$-alcoves. 
The fundamental alcove $C$ is colored dark gray, and the fundamental $l$-alcove $C_l$ is colored light gray.
\end{exam}

\begin{figure}
\begin{center}
   \begin{tikzpicture}[scale=1]
 
  \tikzstyle{axes}=[very thick,>=stealth]
  \tikzstyle{grid lines}=[thin]
  \tikzstyle{alcoves}=[thick,color=blue!80]
  
   \path  (0,0)  coordinate (origin); 
  
   


   \path[fill,gray!30](origin) --++(60:3*0.6) --++(180:3*0.6) --++(300:3*0.6) --cycle;
   \path[fill,gray!75](origin) --++(60:0.6) --++(180:0.6) --++(300:0.6) --cycle;

       \begin{scope}

      \clip(0,0)-- (120:0.6*5.5)--++(0:0.6*5.5)--(0,0); 
        \foreach \i in {0,...,35}
  {
    \path (origin)++(60:0.6*\i cm)  coordinate (a\i);
    \path (origin)++(120:0.6*\i cm)  coordinate (b\i);
    \path (a\i)++(120:18cm) coordinate (ca\i);
    \path (b\i)++(60:18cm) coordinate (cb\i);
}
   \path (b20)++(60:0.6*5.5)  coordinate (step1);
      \path (step1)++(0:0.6*5.5)  coordinate (step2);
      \path (step2)++(-60:0.6*5.5)  coordinate (step3); 
      \path (step3)++(0:0.6*5.5)  coordinate (step4);       
      \path (step4)++(-60:0.6*5.5)  coordinate (step5);

  \foreach \i in {0,...,35}    
  { \draw[grid lines]  (a\i) -- (ca\i)  (b\i) -- (cb\i);
     \draw[grid lines]   (a\i) -- (b\i)  ; } 

 \foreach \i in {0,3,6,...,24}    
  { 
  \draw[ultra thick]  (a\i) -- (ca\i);  
  \draw[ultra thick] (b\i) -- (cb\i);
     \draw[ultra thick]  (a\i) -- (b\i)  ; 
     } 
 
   \end{scope}

       \begin{scope}

       \clip(0,0)-- (120:0.6*5.5)--++(-120:0.6*5.5)--(0,0); 
        \foreach \i in {0,...,35}
  {
    \path (origin)++(180:0.6*\i cm)  coordinate (a\i);
    \path (origin)++(120:0.6*\i cm)  coordinate (b\i);
    \path (a\i)++(120:18cm) coordinate (ca\i);
    \path (b\i)++(180:18cm) coordinate (cb\i);
}
   \path (b20)++(180:0.6*5)  coordinate (step1);
      \path (step1)++(0:0.6*5)  coordinate (step2);
      \path (step2)++(-60:0.6*5)  coordinate (step3); 
      \path (step3)++(0:0.6*5)  coordinate (step4);       
      \path (step4)++(-60:0.6*5)  coordinate (step5);

  \foreach \i in {0,...,35}    
  { \draw[grid lines]  (a\i) -- (ca\i)  (b\i) -- (cb\i);
     \draw[grid lines]   (a\i) -- (b\i)  ; } 

 \foreach \i in {0,3,6,...,24}    
  { 
  \draw[ultra thick]  (a\i) -- (ca\i)  (b\i) -- (cb\i);
     \draw[ultra thick]  (a\i) -- (b\i)  ; 
     } 
 
   \end{scope}

       \begin{scope}

       \clip(0,0)-- (-120:0.6*5.5)--++(120:0.6*5.5)--(0,0); 
        \foreach \i in {0,...,35}
  {
    \path (origin)++(180:0.6*\i cm)  coordinate (a\i);
    \path (origin)++(-120:0.6*\i cm)  coordinate (b\i);
    \path (a\i)++(-120:18cm) coordinate (ca\i);
    \path (b\i)++(180:18cm) coordinate (cb\i);
}
   \path (b20)++(180:0.6*5)  coordinate (step1);
      \path (step1)++(0:0.6*5)  coordinate (step2);
      \path (step2)++(-60:0.6*5)  coordinate (step3); 
      \path (step3)++(0:0.6*5)  coordinate (step4);       
      \path (step4)++(-60:0.6*5)  coordinate (step5);

  \foreach \i in {0,...,35}    
  { \draw[grid lines]  (a\i) -- (ca\i)  (b\i) -- (cb\i);
     \draw[grid lines]   (a\i) -- (b\i)  ; } 

 \foreach \i in {0,3,6,...,24}    
  { 
  \draw[ultra thick]  (a\i) -- (ca\i)  (b\i) -- (cb\i);
     \draw[ultra thick]  (a\i) -- (b\i)  ; 
     } 
 
   \end{scope}

       \begin{scope}

       \clip(0,0)-- (-120:0.6*5.5)--++(0:0.6*5.5)--(0,0); 
        \foreach \i in {0,...,35}
  {
    \path (origin)++(-60:0.6*\i cm)  coordinate (a\i);
    \path (origin)++(-120:0.6*\i cm)  coordinate (b\i);
    \path (a\i)++(-120:18cm) coordinate (ca\i);
    \path (b\i)++(-60:18cm) coordinate (cb\i);
}
   \path (b20)++(-60:0.6*5)  coordinate (step1);
      \path (step1)++(0:0.6*5)  coordinate (step2);
      \path (step2)++(-60:0.6*5)  coordinate (step3); 
      \path (step3)++(0:0.6*5)  coordinate (step4);       
      \path (step4)++(-60:0.6*5)  coordinate (step5);

  \foreach \i in {0,...,35}    
  { \draw[grid lines]  (a\i) -- (ca\i)  (b\i) -- (cb\i);
     \draw[grid lines]   (a\i) -- (b\i)  ; } 

 \foreach \i in {0,3,6,...,24}    
  { 
  \draw[ultra thick]  (a\i) -- (ca\i)  (b\i) -- (cb\i);
     \draw[ultra thick]  (a\i) -- (b\i)  ; 
     } 
 
   \end{scope}

   \begin{scope}

       \clip(0,0)-- (-60:0.6*5.5)--++(60:0.6*5.5)--(0,0); 
        \foreach \i in {0,...,35}
  {
    \path (origin)++(-60:0.6*\i cm)  coordinate (a\i);
    \path (origin)++(0:0.6*\i cm)  coordinate (b\i);
    \path (a\i)++(0:18cm) coordinate (ca\i);
    \path (b\i)++(-60:18cm) coordinate (cb\i);
}
   \path (b20)++(-60:0.6*5)  coordinate (step1);
      \path (step1)++(0:0.6*5)  coordinate (step2);
      \path (step2)++(-60:0.6*5)  coordinate (step3); 
      \path (step3)++(0:0.6*5)  coordinate (step4);       
      \path (step4)++(-60:0.6*5)  coordinate (step5);

  \foreach \i in {0,...,35}    
  { \draw[grid lines]  (a\i) -- (ca\i)  (b\i) -- (cb\i);
     \draw[grid lines]   (a\i) -- (b\i)  ; } 

 \foreach \i in {0,3,6,...,24}    
  { 
  \draw[ultra thick]  (a\i) -- (ca\i)  (b\i) -- (cb\i);
     \draw[ultra thick]  (a\i) -- (b\i)  ; 
     } 
 
   \end{scope}

   \begin{scope}

       \clip(0,0)-- (0:0.6*5.5)--++(120:0.6*5.5)--(0,0); 
        \foreach \i in {0,...,35}
  {
    \path (origin)++(60:0.6*\i cm)  coordinate (a\i);
    \path (origin)++(0:0.6*\i cm)  coordinate (b\i);
    \path (a\i)++(0:18cm) coordinate (ca\i);
    \path (b\i)++(60:18cm) coordinate (cb\i);
}
   \path (b20)++(60:0.6*5)  coordinate (step1);
      \path (step1)++(0:0.6*5)  coordinate (step2);
      \path (step2)++(-60:0.6*5)  coordinate (step3); 
      \path (step3)++(0:0.6*5)  coordinate (step4);       
      \path (step4)++(-60:0.6*5)  coordinate (step5);

  \foreach \i in {0,...,35}    
  { \draw[grid lines]  (a\i) -- (ca\i)  (b\i) -- (cb\i);
     \draw[grid lines]   (a\i) -- (b\i)  ; } 

 \foreach \i in {0,3,6,...,24}    
  { 
  \draw[ultra thick]  (a\i) -- (ca\i)  (b\i) -- (cb\i);
     \draw[ultra thick]  (a\i) -- (b\i)  ; 
     } 
 
   \end{scope}


 \draw[ultra thick](origin) --++(0:0.6*5.5); 
 \draw[ultra thick](origin) --++(120:0.6*5.5); \draw[ultra thick](origin) --++(-120:0.6*5.5);
 \draw[ultra thick](origin) --++(60:0.6*5.5); \draw[ultra thick](origin) --++(-60:0.6*5.5);
 \draw[ultra thick](origin) --++(180:0.6*5.5);


\end{tikzpicture}
\end{center}
\caption{The $l$-alcoves for $l=3$ and $\W$ of type $\widetilde{A}_2$.}
\label{fig:lalcgeom}
\end{figure}

Let
\begin{equation}
\glsuseri*{Wlcosets}=\{w \in \W : \len(sw) > \len(w) \text{ for all $s \in \Ssl$}\} \subset \W \text{.}
\end{equation}
We can think of $\Wlcosets$ geometrically as follows. 
Let $\glsuseri*{alcslfun}$ denote the set of ordinary alcoves contained inside the fundamental $l$-alcove $C_l$.

\begin{lem}
The bijection $\W \xrightarrow{\sim} \alcs$ restricts to a bijection $\Wlcosets \xrightarrow{\sim} \alcslfun$.
\end{lem}

\begin{proof}
We have
\begin{align*}
w \in \Wlcosets & \Leftrightarrow \len(sw)>\len(w) \text{ for all $s \in \Ssl$,} \\
& \Leftrightarrow \len(w^{-1}s)>\len(w^{-1}) \text{ for all $s \in \Ssl$,} \\
& \Leftrightarrow \langle C, w^{-1}(a_s)\rangle>0 \text{ for all $s \in \Ssl$,} \\
& \Leftrightarrow \langle w(C), a_s \rangle>0 \text{ for all $s \in \Ssl$,} \\
& \Leftrightarrow w(C) \subseteq C_l
\end{align*}
by \eqref{eq:lengthposrootcriterion} and \eqref{eq:atildesppositive}.
\end{proof}

The elements $\Wlcosets \subset \W$ form a set of representatives for the right cosets $\Wl \backslash \W$.
A geometric way to see this is as follows.
Call two ordinary alcoves \defnemph{$l$-linked} if they lie in the same $\Wl$-orbit under the left action of $\Wl$ on $\alcs$.
In other words, for $x,y \in \W$ the alcoves $x(C)$ and $y(C)$ are $l$-linked if there is some $z \in \Wl$ such that $x=zy$.
Thus $l$-linkage classes correspond to right cosets $\Wl \backslash \W$ under the bijection $\W \xrightarrow{\sim} \alcs$.
A natural choice of representatives for the $l$-linkage classes of alcoves is $\alcslfun$, which we already know corresponds to $\Wlcosets$.

The right action of $\W$ on $\Wl \backslash \W$ induces a right action on $\Wlcosets$. 
Geometrically this corresponds to the right action of $\W$ on the $l$-linkage class representatives in $\alcslfun$.
For each $s \in \Ss$ and $w(C) \in \alcspfun$, $s$ sends $w(C)$ to $ws(C)$ under this action unless the $s$-wall of $w(C)$ is contained in a wall of $C_l$, in which case $w(C)$ is fixed by $s$.



We conclude this section with two results which will be necessary for the construction of the Frobenius functor in \S\ref{sec:frobfunctor}. 
In particular, the following crucial lemma ensures that the Frobenius functor is ``integral'' in a certain sense. 
This in turn ensures that the matrix recursion functor $\pr$ (which we construct in \S \ref{sec:prconstruct}) is well defined.

\begin{lem} \label{lem:rlzcoefs}
Let $w \in \Wlcosets$ and $s \in \Ss$. The coefficient of $a_{\tilde{s}}$ in $w(a_s) \in V_{A}$ is a multiple of $l$ if and only if $\Wl ws=\Wl w$. When this happens we have $wsw^{-1} \in \Ssl$ and $w(a_s)=a_{wsw^{-1}}$.
\end{lem}

\begin{proof}
When $\Wl ws=\Wl w$, the alcoves $w(C)$ and $ws(C)$ lie in the same $l$-linkage class, and the wall between them is contained in a wall of $C_{l}$, corresponding to some $t \in \Ssl$. 
This means that $wsw^{-1}=t$ and that $w$ maps the reflection hyperplane of $s$ to the reflection hyperplane of $t$. 
By \eqref{eq:nonsimplereflections} we must have $w(a_s)=ka_t$ and $w(a_s^\vee)=k^{-1} a_t^\vee$ for some non-zero scalar $k$. 
Since both $k$ and $k^{-1}$ are integers we must have $k\in \{\pm 1\}$. 
Finally $k>0$ by \eqref{eq:atildesppositive} so $k=1$. 

Conversely, if $l$ divides the coefficient of $a_{\tilde{s}}$ in $w(a_s)$, then $wsw^{-1}$ maps the origin $o \in E$ to
\begin{equation*}
wsw^{-1}(o)=o-\langle o, w(a_s)\rangle w(a_s^\vee) \in o+lZ
\end{equation*}
by \eqref{eq:nonsimplereflections}, where $Z$ is defined as in Proposition~\ref{prop:translation-index}.
If we write $wsw^{-1}$ as a product $t(\lambda)x$ where $\lambda \in \Lambda$ and $x \in \Wf$, this shows that $\lambda \in lZ \cap \Lambda$.
Since $l$ is coprime to $[Z:\Lambda]$ we can infer that $\lambda \in l\Lambda$, and thus $wsw^{-1} \in \Wp$ and $\Wl ws=\Wl w$.
\end{proof}

\begin{lem} \label{lem:conjbraids}
Let $s \in \Ssf$ with $m_{\tilde{s}s}<\infty$. 
There exist $\tilde{s}',s' \in \Ss$ and $w \in \Wlcosets$ such that $m_{\tilde{s}'s'}=m_{\tilde{s}s}$,  $w\tilde{s}'w^{-1}=\tilde{s}_l$ and $ws'w^{-1}=s$.
\end{lem}

Informally, this means that any braid relation in $\Wl$ between $\tilde{s}_l$ and some finite generator $s \in \Ssf$ is conjugate to some braid relation in $\W$.

\begin{proof}
The dihedral reflection subgroup $D$ generated by $\tilde{s}_l$ and $s$ is contained inside a maximal dihedral reflection subgroup $D'$.
Applying the work of Deodhar \cite{deodhar-reflsubgroups}, the order of $D'$ can be determined by counting
\begin{equation*}
\{w(a_{s'}) : w \in W,\ s' \in S\} \cap (\mathbb{R}a_{\tilde{s}_l} \oplus \mathbb{R} a_s) \subset V_{A,\mathbb{R}} \text{.}
\end{equation*}
In particular, this subset is a root system for $D'$ in the sense of Humphreys \cite[Chapter 1.2]{humphreys}, and the simple roots give generators of $D'$.
So $D<D'$ if and only if there is some $w \in W$, $s' \in S$ and $k_{\tilde{s}_l},k_s \in \mathbb{R}$ with different signs such that
\begin{equation*}
w(a_{s'})=k_{\tilde{s}_l} a_{\tilde{s}_l}+k_s a_s \text{.}
\end{equation*}
The right-hand side is equal to
\begin{equation*}
k_{\tilde{s}_l}(l-1)a_{\rm h}+k_s a_s+k_{\tilde{s}_l} la_{\tilde{s}}
\end{equation*}
The coefficients of $w(a_{s'})$ with respect to $\{a_t\}_{t \in S}$ are all integers with the same sign, so without loss of generality we may assume they are all non-negative, and hence $k_{\tilde{s}_l} l \in \NN$. 

Now suppose $k_{\tilde{s}_l}$ is an integer. 
Writing $w=xz$ with $x \in \Wl$ and $z \in \Wlcosets$, it is clear from the action of $\Wl$ that the coefficient of $a_{\tilde{s}}$ in $z(a_{s'})$ is also a multiple of $l$. 
By Lemma~\ref{lem:rlzcoefs} this means $z(a_{s'})=a_t$ for some $t \in \Ssl$, and thus $k_{s}$ and $k_{\tilde{s}_l}$ are both non-negative, which is a contradiction.
This means $k_{\tilde{s}_l}$ cannot be an integer.

In addition, the coefficients of 
\begin{equation*}
k_{\tilde{s}_l}(l-1)a_{\rm h}+k_s a_s
\end{equation*}
with respect to $\{a_t\}_{t \in \Ssf}$ must be non-negative integers. 
Clearing denominators, this means that the coefficients of $a_{\rm h}$ with respect to $\{a_t : t \in \Ssf \setminus\{s\}\}$ must have some non-trivial common divisor with $l$. 
But by Remark~\ref{rem:affrootsystems} these coefficients equal the coefficients of the highest root or the highest short root with respect to the simple roots in some finite root system $\Phif$. 
We can check case-by-case (e.g.~using \cite[Chapter 4, Table 1]{humphreys}) for each highest root whether all but one of the coefficients have a non-trivial common divisor. 
The only common divisors we get in this manner are $2$ and $3$. 
Since $l$ is coprime to both $2$ and $3$, we conclude that $D=D'$.

To complete the proof, we consider the affine action of $\W$ on $E$.
As $m_{\tilde{s}s}=m_{\tilde{s}_l s}<\infty$, the $\tilde{s}_l$- and $s$-walls of the fundamental $l$-alcove $C_l$ must intersect. 
Let $w \in \Wlcosets$ be the label for an alcove in $\alcslfun$ whose closure intersects with the intersection of these walls.
Since $D'=D$, there exist $\tilde{s}',s' \in \Ss$ such that the $\tilde{s}'$- and $s'$-walls of $w(C)$ lie on the $\tilde{s}_l$- and $s$-walls of $C_l$, and the result follows.
\end{proof}

\begin{rem} \label{rem:chrstcassump}
Our assumptions on $l$ are solely to ensure that Lemma~\ref{lem:rlzcoefs} and Lemma~\ref{lem:conjbraids} both hold in full generality.
If $A$ is of simply laced type, then we do not in fact need \emph{any} assumptions on $l$ for these results to hold.
Outside of type $G$ it is enough to assume that $l$ is odd.
If $A$ is of type $\widetilde{G_2}$ or $\widetilde{G_2}^\vee$, i.e.~if $A$ is
\begin{equation*}
\begin{bmatrix}
2 & -3 & 0 \\
-1 & 2 & -1 \\
0 & -1 & 2
\end{bmatrix}
\end{equation*}
or its transpose up to reordering of indices, then it is necessary to assume that $2,3 \nmid l$.
Although we believe it is possible to extend our main results to even $l$ in the simply laced case, for notational convenience we will always assume that $l$ is odd in what follows. 
\end{rem}

\section{Hecke algebras}


\subsection{Ordinary Hecke algebras}
\label{sec:ordhecke}


The \defnemph{Hecke algebra} $\glsuseri*{hecke}=\hecke(\W,\Ss)$ of the affine Weyl group $\W$ is the $\laur$-algebra with generators $\{h_s\}_{s \in \Ss}$ and relations
\begin{align}
h_s^2 & =1+(v^{-1}-v)h_s && \text{for all $s \in \Ss$,} \\
\underbrace{h_s h_t h_s \dotsm}_{\text{$m_{st}$ terms}} & =\underbrace{h_t h_s h_t \dotsm}_{\text{$m_{st}$ terms}}&& \text{for all distinct $s,t \in \Ss$ where $m_{st} \neq \infty$,}
\end{align}
where $m_{st}$ is defined as in \S\ref{sec:affWnotn}.

If $x \in \W$ and $\expr{x}=\expr{s_1 s_2 \dotsm s_m}$ is a rex for $x$, the element $\glsuseri*{hx}=h_{s_1} h_{s_2} \dotsm h_{s_m}$ is well defined, and the set $\{h_x : x \in \W\}$ forms a $\laur$-basis for $\hecke$, called the \defnemph{standard basis}. 
Each generator $h_s$ is invertible, with $h_s^{-1}=h_s+v-v^{-1}$, so each basis element $h_x$ is also invertible. 
The \defnemph{bar involution} or \defnemph{dualization map} $(\overline{\phantom{h}}):\hecke \longrightarrow \hecke$ is the ring homomorphism defined by the following action
\begin{equation}
\begin{split}
\overline{v} & = v^{-1} \text{,} \\
\overline{h_x} & = (h_{x^{-1}})^{-1} \text{.}
\end{split}
\end{equation}
For $s \in \Ss$ we define $\glsuseri*{bs}=h_s+v$, which is self-dual. 
The set $\{b_s\}_{s \in \Ss}$ forms another set of generators for $\hecke$ as a $\laur$-algebra. 

The right action of the self-dual generators on $\hecke$ with respect to the standard basis has a particularly nice form.
Namely, for any $x \in \W$ and any $s \in \Ss$, we have
\begin{equation}
h_x b_s=\begin{cases}
h_{xs}+vh_x & \text{if $xs>x$,} \\
h_{xs}+v^{-1}h_x & \text{if $xs<x$.}
\end{cases} \label{eq:hxbsaction}
\end{equation}
There is a common visual schematic for elements of $\hecke$ using alcoves (e.g.~\cite[\S 8]{soergel-KL}). 
Since the set of alcoves is in bijection with $\W$, we can visualize an element $\sum_{x} k_x h_x \in \hecke$ as the assignment of each non-zero coefficient $k_x$ to its corresponding alcove $x(C)$.
From this perspective, \eqref{eq:hxbsaction} looks like the pair of local relations
\begin{equation*}
\begin{tikzpicture}[scale=0.75,baseline=(origin)]
\node[above right=-0.1,font=\footnotesize,red] at (300:0.5) {$s$};
\draw (300:2) -- (120:2);
\draw (120:1) -- ++(240:2) -- ++(0:2);
\draw (120:1) -- ++(0:2) -- ++(240:2);
\draw[red] (300:1) -- (120:1);
\node at (210:0.577) {$1$};
\coordinate (origin) at (0,0);
\end{tikzpicture}
\cdot b_{{\color{red} s}}=
\begin{tikzpicture}[scale=0.75,baseline=(origin)]
\node[above right=-0.1,font=\footnotesize,red] at (300:0.5) {$s$};
\draw (300:2) -- (120:2);
\draw (120:1) -- ++(240:2) -- ++(0:2);
\draw (120:1) -- ++(0:2) -- ++(240:2);
\draw[red] (300:1) -- (120:1);
\node at (210:0.577) {$v$};
\node at (30:0.577) {$1$};
\coordinate (origin) at (0,0);
\end{tikzpicture}
\end{equation*}
and
\begin{equation*}
\begin{tikzpicture}[scale=0.75,baseline=(origin)]
\node[above right=-0.1,font=\footnotesize,red] at (300:0.5) {$s$};
\draw (300:2) -- (120:2);
\draw (120:1) -- ++(240:2) -- ++(0:2);
\draw (120:1) -- ++(0:2) -- ++(240:2);
\draw[red] (300:1) -- (120:1);
\node at (30:0.577) {$1$};
\coordinate (origin) at (0,0);
\end{tikzpicture}
\cdot b_{{\color{red} s}}=
\begin{tikzpicture}[scale=0.75,baseline=(origin)]
\node[above right=-0.1,font=\footnotesize,red] at (300:0.5) {$s$};
\draw (300:2) -- (120:2);
\draw (120:1) -- ++(240:2) -- ++(0:2);
\draw (120:1) -- ++(0:2) -- ++(240:2);
\draw[red] (300:1) -- (120:1);
\node at (210:0.577) {$1$};
\node at (30:0.577) {$v^{-1}$};
\coordinate (origin) at (0,0);
\end{tikzpicture} \text{,}
\end{equation*}
where length is increasing in the northeast direction.

For any $\expr{x}=\expr{s_1 s_2\dotsm s_m} \in \expr{\Ss}$ let $\glsuseri*{bexprx}=b_{s_1} b_{s_2} \dotsm b_{s_m}$. 
Clearly $b_{\expr{x}}$ is also self-dual.
There is a formula which gives $b_{\expr{x}}$ in terms of the standard basis on $\hecke$.
This requires some more notation.

\begin{notn}
Let $\expr{x}=\expr{s_1 s_2 \dotsm s_m}$ be an expression. 
A \defnemph{subsequence} for $\expr{x}$ is a sequence of the form $\glsuseri*{seqe}=(\seq{e}_1, \dotsc, \seq{e}_m)$, where each $\seq{e}_i$ is an ordered pair $(s_i,t_i)$ with $t_i \in \{0,1\}$. 
We say that $\seq{e}_i$ is a \defnemph{term} with \defnemph{generator} $s_i$ and \defnemph{type} $t_i$, and we refer to the type of $\seq{e}$ to mean the sequence of types of the terms $\seq{e}_i$. 
We write $\hat{e}$ to denote the group element $s_1^{t_1} s_2^{t_2} \dotsm s_m^{t_m} \in \W$. 
We denote the set of all subsequences for $\expr{x}$ by $\glsuseri*{bexprxb}$. 
\end{notn}

Every subsequence $\seq{e}$ can be assigned an integer $\glsuseri*{dfct}(\seq{e})$ called the \defnemph{defect}. 
The defect $\dfct(\seq{e})$ is based on a sequence of elements in $\W$ called the \defnemph{Bruhat stroll} for $\seq{e}$. 
We delay the definition of the Bruhat stroll and the defect until \S\ref{sec:patterns}, where we provide a more general construction.

\begin{lem}[Deodhar's defect formula {\cite[Proposition~3.5]{deodhar-combi}}] \label{lem:ddfhecke}
Let $\expr{x} \in \expr{\Ss}$. 
Then
\begin{equation*}
b_{\expr{x}}=\sum_{\seq{e} \in [\expr{x}]} v^{\dfct(\seq{e})} h_{\hat{e}} \text{.}
\end{equation*}
\end{lem}

Let $\heckel=\hecke(\Wl,\Ssl)$ denote the Hecke algebra of the $l$-affine Weyl group. 
To avoid confusion with the generators of $\hecke$, we write $h^{(l)}_x$ and $b^{(l)}_s$ to denote the standard basis elements and self-dual generators in $\heckel$. 
Note that $\heckel$ is isomorphic to $\hecke$ via the $\laur$-algebra homomorphism mapping $h_s \mapsto h^{(l)}_{F(s)}$ for all $s \in \Ss$. 
Thus the constructions in this section apply equally well to $\heckel$.

\subsection{\texorpdfstring{$\heckelast$}{H_{l|*}} and \texorpdfstring{$\heckeast$}{H_*}}

\label{sec:heckeast-heckepast}


Recall that $\Wlcosets \subset \W$ is a set of representatives for the right cosets $\Wl \backslash \W$.
The map
\begin{align*}
\Wl \times \Wlcosets & \longrightarrow \W \\
(x,w) & \longmapsto xw
\end{align*}
is a bijection, so any objects indexed over $\Wl \times \Wlcosets$ can just as easily be indexed over $\W$ and vice versa. 
We will sometimes write products $xw$ of the above form with a bar in the form $\glspl*{xbarw}$, in order to emphasize this correspondence and to separately keep track of the $\Wl$- and $\Wlcosets$-parts. 

Extending the map above $\ZZ$-linearly gives an isomorphism
\begin{equation}
\bigoplus_{w \in \Wlcosets} \ZZ\Wl w \xrightarrow{\sim} \ZZ\W \label{eq:Wldecomp}
\end{equation}
of left $\ZZ\Wl$-modules. 
On the right $\ZZ\W$ acts faithfully via $\ZZ\Wl$-module endomorphisms, which we can write as a matrix representation $\ZZ\W \rightarrow M_{|\Wlcosets|}(\ZZ\Wl)$.
As in the introduction we define the \defnemph{matrix recursion representation} $\gls*{xi}$ to be the composition
\begin{equation*}
\ZZ\W \longrightarrow M_{|\Wlcosets|}(\ZZ\Wl) \xrightarrow{F^{-1}} M_{|\Wlcosets|}(\ZZ\W) \text{,} 
\end{equation*}
where the second arrow denotes applying $F^{-1}$ to the entries of the matrix. 
It is easy to show that
\begin{equation} \label{eq:xientries}
\xi(s)_{w,z}=\begin{cases}
F^{-1}(wsw^{-1}) & \text{if $\Wl ws=\Wl w$ and $w=z$,} \\
1 & \text{if $\Wl ws \neq \Wl w$ and $ws=z$,} \\
0 & \text{otherwise}
\end{cases}
\end{equation}
for all $w,z \in \Wlcosets$ and $s \in \Ss$. 

In this section we construct Hecke analogues of \eqref{eq:Wldecomp} and \eqref{eq:xientries} via a new bimodule $\heckelast$ and a new algebra $\heckeast$.
Later when we categorify, the $\laur$-module structures on $\heckelast$ and $\heckeast$ will correspond to a \defnemph{valuation} on the diagrammatic Hecke category, distinct from the usual grading. 
Unfortunately, constructing $\heckelast$ and $\heckeast$ is not immediately straightforward as $\heckel$ is not even a subalgebra of $\hecke$! 
To begin with, if we want a free action of $\heckel$ we must build one ourselves. 

\begin{defn}
We define $\glsuseri*{heckelast}$ to be the free left $\heckel$-module with free basis
\begin{equation*}
\{h_{w}^{(l|\ast)} : w \in \Wlcosets\} \text{.}
\end{equation*}
For each $x \in \Wl$ and $w \in \Wlcosets$, we write $h^{(l|\ast)}_{x|w}=h^{(l)}_x h^{(l|\ast)}_w$.
The set 
\begin{equation*}
\{\glsuseri*{hlast}_{x|w} : x|w \in \W\}
\end{equation*}
forms a $\laur$-basis for $\heckelast$, which we call the \defnemph{standard basis}.
%
\end{defn}

Here the decomposition
\begin{equation}
\bigoplus_{w \in \Wlcosets} \heckel h^{(l|\ast)}_{w} \xrightarrow{\sim} \heckelast \text{,} \label{eq:heckelastdecomp}
\end{equation}
which arises by definition, is analogous to \eqref{eq:Wldecomp}.



We introduce new operators on $\heckelast$ using the right $\W$-action on the standard basis elements.
For each $s \in \Ss$, we define an operator $\glsuseri*{basts}$ by
\begin{equation}
h^{(l|\ast)}_{x|w} \cdot b^{(\ast)}_s=\begin{cases}
h^{(l|\ast)}_{xwsw^{-1}|w} + v h^{(l|\ast)}_{x|w} & \text{if $\Wl ws=\Wl w$ and $xwsw^{-1}>x$,} \\
h^{(l|\ast)}_{xwsw^{-1}|w} + v^{-1} h^{(l|\ast)}_{x|w} & \text{if $\Wl ws=\Wl w$ and $xwsw^{-1}<x$,} \\
h^{(l|\ast)}_{x|ws} + h^{(l|\ast)}_{x|w} & \text{otherwise}
\end{cases} \label{eq:hpastxwbasts}
\end{equation}
for all $x \in \Wl$ and $w \in \Wlcosets$.
For any $\expr{x}=\expr{s_1 s_2 \dotsm s_m} \in \expr{\Ss}$ we also write $\glsuseri*{bexprxast}=b^{(\ast)}_{s_1} b^{(\ast)}_{s_2} \dotsm b^{(\ast)}_{s_m}$.

For each $z \in \Wlcosets$, we also define an operator $\glsuseri*{uz}$ by
\begin{equation}
h^{(l|\ast)}_{x|w} \cdot u_z=\begin{cases}
vh^{(l|\ast)}_{x|w} & \text{if $w=z$,} \\
h^{(l|\ast)}_{x|w} & \text{otherwise.} 
\end{cases} \label{eq:uzaction}
\end{equation}
We will abuse notation slightly and write $u_{zs}$ for $u_{z'}$, where $z' \in \Wlcosets$ such that $\Wl z'=\Wl zs$.

As with $\hecke$, we can depict an element $\sum_{x|w} k_{x|w} h^{(l|\ast)}_{x|w} \in \heckelast$ by assigning each non-zero coefficient $k_{x|w}$ to its corresponding alcove $xw(C)$.
Visualizing \eqref{eq:hpastxwbasts} in this way gives the local relations
\begin{equation*}
\begin{tikzpicture}[scale=0.6,baseline=(origin)]
\draw[ultra thick] (300:4) -- (120:4);
\draw[ultra thick] (120:3) -- ++(240:6) -- ++(0:6);
\draw[ultra thick] (120:3) -- ++(0:6) -- ++(240:6);
\path (120:3) ++(240:2) edge[draw] +(0:6) edge[draw] +(300:4);
\path (120:3) ++(240:4) edge[draw] +(0:6) edge[draw] +(300:2);
\path (120:3) ++(0:2) edge[draw] +(240:6) edge[draw] +(300:4);
\path (120:3) ++(0:4) edge[draw] +(240:6) edge[draw] +(300:2);
\path (120:0) ++(210:0.577) node {$1$};
\draw[ultra thick, red] (120:-1) -- (120:1);
\node[above right=-0.1,font=\footnotesize, red] at (120:-0.5) {$s$};
\coordinate (origin) at (0,0);
\end{tikzpicture}
\cdot b^{(\ast)}_{{\color{red} s}}=
\begin{tikzpicture}[scale=0.6,baseline=(origin)]
\draw[ultra thick] (300:4) -- (120:4);
\draw[ultra thick] (120:3) -- ++(240:6) -- ++(0:6);
\draw[ultra thick] (120:3) -- ++(0:6) -- ++(240:6);
\path (120:3) ++(240:2) edge[draw] +(0:6) edge[draw] +(300:4);
\path (120:3) ++(240:4) edge[draw] +(0:6) edge[draw] +(300:2);
\path (120:3) ++(0:2) edge[draw] +(240:6) edge[draw] +(300:4);
\path (120:3) ++(0:4) edge[draw] +(240:6) edge[draw] +(300:2);
\path (120:0) ++(210:0.577) node {$v$};
\path (120:0) ++ (30:0.577) node {$1$};
\draw[ultra thick, red] (120:-1) -- (120:1);
\node[above right=-0.1,font=\footnotesize, red] at (120:-0.5) {$s$};
\coordinate (origin) at (0,0);
\end{tikzpicture} \text{,}
\end{equation*}
\begin{equation*}
\begin{tikzpicture}[scale=0.6,baseline=(origin)]
\draw[ultra thick] (300:4) -- (120:4);
\draw[ultra thick] (120:3) -- ++(240:6) -- ++(0:6);
\draw[ultra thick] (120:3) -- ++(0:6) -- ++(240:6);
\path (120:3) ++(240:2) edge[draw] +(0:6) edge[draw] +(300:4);
\path (120:3) ++(240:4) edge[draw] +(0:6) edge[draw] +(300:2);
\path (120:3) ++(0:2) edge[draw] +(240:6) edge[draw] +(300:4);
\path (120:3) ++(0:4) edge[draw] +(240:6) edge[draw] +(300:2);
\path (120:0) ++(30:0.577) node {$1$};
\draw[ultra thick, red] (120:-1) -- (120:1);
\node[above right=-0.1,font=\footnotesize, red] at (120:-0.5) {$s$};
\coordinate (origin) at (0,0);
\end{tikzpicture}
\cdot b^{(\ast)}_{{\color{red} s}}=
\begin{tikzpicture}[scale=0.6,baseline=(origin)]
\draw[ultra thick] (300:4) -- (120:4);
\draw[ultra thick] (120:3) -- ++(240:6) -- ++(0:6);
\draw[ultra thick] (120:3) -- ++(0:6) -- ++(240:6);
\path (120:3) ++(240:2) edge[draw] +(0:6) edge[draw] +(300:4);
\path (120:3) ++(240:4) edge[draw] +(0:6) edge[draw] +(300:2);
\path (120:3) ++(0:2) edge[draw] +(240:6) edge[draw] +(300:4);
\path (120:3) ++(0:4) edge[draw] +(240:6) edge[draw] +(300:2);
\path (120:0) ++(210:0.577) node {$1$};
\path (120:0) ++ (30:0.577) node {$v^{-1}$};
\draw[ultra thick, red] (120:-1) -- (120:1);
\node[above right=-0.1,font=\footnotesize, red] at (120:-0.5) {$s$};
\coordinate (origin) at (0,0);
\end{tikzpicture} \text{,}
\end{equation*}
and
\begin{equation*}
\begin{tikzpicture}[scale=0.6,baseline=(origin)]
\draw[ultra thick] (300:4) -- (120:4);
\draw[ultra thick] (120:3) -- ++(240:6) -- ++(0:6);
\draw[ultra thick] (120:3) -- ++(0:6) -- ++(240:6);
\path (120:3) ++(240:2) edge[draw] +(0:6) edge[draw] +(300:4);
\path (120:3) ++(240:4) edge[draw] +(0:6) edge[draw] +(300:2);
\path (120:3) ++(0:2) edge[draw] +(240:6) edge[draw] +(300:4);
\path (120:3) ++(0:4) edge[draw] +(240:6) edge[draw] +(300:2);
\path (120:1) ++(210:1.115) node {$1$};
\draw[red] (120:1) -- ++(180:2);
\path (120:1) ++(180:0.5) node[above=-0.05,font=\footnotesize, red] {$s$};
\coordinate (origin) at (0,0);
\end{tikzpicture}
\cdot b^{(\ast)}_{{\color{red} s}}=
\begin{tikzpicture}[scale=0.6,baseline=(origin)]
\draw[ultra thick] (300:4) -- (120:4);
\draw[ultra thick] (120:3) -- ++(240:6) -- ++(0:6);
\draw[ultra thick] (120:3) -- ++(0:6) -- ++(240:6);
\path (120:3) ++(240:2) edge[draw] +(0:6) edge[draw] +(300:4);
\path (120:3) ++(240:4) edge[draw] +(0:6) edge[draw] +(300:2);
\path (120:3) ++(0:2) edge[draw] +(240:6) edge[draw] +(300:4);
\path (120:3) ++(0:4) edge[draw] +(240:6) edge[draw] +(300:2);
\path (120:1) ++(210:1.115) node {$1$};
\path (120:2) ++(210:0.577) node {$1$};
\draw[red] (120:1) -- ++(180:2);
\path (120:1) ++(180:0.5) node[above=-0.05,font=\footnotesize, red] {$s$};
\coordinate (origin) at (0,0);
\end{tikzpicture} \text{,}
\end{equation*}
where the large alcoves are $l$-alcoves, and length is increasing in the northeast direction.
Similarly \eqref{eq:uzaction} looks like
\begin{equation*}
\begin{tikzpicture}[scale=0.6,baseline=(origin)]
\fill[red!10] (120:1) -- ++(180:2) -- ++(300:2) --cycle;
\path (120:1) ++(180:2) ++(300:1.5) node[above right=-0.1,font=\footnotesize,red] {$z$};
\draw[ultra thick] (300:4) -- (120:4);
\draw[ultra thick] (120:3) -- ++(240:6) -- ++(0:6);
\draw[ultra thick] (120:3) -- ++(0:6) -- ++(240:6);
\path (120:3) ++(240:2) edge[draw] +(0:6) edge[draw] +(300:4);
\path (120:3) ++(240:4) edge[draw] +(0:6) edge[draw] +(300:2);
\path (120:3) ++(0:2) edge[draw] +(240:6) edge[draw] +(300:4);
\draw (120:3) ++(0:4) edge[draw] +(240:6) edge[draw] +(300:2);
\path (120:1) ++(210:1.155) node {$1$};
\coordinate (origin) at (0,0);
\end{tikzpicture}
\cdot u_{{\color{red} z}}=
\begin{tikzpicture}[scale=0.6,baseline=(origin)]
\fill[red!10] (120:1) -- ++(180:2) -- ++(300:2) --cycle;
\path (120:1) ++(180:2) ++(300:1.5) node[above right=-0.1,font=\footnotesize,red] {$z$};
\draw[ultra thick] (300:4) -- (120:4);
\draw[ultra thick] (120:3) -- ++(240:6) -- ++(0:6);
\draw[ultra thick] (120:3) -- ++(0:6) -- ++(240:6);
\path (120:3) ++(240:2) edge[draw] +(0:6) edge[draw] +(300:4);
\path (120:3) ++(240:4) edge[draw] +(0:6) edge[draw] +(300:2);
\path (120:3) ++(0:2) edge[draw] +(240:6) edge[draw] +(300:4);
\draw (120:3) ++(0:4) edge[draw] +(240:6) edge[draw] +(300:2);
\path (120:1) ++(210:1.155) node {$v$};
\coordinate (origin) at (0,0);
\end{tikzpicture}
\end{equation*}
and
\begin{equation*}
\begin{tikzpicture}[scale=0.6,baseline=(origin)]
\fill[red!10] (120:1) -- ++(180:2) -- ++(300:2) --cycle;
\path (120:1) ++(180:2) ++(300:1.5) node[above right=-0.1,font=\footnotesize,red] {$z$};
\draw[ultra thick] (300:4) -- (120:4);
\draw[ultra thick] (120:3) -- ++(240:6) -- ++(0:6);
\draw[ultra thick] (120:3) -- ++(0:6) -- ++(240:6);
\path (120:3) ++(240:2) edge[draw] +(0:6) edge[draw] +(300:4);
\path (120:3) ++(240:4) edge[draw] +(0:6) edge[draw] +(300:2);
\path (120:3) ++(0:2) edge[draw] +(240:6) edge[draw] +(300:4);
\draw (120:3) ++(0:4) edge[draw] +(240:6) edge[draw] +(300:2);
\path (120:0) ++(210:0.577) node {$1$};
\coordinate (origin) at (0,0);
\end{tikzpicture}
\cdot u_{{\color{red} z}}=
\begin{tikzpicture}[scale=0.6,baseline=(origin)]
\fill[red!10] (120:1) -- ++(180:2) -- ++(300:2) --cycle;
\path (120:1) ++(180:2) ++(300:1.5) node[above right=-0.1,font=\footnotesize,red] {$z$};
\draw[ultra thick] (300:4) -- (120:4);
\draw[ultra thick] (120:3) -- ++(240:6) -- ++(0:6);
\draw[ultra thick] (120:3) -- ++(0:6) -- ++(240:6);
\path (120:3) ++(240:2) edge[draw] +(0:6) edge[draw] +(300:4);
\path (120:3) ++(240:4) edge[draw] +(0:6) edge[draw] +(300:2);
\path (120:3) ++(0:2) edge[draw] +(240:6) edge[draw] +(300:4);
\draw (120:3) ++(0:4) edge[draw] +(240:6) edge[draw] +(300:2);
\path (120:0) ++(210:0.577) node {$1$};
\coordinate (origin) at (0,0);
\end{tikzpicture} \text{,}
\end{equation*}
where the red alcove labeled $z$ is the unique ordinary alcove inside the $l$-alcove corresponding to the $l$-linkage class $\Wl z(C)$.

\begin{defn}
The \defnemph{$\ast$-Hecke algebra} $\gls*{heckeast}$ is the $\laur$-subalgebra of $\End_{\heckel}(\heckelast)$ generated by 
\begin{align*}
b^{(\ast)}_s &\qquad \text{for each $s \in \Ss$,} \\
u_w^{\pm 1} & \qquad \text{for each $w \in \Wlcosets$.}
\end{align*}
\end{defn}

As our previous notation suggests we interpret $\heckeast$ as acting on $\heckelast$ from the right, and $\heckelast$ as an $(\heckel,\heckeast)$-bimodule. 

Now let
\begin{equation*}
\gls*{psi}:\heckeast \longrightarrow M_{|\Wlcosets|}(\hecke)
\end{equation*}
be the homomorphism defined by
\begin{align}
\psi(b^{(\ast)}_s)_{w,w'}& =\begin{cases}
b_{F^{-1}(wsw^{-1})} & \text{if $\Wl ws=\Wl w$ and $w=w'$,} \\
1 & \text{if $\Wl ws \neq \Wl w$ and $ws=w'$,} \\
1 & \text{if $\Wl ws \neq \Wl w$ and $w=w'$,} \\
0 & \text{otherwise}
\end{cases} \label{eq:psientries1} \\
\psi(u_{z})_{w,w'} & =\begin{cases}
v & \text{if $w=w'=z$,} \\
1 & \text{if $w=w' \neq z$,} \\
0 & \text{otherwise}
\end{cases} \label{eq:psientries2}
\end{align}
for all $w,w',z \in \Wlcosets$ and $s \in \Ss$. 
It is easy to check that $\psi$ is the matrix form of the right $\heckeast$-action on $\heckelast$ with respect to the decomposition \eqref{eq:heckelastdecomp}. 
In other words, we have
\begin{equation*}
h^{(l|\ast)}_{x|w} b=\sum_{z \in \Wlcosets} h^{(l)}_x F(\psi(b)_{w,z}) h^{(l|\ast)}_z
\end{equation*}
for all $b \in \heckeast$, $x \in \Wl$, and $w \in \Wlcosets$. 
We call $\psi$ the \defnemph{matrix recursion representation} of $\heckeast$.

Now write $\glsuseri*{voneheckeast}$ and $\prescript{}{u,v=1}{\hecke}_{l|\ast}$ for the quotients $\heckeast/(u_s-1,v-1 : s \in \Ss)$ and $\heckelast/(v-1)$. 
For $b \in \heckeast$ (resp.~$h \in \heckelast$) write $\prescript{}{u,v=1}{b}$ (resp.~$\prescript{}{v=1}{h}$) for the image of $b$ (resp.~$h$) in the quotient. 
We have a ring isomorphism and a $(\ZZ\Wl,\ZZ\W)$-bimodule isomorphism mapping
\begin{align*}
\prescript{}{u,v=1}{\hecke}_{\ast} & \longrightarrow \ZZ\W & \prescript{}{v=1}{\hecke}_{l|\ast} & \longrightarrow \ZZ\W \\
\glsuseri*{vonebpastps} & \longmapsto s+1 & \prescript{}{v=1}{h}^{(l|\ast)}_{x|w} & \longmapsto xw
\end{align*}
for all $s \in \Ss$ and $x|w \in \W$. 
From \eqref{eq:psientries1}--\eqref{eq:psientries2} we have $\prescript{}{u,v=1}{\psi}=\xi$, and thus $\psi$ is a Hecke analogue of the matrix recursion representation $\xi$ of $\W$. 

\begin{rem}
There is an analogue of the bar involution for $\heckeast$, defined as follows.
Namely, for any $w \in \Wlcosets$ and $s \in \Ss$ we set
\begin{equation}
\begin{split}
\overline{v} & =v^{-1} \text{,} \\
\overline{u_w} & =u_w^{-1} \text{,} \\
\overline{b^{(\ast)}_s} & =b^{(\ast)}_s \text{,}
\end{split}
\end{equation}
and then extend to an involution $\overline{\phantom{H}}:\heckeast \rightarrow \heckeast$. 
One can check that this is well-defined with a basis or a presentation of $\heckeast$, which is straightforward (but tedious) to find.
\end{rem}

\subsection{Twisted \texorpdfstring{$l$}{l}-defects and patterns}
\label{sec:patterns}

The goal of this section is to obtain a Deodhar-like formula for calculating products in $\heckelast$. 
This will involve defining a generalization of the Bruhat stroll and the Deodhar defect.

\begin{defn}
Suppose $w \in \Wlcosets$ and $\seq{e}=(\seq{e}_1,\seq{e}_2,\dotsm \seq{e}_m)$ is a subsequence for some expression in $\expr{\Ss}$, with $\seq{e}_i=(s_i,t_i)$ for each integer $1 \leq i \leq m$. 
The \defnemph{$w$-twisted $l$-Bruhat stroll} for $\seq{e}$ is a sequence $(y_i)_{i=0}^m$ in $\Wl$ defined as follows. 
For each integer $0 \leq i \leq m$, let $x_i=w\widehat{e_{\leq i}}$, where $\seq{e}_{\leq i}$ is the subsequence formed from the first $i$ terms of $\seq{e}$.
Then write $x_i=y_i|w_i$, where $y_i \in \Wl$ and $w_i \in \Wlcosets$.

We decorate the types of the terms $\seq{e}_i$ as follows.
For each integer $1 \leq i \leq m$, the \defnemph{$w$-twisted $l$-decorated type} of the term $\seq{e}_i$ is
\begin{equation*}
t_i'=\begin{cases}
\upseq t_i & \text{if $\Wl w_{i-1} s_i=\Wl w_{i-1}$ and $y_i w_{i-1} s_i w_{i-1}^{-1}>y_{i-1}$,} \\
\downseq t_i & \text{if $\Wl w_{i-1} s_i=\Wl w_{i-1}$ and $y_i w_{i-1} s_i w_{i-1}^{-1}<y_{i-1}$,} \\
\stayseq t_i & \text{if $\Wl w_{i-1} s_i \neq \Wl w_{i-1}$.}
\end{cases}
\end{equation*}
and the decorated type of $\seq{e}$ is the sequence of decorated types of all the terms in $\seq{e}$.
The \defnemph{$w$-twisted $l$-defect} of $\seq{e}$ is 
\begin{equation}
\glsuseri*{dfctlw}_l^w(\seq{e})=|\{i : t_i'=\upseq 0\}| - |\{i : t_i'=\downseq 0\}| \text{.}
\end{equation}
\end{defn}


We can visualize $l$-Bruhat strolls and $l$-defects in the following way.
If $\seq{e}$ is a subsequence for an expression in $\expr{\Ss}$ and $w=1$, the sequence $(x_i)_{i=0}^m$ above is just Deodhar's Bruhat stroll for $\seq{e}$.
This corresponds to a sequence of adjacent ordinary alcoves. 
The $l$-Bruhat stroll corresponds to the sequence of $l$-alcoves containing these ordinary alcoves.
The $l$-defect is calculated like the Deodhar defect, except that we only consider steps in the ordinary Bruhat stroll which (potentially) traverse an $l$-wall. 
These steps are decorated according to whether the length of the corresponding $l$-alcove has gone Up or Down.
All other steps are marked $\stayseq$ (for Neutral) and are ignored.
Twisting by $w \neq 1$ means starting the ordinary Bruhat stroll at $w$ instead of $1$.

\begin{exam}
\label{exam:w-twisted-p-defect}
Suppose $\W$ is of type $\widetilde{A_2}$ and $l=3$.
For succinctness we will write the generators of $\W$ as $0$, $1$, and $2$, with $\Ssf=\{1,2\}$ and $\tilde{s}=0$, and write the identity element as $\ident$. 
Consider the expression $\expr{x}=\expr{01202101} \in \expr{\Ss}$ and the element $w=0 \in \W$. 
Let $\seq{e} \in [\expr{x}]$ with type $10111110$.
The $w$-twisted ordinary Bruhat stroll for $\seq{e}$ is
\begin{equation*}
\ident|0,\ident|\ident,\ident|\ident,2|\ident,2|0,2|02,2|021,2|0210,2|0210 \text{,}
\end{equation*}
so the $w$-twisted $l$-Bruhat stroll for $\seq{e}$ is
\begin{equation*}
\ident,\ident,\ident,2,2,2,2,2,2 \text{.}
\end{equation*}
In terms of alcoves, this looks like
\begin{equation*}
\begin{tikzpicture}[scale=0.9,baseline=(origin),>={Stealth[inset=0pt,angle=30:4pt]}]
\draw[ultra thick] (300:4) -- (120:4);
\draw[ultra thick] (120:3) -- ++(240:6) -- ++(0:6);
\draw[ultra thick] (120:3) -- ++(0:6) -- ++(240:6);
%
\path (120:3) ++(240:2) edge[draw] +(0:6) edge[draw] +(300:4);
\path (120:3) ++(240:4) edge[draw] +(0:6) edge[draw] +(300:2);
\path (120:3) ++(0:2) edge[draw] +(240:6) edge[draw] +(300:4);
\draw (120:3) ++(0:4) edge[draw] +(240:6) edge[draw] +(300:2);
\draw (300:1) edge[draw,blue] +(0:2);
\path (300:1) ++(0:1.5) node[above=-0.05,font=\footnotesize,blue] {$0$};
\draw (300:3) edge[draw,ultra thick,red] +(60:2) edge[draw,ultra thick,green] +(120:2);
\path (300:3) ++(60:0.5) node[below right=-0.1,font=\footnotesize,red] {$1$};
\path (300:2.5) node[below left=-0.1,font=\footnotesize,green] {$2$};
\path (300:1) ++(30:1.155) coordinate (id0) 
             ++(270:1.155) coordinate (idid) 
             ++(210:1.155) coordinate (2id)
             ++(150:1.155) coordinate (20)
              ++(90:1.155) coordinate (202)
             ++(150:1.155) coordinate (2021)
              ++(90:1.155) coordinate (20210);
\draw[->] (id0) to (idid); 
\draw[->] (idid) to[out=355,in=305,loop] (idid); 
\draw[->] (idid) to (2id);
\draw[->] (2id) to (20);
\draw[->] (20) to (202);
\draw[->] (202) to (2021);
\draw[->] (2021) to (20210);
\draw[->] (20210) to[out=5,in=55,loop] (20210);
\coordinate (origin) at (0,0);
\end{tikzpicture} \text{.}
\end{equation*}
The decorated type of $\seq{e}$ is
\begin{equation*}
(\stayseq 1, \upseq 0, \upseq 1, \stayseq 1, \stayseq 1, \stayseq 1, \stayseq 1, \downseq 0)
\end{equation*}
Thus the $w$-twisted $l$-defect $\dfct_l^w(\seq{e})$ is $1-1=0$.
\end{exam}


The $w$-twisted $l$-Bruhat stroll gives rise to an equivalence relation on the set of subsequences of a given expression.
As we will see later, the subsequences in an equivalence class are in bijection with the subsequences of a smaller expression $\Sslexpr$.

\begin{defn}
Let $\expr{x} \in \expr{\Ss}$ and $w \in \Wlcosets$.
Consider the following equivalence relation $\sim_w$ on $[\expr{x}]$. 
For $\seq{e},\seq{f} \in [\expr{x}]$ we write $\seq{e} \sim_w \seq{f}$ if for all integers $1 \leq i \leq \len(\expr{x})$ and all $t \in \{0,1\}$
\begin{equation*}
\text{$\seq{e}_i$ has decorated type $\stayseq t$} \quad \Longleftrightarrow \quad \text{$\seq{f}_i$ has decorated type $\stayseq t$}
\end{equation*}
in the $w$-twisted $l$-Bruhat strolls for $\seq{e}$ and $\seq{f}$ respectively.
We call an equivalence class for $\sim_w$ a \defnemph{$w$-twisted $l$-undecided stroll} for $\expr{x}$.
\end{defn}

To avoid writing down large sets of subsequences, we introduce the following notation.

\begin{notn}
Let $\expr{x}=\expr{s_1 s_2 \dotsm s_m} \in \expr{\Ss}$, $\seq{e} \in [\expr{x}]$, and $w \in \Wlcosets$.
The $\sim_w$-equivalence class of $\seq{e}$ is a $w$-twisted $l$-undecided stroll for $\expr{x}$. 
The \defnemph{pattern} for $\expr{x}$ corresponding to this undecided stroll is the sequence $\glsuseri*{pattr}=(\patt{r}_1,\patt{r}_2,\dotsc,\patt{r}_m)$, where 
\begin{equation*}
\patt{r}_i=\begin{cases}
(s_i,t_i) & \text{if the $w$-twisted $l$-decorated type of $\seq{e}_i$ is $\stayseq t_i$,} \\
(s_i,\ast) & \text{otherwise.}
\end{cases}
\end{equation*}
(It is clear that $\patt{r}$ only depends on the $\sim_w$-equivalence class of $\seq{e}$.)

As with subsequences, $\patt{r}_i$ is called a \defnemph{term} with \defnemph{generator} $s_i$ and \defnemph{type} $t_i$, and the type of $\patt{r}$ is the sequence of types of the terms $\patt{r}_i$. 
We call a term \emph{indeterminate} if it is of type $\ast$; otherwise we call it \defnemph{fixed}. 
We set
\begin{equation*}
[\patt{r}]=\{\seq{f} \in [\expr{x}] : \text{$\seq{f}_i,\patt{r}_i$ have the same type whenever $\patt{r}_i$ is fixed}\} \text{,}
\end{equation*}
which by construction is the $w$-twisted $l$-undecided stroll we started with.
We also write $\hat{r}$ for the product of all the generators in $\patt{r}$ of type $1$.

By construction $w\widehat{r_{\leq i}}$ is the $\Wlcosets$-part of $w\widehat{e_{\leq i}}$ for all integers $0 \leq i \leq m$, where $\patt{r}_{\leq i}$ (resp.~$\seq{e}_{\leq i}$) is the pattern (resp.~subsequence) formed from the first $i$ terms of $\patt{r}$ (resp.~$\seq{e}$).
Thus
\begin{equation*}
(w\widehat{r_{\leq i-1}})s_i(w\widehat{r_{\leq i-1}})^{-1} \in \Ssl
\end{equation*}
whenever $\patt{r}_i$ is indeterminate.
We write $w\patt{r}(w\hat{r})^{-1}$ for the expression in $\Sslexpr$ formed by these generators.
\end{notn}

We can think of patterns as generalized expressions, with fixed terms restricting possible subsequences by pre-selecting certain generators to be included or discarded. 
For $\expr{x} \in \expr{\Ss}$, we write $\glsuseri*{bexprxbast}(w)$ for the set of all patterns $\patt{r}$ for $\expr{x}$ which correspond to some $w$-twisted $l$-undecided stroll in the above manner. 
In other words, $[\expr{x}]_{\ast}(w)$ is the set of all patterns for $\expr{x}$ such that for each integer $0 \leq i \leq \len(\expr{x})$, 
\begin{equation*}
\text{$\patt{r}_i$ is indeterminate} \quad \Longleftrightarrow \quad \Wl w\widehat{r_{\leq i-1}}s_i=\Wl w\widehat{r_{\leq i-1}} \text{.}
\end{equation*}
If $\patt{r} \in [\expr{x}]_{\ast}(w)$, any $\seq{e} \in [\patt{r}]$ gives rise to a subsequence $w\seq{e}(w\hat{r})^{-1} \in [w\patt{r}(w\hat{r})^{-1}]$ defined as follows. 
The type of $w\seq{e}(w\hat{r})^{-1}$ is just the sequence of types of $\seq{e}$ with decorations $\upseq$ or $\downseq$ in the $w$-twisted $l$-Bruhat stroll.
Since the $w$-twisted $l$-defect only depends on these types, we have $\dfct_l^w(\seq{e})=\dfct_l(w\seq{e}(w\hat{r})^{-1})$, where $\glsuseri*{dfctl}_l$ denotes the Deodhar defect for subsequences of expressions in $\Sslexpr$.

\begin{exam}
\label{exam:pattern}
We continue Example~\ref{exam:w-twisted-p-defect} where $\W$ is of type $\widetilde{A_2}$ and $l=3$,
with $\expr{x}=\expr{01202101} \in \expr{\Ss}$ and $w=0 \in \W$.
We can calculate $[\expr{x}]_{\ast}(w)$ inductively, e.g.~
\begin{align*}
[\expr{x}_{\leq 1}]_{\ast}(w)& =\{((0,0)),((0,1))\} \text{,} \\
[\expr{x}_{\leq 2}]_{\ast}(w)& =\{((0,0),(1,0)),((0,0),(1,1)),((0,1),(1,\ast))\} \text{,} \\
[\expr{x}_{\leq 3}]_{\ast}(w)& =\{((0,0),(1,0),(2,0)),((0,0),(1,0),(2,1)),((0,0),(1,1),(2,0)), \\
& \qquad ((0,0),(1,1),(2,1)),((0,1),(1,\ast),(2,\ast))\} \text{,} \\
& \mathrel{\makebox[\widthof{=}]{\vdots}}
\end{align*}
Consider the pattern $\patt{r} \in [\expr{x}]_{\ast}(w)$ of type $1 {\ast} {\ast} 1111 {\ast}$.
There are $2^3=8$ subsequences in $[\patt{r}]$. 
Their $w$-twisted $l$-Bruhat strolls look like this:
\begin{equation*}
\begin{tikzpicture}[scale=0.9,baseline=(origin),>={Stealth[inset=0pt,angle=30:4pt]}]
\draw[ultra thick] (300:9) -- (120:3);
\draw[ultra thick] (120:3) -- ++(240:6) -- ++(0:6);
\draw[ultra thick] (120:3) -- ++(0:6) -- ++(240:6);
\path (300:3) ++(0:6) edge[draw, ultra thick] +(120:6) edge[draw, ultra thick] +(180:6) edge[draw, ultra thick] +(240:6);
%
\path (120:3) ++(240:2) edge[draw] +(0:8) edge[draw] +(300:4);
\path (120:3) ++(240:4) edge[draw] +(0:10) edge[draw] +(300:2);
\path (300:5) edge[draw] +(0:4) edge[draw] +(60:6);
\path (300:7) edge[draw] +(0:2) edge[draw] +(60:6);
\path (120:3) ++(0:2) edge[draw] +(240:6) edge[draw] +(300:10);
\draw (120:3) ++(0:4) edge[draw] +(240:6) edge[draw] +(300:8);
\draw (300:1) edge[draw,blue] +(0:2);
\path (300:1) ++(0:1.5) node[above=-0.05,font=\footnotesize,blue] {$0$};
\draw (300:3) edge[draw,ultra thick,red] +(60:2) edge[draw,ultra thick,green] +(120:2);
\path (300:3) ++(60:0.5) node[below right=-0.1,font=\footnotesize,red] {$1$};
\path (300:2.5) node[below left=-0.1,font=\footnotesize,green] {$2$};
\path (300:1) ++(30:1.155) coordinate (id0) 
             ++(270:1.155) coordinate (idid) 
             ++(210:1.155) coordinate (2id)
             ++(150:1.155) coordinate (20)
              ++(90:1.155) coordinate (202)
             ++(150:1.155) coordinate (2021)
              ++(90:1.155) coordinate (20210);
\path (id0) ++(150:1.155) coordinate (id02)
             ++(90:1.155) coordinate (id021)
            ++(150:1.155) coordinate (id0210);
\path (idid) ++(330:1.155) coordinate (1id)
              ++(30:1.155) coordinate (10)
             ++(330:1.155) coordinate (102)
              ++(30:1.155) coordinate (1021)
             ++(330:1.155) coordinate (10210);
\path (1id) ++(270:1.155) coordinate (21id)
            ++(330:1.155) coordinate (210)
             ++(30:1.155) coordinate (2102)
            ++(330:1.155) coordinate (21021)
             ++(30:1.155) coordinate (210210);
\coordinate (origin) at (300:3);
\draw[->,gray] ([xshift=0.05cm]id0) to ([xshift=0.05cm]idid); 
\draw[->,thick] ([xshift=0.05cm]idid) to[out=355,in=305,looseness=8,min distance=5mm] ([yshift=-0.0433cm]idid); 
\draw[->,thick] ([yshift=-0.0433cm]idid) to (2id);
\draw[->,gray] (2id) to (20);
\draw[->,gray] (20) to (202);
\draw[->,gray] (202) to ([xshift=0.025cm]2021);
\draw[->,gray] ([xshift=0.025cm]2021) to ([yshift=-0.0433cm,xshift=0.025cm]20210);
\draw[->,thick] ([yshift=-0.0433cm,xshift=0.025cm]20210) to[out=340,in=80,looseness=8,min distance=5mm] ([yshift=0.0433cm,xshift=-0.025cm]20210);
\draw[->,thick] ([yshift=-0.0433cm,xshift=0.025cm]20210) to ([yshift=-0.0433cm,xshift=0.025cm]id0210);
\draw[->,thick] ([yshift=-0.0433cm]idid) to[out=235,in=185,looseness=8,min distance=5mm] ([xshift=-0.05cm]idid);
\draw[->,gray] ([xshift=-0.05cm]idid) to ([xshift=-0.05cm]id0);
\draw[->,gray] ([xshift=-0.05cm]id0) to (id02);
\draw[->,gray] (id02) to ([xshift=-0.025cm]id021);
\draw[->,gray] ([xshift=-0.025cm]id021) to ([yshift=0.0433cm,xshift=-0.025cm]id0210);
\draw[->,thick] ([yshift=0.0433cm,xshift=-0.025cm]id0210) to[out=160,in=260,looseness=8,min distance=5mm] ([yshift=-0.0433cm,xshift=0.025cm]id0210);
\draw[->,thick] ([yshift=0.0433cm,xshift=-0.025cm]id0210) to ([yshift=0.0433cm,xshift=-0.025cm]20210);
\draw[->,thick] ([xshift=0.05cm]idid) to ([xshift=-0.05cm]1id);
\draw[->,thick] ([xshift=-0.05cm]1id) to[out=245,in=295,looseness=8,min distance=5mm] ([yshift=0.0433cm]1id);
\draw[->,gray] ([yshift=0.0433cm]1id) to (10);
\draw[->,gray] (10) to (102);
\draw[->,gray] (102) to (1021);
\draw[->,gray] (1021) to ([xshift=-0.05cm]10210);
\draw[->,thick] ([xshift=-0.05cm]10210) to[out=220,in=320,looseness=8,min distance=5mm] ([xshift=0.05cm]10210);
\draw[->,thick] ([xshift=-0.05cm]10210) to ([xshift=-0.05cm]210210);
\draw[->,thick] ([xshift=-0.05cm]1id) to ([xshift=-0.05cm]21id);
\draw[->,gray] ([xshift=-0.05cm]21id) to (210);
\draw[->,gray] (210) to (2102);
\draw[->,gray] (2102) to (21021);
\draw[->,gray] (21021) to ([xshift=0.05cm]210210);
\draw[->,thick] ([xshift=0.05cm]210210) to[out=40,in=140,looseness=8,min distance=5mm] ([xshift=-0.05cm]210210);
\draw[->,thick] ([xshift=0.05cm]210210) to ([xshift=0.05cm]10210);
\end{tikzpicture} \text{.}
\end{equation*}
We have $w\patt{r}(w\hat{r})^{-1}=\expr{122} \in \Sslexpr$.
In particular the subsequence $\seq{e} \in [\patt{r}]$ of type $10111110$ gives rise to the subsequence $w\seq{e}(w\hat{r})^{-1} \in [w\patt{r}(w\hat{r})^{-1}]$ of type $010$.
\end{exam}

We are now ready to prove a version of Deodhar's defect formula for $\heckelast$.

\begin{prop} \label{prop:ddfheckepasttwisted} 
Let $w \in \Wlcosets$ and $\expr{x} \in \expr{\Ss}$. Then
\begin{equation*}
h^{(l|\ast)}_w b^{(\ast)}_{\expr{x}}=\sum_{\patt{r} \in [\expr{x}]_{\ast}(w)} b^{(l)}_{w\patt{r}(w\hat{r})^{-1}} h^{(l|\ast)}_{w\hat{r}}=\sum_{\seq{e} \in [\expr{x}]} v^{\dfct_l^w(\seq{e})} h^{(l|\ast)}_{w\hat{e}} \text{.}
\end{equation*}
%
\end{prop}

\begin{proof}
Note that the second equality follows from the first, by applying Deodhar's defect formula to each $h^{(l)}_{w\patt{r}(w\hat{r})^{-1}}$ and using $\dfct_l^w(\seq{e})=\dfct_l(w\seq{e}(w\hat{r})^{-1})$ for all $\seq{e} \in [\patt{r}]$.
To prove the first equality, we induct on the length of $\expr{x}$. 
When $\len(\expr{x})=0$ the result holds trivially. 
Now suppose $\len(\expr{x})=m$ and that the result holds for expressions length less than $m$. 
Write $\expr{x}=\expr{z s}$ for some $\expr{z} \in \expr{\Ss}$ and $s \in \Ss$. 
Then we have
\begin{align*}
h^{(l|\ast)}_w b^{(\ast)}_{\expr{x}}& =(h^{(l|\ast)}_w  b^{(\ast)}_{\expr{z}}) b^{(\ast)}_s \\
& =\left(\sum_{\patt{q} \in [\expr{z}]_{\ast}(w)} b^{(l)}_{w\patt{q}(w\hat{q})^{-1}} h^{(l|\ast)}_{w\hat{q}}\right)b^{(\ast)}_s \\
& =\sum_{\substack{\patt{q} \in [\expr{z}]_{\ast}(w) \\ \Wl w\hat{q}s=\Wl w\hat{q}}} (b^{(l)}_{w\patt{q}(w\hat{q})^{-1}} b^{(l)}_{(w\hat{q})s(w\hat{q})^{-1}})h^{(l|\ast)}_{w\hat{q}} \\
& \quad +\sum_{\substack{\patt{q} \in [\expr{z}]_{\ast}(w)\\ \Wl w\hat{q}s \neq \Wl w\hat{q}}} (b^{(l)}_{w\patt{q}(w\hat{q})^{-1}}h^{(l|\ast)}_{w\hat{q}s}+b^{(l)}_{w\patt{q}(w\hat{q})^{-1}}h^{(l|\ast)}_{w\hat{q}}) \text{.}
\end{align*}
This sum is equal to
\begin{equation*}
\sum_{\substack{\patt{r} \in [\expr{x}]_{\ast}(w) \\ \text{$\patt{r}_m$ is indeterminate}}}
b^{(l)}_{w\patt{r}(w\hat{r})^{-1}}h^{(l|\ast)}_{w\hat{r}}+\sum_{\substack{\patt{r} \in [\expr{x}]_{\ast}(w) \\ \text{$\patt{r}_m$ is fixed}}}
b^{(l)}_{w\patt{r}(w\hat{r})^{-1}}h^{(l|\ast)}_{w\hat{r}}=\sum_{\patt{r} \in [\expr{x}]_{\ast}(w)}
b^{(l)}_{w\patt{r}(w\hat{r})^{-1}}h^{(l|\ast)}_{w\hat{r}}
\end{equation*}
which proves the result.
\end{proof}

Let $\glsuseri*{linklexpr}$ denote the following subset
\begin{equation*}
\linklexpr=\{\expr{x w} : \expr{x} \in \Sslexpr,\ \expr{w} \in \expr{\Ss}\}
\end{equation*}
of expressions involving generators in $\Ssl$ and $\Ss$. 
We call such expressions \defnemph{$(l|1)$-expressions}. 
As with similar products in $\W$, we will sometimes write expressions in $\linklexpr$ with a bar in the form $\glsuseri*{exprxbarexprw}$ in order to emphasize that $\expr{x} \in \Sslexpr$ and $\expr{w} \in \expr{\Ss}$. 
We set $\len_{l|1}(\expr{x}|\expr{w})=\len_l(\expr{x})+\len(\expr{w})$. 
We also extend the notion of a subsequence to $(l|1)$-expressions in the obvious manner.

Every construction in this section which is described as ``$w$-twisted'' (or is dependent on some $w \in \Wlcosets$) has an untwisted $\linklexpr$-analogue. 
For example, if $\expr{x}|\expr{z} \in \linklexpr$, the $l$-Bruhat stroll for a subsequence $\seq{e}|\seq{f} \in [\expr{x}|\expr{z}]$ has the exact same definition as the $w$-twisted $l$-Bruhat stroll, but with $w=1$.
This has the effect that the $l$-defect $\glsuseri*{dfctl}_l(\seq{e}|\seq{f})$ is just $\dfct_l(\seq{e})+\dfct_l^1(\seq{f})$.
Similarly, we may define the set of patterns associated to $\expr{x}|\expr{z}$ by $\glsuseri*{bexprxzlast}=\expr{x}[\expr{z}]_{\ast}(1)$. 
Here $\expr{x}$ is interpreted as a single pattern with only indeterminate terms, acting on $[\expr{z}]_{\ast}(1)$ via the free monoid structure on patterns.

As with $\heckel$ and $\heckeast$, let $\glspl*{blastexprxbarexprw}^{(l|\ast)}_{\expr{x}|\expr{y}}=b^{(l)}_{\expr{x}} b^{(\ast)}_{\expr{y}}$.
The following $\linklexpr$-analogue of Deodhar's defect formula follows directly from Lemma~\ref{lem:ddfhecke} and Proposition~\ref{prop:ddfheckepasttwisted}.

\begin{cor} \label{cor:ddfheckepastuntwisted}
Let $\expr{x}|\expr{w} \in \linklexpr$. Then
\begin{equation*}
b^{(l|\ast)}_{\expr{x}|\expr{w}}=\sum_{\patt{r} \in [\expr{x}|\expr{w}]_{l|\ast}} b^{(l)}_{\patt{r}\hat{r}^{-1}} h^{(l|\ast)}_{\hat{r}}=\sum_{\seq{e}|\seq{f} \in [\expr{x}|\expr{w}]} v^{\dfct_l(\seq{e}|\seq{f})} h^{(l|\ast)}_{\hat{e}\hat{f}} \text{.}
\end{equation*}
\end{cor}

\section{Diagrammatic Soergel bimodules}

\subsection{The diagrammatic Hecke category}

Let $\gls*{V}$ be a realization of $(\W,\Ss)$. 
In this section we will construct a monoidal category $\dgrm$ from $V$. 
For the purposes of drawing diagrams we identify $\Ss$ with a set of colors.
We set $\gls*{R}=\mathrm{Sym}(V)$, the symmetric algebra in $V$, which we view as a graded algebra with $\glsuseri*{deg}(V)=2$. 
The algebra $R$ inherits a graded $\W$-action from $V$. 

We summarize the construction of the \defnemph{(diagrammatic) Bott--Samelson category} below following \cite[\S 2.2]{amrw-mixedtilting}.

\begin{defn} \label{defn:tildediagBS}
The category $\widetilde{\dgrm}_{\rm BS}$ is the $\field$-linear graded strict monoidal category defined as follows.
\begin{itemize}
\item The objects of $\widetilde{\dgrm}_{\rm BS}$ are the formal symbols $\glsuseri*{Bexprx}$ for each expression $\expr{x} \in \expr{\Ss}$, with tensor product $B_{\expr{x}} \otimes B_{\expr{y}}=B_{\expr{x y}}$.

\item The morphisms in $\widetilde{\dgrm}_{\rm BS}$ are generated (under $\field$-linear combinations, compositions, and tensor products) by the following elementary morphisms. 
We depict these morphisms as string diagrams, with the bottom and top boundaries corresponding to the domain and codomain respectively.
Vertical and horizontal concatenation of diagrams correspond to composition and the tensor product of morphisms respectively.
In the diagrams below $s$ is colored red and $t$ is colored blue.
\begin{itemize}
\item For each homogeneous $f \in R$, there is a morphism 
\begin{gather*}
f:B_{\emptexpr} \longrightarrow B_{\emptexpr} \\
\begin{tikzpicture}[baseline=(origin)]
\coordinate (origin) at (0,0) {};
\node[circle, dashed, minimum size=0.5cm] (spacer) at (origin) {$f$};
\end{tikzpicture}
\end{gather*}
of degree $\deg(f)$.

\item For each $s \in \Ss$ there are morphisms 
\begin{align*}
\begin{gathered}
\glsuseri*{dots} : B_{\expr{s}} \longrightarrow B_{\emptexpr} \text{,} \\
\begin{tikzpicture}[scale=0.5,baseline=(uni.center)]
    \node[dot, fill=red] (uni) at (0,0) {};
    \node[circle, minimum size=1cm] (spacer) at (uni) {};
    \path (uni.center) ++(0,-1) coordinate (unibot);
    \draw[red, line width=\superthick] (uni.center) to (unibot);
\end{tikzpicture}
\end{gathered} & & 
\begin{gathered}
\overline{{\rm dot}_s} : B_{\emptexpr} \longrightarrow B_{\expr{s}} \\
\begin{tikzpicture}[xscale=0.5,yscale=-0.5,baseline=(uni.center)]
    \node[dot, fill=red] (uni) at (0,0) {};
    \path (uni.center) ++(0,-1) coordinate (unibot);
    \node[circle, minimum size=1cm] (spacer) at (uni) {};
    \draw[red, line width=\superthick] (uni.center) to (unibot);
\end{tikzpicture}
\end{gathered} &
\end{align*}
of degree $1$ and 
\begin{align*}
\begin{gathered}
\glsuseri*{forks} : B_{\expr{ss}} \longrightarrow B_{\expr{s}} \text{,} \\
\begin{gathered}
\begin{tikzpicture}[xscale=0.5,yscale=-0.5]
    \clip (-1,-1) rectangle (1,0.65);
    \coordinate (tri) at (0,0);
    \path (tri) ++(45:1) coordinate (abovelefttri);
    \path (tri) ++(135:1) coordinate (aboverighttri);
    \path (tri) ++(270:1) coordinate (belowtri);
    \draw[red, line width=\superthick] (tri) edge (abovelefttri)
                                             edge (aboverighttri)
                                             edge (belowtri);
\end{tikzpicture}
\end{gathered}
\end{gathered} & &
\begin{gathered}
\overline{{\rm fork}_s} : B_{\expr{s}} \longrightarrow B_{\expr{ss}} \\
\begin{gathered}
\begin{tikzpicture}[scale=0.5]
    \clip (-1,-1) rectangle (1,0.65);
    \coordinate (tri) at (0,0);
    \path (tri) ++(45:1) coordinate (abovelefttri);
    \path (tri) ++(135:1) coordinate (aboverighttri);
    \path (tri) ++(270:1) coordinate (belowtri);
    \draw[red, line width=\superthick] (tri) edge (abovelefttri)
                                             edge (aboverighttri)
                                             edge (belowtri);
\end{tikzpicture}
\end{gathered}
\end{gathered} &
\end{align*}
of degree $-1$.

\item For each pair $(s,t) \in \Ss \times \Ss$ with $s \neq t$ and $m_{st}<\infty$, there is a morphism
\begin{gather*}
\glsuseri*{braidst} : B_{\underbrace{\scriptstyle\expr{sts \dotsm }}_{m_{st}}} \longrightarrow B_{\underbrace{\scriptstyle\expr{tst \dotsm }}_{m_{st}}} \\
\begin{aligned}
\begin{gathered}
\begin{tikzpicture}[scale=0.66]
\clip (-1.9,-0.9) rectangle (1.9,0.9);
\coordinate (origin) at (0,0) {};
\coordinate (dom1) at (-2,-1) {};
\coordinate (dom2) at (-1.25,-1) {};
\coordinate (dom3) at (-0.5,-1) {};
\coordinate (domlast) at (2,-1) {};
\coordinate (cod1) at (-2,1) {};
\coordinate (cod2) at (-1.25,1) {};
\coordinate (cod3) at (-0.5,1) {};
\coordinate (codlast) at (2,1) {};
\draw[red, line width=\superthick] (origin) edge (dom1)
                                            edge (dom3)
                                            edge (domlast)
                                            edge (cod2);
\draw[blue, line width=\superthick] (origin) edge (dom2)
                                             edge (cod1)
                                             edge (cod3)
                                             edge (codlast);
\node (domdots) at (0.4,-0.66) {$\dotsm$};
\node (coddots) at (0.4,0.66) {$\dotsm$};
\end{tikzpicture}
\end{gathered} & & 
\text{($m_{st}$ odd)} \\
\begin{gathered}
\begin{tikzpicture}[scale=0.66]
\clip (-1.9,-0.9) rectangle (1.9,0.9);
\coordinate (origin) at (0,0) {};
\coordinate (dom1) at (-2,-1) {};
\coordinate (dom2) at (-1.25,-1) {};
\coordinate (dom3) at (-0.5,-1) {};
\coordinate (domlast) at (2,-1) {};
\coordinate (cod1) at (-2,1) {};
\coordinate (cod2) at (-1.25,1) {};
\coordinate (cod3) at (-0.5,1) {};
\coordinate (codlast) at (2,1) {};
\draw[red, line width=\superthick] (origin) edge (dom1)
                                            edge (dom3)
                                            edge (cod2)
                                            edge (codlast);
\draw[blue, line width=\superthick] (origin) edge (dom2)
                                             edge (domlast)
                                             edge (cod1)
                                             edge (cod3);
\node (domdots) at (0.4,-0.66) {$\dotsm$};
\node (coddots) at (0.4,0.66) {$\dotsm$};
\end{tikzpicture}
\end{gathered} & & 
\text{($m_{st}$ even)}
\end{aligned}
\end{gather*}
of degree $0$.
\end{itemize}
These morphisms are subject to a number of relations, which can be found in e.g.~\cite[(10.8), (10.12)--(10.15)]{emtw}.
\end{itemize}
\end{defn}

We call a string diagram arising from compositions and tensor products of the above elementary morphisms an \defnemph{$\Ss$-diagram}.
For each $s \in \Ss$ we will also use the following shorthand in our $\Ss$-diagrams:
\begin{align*}
\begin{gathered}
{\rm cap}_s={\rm dot}_s \circ {\rm fork}_s : B_{\expr{ss}} \longrightarrow B_{\emptexpr} \text{,} \\
    \begin{tikzpicture}[scale=0.5,baseline=(flex)]
        \coordinate (flex) at (0,0.65);
        \coordinate (right) at (0.65,0);
        \coordinate (left) at (-0.65,0);
        \draw[red, line width=\superthick] (left) 
                                      arc [start angle=180, end angle=0, radius=0.65]
                                        to (right);
    \end{tikzpicture}=\begin{tikzpicture}[xscale=0.5,yscale=-0.5,baseline=(tri)]
        \clip (-0.8,-1) rectangle (0.8,0.65);
        \coordinate (tri) at (0,0);
        \path (tri) ++(45:1) coordinate (abovelefttri);
        \path (tri) ++(135:1) coordinate (aboverighttri);
        \path (tri) ++(270:0.5) node[dot, fill=red] (belowtri) {};
        \draw[red, line width=\superthick] (tri) edge (abovelefttri)
                                                 edge (aboverighttri)
                                                 edge (belowtri.center);
    \end{tikzpicture}
\end{gathered} & & 
\begin{gathered}
{\rm cup}_s=\overline{{\rm fork}_s} \circ \overline{{\rm dot}_s} : B_{\emptexpr} \longrightarrow B_{\expr{ss}} \text{.} \\
    \begin{tikzpicture}[xscale=0.5,yscale=-0.5,baseline=(flex)]
        \coordinate (flex) at (0,0.65);
        \coordinate (right) at (0.65,0);
        \coordinate (left) at (-0.65,0);
        \draw[red, line width=\superthick] (left) 
                                      arc [start angle=180, end angle=0, radius=0.65]
                                        to (right);
    \end{tikzpicture}=\begin{tikzpicture}[scale=0.5,baseline=(tri)]
        \clip (-0.8,-1) rectangle (0.8,0.65);
        \coordinate (tri) at (0,0);
        \path (tri) ++(45:1) coordinate (abovelefttri);
        \path (tri) ++(135:1) coordinate (aboverighttri);
        \path (tri) ++(270:0.5) node[dot, fill=red] (belowtri) {};
        \draw[red, line width=\superthick] (tri) edge (abovelefttri)
                                                 edge (aboverighttri)
                                                 edge (belowtri.center);
    \end{tikzpicture}
\end{gathered} 
\end{align*}
More generally, for 
$\expr{w}=\expr{st \dotsm} \in \expr{\Ss}$ we write 
\begin{align*}
{\rm cap}_{\expr{w}}& ={\rm cap}_s \circ (\ident_{B_{\expr{s}}} \otimes {\rm cap}_t \otimes \ident_{B_{\expr{s}}}) \circ \dotsb
=\begin{gathered}
    \begin{tikzpicture}[scale=1,baseline=(origin)]
        \coordinate (origin) at (0,0);
        \node (dots) at (origin) {$\dotsm$}; 
        \path (origin) ++(1,0) coordinate (farright);
        \path (origin) ++(-1,0) coordinate (farleft) node[anchor=north] {${\color{red} s}$};
        \path (origin) ++(0.5,0) coordinate (right);
        \path (origin) ++(-0.5,0) coordinate (left) node[anchor=north] {${\color{blue} t}$};
        %
        \draw[red,string] (farleft) arc [start angle=180, end angle=0, radius=1] to (farright);
        \draw[blue,string] (left) arc [start angle=180, end angle=0, radius=0.5] to (right);
    \end{tikzpicture}
\end{gathered} \text{,} \\
{\rm cup}_{\expr{w}}& = \dotsb \circ (\ident_{B_{\expr{s}}} \otimes {\rm cup}_t \otimes \ident_{B_{\expr{s}}}) \circ {\rm cup}_s 
=\begin{gathered}
    \begin{tikzpicture}[xscale=1,yscale=-1,baseline=(origin)]
        \coordinate (origin) at (0,0);
        \node (dots) at (origin) {$\dotsm$}; 
        \path (origin) ++(1,0) coordinate (farright);
        \path (origin) ++(-1,0) coordinate (farleft) node[anchor=south] {${\color{red} s}$};
        \path (origin) ++(0.5,0) coordinate (right);
        \path (origin) ++(-0.5,0) coordinate (left) node[anchor=south] {${\color{blue} t}$};
        %
        \draw[red,string] (farleft) arc [start angle=180, end angle=0, radius=1] to (farright);
        \draw[blue,string] (left) arc [start angle=180, end angle=0, radius=0.5] to (right);
    \end{tikzpicture}
\end{gathered}
\end{align*}
for a nested series of caps or cups.

The $\Hom$-spaces in $\widetilde{\dgrm}_{\rm BS}$ have graded left and right $R$-actions induced by taking tensor products with polynomial morphisms on the left or right sides.
Thus $\widetilde{\dgrm}_{\rm BS}$ has the structure of a category enriched in graded $R$-bimodules. 
We write $\Hom_{\widetilde{\dgrm}_{\rm BS}}^i(B_{\expr{x}},B_{\expr{y}})$ to denote the subspace of degree $i$ morphisms in $\Hom_{\widetilde{\dgrm}_{\rm BS}}(B_{\expr{x}}, B_{\expr{y}})$.

\begin{defn}
The category $\glsuseri*{dgrmBS}$ is the $\field$-linear monoidal category defined as follows. 
\begin{itemize}
\item The objects of $\dgrmBS$ are the formal symbols $B_{\expr{x}}(m)$, for $\expr{x} \in \expr{\Ss}$ and $m \in \ZZ$, with tensor product $B_{\expr{x}}(m) \otimes B_{\expr{y}}(n)=B_{\expr{x y}}(m+n)$.

\item The morphisms in $\dgrmBS$ are given by
\begin{equation*}
\Hom_{\dgrmBS}(B_{\expr{x}}(m),B_{\expr{y}}(n))=\Hom_{\widetilde{\dgrm}_{\rm BS}}^{n-m}(B_{\expr{x}},B_{\expr{y}}) \text{,}
\end{equation*}
with composition and tensor product defined via $\widetilde{\dgrm}_{\rm BS}$.
\end{itemize}

The \defnemph{diagrammatic Hecke category} $\gls*{dgrm}$ is the Karoubi envelope of $\dgrmBS$. 
In other words $\dgrm$ is the closure of $\dgrmBS$ with respect to all finite direct sums and all direct summands of objects and morphisms in $\dgrmBS$. 
\end{defn}
 
We call objects in $\widetilde{\dgrm}_{\rm BS}$ or $\dgrmBS$ \defnemph{(diagrammatic) Bott--Samelson bimodules} and objects in $\dgrm$ \defnemph{(diagrammatic) Soergel bimodules}. 
We note the useful fact that the Karoubi envelope construction ensures that $\dgrmBS$ is a \emph{full} subcategory of $\dgrm$, i.e.~
\begin{equation*}
\Hom_{\dgrm}(B_{\expr{x}}(m),B_{\expr{y}}(n))=\Hom_{\dgrmBS}(B_{\expr{x}}(m),B_{\expr{y}}(n))
\end{equation*}
for all $\expr{x},\expr{y} \in \expr{\Ss}$ and $m,n \in \ZZ$.

The category $\dgrmBS$ has a natural autoequivalence $(1)$ (whose $i$th power we denote by $\glsuseri*{pip}$) called the \defnemph{grade shift}, which maps $B_{\expr{x}}(m)$ to $B_{\expr{x}}(m+1)$ for any $\expr{x} \in \expr{\Ss}$ and $m \in \ZZ$.
This autoequivalence extends naturally to $\dgrm$. 
We also define a graded vector space
\begin{equation*}
\glsuseri*{HomdgrmBSbullet}(B_{\expr{x}}(m),B_{\expr{y}}(n))=\bigoplus_{i \in \ZZ} \Hom_{\dgrmBS}(B_{\expr{x}}(m),B_{\expr{y}}(i+n))
\end{equation*}
for any $\expr{x},\expr{y} \in \expr{\Ss}$ and $m,n \in \ZZ$, and similarly for $\dgrm$. 
When $m=n=0$ this is just the corresponding $\Hom$-space in $\widetilde{\dgrm}_{\rm BS}$.
In this way we can recover $\widetilde{\dgrm}_{\rm BS}$ from $\dgrmBS$, and it makes no difference whether we work with graded $\Hom$-spaces in $\widetilde{\dgrm}_{\rm BS}$ or grade shifts in $\dgrmBS$ (see also \cite[\S 2.2]{achar-riche-vay}).
For this reason, we will not always distinguish between morphisms in $\widetilde{\dgrm}_{\rm BS}$ and morphisms in $\dgrmBS$, and we will usually write $B_{\expr{x}}(0)$ simply as $B_{\expr{x}}$.

There is also a monoidal functor $(\overline{\phantom{B}}):\dgrmBS \rightarrow \dgrmBS^{\rm op}$ on $\dgrmBS$ called \defnemph{duality} defined as follows. 
For each $\expr{x} \in \expr{\Ss}$ and $m \in \ZZ$ we set $\overline{B_{\expr{x}}(m)}=B_{\expr{x}}(-m)$. 
For any morphism $\phi: B_{\expr{x}}(m)\rightarrow B_{\expr{y}}(n)$, the morphism $\overline{\phi}:\overline{B_{\expr{y}}(n)} \rightarrow \overline{B_{\expr{x}}(m)}$ corresponds to flipping the diagrams representing $\phi$ upside-down. 
The duality functor extends naturally to an additive monoidal functor on $\dgrm$.


Although we have not explicitly written down the relations defining $\dgrmBS$, we will need to check these relations at various points.
We conclude this subsection with the following indirect version of the Jones--Wenzl relation \cite[(10.8i)]{emtw}, which we will use to construct the Frobenius functor later.

\begin{prop} \label{prop:JWindirect}
Let $s,t \in \Ss$ with $m_{st}<\infty$. 
If ${\rm braid}_{s,t}$ is cyclic \cite[(10.15d)]{emtw} and
\begin{multline*}
({\rm spot}_s \otimes \ident) \overline{{\rm braid}_{s,t}}=
\begin{cases}
\ident \otimes {\rm spot}_t & \text{if $m_{st}$ is odd,} \\
\ident \otimes {\rm spot}_s & \text{if $m_{st}$ is even}
\end{cases} \\
+ (\text{linear combination of other $\Ss$-diagrams generated by spots and forks})
\end{multline*}
then the Jones--Wenzl relation \cite[(10.8i)]{emtw} holds.
\end{prop}

\begin{proof}
By degree considerations, $({\rm spot}_s \otimes \ident) \overline{{\rm braid}_{s,t}}$ must be a linear combination of diagrams of the form
\begin{equation*}
({\rm fork}_t \otimes \ident) \circ (\ident_{B_t} \otimes D)
\end{equation*}
with $D$ a diagram of degree $+2$ generated by spots and forks.
In particular this means that ${\rm braid}_{s,t}$ has the ``death by pitchfork'' property (see \cite[(5.17)]{ew-soergelcalc}). 
This further implies that
\begin{equation*}
({\rm spot}_s \otimes \ident) \circ \overline{{\rm braid}_{s,t}} \circ (\overline{{\rm spot}_t} \otimes \ident)
\end{equation*}
has the death by pitchfork property, one of the defining properties of the Jones--Wenzl morphism. 
The other defining property is that the coefficient of the leading term is $1$, which we have already assumed.
Thus the Jones--Wenzl relation \cite[(10.8i)]{emtw} holds.
\end{proof}

\subsection{Light leaves} \label{sec:lightleaves}

We briefly discuss bases for the $\Hom$-spaces in $\dgrm$, as described in \cite[{\S 6}]{ew-soergelcalc}. 
These bases yield a classification of the indecomposable objects in $\dgrm$.

Let $\expr{x} \in \expr{\Ss}$ and $\seq{e} \in [\expr{x}]$. 
Choose a rex $\expr{w}$ for $\hat{e}$. 
A \defnemph{light leaves morphism} $\glsuseri*{LLseqeexprw}:B_{\expr{x}} \rightarrow B_{\expr{w}}(\dfct(\seq{e}))$ is a morphism in $\dgrmBS$, defined inductively in \cite[{Construction~6.1}]{ew-soergelcalc}. 
This construction depends on several non-canonical choices, and as such $\LL_{\seq{e},\expr{w}}$ is not uniquely defined, but this will not matter for what follows. 
For $\expr{x},\expr{w} \in \expr{\Ss}$ with $\expr{w}$ a rex, let $\glsuseri*{LLbexprxbexprw}$ denote a complete selection of light leaves morphisms $\{\LL_{\seq{e},\expr{w}}\}$ over all subsequences $\seq{e} \in [\expr{x}]$ such that $\hat{e}=w$. 

Now suppose $\expr{x},\expr{y} \in \expr{\Ss}$, and that for each $w \in \W$ we have fixed a corresponding rex $\expr{w}$. 
For subsequences $\seq{e} \in [\expr{x}]$ and $\seq{f} \in [\expr{y}]$ such that $\hat{e}$ and $\hat{f}$ are the same element $w \in \W$, the \defnemph{double leaves morphism} is defined to be 
\begin{equation}
\glsuseri*{LLdblseqeseqf}= \overline{\LL_{\seq{f},\expr{w}}} \circ \LL_{\seq{e},\expr{w}}:B_{\expr{x}} \longrightarrow B_{\expr{y}}(\dfct(\seq{e})+\dfct(\seq{f})) \text{,}
\end{equation}
where $\expr{w}$ is the rex corresponding to $w$. 
We write $\glsuseri*{LLdblbexprxbbexpryb}$ to denote a complete selection of double leaves morphisms $B_{\expr{x}} \rightarrow B_{\expr{y}}$ over all such pairs of subsequences. 
Regardless of the realization of $\W$ and any choices made during the construction of the light leaves morphisms, we have the following theorem.

\begin{thm}[{\cite[Theorem~6.12]{ew-soergelcalc}}]
\label{thm:LLdblbasis}
Let $\expr{x},\expr{y} \in \expr{\Ss}$. 
The double leaves morphisms $\LLdbl_{[\expr{x}]}^{[\expr{y}]}$ form a basis for 
\begin{equation*}
\Hom_{\dgrm}^{\bullet}(B_{\expr{x}},B_{\expr{y}})=\Hom_{\dgrm}^{\bullet}(B_{\expr{x}},B_{\expr{y}})
\end{equation*}
as a graded left (or right) $R$-module.
\end{thm}

There is also a light leaves variant of this basis result. 
Let $w \in \W$, and take $I_{\ngeq w}=\{z \in \W : z \ngeq w\} \subseteq \W$ an ideal in the Bruhat order (i.e.~$z \in I$ and $y \leq z$ implies $y \in I_{\ngeq w}$). 
Let $\LLdbl_{\ngeq w}$ denote the span of the $\LLdbl$ maps $\{\LLdbl_{\seq{e}}^{\seq{f}} : \hat{e}=\hat{f} \in I_{\ngeq w}\}$. 
In other words, $\LLdbl_{\ngeq w}$ is spanned by double leaves morphisms which factor through $I_{\ngeq w}$. 
By \cite[Claim~6.19]{ew-soergelcalc} $\LLdbl_{\ngeq w}$ is a $2$-sided ideal of morphisms in $\dgrmBS$. 
We define $\dgrm_{\rm BS}^{\geq w}=\dgrmBS/\LLdbl_{\ngeq w}$ and write $\dgrm^{\geq w}$ for the additive graded Karoubi envelope of $\dgrm_{\rm BS}^{\geq w}$. 
In \cite[{\S 6.5}]{ew-soergelcalc} Elias--Williamson show that the light leaves morphisms give a basis for certain $\Hom$-spaces in $\dgrm_{\rm BS}^{\geq w}$.

\begin{thm} \label{thm:LLbasis}
Let $\expr{x},\expr{w} \in \expr{\Ss}$ with $\expr{w}$ a rex. 
The light leaves morphisms $\LL_{[\expr{x}],\expr{w}}$ form a basis for $\Hom_{\dgrm_{\rm BS}^{\geq w}}^{\bullet}(B_{\expr{x}},B_{\expr{w}})$ as a graded left (or right) $R$-module.
\end{thm}

Another consequence of the double leaves basis is the following classification of the indecomposable Soergel bimodules.

\begin{thm}[{\cite[Theorem~6.26]{ew-soergelcalc}}]
\label{thm:DBSindecomp}
Suppose $w \in \W$, and let $\expr{w}$ be a rex for $w$. 
There is a unique indecomposable summand $\glsuseri*{Bw}$ of $B_{\expr{w}}$ which is not a summand of $B_{\expr{y}}$ for $\expr{y}$ a rex with $y<w$. 
Up to isomorphism, the object $B_w$ does not depend on the choice of rex for $w$. 
Every indecomposable Soergel bimodule is isomorphic to a grade shift of $B_w$ for some $w \in \W$.
\end{thm}

\begin{exam} \label{exam:singgenindecomp}
Let $s \in \Ss$. 
By Theorem~\ref{thm:LLdblbasis} the endomorphisms
\begin{align*}
\begin{gathered}
    \begin{tikzpicture}[scale=0.5,baseline=(origin)]
        \coordinate (origin) at (0,0);
        \node[circle, minimum size=0.5cm] (spacer) at (0,0) {};
        \path (origin) ++(0,1) coordinate (cod);
        \path (origin) ++(0,-1) coordinate (dom);
        \path (origin) ++(0,0.5) coordinate (id);
        \path (origin) ++(0,-0.5) coordinate (di);
        \draw[red, string] (dom) to (cod);
    \end{tikzpicture}
\end{gathered} \text{,}& &  \begin{gathered}
    \begin{tikzpicture}[scale=0.5,baseline=(spacer)]
        \node[circle, minimum size=0.5cm] (spacer) at (0,0) {};
        \coordinate (above) at (0,1);
        \coordinate (below) at (0,-1);
        \node[dot, fill=red] (updot) at (0,0.5) {};
        \node[dot, fill=red] (downdot) at (0,-0.5) {};
%
        \draw[red, string] (above) to (updot.center)
                           (below) to (downdot.center);
    \end{tikzpicture}
\end{gathered} &
\end{align*}
form a left $R$-module basis for $\End_{\dgrm}^{\bullet}(B_{\expr{s}})$. 
It follows that $\End_{\dgrm}^{\bullet}(B_{\expr{s}}) \iso R[x]/(x^2-a_s x)$ as an $R$-module and as a $\field$-algebra, with $x$ corresponding to the second basis element above. 
As $x$ is of degree $2$, this means that $\End_{\dgrm}^0(B_{\expr{s}}) \iso \field$. 
Thus $\End_{\dgrm}^{\bullet}(B_{\expr{s}})$ contains no non-trivial idempotents, so $B_{\expr{s}}$ is indecomposable and hence $B_s=B_{\expr{s}}$.
\end{exam}

For $B$ a Soergel bimodule, write $\gls*{bBb}$ to denote the isomorphism class of $B$, and $\glsuseri*{bdgrmb}$ for the Grothendieck ring of $\dgrm$. 
Since $\dgrm$ has grade shifts, $[\dgrm]$ has the structure of a $\laur$-algebra, with $v=[R(1)]$. 
Theorem~\ref{thm:LLbasis} tells us that for any rex $\expr{w}$ and any Soergel bimodule $B$, the left $R$-module $\Hom_{\dgrm_{\rm BS}^{\geq w}}^{\bullet}(B,B_{\expr{w}})$ is a graded projective $R$-module. 
Yet $R$ is a polynomial algebra over a field $\field$, so in fact all projective $R$-modules are free. 

The \defnemph{diagrammatic character} \cite[Definition 6.24]{ew-soergelcalc} is the $\laur$-linear map
\begin{align*}
\gls*{charctr} : [\dgrm] & \longrightarrow \hecke \\
[B] & \longmapsto \sum_{w \in \W} \grk \Hom_{\dgrm^{\geq w}}^{\bullet}(B,B_{w})H_w \text{,}
\end{align*}
where $\grk$ denotes the graded rank as a free $R$-module.
A combination of Lemma~\ref{lem:ddfhecke}, Theorem~\ref{thm:LLbasis}, and Theorem~\ref{thm:DBSindecomp} then yields the following categorification result.

\begin{cor}[{\cite[Corollary 6.27]{ew-soergelcalc}}] \label{cor:charctrheckeiso}
The map $\charctr:[\dgrm] \rightarrow \hecke$ is an isomorphism of $\laur$-algebras.
\end{cor}


\subsection{Mixed diagrams and localization}

\label{sec:localization}
Let $\gls*{Q}$ be the homogeneous fraction field of $R$. 
In other words, $Q$ is the localization of $R$ with respect to all non-zero homogeneous elements. 
We write $Q \otimes \widetilde{\dgrm}_{\rm BS}$ for the (left) scalar extension of $\widetilde{\dgrm}_{\rm BS}$ to a $Q$-linear category. 
In $Q \otimes \widetilde{\dgrm}_{\rm BS}$ diagrams can have elements of $Q$ as left coefficients. 
Since we can ``push'' coefficients through strings in an $\Ss$-diagram using the relation \cite[(10.8d)]{emtw}, this is equivalent to the right scalar extension by $Q$, or indeed the extension of $\widetilde{\dgrm}_{\rm BS}$ where coefficients in $Q$ are allowed in \emph{any} region of a diagram.
For this reason $Q \otimes \widetilde{\dgrm}_{\rm BS}$ inherits the monoidal structure from $\widetilde{\dgrm}_{\rm BS}$.
For convenience we will abuse notation and generally write the object $Q \otimes B_{\expr{x}}$ in $Q \otimes \widetilde{\dgrm}_{\rm BS}$ as $B_{\expr{x}}$ whenever it is unambiguous.
We next construct a certain extension of this category which is closed under direct summands.


\begin{defn}[{\cite[{\S 5.4}]{ew-soergelcalc}}]
Temporarily abusing notation, the category $\glsuseri*{QotimesdgrmtildeBSstd}$ is the following $Q$-linear extension of $Q \otimes \widetilde{\dgrm}_{\rm BS}$.
\begin{itemize}
\item For each $\expr{x} \in \expr{\Ss}$ add the object $\glsuseri*{Qexprx}$, which is called a \defnemph{(diagrammatic) standard bimodule}. 
As with Bott--Samelson bimodules the tensor product of standard bimodules is defined by $Q_{\expr{x}} \otimes Q_{\expr{y}}=Q_{\expr{x y}}$.

\item 
For any $s \in \Ss$ we represent the identity morphism on $Q_{\expr{s}}$ with a dashed $s$-colored vertical string. 
We add two new morphisms between Bott--Samelson bimodules and standard bimodules, which are both of degree $1$:
\begin{align*}
\begin{gathered}
\glsuseri*{bivalents} : B_s \longrightarrow Q_s \text{,} \\
\begin{gathered}
    \begin{tikzpicture}[baseline=(origin)]
        \coordinate (origin) at (0,0);
        \node[circle, minimum size=0.5cm] (spacer) at (origin) {};
        \path (origin) ++(0,0.5) coordinate (cod);
        \path (origin) ++(0,-0.5) coordinate (dom);
        \draw[red, std] (cod) to (origin);
        \draw[red, string] (origin) to (dom);
    \end{tikzpicture}
\end{gathered}
\end{gathered}
 & & 
\begin{gathered}
\overline{{\rm bivalent}_s} : Q_s \longrightarrow B_s \text{.} \\
\begin{gathered}
    \begin{tikzpicture}[baseline=(origin)]
        \coordinate (origin) at (0,0);
        \node[circle, minimum size=0.5cm] (spacer) at (origin) {};
        \path (origin) ++(0,0.5) coordinate (cod);
        \path (origin) ++(0,-0.5) coordinate (dom);
        \draw[red, std] (dom) to (origin);
        \draw[red,string] (origin) to (cod);
    \end{tikzpicture}
\end{gathered}
\end{gathered} &
\end{align*}

\item 
We add the following relations involving ${\rm bivalent}_s$ and $\overline{{\rm bivalent}_s}$:
\input{fig/mixedrelations.tex} \unskip
\end{itemize}
\end{defn}

We call a string diagram arising from compositions and tensor products of $\Ss$-diagrams and the new morphisms above a \defnemph{mixed diagram}.
We will also use the following shorthand in our mixed diagrams.
For each $s \in \Ss$, we define the \defnemph{$s$-colored standard cap} and \defnemph{cup}
\begin{align*}
\begin{gathered}
\glsuseri*{stdcaps} : Q_{\expr{s s}} \longrightarrow Q_{\emptexpr} \text{,} \\
\begin{gathered}
    \begin{tikzpicture}[scale=0.5,baseline=(origin)]
        \coordinate (origin) at (0,0);
        \path (origin) ++(0.5,0) coordinate (right);
        \path (origin) ++(-0.5,0) coordinate (left);
        \path (left) ++(0,-0.5) coordinate (leftid);
        \path (right) ++(0,-0.5) coordinate (rightid);
        \path (left) ++(0,-1) coordinate (belowleft);
        \path (right) ++(0,-1) coordinate (belowright) {};
        \draw[red,std] (belowleft) to (leftid) to (left) 
                                      arc [start angle=180, end angle=0, radius=0.5]
                                        to (right) to (rightid) to (belowright);
    \end{tikzpicture}
\end{gathered}
=
\begin{gathered}
    \begin{tikzpicture}[scale=0.5,baseline=(origin)]
        \coordinate (origin) at (0,0);
        \path (origin) ++(0.5,0) coordinate (right);
        \path (origin) ++(-0.5,0) coordinate (left);
        \path (left) ++(0,-0.25) coordinate (leftid);
        \path (right) ++(0,-0.25) coordinate (rightid) node[right=2pt] {$a^{-1}_{\color{red} s}$};
        \path (left) ++(0,-1) coordinate (belowleft);
        \path (right) ++(0,-1) coordinate (belowright) {};
        \draw[red,string] (leftid) to (left) arc [start angle=180, end angle=0, radius=0.5]
                    to (right) to (rightid);
        \draw[red,std] (rightid) to (belowright)
                          (leftid) to (belowleft);
    \end{tikzpicture}
\end{gathered}
\end{gathered} & & 
\begin{gathered}
\glsuseri*{stdcups} : Q_{\emptexpr} \longrightarrow Q_{\expr{s s}} \text{,} \\
\begin{gathered}
    \begin{tikzpicture}[xscale=0.5,yscale=-0.5,baseline=(origin)]
        \coordinate (origin) at (0,0);
        \path (origin) ++(0.5,0) coordinate (right);
        \path (origin) ++(-0.5,0) coordinate (left);
        \path (left) ++(0,-0.5) coordinate (leftid);
        \path (right) ++(0,-0.5) coordinate (rightid);
        \path (left) ++(0,-1) coordinate (belowleft);
        \path (right) ++(0,-1) coordinate (belowright) {};
        \draw[red,std] (belowleft) to (leftid) to (left) 
                                      arc [start angle=180, end angle=0, radius=0.5]
                                        to (right) to (rightid) to (belowright);
    \end{tikzpicture}
\end{gathered}
=
\begin{gathered}
    \begin{tikzpicture}[xscale=0.5,yscale=-0.5,baseline=(origin)]
        \coordinate (origin) at (0,0);
        \path (origin) ++(0.5,0) coordinate (right);
        \path (origin) ++(-0.5,0) coordinate (left);
        \path (left) ++(0,-0.25) coordinate (leftid);
        \path (right) ++(0,-0.25) coordinate (rightid) node[right=2pt] {$a^{-1}_{\color{red} s}$};
        \path (left) ++(0,-1) coordinate (belowleft);
        \path (right) ++(0,-1) coordinate (belowright) {};
        \draw[red,string] (leftid) to (left) arc [start angle=180, end angle=0, radius=0.5]
                    to (right) to (rightid);
        \draw[red,std] (rightid) to (belowright)
                          (leftid) to (belowleft);
    \end{tikzpicture}
\end{gathered}
\end{gathered} & 
\end{align*}
which extends to ${\rm stdcap}_{\expr{w}}$ and ${\rm stdcup}_{\expr{w}}$ for arbitrary $\expr{w} \in \expr{\Ss}$ in an analogous manner to ${\rm cup}_{\expr{w}}$ and ${\rm cap}_{\expr{w}}$. 
For all $(s,t) \in \Ss \times \Ss$ with $s \neq t$ and $m_{st}<\infty$, we also define the \defnemph{$s$-and-$t$-colored standard braid}:
\begin{gather*}
\glsuseri*{stdbraidst} : Q_{\underbrace{\scriptstyle\expr{sts \dotsm }}_{m_{st}}} \longrightarrow Q_{\underbrace{\scriptstyle\expr{tst \dotsm }}_{m_{st}}} \\
\begin{aligned}
\begin{gathered}
\begin{tikzpicture}[scale=0.9,baseline=(origin.center),font=\scriptsize]
\coordinate (origin) at (0,0) {} node[rectangle, minimum width=4.5cm] {};
\path (origin) ++(-2,-1) coordinate (dom1) {} ++(0,-0.25) coordinate (below1) {} ++(0,-0.42) coordinate (stddom1) {};
\path (origin) ++(-1.25,-1) coordinate (dom2) {} ++(0,-0.25) coordinate (below2) {} ++(0,-0.42) coordinate (stddom2) {};
\path (origin) ++(-0.5,-1) coordinate (dom3) {} ++(0,-0.25) coordinate (below3) {} ++(0,-0.42) coordinate (stddom3) {};
\path (origin) ++(2,-1) coordinate (domlast) {} ++(0,-0.25) coordinate (belowlast) {} ++(0,-0.42) coordinate (stddomlast) {};
\path (origin) ++(-2,1) coordinate (cod1) {} ++(0,0.25) coordinate (above1) {} ++(0,0.42) coordinate (stdcod1) {};
\path (origin) ++(-1.25,1) coordinate (cod2) {} ++(0,0.25) coordinate (above2) {} ++(0,0.42) coordinate (stdcod2) {};
\path (origin) ++(-0.5,1) coordinate (cod3) {} ++(0,0.25) coordinate (above3) {} ++(0,0.42) coordinate (stdcod3) {};
\path (origin) ++(2,1) coordinate (codlast) {} ++(0,0.25) coordinate (abovelast) {} ++(0,0.42) coordinate (stdcodlast) {};
\draw[red, corner, line width=\superthick, std] (origin) to (dom1) to (below1)
                                   (origin) to (dom3) to (below3)
                                   (origin) to (domlast) to (belowlast)
                                   (origin) to (cod2) to (above2);
\draw[red, line width=\superthick, std] (below1) to (stddom1)
                                        (below3) to (stddom3)
                                        (belowlast) to (stddomlast)
                                        (above2) to (stdcod2);
\draw[blue, corner, line width=\superthick, std] (origin) to (dom2) to (below2)
                                            (origin) to (cod1) to (above1)
                                            (origin) to (cod3) to (above3)
                                            (origin) to (codlast) to (abovelast);
\draw[blue, line width=\superthick, std] (below2) to (stddom2)
                                         (above1) to (stdcod1)
                                         (above3) to (stdcod3)
                                         (abovelast) to (stdcodlast);
\path (origin) ++(0.75,-1.25) node[font=\large] (domdots) {$\dotsm$};
\path (origin) ++(0.75,1.25) node[font=\large] (coddots) {$\dotsm$};
\end{tikzpicture}
\end{gathered}
& =
\begin{gathered}
\begin{tikzpicture}[scale=0.9,baseline=(origin.center),font=\scriptsize]
\coordinate (origin) at (0,0) {} node[rectangle, minimum width=4.5cm] {};
\path (origin) ++(-2,-1) coordinate (dom1) {} ++(0,-0.25) coordinate (below1) {} ++(0,-0.42) coordinate (stddom1) {};
\path (origin) ++(-1.25,-1) coordinate (dom2) {} ++(0,-0.25) coordinate (below2) {} ++(0,-0.42) coordinate (stddom2) {};
\path (origin) ++(-0.5,-1) coordinate (dom3) {} ++(0,-0.25) coordinate (below3) {} ++(0,-0.42) coordinate (stddom3) {};
\path (origin) ++(2,-1) coordinate (domlast) {} ++(0,-0.25) coordinate (belowlast) {} ++(0,-0.42) coordinate (stddomlast) {};
\path (origin) ++(-2,1) coordinate (cod1) {} ++(0,0.25) coordinate (above1) {} node[right] {$a^{-1}_{\color{blue} t}$} ++(0,0.42) coordinate (stdcod1) {};
\path (origin) ++(-1.25,1) coordinate (cod2) {} ++(0,0.25) coordinate (above2) {} node[right] {$a^{-1}_{\color{red} s}$} ++(0,0.42) coordinate (stdcod2) {};
\path (origin) ++(-0.5,1) coordinate (cod3) {} ++(0,0.25) coordinate (above3) {} node[right] {$a^{-1}_{\color{blue} t}$} ++(0,0.42) coordinate (stdcod3) {};
\path (origin) ++(2,1) coordinate (codlast) {} ++(0,0.25) coordinate (abovelast) {} node[right] {$a^{-1}_{\color{blue} t}$} ++(0,0.42) coordinate (stdcodlast) {};
\draw[red, corner, line width=\superthick] (origin) to (dom1) to (below1)
                                   (origin) to (dom3) to (below3)
                                   (origin) to (domlast) to (belowlast)
                                   (origin) to (cod2) to (above2);
\draw[red, line width=\superthick, std] (below1) to (stddom1)
                                        (below3) to (stddom3)
                                        (belowlast) to (stddomlast)
                                        (above2) to (stdcod2);
\draw[blue, corner, line width=\superthick] (origin) to (dom2) to (below2)
                                            (origin) to (cod1) to (above1)
                                            (origin) to (cod3) to (above3)
                                            (origin) to (codlast) to (abovelast);
\draw[blue, line width=\superthick, std] (below2) to (stddom2)
                                         (above1) to (stdcod1)
                                         (above3) to (stdcod3)
                                         (abovelast) to (stdcodlast);
\path (origin) ++(0.75,-1.25) node[font=\large] (domdots) {$\dotsm$};
\path (origin) ++(0.75,1.25) node[font=\large] (coddots) {$\dotsm$};
\end{tikzpicture}
\end{gathered} & & \text{($m_{st}$ odd)} \\
\begin{gathered}
\begin{tikzpicture}[scale=0.9,baseline=(origin.center),font=\scriptsize]
\coordinate (origin) at (0,0) {} node[rectangle, minimum width=4.5cm] {};
\path (origin) ++(-2,-1) coordinate (dom1) {} ++(0,-0.25) coordinate (below1) {} ++(0,-0.42) coordinate (stddom1) {};
\path (origin) ++(-1.25,-1) coordinate (dom2) {} ++(0,-0.25) coordinate (below2) {} ++(0,-0.42) coordinate (stddom2) {};
\path (origin) ++(-0.5,-1) coordinate (dom3) {} ++(0,-0.25) coordinate (below3) {} ++(0,-0.42) coordinate (stddom3) {};
\path (origin) ++(2,-1) coordinate (domlast) {} ++(0,-0.25) coordinate (belowlast) {} ++(0,-0.42) coordinate (stddomlast) {};
\path (origin) ++(-2,1) coordinate (cod1) {} ++(0,0.25) coordinate (above1) {} ++(0,0.42) coordinate (stdcod1) {};
\path (origin) ++(-1.25,1) coordinate (cod2) {} ++(0,0.25) coordinate (above2) {} ++(0,0.42) coordinate (stdcod2) {};
\path (origin) ++(-0.5,1) coordinate (cod3) {} ++(0,0.25) coordinate (above3) {} ++(0,0.42) coordinate (stdcod3) {};
\path (origin) ++(2,1) coordinate (codlast) {} ++(0,0.25) coordinate (abovelast) {} ++(0,0.42) coordinate (stdcodlast) {};
\draw[red, corner, line width=\superthick, std] (origin) to (dom1) to (below1)
                                   (origin) to (dom3) to (below3)
                                   (origin) to (cod2) to (above2)
                                   (origin) to (codlast) to (abovelast);
\draw[red, line width=\superthick, std] (below1) to (stddom1)
                                        (below3) to (stddom3)
                                        (above2) to (stdcod2)
                                        (abovelast) to (stdcodlast);
\draw[blue, corner, line width=\superthick, std] (origin) to (dom2) to (below2)
                                            (origin) to (domlast) to (belowlast)
                                            (origin) to (cod1) to (above1)
                                            (origin) to (cod3) to (above3);
\draw[blue, line width=\superthick, std] (below2) to (stddom2)
                                         (belowlast) to (stddomlast)
                                         (above1) to (stdcod1)
                                         (above3) to (stdcod3);
\path (origin) ++(0.75,-1.25) node[font=\large] (domdots) {$\dotsm$};
\path (origin) ++(0.75,1.25) node[font=\large] (coddots) {$\dotsm$};
\end{tikzpicture}
\end{gathered}
& =
\begin{gathered}
\begin{tikzpicture}[scale=0.9,baseline=(origin.center),font=\scriptsize]
\coordinate (origin) at (0,0) {} node[rectangle, minimum width=4.5cm] {};
\path (origin) ++(-2,-1) coordinate (dom1) {} ++(0,-0.25) coordinate (below1) {} ++(0,-0.42) coordinate (stddom1) {};
\path (origin) ++(-1.25,-1) coordinate (dom2) {} ++(0,-0.25) coordinate (below2) {} ++(0,-0.42) coordinate (stddom2) {};
\path (origin) ++(-0.5,-1) coordinate (dom3) {} ++(0,-0.25) coordinate (below3) {} ++(0,-0.42) coordinate (stddom3) {};
\path (origin) ++(2,-1) coordinate (domlast) {} ++(0,-0.25) coordinate (belowlast) {} ++(0,-0.42) coordinate (stddomlast) {};
\path (origin) ++(-2,1) coordinate (cod1) {} ++(0,0.25) coordinate (above1) {} node[right] {$a^{-1}_{\color{blue} t}$} ++(0,0.42) coordinate (stdcod1) {};
\path (origin) ++(-1.25,1) coordinate (cod2) {} ++(0,0.25) coordinate (above2) {} node[right] {$a^{-1}_{\color{red} s}$} ++(0,0.42) coordinate (stdcod2) {};
\path (origin) ++(-0.5,1) coordinate (cod3) {} ++(0,0.25) coordinate (above3) {} node[right] {$a^{-1}_{\color{blue} t}$} ++(0,0.42) coordinate (stdcod3) {};
\path (origin) ++(2,1) coordinate (codlast) {} ++(0,0.25) coordinate (abovelast) {} node[right] {$a^{-1}_{\color{red} s}$} ++(0,0.42) coordinate (stdcodlast) {};
\draw[red, corner, line width=\superthick] (origin) to (dom1) to (below1)
                                   (origin) to (dom3) to (below3)
                                   (origin) to (cod2) to (above2)
                                   (origin) to (codlast) to (abovelast);
\draw[red, line width=\superthick, std] (below1) to (stddom1)
                                        (below3) to (stddom3)
                                        (above2) to (stdcod2)
                                        (abovelast) to (stdcodlast);
\draw[blue, corner, line width=\superthick] (origin) to (dom2) to (below2)
                                            (origin) to (domlast) to (belowlast)
                                            (origin) to (cod1) to (above1)
                                            (origin) to (cod3) to (above3);
\draw[blue, line width=\superthick, std] (below2) to (stddom2)
                                         (belowlast) to (stddomlast)
                                         (above1) to (stdcod1)
                                         (above3) to (stdcod3);
\path (origin) ++(0.75,-1.25) node[font=\large] (domdots) {$\dotsm$};
\path (origin) ++(0.75,1.25) node[font=\large] (coddots) {$\dotsm$};
\end{tikzpicture}
\end{gathered} & & \text{($m_{st}$ even)}
\end{aligned}
\end{gather*}

\begin{rem} \hfill 
\label{rem:dgrmBsstdsgnconv}
\begin{enumerate}
\item We warn the reader that the morphisms ${\rm bivalent}_s$ and $\overline{{\rm bivalent}_s}$ are \emph{not} cyclic.
In other words, twisting a bivalent vertex by 180 degrees does not always result in the other bivalent vertex, e.g.~
\begin{equation*}
\begin{gathered}
    \begin{tikzpicture}[xscale=-0.5,yscale=0.5,baseline=(flex)]
        \coordinate (flex) at (0,0);
        \node[circle, minimum size=1.5cm] (spacer) at (flex) {};
        \path (flex) ++(0,0.5) coordinate (above);
        \path (flex) ++(0,-0.5) coordinate (below);
        \path (flex) ++(1,-0.5) coordinate (belowright);
        \path (flex) ++(-1,0.5) coordinate (aboveleft);
        \path (aboveleft) ++(0,-2) coordinate (belowleft);
        \path (belowright) ++(0,2) coordinate (aboveright) {};
        \draw[red,std] (belowleft) to (aboveleft) 
                                      arc [start angle=180, end angle=0, radius=0.5]
                                        to (above) to (flex);
        \draw[red, string] (flex) to (below) arc [start angle=180, end angle=360, radius=0.5]
                                        to (belowright) 
                                        to (aboveright);
    \end{tikzpicture}
\end{gathered}
=
-
\begin{gathered}
    \begin{tikzpicture}[scale=0.5,baseline=(flex)]
        \coordinate (flex) at (0,0);
        \node[circle, minimum size=0.5cm] (spacer) at (flex) {};
        \path (flex) ++(0,-1.5) coordinate (below);
        \path (flex) ++(0,1.5) coordinate (above);
        \draw[std, red, line width=\superthick] (below) to (flex);
        \draw[string, red] (flex) to (above);
    \end{tikzpicture}
\end{gathered}
\neq
\begin{gathered}
    \begin{tikzpicture}[scale=0.5,baseline=(flex)]
        \coordinate (flex) at (0,0);
        \node[circle, minimum size=0.5cm] (spacer) at (flex) {};
        \path (flex) ++(0,-1.5) coordinate (below);
        \path (flex) ++(0,1.5) coordinate (above);
        \draw[std, red, line width=\superthick] (below) to (flex);
        \draw[string, red] (flex) to (above);
    \end{tikzpicture}
\end{gathered}
\text{.}
\end{equation*}
Thankfully the failure of cyclicity is only up to a sign change. 
In particular our sign convention differs from that in \cite{ew-soergelcalc,ew-localizedcalc} by a sign. 

\item The ``$Q$'' in the notation $Q_{\expr{x}}$ for a standard bimodule is a holdover from classical Soergel bimodules, where each standard bimodule is isomorphic to $Q$ as a left $Q$-module. 
This provides us with a mnemonic that this category is defined over $Q$. 
\end{enumerate}
\end{rem}

\begin{defn}
The category $\widetilde{\dgrm}_{\rm std}$ is the graded $R$-linear monoidal subcategory of $Q \otimes \widetilde{\dgrm}_{\rm BS,std}$ defined as follows.
\begin{itemize}
\item The objects of $\widetilde{\dgrm}_{\rm std}$ are the standard bimodules.

\item The morphisms in $\widetilde{\dgrm}_{\rm std}$ are generated by
\begin{itemize}
\item ${\rm stdcap}_s$ and ${\rm stdcup}_s$ for each $s \in \Ss$;

\item ${\rm stdbraid}_{s,t}$ for each $(s,t) \in \Ss \times \Ss$ such that $s \neq t$ and $m_{st}<\infty$.
\end{itemize}
\end{itemize}

The category $\widetilde{\dgrm}_{\rm BS,std}$ is the graded $R$-linear monoidal subcategory of $Q \otimes \widetilde{\dgrm}_{\rm BS,std}$ defined as follows. 
\begin{itemize}
\item Every object in $Q \otimes \widetilde{\dgrm}_{\rm BS,std}$ is an object in $\widetilde{\dgrm}_{\rm BS,std}$.

\item The morphisms in $\widetilde{\dgrm}_{\rm BS,std}$ are generated by
\begin{itemize}
\item all morphisms in the $R$-linear subcategory $\widetilde{\dgrm}_{\rm BS}$;

\item all morphisms in the $R$-linear subcategory $\widetilde{\dgrm}_{\rm std}$;

\item ${\rm bivalent}_s$ and $\overline{{\rm bivalent}_s}$ for each $s \in \Ss$.
\end{itemize}
\end{itemize}
\end{defn}

As with $\widetilde{\dgrm}_{\rm BS}$, we may similarly construct categories $\gls*{dgrmstd}$ and $\gls*{dgrmBSstd}$ from $\widetilde{\dgrm}_{\rm std}$ and $\widetilde{\dgrm}_{\rm BS,std}$, which have a grade shift functor instead of graded $\Hom$-spaces, and we will not usually distinguish between the tilde and non-tilde versions of these categories.
We call $\dgrmstd$ the \defnemph{(diagrammatic) standard category} and $\dgrmBSstd$ the \defnemph{mixed category}.
As our notation suggests, the scalar extension of $\widetilde{\dgrm}_{\rm BS,std}$ to a $Q$-linear category is the category we have labeled $Q \otimes \widetilde{\dgrm}_{\rm BS,std}$, which contains the scalar extension $Q \otimes \widetilde{\dgrm}_{\rm std}$ of $\widetilde{\dgrm}_{\rm std}$.
Abusing notation slightly we will write $Q \otimes \dgrmBSstd$ and $Q \otimes \dgrmstd$ for the versions of these categories with grade shift functors instead of graded $\Hom$-spaces.
For consistency with the convention in Remark~\ref{rem:dgrmBsstdsgnconv}, we label standard bimodules in $\widetilde{\dgrm}_{\rm std}$ and $\widetilde{\dgrm}_{\rm BS,std}$ (and their non-tilde counterparts) by $\glsuseri*{Rexprx}$ for $\expr{x} \in \expr{\Ss}$.
We also note that the duality functor on $\dgrmBS$ carries over to all of the categories above in the obvious way.

We call a string diagram arising from compositions and tensor products of standard caps, cups, and braids a \defnemph{basic standard diagram}. 
A \defnemph{standard diagram} is any $Q$-linear multiple of a basic standard diagram.
The morphisms in $\dgrmstd$ have the following description (see \cite[Theorem~4.8]{ew-soergelcalc} and \cite[Proposition~5.23]{ew-soergelcalc}).

\begin{thm} \label{thm:stdhoms}
Let $\expr{x},\expr{y} \in \expr{\Ss}$.
\begin{enumerate}[label={\rm (\roman*)}]
\item If $x=y$, then $\Hom_{\dgrmstd}^{\bullet}(R_{\expr{x}},R_{\expr{y}}) \iso R$. 
Moreover, all basic standard diagrams in this $\Hom$-space are equal, and span the entire $\Hom$-space as a left or right $R$-module.

\item If $x \neq y$, then $\Hom_{\dgrmstd}^{\bullet}(R_{\expr{x}},R_{\expr{y}})=0$. 
\end{enumerate}
\end{thm}

An immediate consequence of this result is that any basic standard diagram is an isomorphism. 
Thus we may label each standard bimodule up to isomorphism by an element of $\W$ instead of an expression, e.g.~$\glsuseri*{Rx}$.

From \eqref{eq:stdBSstd} and \eqref{eq:BSstdBS} we observe that 
\begin{align}
\begin{gathered}
    \begin{tikzpicture}[baseline=(origin)]
        \coordinate (origin) at (0,0) {};
        \node[right=3pt] (label) at (origin) {$a^{-1}_{\color{red} s}$};
        \path (origin) ++(0,0.875) coordinate (cod);
        \path (origin) ++(0,-0.875) coordinate (dom);
        \path (origin) ++(0,0.375) coordinate (dottop);
        \path (origin) ++(0,-0.375) coordinate (dotbot);
        \node[dot,fill=red] (dot1) at (dottop) {};
        \node[dot,fill=red] (dot2) at (dotbot) {};
        \draw[red, string] (cod) to (dottop) 
                           (dotbot) to (dom);
    \end{tikzpicture}
\end{gathered}  \text{,}& &
\begin{gathered}
    \begin{tikzpicture}[baseline=(origin)]
        \coordinate (origin) at (0,0) {};
        \node[right=3pt] (label) at (origin) {$a^{-1}_{\color{red} s}$};
        \path (origin) ++(0,0.875) coordinate (cod);
        \path (origin) ++(0,-0.875) coordinate (dom);
        \path (origin) ++(0,0.375) coordinate (bivtop);
        \path (origin) ++(0,-0.375) coordinate (bivbot);
        \draw[red, std] (bivtop) to (bivbot);
        \draw[red, string] (cod) to (bivtop) 
                           (bivbot) to (dom);
    \end{tikzpicture}
\end{gathered} &  \label{eq:dotbividempotents}
\end{align}
are complementary idempotent endomorphisms in $Q \otimes \widetilde{\dgrm}_{\rm BS,std}$. 
Thus the Bott--Samelson bimodule $B_{\expr{s}}$ decomposes as a direct sum $Q_{\emptexpr} \oplus Q_{\expr{s}}$. 
To fix such a decomposition, we choose factorizations
\begin{align*}
B_{\expr{s}} \xrightarrow{\pi_{s,0}} Q_{\emptexpr} & \xrightarrow{\iota_{s,0}} B_{\expr{s}} \\
B_{\expr{s}} \xrightarrow{\pi_{s,1}} Q_{\expr{s}} & \xrightarrow{\iota_{s,1}} B_{\expr{s}}
\end{align*}
for the idempotents in \eqref{eq:dotbividempotents} following \cite[\S 5.4]{ew-soergelcalc} (see also \cite[\S 4]{ew-localizedcalc}):
\begin{align}
\pi_{{\color{red} s},0}& =\begin{gathered}
    \begin{tikzpicture}[baseline=(origin)]
        \coordinate (origin) at (0,0) {};
        \node[right=3pt] (label) at (origin) {$a^{-1}_{\color{red} s}$};
        \path (origin) ++(0,0.5) coordinate (cod);
        \path (origin) ++(0,-0.5) coordinate (dom);
        \node[dot,fill=red] (dot) at (origin) {};
        \draw[red, string] (origin) to (dom);
    \end{tikzpicture}
\end{gathered}\text{,} & 
\pi_{{\color{red} s},1}& =
\begin{gathered}
    \begin{tikzpicture}[baseline=(origin)]
        \coordinate (origin) at (0,0) {};
        \node[right=3pt] (label) at (origin) {$a^{-1}_{\color{red} s}$};
        \path (origin) ++(0,0.5) coordinate (cod);
        \path (origin) ++(0,-0.5) coordinate (dom);
        \draw[red, std] (cod) to (origin);
        \draw[red, string] (origin) to (dom);
    \end{tikzpicture}
\end{gathered}\text{,} \label{eq:locprojsingle} \\
\iota_{{\color{red} s},0}& =\begin{gathered}
    \begin{tikzpicture}[baseline=(origin)]
        \coordinate (origin) at (0,0) {};
        \node[right=3pt] (label) at (origin) {$1$};
        \path (origin) ++(0,0.5) coordinate (cod);
        \path (origin) ++(0,-0.5) coordinate (dom);
        \node[dot,fill=red] (dot) at (origin) {};
        \draw[red,string] (origin) to (cod);
    \end{tikzpicture}
\end{gathered} \text{,} & 
\iota_{{\color{red} s},1}& =\begin{gathered}
    \begin{tikzpicture}[baseline=(origin)]
        \coordinate (origin) at (0,0) {};
        \node[right=3pt] (label) at (origin) {$1$};
        \path (origin) ++(0,0.5) coordinate (cod);
        \path (origin) ++(0,-0.5) coordinate (dom);
        \draw[red, std] (dom) to (origin);
        \draw[red,string] (origin) to (cod);
    \end{tikzpicture}
\end{gathered} \text{.} \label{eq:locinclsingle}
\end{align}
For $\expr{x} \in \expr{\Ss}$ and $\seq{e}=(\seq{e}_1,\seq{e}_2,\dotsc,\seq{e}_m) \in [\expr{x}]$, we write 
\begin{equation*}
Q_{\seq{e}}= \bigotimes_{i=1}^m \begin{cases}
Q_{\emptexpr} & \text{if $\seq{e}_i=(s_i,0),$} \\
Q_{s_i} & \text{if $\seq{e}_i=(s_i,1)$}
\end{cases}
\end{equation*}
in $Q \otimes \dgrmBSstd$. 
It is clear that $Q_{\seq{e}} \iso Q_{\hat{e}}$.
Similarly we write 
\begin{align*}
\gls*{pie} &= 
\bigotimes_{i=1}^m \pi_{\seq{e}_i} & 
\iota_{\seq{e}} &= 
\bigotimes_{i=1}^m \iota_{\seq{e}_i}
\end{align*}
which are morphisms $B_{\expr{x}} \longrightarrow Q_{\patt{r}}$ and $Q_{\seq{e}} \longrightarrow B_{\expr{x}}$ respectively in $Q \otimes \dgrmBSstd$.
By induction, for $\expr{x} \in \expr{\Ss}$ we obtain a fixed decomposition in $Q \otimes \widetilde{\dgrm}_{\rm BS,std}$ (see also \cite[\S 5.5]{ew-soergelcalc}):
\begin{equation}
B_{\expr{x}} \iso \bigoplus_{\seq{e} \in [\expr{x}]} Q_{\seq{e}} \text{.}
\end{equation}

For any morphism $\phi:B_{\expr{x}} \rightarrow B_{\expr{y}}$ in $\dgrmBS$, the $Q$-scalar extension of $\phi$ decomposes into a $[\expr{y}] \times [\expr{x}]$ matrix of standard morphisms called the \defnemph{localization matrix}.
The $(\seq{f},\seq{e})$-entry in the localization matrix is the composition
\begin{equation*}
Q_{\seq{e}} \xrightarrow{\iota_{\seq{e}}} B_{\expr{x}} \xrightarrow{\phi} B_{\expr{y}} \xrightarrow{\pi_{\seq{f}}} Q_{\seq{f}} \text{.}
\end{equation*}
Diagrammatically, calculating these entries involves adding dots or bivalent vertices to each of the generators in the domain and codomain. 
We say that a generator has been \defnemph{standardized} if a bivalent vertex has been added to it (i.e.~it corresponds to a subsequence term of type $1$).
The most important fact about localization is that it is faithful, i.e.~two morphisms in $\dgrmBS$ are equal if and only if they have the same localization matrices \cite[{\S 5.5}]{ew-soergelcalc}.
When it is clear from context, we will omit the basic standard diagrams and only write the left scalar coefficients in the entries of the localization matrix.


\begin{exam} \label{exam:localizationsingle} \hfill
\begin{enumerate}
\item The localization matrix of ${\rm dot}_{\color{red} s}$ is
\begin{equation*}
\begin{gathered}
    \begin{tikzpicture}[baseline=(origin)]
        \coordinate (origin) at (0,0) {};
        \path (origin) ++(0,0.5) coordinate (cod);
        \path (origin) ++(0,-0.5) coordinate (dom);
        \node[dot,fill=red] (dot) at (origin) {};
        \draw[red, string] (origin) to (dom);
    \end{tikzpicture}
\end{gathered}
\end{equation*}
is 
\begin{equation*}
\begin{bmatrix}
\begin{gathered}
    \begin{tikzpicture}[baseline=(origin)]
        \coordinate (origin) at (0,0) {};
        \path (origin) ++(0,-0.5) coordinate (biv);
        \path (biv)    ++(0,-0.5) coordinate (dom);
        \node[dot,fill=red] (dot) at (origin) {};
        \draw[red, string] (origin) to (biv);
        \draw[red,std] (biv) to (dom);
        \node[right=3pt] (label) at (biv) {};
    \end{tikzpicture}
\end{gathered} & 
\begin{gathered}
    \begin{tikzpicture}[baseline=(origin)]
        \coordinate (origin) at (0,0) {};
        \path (origin) ++(0,-0.5) coordinate (origin2);
        \node[dot,fill=red] (dot) at (origin) {};
        \node[dot,fill=red] (dot2) at (origin2) {};
        \draw[red, string] (origin) to (origin2);
        \node[right=3pt] (label) at (dot2) {};
    \end{tikzpicture}
\end{gathered}
\end{bmatrix}=\begin{bmatrix}
0 & a_{\color{red} s}
\end{bmatrix} \text{.}
\end{equation*}

\item The localization matrix of $\overline{{\rm dot}_{\color{red} s}}$ is
\begin{equation*}
\begin{gathered}
    \begin{tikzpicture}[baseline=(origin),yscale=-1]
        \coordinate (origin) at (0,0) {};
        \path (origin) ++(0,0.5) coordinate (cod);
        \path (origin) ++(0,-0.5) coordinate (dom);
        \node[dot,fill=red] (dot) at (origin) {};
        \draw[red, string] (origin) to (dom);
    \end{tikzpicture}
\end{gathered}
\end{equation*}
is 
\begin{equation*}
\begin{bmatrix}
\begin{gathered}
    \begin{tikzpicture}[baseline=(origin),yscale=-1]
        \coordinate (origin) at (0,0) {};
        \path (origin) ++(0,-0.5) coordinate (biv);
        \path (biv)    ++(0,-0.5) coordinate (dom);
        \node[dot,fill=red] (dot) at (origin) {};
        \draw[red, string] (origin) to (biv);
        \draw[red,std] (biv) to (dom);
        \node[right=3pt] (label) at (biv) {$a^{-1}_{\color{red} s}$};
    \end{tikzpicture}
\end{gathered} \\
\begin{gathered}
    \begin{tikzpicture}[baseline=(origin),yscale=-1]
        \coordinate (origin) at (0,0) {};
        \path (origin) ++(0,-0.5) coordinate (origin2);
        \node[dot,fill=red] (dot) at (origin) {};
        \node[dot,fill=red] (dot2) at (origin2) {};
        \draw[red, string] (origin) to (origin2);
        \node[right=3pt] (label) at (dot2) {$a^{-1}_{\color{red} s}$};
    \end{tikzpicture}
\end{gathered}
\end{bmatrix}=\begin{bmatrix}
0 \\
1
\end{bmatrix} \text{.}
\end{equation*}
\end{enumerate}
\end{exam}

\subsection{Dual-valued scalars}

The previous example shows that localization does not behave well with respect to dualization.
This is because the projection and inclusion maps in \eqref{eq:locprojsingle}--\eqref{eq:locinclsingle} are not dual to each other.
Moreover, we cannot avoid this problem by choosing a different factorization of the idempotents \eqref{eq:dotbividempotents}, because the scalar factor $a^{-1}_{s}$ is not a square in $Q$.
To remedy this inconvenience and to avoid drawing diagrams twice, we will adopt the following notation throughout.

\begin{notn}
Let $f,g \in Q$. 
We write $\gls*{binomfg}$ in a region of a mixed diagram to indicate a \defnemph{dual-valued} scalar, i.e.~a possible choice between the scalars $f$ or $g$ in this region, depending on whether the diagram is viewed right-side-up or upside-down.
We call a mixed diagram with a region containing one or more dual-valued scalars a dual-valued diagram.
We write expressions with dual-valued diagrams as shorthand for \emph{two} expressions.
When the expression (and all the diagrams therein) is viewed ``right-side-up'', we choose the upper scalars $f_1,f_2,\dotsc$ for each dual-valued scalar $\binom{f_1}{g_1},\binom{f_2}{g_2},\dotsc$; when the expression is viewed upside-down, we choose the lower scalars $g_1,g_2,\dots$ instead.
In this manner the notation closely resembles the symbols $\pm$ and $\mp$.
For example, if we write
\begin{equation*}
\pi_{{\color{red} s},0}=\begin{gathered}
    \begin{tikzpicture}[baseline=(origin)]
        \coordinate (origin) at (0,0) {};
        \node[right=3pt] (label) at (origin) {$\binom{a^{-1}_{\color{red} s}}{1}$};
        \path (origin) ++(0,0.5) coordinate (cod);
        \path (origin) ++(0,-0.5) coordinate (dom);
        \node[dot,fill=red] (dot) at (origin) {};
        \draw[red, string] (origin) to (dom);
    \end{tikzpicture}
\end{gathered}\text{,}
\end{equation*}
this is shorthand for setting
\begin{equation*}
\pi_{{\color{red} s},0}=\begin{gathered}
    \begin{tikzpicture}[baseline=(origin)]
        \coordinate (origin) at (0,0) {};
        \node[right=3pt] (label) at (origin) {$a^{-1}_{\color{red} s}$};
        \path (origin) ++(0,0.5) coordinate (cod);
        \path (origin) ++(0,-0.5) coordinate (dom);
        \node[dot,fill=red] (dot) at (origin) {};
        \draw[red, string] (origin) to (dom);
    \end{tikzpicture}
\end{gathered}\text{,}
\qquad \text{and} \qquad
\overline{\pi_{{\color{red} s},0}}=\begin{gathered}
    \begin{tikzpicture}[baseline=(origin),yscale=-1]
        \coordinate (origin) at (0,0) {};
        \node[right=3pt] (label) at (origin) {$1$};
        \path (origin) ++(0,0.5) coordinate (cod);
        \path (origin) ++(0,-0.5) coordinate (dom);
        \node[dot,fill=red] (dot) at (origin) {};
        \draw[red, string] (origin) to (dom);
    \end{tikzpicture}
\end{gathered} \text{.}
\end{equation*}
Note the abuse of notation here: the single-valued diagram $\overline{\pi_{s,0}}$ above is \emph{not} the dual of the single-valued diagram $\pi_{s,0}$, but a separate diagram to be used \emph{instead} of the dual of $\pi_{s,0}$.
\end{notn}

It is easiest whenever possible to perform all calculations in the dual-valued setting, and to subsequently interpret any result as a single-valued diagram if needed.
This makes use of some obvious rules of dual-valued arithmetic, e.g.~$\binom{f_1}{g_1}+\binom{f_2}{g_2}=\binom{f_1+f_2}{g_1+g_2}$, $\binom{f_1}{g_1}\binom{f_2}{g_2}=\binom{f_1f_2}{g_1g_2}$, etc.
It is clear that a single-valued scalar $f$ in any region of a dual-valued diagram is equivalent to the dual-valued scalar $\binom{f}{f}$.
We define the dual of a dual-valued diagram to be the dual-valued diagram obtained by flipping the diagram upside-down and replacing every dual-valued scalar $\binom{f}{g}$ with $\overline{\binom{f}{g}}=\binom{g}{f}$.
For homogeneous $f,g$ we also set $\deg \binom{f}{g}=\frac{1}{2}(\deg f+\deg g)$.

Using our new notation, the following dual-valued diagrams
\begin{align}
\glsuseri*{pisubs}& =\begin{gathered}
    \begin{tikzpicture}[baseline=(origin)]
        \coordinate (origin) at (0,0) {};
        \node[right=3pt] (label) at (origin) {$\binom{a^{-1}_{\color{red} s}}{1}$};
        \path (origin) ++(0,0.5) coordinate (cod);
        \path (origin) ++(0,-0.5) coordinate (dom);
        \node[dot,fill=red] (dot) at (origin) {};
        \draw[red, string] (origin) to (dom);
    \end{tikzpicture}
\end{gathered}\text{,} &
\pi_{{\color{red} s},1} & =\begin{gathered}
    \begin{tikzpicture}[baseline=(origin)]
        \coordinate (origin) at (0,0) {};
        \node[right=3pt] (label) at (origin) {$\binom{a^{-1}_{\color{red} s}}{1}$};
        \path (origin) ++(0,0.5) coordinate (cod);
        \path (origin) ++(0,-0.5) coordinate (dom);
        \draw[red, std] (cod) to (origin);
        \draw[red, string] (origin) to (dom);
    \end{tikzpicture}
\end{gathered}\text{,}  \label{eq:locproj} \\
\overline{\pi_{{\color{red} s},0}}& =\begin{gathered}
    \begin{tikzpicture}[baseline=(origin)]
        \coordinate (origin) at (0,0) {};
        \node[right=3pt] (label) at (origin) {$\binom{1}{a^{-1}_{\color{red} s}}$};
        \path (origin) ++(0,0.5) coordinate (cod);
        \path (origin) ++(0,-0.5) coordinate (dom);
        \node[dot,fill=red] (dot) at (origin) {};
        \draw[red,string] (origin) to (cod);
    \end{tikzpicture}
\end{gathered} \text{,} & 
\overline{\pi_{{\color{red} s},1}}& =\begin{gathered}
    \begin{tikzpicture}[baseline=(origin)]
        \coordinate (origin) at (0,0) {};
        \node[right=3pt] (label) at (origin) {$\binom{1}{a^{-1}_{\color{red} s}}$};
        \path (origin) ++(0,0.5) coordinate (cod);
        \path (origin) ++(0,-0.5) coordinate (dom);
        \draw[red, std] (dom) to (origin);
        \draw[red,string] (origin) to (cod);
    \end{tikzpicture}
\end{gathered} \text{,} \label{eq:locincl}
\end{align}
can be used to calculate the localization of a morphism and its dual at the same time.
We define the \defnemph{dual-valued localization matrix} in an analogous manner to its single-valued counterpart, but using \eqref{eq:locproj}--\eqref{eq:locincl} instead of \eqref{eq:locprojsingle}--\eqref{eq:locinclsingle}.


\begin{exam} \label{exam:localization}
The dual-valued localization matrix of the dot morphism
\begin{equation*}
\begin{gathered}
    \begin{tikzpicture}[baseline=(origin)]
        \coordinate (origin) at (0,0) {};
        \path (origin) ++(0,0.5) coordinate (cod);
        \path (origin) ++(0,-0.5) coordinate (dom);
        \node[dot,fill=red] (dot) at (origin) {};
        \draw[red, string] (origin) to (dom);
    \end{tikzpicture}
\end{gathered}
\end{equation*}
is 
\begin{equation*}
\begin{bmatrix}
\begin{gathered}
    \begin{tikzpicture}[baseline=(origin)]
        \coordinate (origin) at (0,0) {};
        \path (origin) ++(0,-0.5) coordinate (biv);
        \path (biv)    ++(0,-0.5) coordinate (dom);
        \node[dot,fill=red] (dot) at (origin) {};
        \draw[red, string] (origin) to (biv);
        \draw[red,std] (biv) to (dom);
        \node[right=3pt] (label) at (biv) {$\binom{1}{a^{-1}_{\color{red} s}}$};
    \end{tikzpicture}
\end{gathered} & 
\begin{gathered}
    \begin{tikzpicture}[baseline=(origin)]
        \coordinate (origin) at (0,0) {};
        \path (origin) ++(0,-0.5) coordinate (origin2);
        \node[dot,fill=red] (dot) at (origin) {};
        \node[dot,fill=red] (dot2) at (origin2) {};
        \draw[red, string] (origin) to (origin2);
        \node[right=3pt] (label) at (dot2) {$\binom{1}{a^{-1}_{\color{red} s}}$};
    \end{tikzpicture}
\end{gathered}
\end{bmatrix}=\begin{bmatrix}
0 & \binom{a_{\color{red} s}}{1}
\end{bmatrix} \text{.}
\end{equation*}
\end{exam}

Dual-valued localization matrices satisfy the following algebraic properties.

\begin{lem}
Let $\phi$ be a morphism in $\dgrmBS$.
\begin{enumerate}[label={\rm (\roman*)}]
\item If $\phi$ is homogeneous, then all entries of the dual-valued localization matrix of $\phi$ have degree equal to $\deg \phi$.

\item The dual-valued localization matrix of $\overline{\phi}$ is equal to the transpose-dual of the dual-valued localization matrix of $\phi$.
\end{enumerate}
\end{lem}

For the rest of this paper we will use the dual-valued versions of $\pi_{s,0}$ and $\pi_{s,1}$ in \eqref{eq:locproj}--\eqref{eq:locincl}, and we will use the term ``localization matrix'' to refer solely to the dual-valued localization matrix.




We now use $\pi_{s,1}$ to define the following dual-valued morphisms that behave very similarly to ${\rm stdcup}_s$ and ${\rm stdcap}_s$, which will appear extensively in the next section.

\begin{defn}
For $s \in S$ let
\begin{equation*}
\glsuseri*{epsilons} = {\rm cap}_s \circ (\overline{\pi_{s,1}} \otimes \overline{\pi_{s,1}})=
\begin{gathered}
    \begin{tikzpicture}[scale=1,baseline=(origin)]
        \coordinate (origin) at (0,0);
        \path (origin) ++(0.5,0) coordinate (right);
        \path (origin) ++(-0.5,0) coordinate (left);
        \path (left) ++(0,-0.25) coordinate (leftid) node[font=\scriptsize,right=0pt] {$\binom{1}{a^{-1}_{\color{red} s}}$};
        \path (right) ++(0,-0.25) coordinate (rightid) node[font=\scriptsize,right=0pt] {$\binom{1}{a^{-1}_{\color{red} s}}$};
        \path (left) ++(0,-0.75) coordinate (belowleft);
        \path (right) ++(0,-0.75) coordinate (belowright) {};
        \draw[red,string] (leftid) to (left) arc [start angle=180, end angle=0, radius=0.5]
                    to (right) to (rightid);
        \draw[red,std] (rightid) to (belowright)
                          (leftid) to (belowleft);
    \end{tikzpicture}
\end{gathered}=
\begin{gathered}
    \begin{tikzpicture}[scale=1,baseline=(origin)]
        \coordinate (origin) at (0,0);
        \path (origin) ++(0.5,0) coordinate (right);
        \path (origin) ++(-0.5,0) coordinate (left);
        \path (left) ++(0,-0.25) coordinate (leftid);
        \path (right) ++(0,-0.25) coordinate (rightid) node[font=\scriptsize,right=0pt] {$\binom{a_{\color{red} s}}{-a^{-1}_{\color{red} s}}$};
        \path (left) ++(0,-0.75) coordinate (belowleft);
        \path (right) ++(0,-0.75) coordinate (belowright) {};
        \draw[red,std] (belowleft) to (leftid) to (left) 
                                      arc [start angle=180, end angle=0, radius=0.5]
                                        to (right) to (rightid) to (belowright);
    \end{tikzpicture}
\end{gathered}
={\rm stdcap}_s \binom{a_s}{-a^{-1}_s} \text{,}
\end{equation*}
a dual-valued morphism in $Q \otimes \dgrmstd$. 
For $\expr{w}=\expr{st \dotsm} \in \expr{\Ss}$ we similarly define $\epsilon_{\expr{w}}$ in a manner analogous to ${\rm cap}_{\expr{w}}$ or ${\rm stdcap}_{\expr{w}}$.
We call a dual-valued morphism arising from compositions and tensor products of standard braids, $\epsilon_{s}$ and $\overline{\epsilon_{s}}$ for $s \in \Ss$ an \defnemph{$\epsilon$-basic standard diagram}.
\end{defn}

The following lemma shows that $\epsilon_{s}$ behaves well with respect to $\pi_{s,1}$ (see \eqref{eq:locproj}--\eqref{eq:locincl}). 

\begin{lem} \label{lem:epspisubs}
Let $s \in \Ss$.
We have
\begin{align*}
\epsilon_{s} \circ (\ident \otimes \pi_{s,1})& ={\rm cap}_s \circ (\overline{\pi_{s,1}} \otimes \ident) \text{,} & 
\epsilon_{s} \circ (\pi_{s,1} \otimes \ident)& ={\rm cap}_s \circ (\ident \otimes \overline{\pi_{s,1}}) \text{,} \\
(\ident \otimes \overline{\pi_{s,1}}) \circ \overline{\epsilon_{s}}& =(\pi_{s,1} \otimes \ident) \circ {\rm cup}_s \text{,} & 
(\overline{\pi_{s,1}} \otimes \ident) \circ \overline{\epsilon_{s}}& =(\ident \otimes \pi_{s,1}) \circ {\rm cup}_s \text{.}
\end{align*}
\end{lem}

\begin{proof}
We prove the first equation, as the others are very similar.
\begin{equation*}
\begin{gathered}
    \begin{tikzpicture}[scale=1,baseline=(origin)]
        \coordinate (origin) at (0,0);
        \path (origin) ++(0.5,0) coordinate (right);
        \path (origin) ++(-0.5,0) coordinate (left);
        \path (left) ++(0,-0.33) coordinate (leftid) node[font=\scriptsize,right=0pt] {$\binom{1}{a^{-1}_{\color{red} s}}$};
        \path (leftid) ++(0,-0.33) coordinate (leftdi);
        \path (right) ++(0,-0.33) coordinate (rightid) node[font=\scriptsize,above right=0pt] {$\binom{1}{a^{-1}_{\color{red} s}}$};
        \path (rightid) ++(0,-0.33) coordinate (rightdi) node[font=\scriptsize,right=0pt] (pilabel) {$\binom{a^{-1}_{\color{red} s}}{1}$};
        \path (left) ++(0,-1) coordinate (belowleft);
        \path (right) ++(0,-1) coordinate (belowright);
        \draw[red,std] (belowleft) to (leftid) (rightid) to (rightdi);
        \draw[red,string] (leftid) to (left) arc [start angle=180, end angle=0, radius=0.5]
                            to (right) to (rightid) (rightdi) to (belowright);
    \end{tikzpicture}
\end{gathered}=
\begin{gathered}
    \begin{tikzpicture}[scale=1,baseline=(origin)]
        \coordinate (origin) at (0,0);
        \path (origin) ++(0.5,0) coordinate (right);
        \path (origin) ++(-0.5,0) coordinate (left);
        \path (left) ++(0,-0.33) coordinate (leftid) node[font=\scriptsize,right=0pt] {$\binom{1}{a^{-1}_{\color{red} s}}$};
        \path (leftid) ++(0,-0.33) coordinate (leftdi);
        \path (right) ++(0,-0.33) coordinate (rightid);
        \path (rightid) ++(0,-0.33) coordinate (rightdi);
        \path (left) ++(0,-1) coordinate (belowleft);
        \path (right) ++(0,-1) coordinate (belowright);
        \draw[red,std] (belowleft) to (leftid) (rightid) to node[black,font=\scriptsize,right=0pt] {$a^{-1}_{\color{red} s}$} (rightdi);
        \draw[red,string] (leftid) to (left) arc [start angle=180, end angle=0, radius=0.5]
                            to (right) to (rightid) (rightdi) to (belowright);
    \end{tikzpicture}
\end{gathered}=
\begin{gathered}
    \begin{tikzpicture}[scale=1,baseline=(origin)]
        \coordinate (origin) at (0,0);
        \path (origin) ++(0.5,0) coordinate (right);
        \path (origin) ++(-0.5,0) coordinate (left);
        \path (left) ++(0,-0.33) coordinate (leftid) node[font=\scriptsize,right=0pt] {$\binom{1}{a^{-1}_{\color{red} s}}$};
        \path (leftid) ++(0,-0.33) coordinate (leftdi);
        \path (right) ++(0,-0.33) coordinate (rightid);
        \path (rightid) ++(0,-0.33) coordinate (rightdi);
        \path (left) ++(0,-1) coordinate (belowleft);
        \path (right) ++(0,-1) coordinate (belowright);
        \draw[red,std] (belowleft) to (leftid) ;
        \draw[red,string] (leftid) to (left) arc [start angle=180, end angle=0, radius=0.5]
                            to (right) to (rightid) to (rightdi) to (belowright);
    \end{tikzpicture}
\end{gathered}-
\begin{gathered}
    \begin{tikzpicture}[scale=1,baseline=(origin)]
        \coordinate (origin) at (0,0);
        \path (origin) ++(0.5,0) coordinate (right);
        \path (origin) ++(-0.5,0) coordinate (left);
        \path (left) ++(0,-0.33) coordinate (leftid) node[font=\scriptsize,right=0pt] {$\binom{1}{a^{-1}_{\color{red} s}}$};
        \path (leftid) ++(0,-0.33) coordinate (leftdi);
        \path (right) ++(0,-0.33) coordinate (rightid);
        \path (rightid) ++(0,-0.33) coordinate (rightdi);
        \path (left) ++(0,-1) coordinate (belowleft);
        \path (right) ++(0,-1) coordinate (belowright);
        \draw[red,std] (belowleft) to (leftid);
        \node[dot,fill=red] (updot) at (rightid) {};
        \node[dot,fill=red] (downdot) at (rightdi) {};
        \path (rightid) to node[black,font=\scriptsize,right=0pt] {$a^{-1}_{\color{red} s}$} (rightdi);
        \draw[red,string] (leftid) to (left) arc [start angle=180, end angle=0, radius=0.5]
                            to (right) to (rightid) (rightdi) to (belowright);
    \end{tikzpicture}
\end{gathered}
=\begin{gathered}
    \begin{tikzpicture}[scale=1,baseline=(origin)]
        \coordinate (origin) at (0,0);
        \path (origin) ++(0.5,0) coordinate (right);
        \path (origin) ++(-0.5,0) coordinate (left);
        \path (left) ++(0,-0.33) coordinate (leftid) node[font=\scriptsize,right=0pt] {$\binom{1}{a^{-1}_{\color{red} s}}$};
        \path (leftid) ++(0,-0.33) coordinate (leftdi);
        \path (right) ++(0,-0.33) coordinate (rightid);
        \path (rightid) ++(0,-0.33) coordinate (rightdi);
        \path (left) ++(0,-1) coordinate (belowleft);
        \path (right) ++(0,-1) coordinate (belowright);
        \draw[red,std] (belowleft) to (leftid) ;
        \draw[red,string] (leftid) to (left) arc [start angle=180, end angle=0, radius=0.5]
                            to (right) to (rightid) to (rightdi) to (belowright);
    \end{tikzpicture}
\end{gathered} \qedhere
\end{equation*}
\end{proof}

It is then easy to show that $\epsilon_s$ defines a biadjunction, and that a partially standardized braid is cyclic with respect to this biadjunction.

\begin{lem} \hfill \label{lem:epsbiadjbraidcyclic}
\begin{enumerate}[label={\rm (\roman*)}]
\item For any $s \in \Ss$, we have
\begin{align*}
(\epsilon_{s} \otimes \ident_{Q_s}) \circ (\ident_{Q_s} \otimes \overline{\epsilon_{s}})& =\ident_{Q_s} \text{,} & (\ident_{Q_s} \otimes \epsilon_{s}) \circ (\overline{\epsilon_{s}} \otimes \ident_{Q_s}) & =\ident_{Q_s} \text{.}
\end{align*}

\item Suppose $(s,t) \in \Ss \times \Ss$ with $s \neq t$ and $m_{st}<\infty$. If $m_{st}$ is even, then
\begin{multline*}
(\ident \otimes \epsilon_{s}) \circ (\ident_{Q_s} \otimes ((\ident \otimes \pi_{s,1}) \circ {\rm braid}_{s,t} \circ (\overline{\pi_{s,1}} \otimes \ident)) \otimes \ident_{Q_s}) \circ (\overline{\epsilon_{s}} \otimes \ident) \\
 =(\pi_{s,1} \otimes \ident) \circ {\rm braid}_{t,s} \circ (\ident \otimes \overline{\pi_{s,1}}) \text{,} 
\end{multline*}
\begin{multline*}
(\epsilon_{t} \otimes \ident) \circ (\ident_{Q_t} \otimes ((\pi_{t,1} \otimes \ident) \circ {\rm braid}_{s,t} \circ (\ident \otimes \overline{\pi_{t,1}})) \otimes \ident_{Q_t}) \circ (\ident \otimes \overline{\epsilon_{t}}) \\
=(\ident \otimes \pi_{t,1}) \circ {\rm braid}_{t,s} \circ (\overline{\pi_{t,1}} \otimes \ident)\text{.} 
\end{multline*}
If $m_{st}$ is odd, then
\begin{multline*}
(\ident \otimes \epsilon_{t}) \circ (\ident_{Q_s} \otimes ((\ident \otimes \pi_{t,1}) \circ {\rm braid}_{s,t} \circ (\overline{\pi_{s,1}} \otimes \ident)) \otimes \ident_{Q_t}) \circ (\overline{\epsilon_{s}} \otimes \ident) \\
=(\pi_{s,1} \otimes \ident) \circ {\rm braid}_{t,s} \circ (\ident \otimes \overline{\pi_{t,1}}) \text{,} 
\end{multline*}
\begin{multline*}
(\epsilon_{t} \otimes \ident) \circ (\ident_{Q_t} \otimes ((\pi_{t,1} \otimes \ident) \circ {\rm braid}_{s,t} \circ (\ident \otimes \overline{\pi_{s,1}})) \otimes \ident_{Q_s}) \circ (\ident \otimes \overline{\epsilon_{s}}) \\
=(\ident \otimes \pi_{s,1}) \circ {\rm braid}_{t,s} \circ (\overline{\pi_{t,1}} \otimes \ident) \text{.} 
\end{multline*}
\end{enumerate}
\end{lem}

\begin{proof}
To prove the first claim, standardize the usual biadjunction relation in $\dgrmBS$ and apply Lemma~\ref{lem:epspisubs}.
To prove the second claim, standardize one strand in the cyclicity relation for ${\rm braid}_{s,t}$ in $\dgrmBS$ and apply Lemma~\ref{lem:epspisubs}.
\end{proof}

In \cite[Definition~4.3]{ew-soergelcalc} Elias--Williamson give a presentation of $\dgrmstd$. 
The previous result shows that $\epsilon_s$ satisfies many of the same relations that ${\rm stdcap}_s$ does in this presentation.
Since the presentation of $\dgrmstd$ is enough to show that all basic standard diagrams are equal (Theorem~\ref{thm:stdhoms}), a natural question is whether all $\epsilon$-basic standard diagrams are equal too.

\begin{prop} \label{prop:epscupcapinv}
For any $s \in \Ss$, we have
\begin{align*}
\epsilon_{s} \circ \overline{\epsilon_{s}} & =-\ident_Q \text{,} & \overline{\epsilon_{s}} \circ \epsilon_{s} & =-\ident_Q \text{.}
\end{align*}
More generally, all $\epsilon$-basic standard diagrams with the same domain and codomain are equal up to sign. 
%
\end{prop}

\begin{proof}
For any $s \in \Ss$ we check that
\begin{equation*}
\epsilon_{s} \circ \overline{\epsilon_{s}}= 
\begin{gathered}
    \begin{tikzpicture}[scale=0.5,baseline=(origin)]
        \coordinate (origin) at (0,0);
        \path (origin) ++(0.5,0) coordinate (right);
        \path (origin) ++(-0.5,0) coordinate (left);
        \path (right) node[above right=0.1] (uplabel) {$\binom{a_{\color{red} s}}{-a^{-1}_{\color{red} s}}$};
        \path (right) node[below right=0.1] (downlabel) {$\binom{-a^{-1}_{\color{red} s}}{a_{\color{red} s}}$};
        %
        \draw[red,std] (right) arc [start angle=0, end angle=360, radius=0.5];
    \end{tikzpicture}
\end{gathered}
=\begin{gathered}
    \begin{tikzpicture}[scale=0.5,baseline=(origin)]
        \coordinate (origin) at (0,0);
       \node[circle, minimum size=1cm] (spacer) at (origin) {$\binom{-1}{-1}$};
        %
    \end{tikzpicture}
\end{gathered}
=-\ident_Q \text{,}
\end{equation*}
and
\begin{align*}
\overline{\epsilon_{s}} \circ \epsilon_{s} & =\binom{-a^{-1}_s}{a_s} \binom{a_s}{-a^{-1}_s} \ident_{Q_{\expr{ss}}} \\
& = -\ident_{Q_{\expr{ss}}} \text{.}
\end{align*}

Next we check that $\epsilon_s$ and ${\rm stdbraid}_{s,t}$ satisfy the relations \cite[(4.1)--(4.9)]{ew-soergelcalc} as well as the biadjunction and cyclicity relations up to sign. 
We have already proved that \cite[(4.2)--(4.3)]{ew-soergelcalc} holds up to a sign. 
None of the other numbered relations \cite[(4.1), (4.4)--(4.9)]{ew-soergelcalc} involve standard cups or caps, so they follow immediately.
The first part of Lemma~\ref{lem:epsbiadjbraidcyclic} is the biadjunction relation.
Finally, cyclicity of the standard braid follows from the second part of Lemma~\ref{lem:epsbiadjbraidcyclic} after standardizing all strands. 
The second claim then follows by Theorem~\ref{thm:stdhoms}.
\end{proof}

\begin{notn}
%
Suppose $s \in \Ss$ is colored red. 
To save space in our diagrams, we will write 
\begin{align*} 
\inclscal{red}& = \binom{1}{a^{-1}_{\color{red} s}} \text{,} & \epsscal{red} & =\binom{a_{\color{red} s}}{-a^{-1}_{\color{red} s}} \text{,} \\
\projscal{red}& =\binom{a^{-1}_{\color{red} s}}{1} \text{,} & \spescal{red}& =\binom{-a^{-1}_{\color{red} s}}{a_{\color{red} s}} \text{.}
\end{align*}
\end{notn}

\section{The Frobenius functor}
\label{sec:frobfunctor}

For the rest of this paper, we fix $\field$ to be a field of positive characteristic $p \neq 2,3$ and set $l=p$.
Let $\glsuseri*{V}=\Bbbk \otimes V_A$ be the universal realization of $(\W,\Ss)$ over $\field$ with respect to the affine Cartan matrix $A=(a_{st})_{s,t \in \Ss}$, and let $\dgrmBS$, $\dgrm$, $\dgrmBSstd$, and $\dgrmstd$ be the associated categories over this realization defined in the previous section. 
Similarly we write $\gls*{VF}=\Bbbk \otimes V_A^F$ for the $F$-twist of $V$ (see Definition~\ref{defn:Ftwist}). 
Write $\glsuseri*{dgrmBSF}$, $\dgrmF$, and $\dgrm_{\rm std}^F$ for the associated categories over $V^F$. 
In general we will write a superscript $F$ to denote any Soergel bimodule construction which is applied using the realization $V^F$, such as $B^F_s$ for a Bott--Samelson bimodule in $\dgrmBSF$.
In particular, we note that the polynomial ring $R^F=\mathrm{Sym}(V^F)$ is isomorphic to $R$ as a ring, because $V^F$ and $V$ are identical as vector spaces.

The goal of this section is to construct a monoidal functor $F$ which embeds $\dgrmBSF$ inside a scalar extension of $\dgrmBSstd$. 
The action of this functor on the degraded Grothendieck rings of the respective categories (which are both just $\ZZ\W$) is the Frobenius map, so we call $F$ the \defnemph{Frobenius functor}. 
We first describe a ring extension $\hat{R}$ of $R$ over which we will define the functor $F$.
Next, we will construct an alternative embedding of the standard category in $\dgrmBSstd$. 
Finally we will use this to construct our desired embedding of $\dgrmBSF$ inside $\hat{R} \otimes \dgrmBSstd$.

\subsection{A new ring}

The functor $F$ requires working over an extension $\hat{R}$ of $R$ in which certain elements are invertible.
It will be convenient for us to further assume that:
\begin{enumerate}[label=($\hat{R}$\arabic*)]
\item \label{enum:Rhataxiom1} the non-invertible elements in $V \subseteq \hat{R}$ are precisely $\bigoplus_{s \in \Ssf} \field a_s$;

\item \label{enum:Rhataxiom2} $\hat{R}$ is a complete discrete valuation ring.
\end{enumerate}
We can use standard commutative algebra facts to construct such an extension. 
Take the prime ideal $\mathfrak{p}=(a_s : s \in \Ssf)$ in $R$ and consider the localization $R_{\mathfrak{p}}$. 
This ring is a Noetherian local domain, so it is dominated by a discrete valuation ring (e.g.~\cite[{Exercises~11.2--11.3}]{eisenbud}), whose completion $\hat{R}$ satisfies both axioms. 
For completeness we describe this construction in more detail below. 

\begin{defn} \label{defn:Rhat}
The ring $\gls*{hatR}$ is the following iterated extension of $R$. 
First let $R'=R[a_s/a_t : s,t \in \Ssf]$. 
This is the subring of $Q$ generated by $R$ and the fractions 
\begin{equation*}
\{a_s/a_t : s,t \in \Ssf\} \subset Q \text{.}
\end{equation*}
In $R'$, for any $s \in \Ssf$ the ideal $\mathfrak{p}'=R'a_s$ is prime and does not depend on $s$, so we may consider the localization $R'_{\mathfrak{p}'}$ of $R'$ at $\mathfrak{p}'$. 
Finally let $\hat{R}$ be the completion of $R'_{\mathfrak{p}'}$ with respect to $\mathfrak{p}'_{\mathfrak{p}'}$: 
\begin{equation*}
R \subset R'=R[a_s/a_t : s,t \in \Ssf] \subset R'_{\mathfrak{p}'} \subset \hat{R} \text{.}
\end{equation*}
As $\mathfrak{p}'$ is a principal ideal, $R'_{\mathfrak{p}'}$ and therefore $\hat{R}$ are discrete valuation rings. 
Let $\glsuseri*{val}$ be the valuation on $\hat{R}$ normalized so that $\val a_s=2$ for all $s \in \Ssf$ and $\val a_{\tilde{s}}=0$. 
For $f,g \in \hat{R}$ we also set $\val \binom{f}{g}=\frac{1}{2}(\val f+\val g)$. 
Finally, we write $\gls*{hatQ}$ for the fraction field of $\hat{R}$.
\end{defn}


\begin{rem}
The structure of the maximal ideal of $\hat{R}$ ensures that for each $s \in \Ss$, the object $B^{(F)}_s$ (defined later in \S\ref{sec:Frobdotsforks}) is indecomposable because $a_{F(s)}$ is not invertible in $\hat{R}$. 
This is analogous to $B_s$ being indecomposable in $\dgrm$ because $a_s$ is not invertible in $R$ (see Example~\ref{exam:singgenindecomp}). 
\end{rem}

The valuation in $\hat{R}$ gives a useful invariant, playing a role largely analogous to degree in $R$. 
To emphasize this connection (and to avoid confusion with the generator of $\laur$) we have chosen the notation $\val$ instead of the more usual $v$ or $\nu$. 
It is also important that $\hat{R}$ is a complete local ring so that when we later define the category $\dgrmpast$ over $\hat{R}$, it retains the Krull--Schmidt property (cf.~\cite[Lemma~6.25]{ew-soergelcalc}).

In order for $\hat{R}$ to take the role of $R$ in later sections, we will require a suitable replacement for the notion of a graded $R$-module. 
The valuation on $\hat{R}$ gives rise to a (non-archimedean) absolute value. 
It turns out that the appropriate replacement for ``graded $R$-modules'' is ``normed $\hat{R}$-modules''. 
We give a brief overview of the relevant theory, translated into the language of valuations (see also e.g.~\cite[\S 2.2]{ardakov-wadsley}).

\begin{defn} \label{defn:valmod} \hfill
\begin{enumerate}
\item Let $M$ be an $\hat{R}$-module. 
A \defnemph{valuation} on $M$ is a function $\val : M \rightarrow \ZZ \cup \{\infty\}$ which satisfies the following properties for all $m,m' \in M$ and $f \in \hat{R}$:
\begin{itemize}
\item $\val m=\infty$ if and only if $m=0$;

\item $\val(m+m') \geq \min(\val m,\val m')$;

\item $\val(fm)=\val f + \val m$.
\end{itemize}
An \defnemph{$\hat{R}$-valuation module} is an $\hat{R}$-module equipped with a valuation. 
We say that two $\hat{R}$-valuation modules $M$ and $N$ are \defnemph{valuation isomorphic} if there are mutually inverse $\hat{R}$-module isomorphisms $M \rightarrow N$ and $N \rightarrow M$ which preserve the valuation.

\item Let $\mathcal{M}$ be a category enriched in $\hat{R}$-modules. 
A \defnemph{valuation} on the $\Hom$-spaces of $\mathcal{M}$ is a collection of functions $\val: \Hom_{\mathcal{M}}(M,N) \rightarrow \ZZ \cup \{\infty\}$ for all objects $M,N$ in $\mathcal{M}$ which give each $\Hom$-space the structure of a valuation module, with the following additional properties for all morphisms $\beta$ and $\gamma$ in $\mathcal{M}$:
\begin{itemize}
\item $\val(\ident_M)=0$;

\item $\val(\beta \circ \gamma) \geq \val \beta+\val \gamma$ if $\beta$ and $\gamma$ are composable.
\end{itemize}
We write $\Hom_{\mathcal{M}}^{\blacklozenge}$ to refer to the $\Hom$-space considered as a valuation module. 
We say that two objects $M$ and $N$ are \defnemph{valuation isomorphic} in $\mathcal{M}$ if there are two mutually inverse isomorphisms $M \rightarrow N$ and $N \rightarrow M$ in $\mathcal{M}$ of valuation $0$.
\end{enumerate}
\end{defn}

\begin{exam} \label{exam:valmodind} \hfill
\begin{enumerate}
\item The ring $\hat{R}$ as a left $\hat{R}$-module is a valuation module, with valuation equal to the ring valuation. 
More generally, any finitely generated free $\hat{R}$-module $\hat{R}^{\oplus m}$ is a valuation module with valuation 
\begin{equation*}
\val (r_1, r_2,\dotsc,r_m)=\min_{i} \val(r_i) \text{.}
\end{equation*}

\item Consider the category $\mathcal{M}$ of all $\hat{R}$-valuation modules. For $\beta : M \rightarrow N$ a morphism in this category, let 
\begin{equation*}
\val_{\mathcal{M}} \beta=\inf_{m \in M \setminus \{0\}} (\val_N(\beta(m))-\val_M m) \text{.} 
\end{equation*}
Then $\val_{\mathcal{M}}$ is a valuation on $\mathcal{M}$, called the \defnemph{induced valuation}. 
Moreover, if $\beta$ is an isomorphism then $\beta$ preserves the valuation if and only if $\val_{\mathcal{M}} \beta=0$. 
In other words, the two notions of valuation isomorphism above coincide.
\end{enumerate}
\end{exam}

Let $M$ be a valuation module with valuation $\val_M$. 
For $i \in \ZZ$, the \defnemph{valuation shift} $M\glsplural{pip}$ is the valuation module given by $M$ with the new valuation $\val_{M\langle i\rangle} m=i+\val m$. 
Valuation shift plays a similar role for valuation modules as the grade shift functor does for graded modules. 

For valuation modules $M,N$ with valuations $\val_M,\val_N$ respectively, the direct sum $M \oplus N$ is a valuation module with valuation 
\begin{equation*}
\val_{M \oplus N}(m \oplus n)=\min(\val m,\val n) \text{.}
\end{equation*} 
This is a generalization of the valuations on free $\hat{R}$-modules from the previous example. 

We call a valuation module a \defnemph{valuation free module} if it is valuation isomorphic to a valuation module of the form $\bigoplus_i \hat{R}\langle d_i \rangle$ for some $d_i \in \ZZ$. 
We call a basis $(b_i)$ for a valuation free module $M$ a \defnemph{valuation basis} if for some $d_i \in \ZZ$ the $\hat{R}$-module homomorphism
\begin{align*}
\bigoplus_i \hat{R}\langle d_i\rangle & \longrightarrow M\\
e_i & \longmapsto b_i
\end{align*}
is a valuation isomorphism.
Not every basis is a valuation basis, as the following example shows.

\begin{exam}
Let $M=\hat{R} \oplus \hat{R} \langle 1\rangle$.
Clearly $\{(a,b),(c,d)\}$ is a basis of $M$ if and only if $ad-bc$ is invertible in $\hat{R}$. 
However, $\{(a,b),(c,d)\}$ is a valuation basis of $M$ only if in addition either $b=0$ or $d=0$.
\end{exam}

\begin{rem} \hfill
\begin{enumerate}
\item The ring $\hat{R}$ has the structure of a filtered $\field$-algebra, with decreasing filtration given by
\begin{equation*}
\hat{R}^i=\{f \in \hat{R} : \val f \geq i\} \text{.}
\end{equation*}
Every $\hat{R}$-module $M$ is filtered in a similar way, giving it the structure of a filtered $\hat{R}$-module as defined in e.g.~\cite[Definition~1.2]{rigid-tilting}. In this context valuation isomorphisms and valuation shifts are special cases of filtered isomorphisms and filtration shifts respectively. 
\item Valuation free modules are analogous to graded free modules. 
Just as a graded module which is free as an ungraded module is not necessarily graded free, a valuation module which is a free module is not necessarily a valuation free module. 
(In fact, from the axioms it is clear that \emph{any} finitely generated valuation module is torsion free and thus free, because $\hat{R}$ is a discrete valuation ring.)
\end{enumerate}
\end{rem}

\subsection{One-color generators}
\label{sec:Frobdotsforks}


Let $\hat{R} \otimes \dgrmBSstd$ denote the scalar extension of $\dgrmBSstd$ to an $\hat{R}$-linear category. 
As with the mixed category we will generally omit the ``$\hat{R} \otimes (-)$'' when describing objects in $\hat{R} \otimes \dgrmBSstd$ whenever possible.
We proceed to construct the Frobenius functor as an embedding of $\dgrmBSF$ inside $\hat{R} \otimes \dgrmBSstd$.
This construction is straightforward for dots and forks.

Recall from \eqref{eq:tildesp} that
\begin{equation*}
\tilde{s}_p=F(\tilde{s})=(\tilde{s}\sh)^{(p-1)/2}\tilde{s} (\sh\tilde{s})^{(p-1)/2} \text{,}
\end{equation*}
with the reflection $\sh \in \Wf$ defined as in Proposition~\ref{prop:tildes-defn}. 
The following objects are essentially ``$p$-dilated'' versions of $B^F_{s}$ inside $\hat{R} \otimes \dgrmBSstd$. 

\begin{defn}
Fix a rex $\glsuseri*{exprsh} \in \expr{\Ss}_{\rm f}$ for the reflection $\sh$. 
For each $s \in \Ss$ we write $\hat{R}^{(F)}_s$ and $B^{(F)}_s$ to mean the following objects in $\hat{R} \otimes \dgrmBSstd$:
\begin{align*}
\hat{R}^{(F)}_s& =\begin{cases}
\hat{R}_s & \text{if $s \in \Ssf$,} \\
\hat{R}_{(\expr{\tilde{s}\sh})^{(p-1)/2}\expr{\tilde{s}} (\expr{\sh\tilde{s}})^{\otimes (p-1)/2}} & \text{if $s=\tilde{s}$,}
\end{cases} \\
B^{(F)}_s& =\begin{cases}
B_s & \text{if $s \in \Ssf$,} \\
(\hat{R}_{\expr{\tilde{s}\sh}})^{\otimes (p-1)/2} \otimes B_{\tilde{s}} \otimes (\hat{R}_{\expr{\sh\tilde{s}}})^{\otimes (p-1)/2} & \text{if $s=\tilde{s}$.}
\end{cases}
\end{align*}
For $\expr{x}=\expr{s_1 s_2 \dotsm s_m} \in \expr{\Ss}$ we write $\glsuseri*{BexprxpFp}=B^{(F)}_{s_1} \otimes B^{(F)}_{s_2} \otimes \dotsm B^{(F)}_{s_m}$.
\end{defn}

The following morphisms in $\hat{R} \otimes \dgrmBSstd$ are analogous to ${\rm dot}^F_s$, ${\rm fork}^F_s$, and ${\rm bivalent}^F_s$.

\begin{defn} 
For each $s \in \Ss$, we write 
\begin{align*}
\glsuseri*{dotspFp} & : B^{(F)}_s  \longrightarrow B^{(F)}_{\emptexpr} \text{,} \\
\glsuseri*{forkspFp} & : B^{(F)}_s \otimes B^{(F)}_s  \longrightarrow B^{(F)}_s \text{,} \\
{\rm bivalent}^{(F)}_s & : B^{(F)}_s  \longrightarrow \hat{R}^{(F)}_s
\end{align*}
for the following dual-valued morphisms defined in $\hat{R} \otimes \dgrmBSstd$.
If $s \in \Ssf$ then ${\rm dot}^{(F)}_s={\rm dot}_s$, ${\rm fork}^{(F)}_s={\rm fork}_s$, and ${\rm bivalent}^{(F)}_s={\rm bivalent}_s$.
Otherwise if $s=\tilde{s}$ we set
\begin{align*}
{\rm dot}^{(F)}_{\tilde{s}}& =\epsilon_{(\expr{\tilde{s}\sh})^{(p-1)/2}} \circ (\ident_{\hat{R}_{\expr{\tilde{s}\sh}}}^{\otimes (p-1)/2} \otimes {\rm dot}_{\tilde{s}} \otimes \ident_{\hat{R}_{\expr{\sh\tilde{s}}}}^{\otimes (p-1)/2}) \text{,} \\
{\rm fork}^{(F)}_{\tilde{s}}& =(\ident_{\hat{R}_{\expr{\tilde{s}\sh}}}^{\otimes (p-1)/2} \otimes {\rm fork}_{\tilde{s}} \otimes \ident_{\hat{R}_{\expr{\sh\tilde{s}}}}^{\otimes (p-1)/2}) \\
& \quad \circ (\ident_{\hat{R}_{\expr{\tilde{s}\sh}}}^{\otimes (p-1)/2} \otimes \ident_{\hat{R}_{\expr{\tilde{s}}}} \otimes \epsilon_{(\expr{\sh\tilde{s}})^{(p-1)/2}} \otimes \ident_{\hat{R}_{\expr{\tilde{s}}}} \otimes \ident_{\hat{R}_{\expr{\sh\tilde{s}}}}^{\otimes (p-1)/2}) \text{,} \\
{\rm bivalent}^{(F)}_{\tilde{s}}& =\ident_{\hat{R}_{\expr{\tilde{s}\sh}}^{\otimes (p-1)/2}} \otimes {\rm bivalent}_{\tilde{s}} \otimes \ident_{\hat{R}_{\expr{\sh\tilde{s}}}^{\otimes (p-1)/2}} \text{.}
\end{align*}
\end{defn}


\begin{exam} \label{exam:frobdotfork}
Let $p=3$ and $\W$ be of type $\widetilde{A_2}$, and let us follow the same convention as Example~\ref{exam:w-twisted-p-defect} in labeling the generators. 
Set $\expr{\sh}=\expr{121}$, and color the generators $0,1,2$ blue, red, and green respectively.
Figure~\ref{fig:affpdotfork} depicts ${\rm dot}^{(F)}_0$ and ${\rm fork}^{(F)}_0$. 
\begin{figure}
    \centering
    \begin{subfigure}{0.35\textwidth}
        \centering
        \begin{tikzpicture}[scale=0.5]
        \node[anchor=south west,inner sep=0] at (0,0) {\includegraphics[scale=0.85]{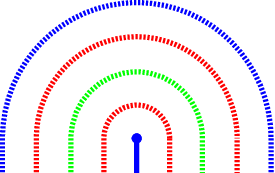}};
         \node[font=\footnotesize] at (150pt,15pt){$\epsscal{red}$};
         \node[font=\footnotesize] at (178pt,15pt){$\epsscal{green}$};
         \node[font=\footnotesize] at (206pt,15pt){$\epsscal{red}$};
         \node[font=\footnotesize] at (234pt,15pt){$\epsscal{blue}$};
        \end{tikzpicture}
        \caption{${\rm dot}^{(F)}_0$}
    \end{subfigure}
    ~
    \begin{subfigure}{0.63\textwidth}
        \centering
        \begin{tikzpicture}[scale=0.6]
        \node[anchor=south west,inner sep=0] at (0,0) {\includegraphics[scale=0.85]{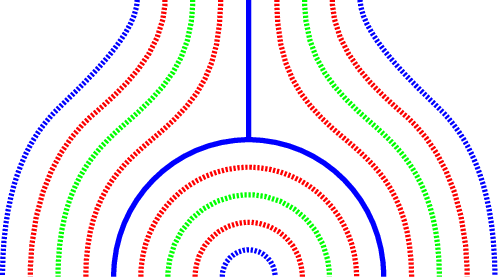}};
         \node[font=\footnotesize] at (197pt,7pt){$\epsscal{blue}$};
         \node[font=\footnotesize] at (216pt,7pt){$\epsscal{red}$};
         \node[font=\footnotesize] at (235pt,7pt){$\epsscal{green}$};
         \node[font=\footnotesize] at (254pt,7pt){$\epsscal{red}$};
        \end{tikzpicture}
        \caption{${\rm fork}^{(F)}_0$}
    \end{subfigure}
    \caption{Examples of ${\rm dot}^{(F)}_s$ and ${\rm fork}^{(F)}_s$ ($p=3$, $\W$ of type $\widetilde{A_2}$).}
    \label{fig:affpdotfork}
\end{figure}
\end{exam}

\begin{rem} \label{rem:Rhatmorphisms}
We are implicitly using Lemma~\ref{lem:rlzcoefs} in a crucial way whenever we claim that ${\rm dot}^{(F)}_s$ and ${\rm fork}^{(F)}_s$ (and later ${\rm braid}^{(F)}_{st}$) are morphisms in $\hat{R} \otimes \dgrmBSstd$. 
For example, the leftmost $\epsscal{red}$ in ${\rm dot}^{(F)}_0$ in Figure~\ref{fig:affpdotfork} can be pushed all the way to the left to yield ${0}{1}{2}\binom{a_{1}}{-a^{-1}_{1}}$. 
The resulting dual-valued scalar lies in $\hat{R}$ by Lemma~\ref{lem:rlzcoefs} because $012 \in \Wpcosets$ and $\Wp 0121 \neq \Wp 012$. 
\end{rem}

\begin{prop} \label{prop:dotforkFrelns} 
The morphisms ${\rm dot}^{(F)}_s$, ${\rm fork}^{(F)}_s$, and ${\rm bivalent}^{(F)}_s$ for all $s \in \Ss$ satisfy all the relations in $\dgrmBSstdF$ which do not involve braids. 
In other words, the relations \cite[(10.8a)--(10.8e)]{emtw} and \eqref{eq:stdBSstd}--\eqref{eq:BSstdBS}
hold when the morphisms ${\rm dot}^{(F)}_s$, ${\rm fork}^{(F)}_s$, and ${\rm bivalent}^{(F)}_s$ are substituted for all corresponding dots, forks, and bivalent vertices respectively.
\end{prop}

\begin{proof}
The only non-trivial relations to check are those which involve $\tilde{s}$-colored dots, forks, and bivalent vertices.
Following Example~\ref{exam:frobdotfork} we will draw diagrams for the case $p=3$ and $\W$ of type $\widetilde{A_2}$.

Checking the one-color relations \cite[(10.8a), (10.8b), (10.8e)]{emtw} is quite straightforward. 
For example, one form of the relation \cite[(10.8a)]{emtw} looks like
\begin{equation*}
\begin{gathered}
    \begin{tikzpicture}[scale=0.85]
    \node[anchor=south west,inner sep=0] at (0,0) {\includegraphics[scale=0.85]{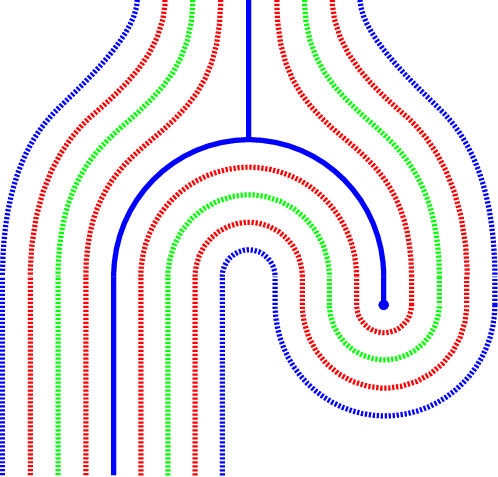}};
     \node[font=\footnotesize] at (139pt,100pt){$\epsscal{blue}$};
     \node[font=\footnotesize] at (152pt,100pt){$\epsscal{red}$};
     \node[font=\footnotesize] at (165pt,100pt){$\epsscal{green}$};
     \node[font=\footnotesize] at (178pt,100pt){$\epsscal{red}$};
     \node[font=\footnotesize] at (205pt,93pt){$\spescal{red}$};
     \node[font=\footnotesize] at (218pt,93pt){$\spescal{green}$};
     \node[font=\footnotesize] at (231pt,93pt){$\spescal{red}$};
     \node[font=\footnotesize] at (244pt,93pt){$\spescal{blue}$};
    \end{tikzpicture}
\end{gathered}=
\begin{gathered}
    \begin{tikzpicture}[scale=0.85]
    \node[anchor=south west,inner sep=0] at (0,0) {\includegraphics[scale=0.85]{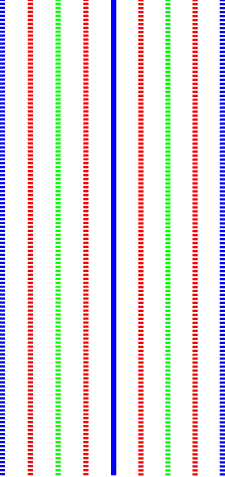}};
    \end{tikzpicture}
\end{gathered}
\end{equation*}
In general, the number of sign changes in these computations (which come from applications of Proposition~\ref{prop:epscupcapinv}) is a multiple of $\len((\tilde{s} \sh)^{(p-1)/2})$, which is even.

For the barbell relation \cite[(10.8c)]{emtw}, we use the fact that $a_{\tilde{s}_p}=(\tilde{s} \sh)^{(p-1)/2}(a_{\tilde{s}})$:
\begin{equation*}
\begin{gathered}
    \begin{tikzpicture}[scale=0.5]
    \node[anchor=south west,inner sep=0] at (0,0) {\includegraphics[scale=0.85]{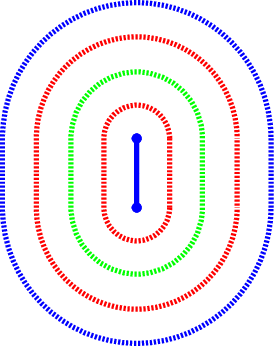}};
     \node[font=\footnotesize] at (150pt,128pt){$\spescal{red}$};
     \node[font=\footnotesize] at (178pt,128pt){$\spescal{green}$};
     \node[font=\footnotesize] at (206pt,128pt){$\spescal{red}$};
     \node[font=\footnotesize] at (234pt,128pt){$\spescal{blue}$};
     \node[font=\footnotesize] at (150pt,154pt){$\epsscal{red}$};
     \node[font=\footnotesize] at (178pt,154pt){$\epsscal{green}$};
     \node[font=\footnotesize] at (206pt,154pt){$\epsscal{red}$};
     \node[font=\footnotesize] at (234pt,154pt){$\epsscal{blue}$};
    \end{tikzpicture}
\end{gathered}
=(-1)^{4} a_{\tilde{s}_p}=a_{F(\tilde{s})}
\end{equation*}
For the relation \cite[(10.8d)]{emtw}, we recall that $F(\tilde{s})=(\tilde{s}\sh)^{(p-1)/2} \tilde{s} (\sh \tilde{s})^{(p-1)/2}$:
\begin{multline*}
f
\begin{gathered}
    \begin{tikzpicture}[scale=1]
    \node[anchor=south west,inner sep=5pt] at (0,0) {\includegraphics[scale=1]{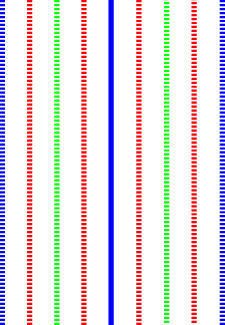}};
    \end{tikzpicture}
\end{gathered}
-
\begin{gathered}
    \begin{tikzpicture}[scale=1]
    \node[anchor=south west,inner sep=5pt] at (0,0) {\includegraphics[scale=1]{fig/affdemazure2}};
    \end{tikzpicture}
\end{gathered}
F(\tilde{s})(f) \\
\begin{aligned}
& =
\begin{gathered}
    \begin{tikzpicture}[scale=1]
    \node[anchor=south west,inner sep=5pt] at (0,0) {\includegraphics[scale=1]{fig/affdemazure2}};
     \node at (53pt,83pt){$f'$};
    \end{tikzpicture}
\end{gathered}
-
\begin{gathered}
    \begin{tikzpicture}[scale=1]
    \node[anchor=south west,inner sep=5pt] at (0,0) {\includegraphics[scale=1]{fig/affdemazure2}};
     \node at (66pt,83pt){$f''$};
    \end{tikzpicture}
\end{gathered} \\
& =
\begin{gathered}
    \begin{tikzpicture}[scale=1]
    \node[anchor=south west,inner sep=5pt] at (0,0) {\includegraphics[scale=1]{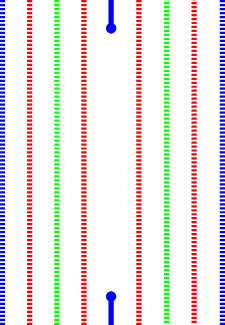}};
     \node at (59pt,83pt){$\partial_{\tilde{s}}(f')$};
    \end{tikzpicture}
\end{gathered}
=
\begin{gathered}
    \begin{tikzpicture}[scale=1]
    \node[anchor=south west,inner sep=5pt] at (0,0) {\includegraphics[scale=1]{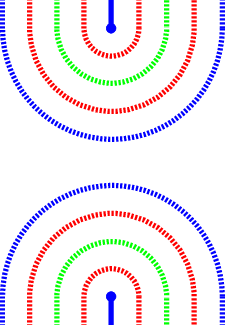}};
     \node at (59pt,83pt){$(-1)^4 \partial_{\tilde{s}}^F(f)$};
     \node[font=\footnotesize] at (77pt,11pt){$\epsscal{red}$};
     \node[font=\footnotesize] at (91pt,11pt){$\epsscal{green}$};
     \node[font=\footnotesize] at (104pt,11pt){$\epsscal{red}$};
     \node[font=\footnotesize] at (118pt,11pt){$\epsscal{blue}$};
     \node[font=\footnotesize] at (77pt,150pt){$\spescal{red}$};
     \node[font=\footnotesize] at (91pt,150pt){$\spescal{green}$};
     \node[font=\footnotesize] at (104pt,150pt){$\spescal{red}$};
     \node[font=\footnotesize] at (118pt,150pt){$\spescal{blue}$};
    \end{tikzpicture}
\end{gathered}
\end{aligned}
\end{multline*}
where $f'=(\sh \tilde{s})^{(p-1)/2}(f)$, $f''=\tilde{s}(\sh \tilde{s})^{(p-1)/2}(f)$, $\partial_{\tilde{s}}(f')=a^{-1}_{\tilde{s}}(f'-\tilde{s}(f'))$ and $\partial_{\tilde{s}}^F(f)=a^{-1}_{\tilde{s}_p}(f-\tilde{s}_p(f))$.
Checking the mixed relations \eqref{eq:stdBSstd}--\eqref{eq:BSstdBS} is essentially identical.
\end{proof}

In light of Proposition~\ref{prop:dotforkFrelns} we may similarly define ${\rm cap}^{(F)}_s={\rm dot}^{(F)}_s \circ {\rm fork}^{(F)}_s$, ${\rm stdcap}^{(F)}_s={\rm cap}^{(F)}_s \circ (\overline{{\rm bivalent}^{(F)}_s} \otimes \overline{{\rm bivalent}^{(F)}_s} (a_{F(s)})^{-1})$, etc.

\begin{rem}
From the proof of Proposition~\ref{prop:dotforkFrelns} it is apparent that a similar result would hold if we had defined ${\rm dot}^{(F)}_{\tilde{s}}$ and ${\rm fork}^{(F)}_{\tilde{s}}$ using standard caps instead of $\epsilon_{\expr{\tilde{s}\sh}}$.
Ultimately, the reason for preferring the latter is to ensure that every menorah vertex in \S\ref{sec:Dpastconstr} is semi-cyclic.
\end{rem}

\subsection{Braids}

We continue our construction of the Frobenius functor by defining its image on braids. 
As with dots and forks, this is entirely trivial for braids only involving $\Ssf$; when $\tilde{s}$ is involved more work is needed.

First we introduce some graph-theoretic terminology for braids from \cite{straightahead}. 
Suppose $(s,t) \in \Ss \times \Ss$ with $s \neq t$ and $m_{st}<\infty$. 
Recall that diagrammatically, the morphism ${\rm braid}_{s,t}$ is a vertex of valency $2m_{st}$.
We call two strands of ${\rm braid}_{s,t}$ \defnemph{opposite} if they are $m_{st}$ strands apart in the cyclic ordering of the strands around the braid vertex.
If $m_{st}$ is odd, then opposite strands have distinct colors; if $m_{st}$ is even, then opposite strands have the same color.

The basic building block in this section is the braid standardized at all strands except for a pair of opposite strands. 
For example, we define
\begin{align*}
\begin{gathered}
{\rm conjbraid}_{s,t} : B_s \otimes \hat{R}_{\underbrace{\scriptstyle\expr{tst \dotsm s}}_{m_{st}-1}} \longrightarrow \hat{R}_{\underbrace{\scriptstyle\expr{tst \dotsm s}}_{m_{st}-1}} \otimes B_t \\
\begin{gathered}
\begin{tikzpicture}[xscale=0.5,yscale=0.66,baseline=(origin.center),font=\footnotesize]
\coordinate (origin) at (0,0) {} node[rectangle, minimum width=4.5cm] {};
\path (origin) ++(-2,-1) coordinate (dom1) {} ++(0,-0.25) coordinate (below1) {} node[right] {} ++(0,-0.42) coordinate (stddom1) {};
\path (origin) ++(-1.25,-1) coordinate (dom2) {} ++(0,-0.25) coordinate (below2) {} node[right=-2pt] {$\inclscal{blue}$} ++(0,-0.42) coordinate (stddom2) {};
\path (origin) ++(-0.5,-1) coordinate (dom3) {} ++(0,-0.25) coordinate (below3) {} node[right=-2pt] {$\inclscal{red}$} ++(0,-0.42) coordinate (stddom3) {};
\path (origin) ++(2,-1) coordinate (domlast) {} ++(0,-0.25) coordinate (belowlast) {} node[right=-2pt] {$\inclscal{red}$} ++(0,-0.42) coordinate (stddomlast) {};
\path (origin) ++(-2,1) coordinate (cod1) {} ++(0,0.25) coordinate (above1) {} node[right=-2pt] {$\projscal{blue}$} ++(0,0.42) coordinate (stdcod1) {};
\path (origin) ++(-1.25,1) coordinate (cod2) {} ++(0,0.25) coordinate (above2) {} node[right=-2pt] {$\projscal{red}$} ++(0,0.42) coordinate (stdcod2) {};
\path (origin) ++(-0.5,1) coordinate (cod3) {} ++(0,0.25) coordinate (above3) {} node[right=-2pt] {$\projscal{blue}$} ++(0,0.42) coordinate (stdcod3) {};
\path (origin) ++(2,1) coordinate (codlast) {} ++(0,0.25) coordinate (abovelast) {} node[right] {} ++(0,0.42) coordinate (stdcodlast) {};
\draw[red, corner, line width=\superthick] (origin) to (dom1) to (below1)
                                   (below1) to (stddom1)
                                   (origin) to (dom3) to (below3)
                                   (origin) to (domlast) to (belowlast)
                                   (origin) to (cod2) to (above2);
\draw[red, line width=\superthick, std] 
                                        (below3) to (stddom3)
                                        (belowlast) to (stddomlast)
                                        (above2) to (stdcod2);
\draw[blue, corner, line width=\superthick] (origin) to (dom2) to (below2)
                                            (origin) to (cod1) to (above1)
                                            (origin) to (cod3) to (above3)
                                            (origin) to (codlast) to (abovelast)
                                            (abovelast) to (stdcodlast);
\draw[blue, line width=\superthick, std] (below2) to (stddom2)
                                         (above1) to (stdcod1)
                                         (above3) to (stdcod3);
\path (origin) ++(0.75,-1.25) node[font=\large] (domdots) {$\dotsm$};
\path (origin) ++(0.75,1.25) node[font=\large] (coddots) {$\dotsm$};
\end{tikzpicture}
\end{gathered}
\end{gathered} & & & \text{($m_{st}$ odd)} \\
\begin{gathered}
{\rm conjbraid}_{s,t} : B_s \otimes \hat{R}_{\underbrace{\scriptstyle\expr{tst \dotsm t}}_{m_{st}-1}} \longrightarrow \hat{R}_{\underbrace{\scriptstyle\expr{tst \dotsm t}}_{m_{st}-1}} \otimes B_s \\
\begin{gathered}
\begin{tikzpicture}[xscale=0.5,yscale=0.66,baseline=(origin.center),font=\footnotesize]
\coordinate (origin) at (0,0) {} node[rectangle, minimum width=4.5cm] {};
\path (origin) ++(-2,-1) coordinate (dom1) {} ++(0,-0.25) coordinate (below1) {} node[right] {} ++(0,-0.42) coordinate (stddom1) {};
\path (origin) ++(-1.25,-1) coordinate (dom2) {} ++(0,-0.25) coordinate (below2) {} node[right=-2pt] {$\inclscal{blue}$} ++(0,-0.42) coordinate (stddom2) {};
\path (origin) ++(-0.5,-1) coordinate (dom3) {} ++(0,-0.25) coordinate (below3) {} node[right=-2pt] {$\inclscal{red}$} ++(0,-0.42) coordinate (stddom3) {};
\path (origin) ++(2,-1) coordinate (domlast) {} ++(0,-0.25) coordinate (belowlast) {} node[right=-2pt] {$\inclscal{blue}$} ++(0,-0.42) coordinate (stddomlast) {};
\path (origin) ++(-2,1) coordinate (cod1) {} ++(0,0.25) coordinate (above1) {} node[right=-2pt] {$\projscal{blue}$} ++(0,0.42) coordinate (stdcod1) {};
\path (origin) ++(-1.25,1) coordinate (cod2) {} ++(0,0.25) coordinate (above2) {} node[right=-2pt] {$\projscal{red}$} ++(0,0.42) coordinate (stdcod2) {};
\path (origin) ++(-0.5,1) coordinate (cod3) {} ++(0,0.25) coordinate (above3) {} node[right=-2pt] {$\projscal{blue}$} ++(0,0.42) coordinate (stdcod3) {};
\path (origin) ++(2,1) coordinate (codlast) {} ++(0,0.25) coordinate (abovelast) {} node[right] {} ++(0,0.42) coordinate (stdcodlast) {};
\draw[red, corner, line width=\superthick] (origin) to (dom1) to (below1)
                                    (below1) to (stddom1)
                                   (origin) to (dom3) to (below3)
                                   (origin) to (cod2) to (above2)
                                   (origin) to (codlast) to (abovelast)
                                   (abovelast) to (stdcodlast);
\draw[red, line width=\superthick, std] 
                                        (below3) to (stddom3)
                                        (above2) to (stdcod2);
\draw[blue, corner, line width=\superthick] (origin) to (dom2) to (below2)
                                            (origin) to (domlast) to (belowlast)
                                            (origin) to (cod1) to (above1)
                                            (origin) to (cod3) to (above3);
\draw[blue, line width=\superthick, std] (below2) to (stddom2)
                                         (belowlast) to (stddomlast)
                                         (above1) to (stdcod1)
                                         (above3) to (stdcod3);
\path (origin) ++(0.75,-1.25) node[font=\large] (domdots) {$\dotsm$};
\path (origin) ++(0.75,1.25) node[font=\large] (coddots) {$\dotsm$};
\end{tikzpicture}
\end{gathered}
\end{gathered} & & & \text{($m_{st}$ even)}
\end{align*}
A \defnemph{twist} of ${\rm conjbraid}_{s,t}$ is simply a morphism of the form
\begin{equation*}
(\epsilon_{\expr{x}} \otimes \ident) \circ (\ident_{\hat{R}_{\expr{x}^{-1}}} \otimes {\rm conjbraid}_{s,t}\otimes \ident_{\hat{R}_{\expr{y}}}) \circ (\ident \otimes \overline{\epsilon_{\expr{y}}})
\end{equation*}
for some $\expr{x},\expr{y} \in \expr{\Ss}$. 
We will frequently redraw twists of ${\rm conjbraid}_{s,t}$ using Lemma~\ref{lem:epsbiadjbraidcyclic}.
We will use these twists to define the following morphisms in $\dgrmBSstd$.

\begin{defn} \label{defn:rhomaps}
Suppose $w \in \Wpcosets$ and $t,t' \in \Ss$ such that $wt'w^{-1}=F(t)$. 
Fix a reduced expression $\expr{w}$ for $w$ and an $\epsilon$-basic standard diagram
\begin{equation*}
\sigma_{t}^{\expr{w},t'}: 
\hat{R}^{(F)}_{t} \longrightarrow \hat{R}_{\expr{w} \expr{t'} \expr{w}^{-1}}
\end{equation*}
in $\hat{R} \otimes \dgrmstd$. 
The \defnemph{straight-ahead walk} in $\sigma_t^{\expr{w},t'}$ is an Eulerian walk through the underlying graph defined as follows (see also \cite{straightahead}).
At the bottom of the diagram, start from the middle strand (i.e.~the middle $\tilde{s}$-colored strand if $t=\tilde{s}$, or the only strand if $t \in \Ssf$). 
If this strand does not reach the top of the diagram, then it adjoins some braid vertex. 
Move to the opposite strand; if it does not reach the top of the diagram then it must adjoin some new braid vertex.
In this manner continue passing from strand to opposite strand until the top of the diagram is reached. 

Call the braids which the straight-ahead walk passes through \defnemph{central}.
We define a dual-valued morphism in $\hat{R} \otimes \dgrmBSstd$
\begin{equation*}
\gls*{rhotexprwtd} : B^{(F)}_{t} \longrightarrow \hat{R}_{\expr{w}} \otimes B_{t'} \otimes \hat{R}_{\expr{w}^{-1}} 
\end{equation*}
by replacing central standard braids in $\sigma_{t}^{\expr{w},t'}$ with twists of ${\rm conjbraid}$. 
(By definition there is only one consistent way to do this.)
\end{defn}

It can be shown that $\rho_{t}^{\expr{w},t'}$is well defined up to sign, but this will not matter for the sequel. 

Suppose $s \in \Ssf$ such that $m_{s\tilde{s}}<\infty$. 
By Lemma~\ref{lem:conjbraids} there exist $s',\tilde{s}' \in \Ss$ and $w \in \Wpcosets$ such that $ws'w^{-1}=s=F(s)$ and $w\tilde{s}'w^{-1}=\tilde{s}_p=F(\tilde{s})$. 
In this section we are primarily interested in the morphisms $\rho_{s}^{\expr{w},s'}$ and $\rho_{\tilde{s}}^{\expr{w},\tilde{s}'}$ for some fixed rex $\expr{w}$ for $w$. 
Later we will use $\rho_t^{\expr{w},t'}$ in full generality to define a categorification of $\heckepast$.

\begin{exam} \label{exam:affprho} \hfill
\begin{enumerate}
\item Let $p=3$ and $\W$ be of type $\widetilde{A_2}$, following the same coloring conventions as Example~\ref{exam:frobdotfork}. 
Suppose $s=1$, and set $\expr{w}=\expr{0120}$, $\tilde{s}'=1$, and $s'=2$. 
Figure~\ref{fig:affprho} depicts $\rho_{\tilde{s}}^{\expr{w},\tilde{s}'}$ and $\rho_{s}^{\expr{w},s'}$. 

\item Let $p=3$ and $\W$ be of type $\widetilde{C_2}$.
As with $\widetilde{A}_2$ we label the generators $\tilde{s}=0$, $1$, and $2$ and color them blue, red, and green respectively, such that $0101=1010$ and $02=20$. 
Suppose $s=1$, and set $\expr{w}=\expr{012012}$, $\tilde{s}'=0$, and $s'=1$. 
Figure~\ref{fig:affprhotypeB} depicts $\rho_{\tilde{s}}^{\expr{w},\tilde{s}'}$ and $\rho_{s}^{\expr{w},s'}$.
\end{enumerate}
In both cases we have used Lemma~\ref{lem:epsbiadjbraidcyclic} in order to twist some of the solid strands. 
This ensures that the straight-ahead walk in each diagram starts from the middle strand at the bottom and proceeds vertically upwards in a straight line. 
\end{exam}

\begin{figure}
    \centering
    \begin{subfigure}{0.4\textwidth}
        \centering
        \begin{tikzpicture}
        \node[anchor=south west,inner sep=0] at (0,0) {\includegraphics{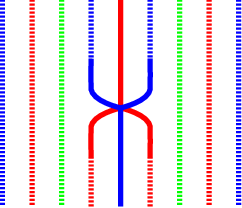}};
        \node[font=\footnotesize] at (80pt,23pt){$\inclscal{red}$};
        \node[font=\footnotesize] at (80pt,72pt){$\projscal{blue}$};
        \node[font=\footnotesize] at (51pt,23pt){$\inclscal{red}$};
        \node[font=\footnotesize] at (51pt,72pt){$\projscal{blue}$};
        \end{tikzpicture}
        \caption{$\rho_{\tilde{s}}^{\expr{w},\tilde{s}'}$}
    \end{subfigure} \qquad
    \begin{subfigure}{0.4\textwidth}
        \centering
        \begin{tikzpicture}[font=\footnotesize]
        \node[anchor=south west,inner sep=0] at (0,0) {\includegraphics{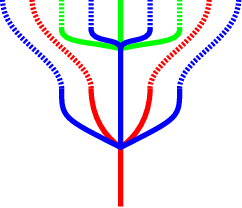}};
        \node at (37pt,58pt) {$\projscal{blue}$};
        \node at (51pt,58pt) {$\projscal{red}$};
        \node at (37pt,88pt) {$\projscal{green}$};
        \node at (51pt,88pt) {$\projscal{blue}$};
        \node at (79pt,88pt) {$\projscal{blue}$};
        \node at (93pt,93pt) {$\projscal{green}$};
        \node at (79pt,58pt) {$\projscal{red}$};
        \node at (94pt,58pt) {$\projscal{blue}$};
        \end{tikzpicture}
        \caption{$\rho_{s}^{\expr{w},s'}$}
    \end{subfigure}
    \caption{Examples of $\rho_{\tilde{s}}^{\expr{w},\tilde{s}'}$ and $\rho_{s}^{\expr{w},s'}$ for $p=3$ in type $\widetilde{A_2}$.}
    \label{fig:affprho}
\end{figure}

\begin{figure}
    \centering
    \begin{subfigure}{0.4\textwidth}
        \centering
        \begin{tikzpicture}[scale=0.8, font=\footnotesize]
         \node[anchor=south west,inner sep=0] at (0,0) {\scalebox{0.8}[0.8]{\includegraphics{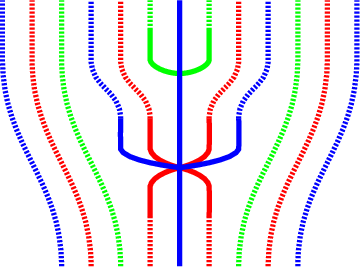}}};
         \node at (79pt,24pt){$\inclscal{red}$};
         \node at (107pt,24pt){$\inclscal{red}$};
         \node at (64pt,72pt){$\projscal{blue}$};
         \node at (79pt,72pt){$\projscal{red}$};
         \node at (107pt,72pt){$\projscal{red}$};
         \node at (121pt,72pt){$\projscal{blue}$};
         \node at (79pt,114pt){$\projscal{green}$};
         \node at (107pt,114pt){$\projscal{green}$};
        \end{tikzpicture}
        \caption{$\rho_{\tilde{s}}^{\expr{w},\tilde{s}'}$}
    \end{subfigure} \qquad
    \begin{subfigure}{0.4\textwidth}
        \centering
        \begin{tikzpicture}[scale=0.8,font=\footnotesize]
         \node[anchor=south west,inner sep=0] at (0,0) {\scalebox{0.8}[0.8]{\includegraphics{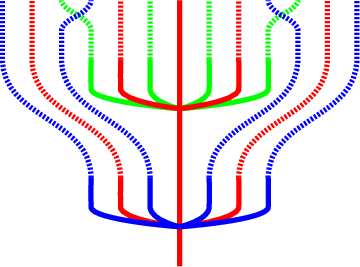}}};
         \node at (50pt,42pt){$\projscal{blue}$};
         \node at (64pt,42pt){$\projscal{red}$};
         \node at (78pt,42pt){$\projscal{blue}$};
         \node at (107pt,42pt){$\projscal{blue}$};
         \node at (121pt,42pt){$\projscal{red}$};
         \node at (135pt,42pt){$\projscal{blue}$};
         \node at (50pt,100pt){$\projscal{green}$};
         \node at (64pt,100pt){$\projscal{red}$};
         \node at (78pt,100pt){$\projscal{green}$};
         \node at (107pt,100pt){$\projscal{green}$};
         \node at (121pt,100pt){$\projscal{red}$};
         \node at (135pt,100pt){$\projscal{green}$};
        \end{tikzpicture}
        \caption{$\rho_{s}^{\expr{w},s'}$}
    \end{subfigure}
    \caption{Examples of $\rho_{\tilde{s}}^{\expr{w},\tilde{s}'}$ and $\rho_{s}^{\expr{w},s'}$ for $p=3$ in type $\widetilde{C_2}$.}
    \label{fig:affprhotypeB}
\end{figure}

\begin{lem} \label{lem:rhomapsdotsforks}
Suppose we have $w \in \Wpcosets$ (with a rex $\expr{w}$) and $t,t' \in \Ss$ as in Definition~\ref{defn:rhomaps}.
Then there exists a sign $\delta_{t}^{\expr{w},t'} \in \{\pm 1\}$ such that
\begin{gather}
\rho_{t}^{\expr{w},t'} \circ \overline{{\rm dot}^{(F)}_t} = \delta_{t}^{\expr{w},t'} (\ident_{\hat{R}_{\expr{w}}} \otimes \overline{{\rm dot}_{t'}} \otimes \ident_{\hat{R}_{\expr{w}^{-1}}}) \circ \overline{\epsilon_{\expr{w}}} \text{,}  \label{eq:rhomapsdots1}\\
(\ident_{\hat{R}_{\expr{w}}} \otimes {\rm dot}_{t'} \otimes \ident_{\hat{R}_{\expr{w}^{-1}}}) \circ \rho_{t}^{\expr{w},t'}= \delta_{t}^{\expr{w},t'}\overline{\epsilon_{\expr{w}}} \circ {\rm dot}^{(F)}_{t} \text{,} \label{eq:rhomapsdots2} \\
\begin{multlined}
\overline{\rho_{t}^{\expr{w},t'}} \circ (\ident_{\hat{R}_{\expr{w}}} \otimes {\rm fork}_{t'} \otimes \ident_{\hat{R}_{\expr{w}^{-1}}}) \circ (\ident_{\hat{R}_{\expr{w}}} \otimes \ident_{B_{t'}} \otimes \epsilon_{\expr{w}^{-1}} \otimes \ident_{B_{t'}} \otimes \ident_{\hat{R}_{\expr{w}^{-1}}}) \\
\circ (\rho_{t}^{\expr{w},t'} \otimes \ident) =\delta_{t}^{\expr{w},t'} {\rm fork}^{(F)}_{t} \circ (\ident \otimes \overline{\rho_{t}^{\expr{w},t'}}) \text{.}
\end{multlined}
\label{eq:rhomapsforks}
\end{gather}
\end{lem}

\begin{proof}
For any $s_1,s_2 \in \Ss$ we have
\begin{equation}
\begin{aligned}
(\ident \otimes {\rm dot}_{s_2})\circ {\rm conjbraid}_{s_1,s_2} & ={\rm dot}_{s_1} \otimes \ident & & \text{if $m_{s_1 s_2}$ is odd,} \\
(\ident \otimes {\rm dot}_{s_1})\circ {\rm conjbraid}_{s_1,s_2} & ={\rm dot}_{s_1} \otimes \ident & & \text{if $m_{s_1 s_2}$ is even,}
\end{aligned} \label{eq:singletermJW}
\end{equation}
by the Jones--Wenzl relation \cite[(10.8i)]{emtw}. 
Without loss of generality we will assume $t=\tilde{s}$ (the general case is very similar).
Following Example~\ref{exam:affprho} we will draw diagrams for the case where $p=3$ and $\W$ is of type $\widetilde{C_2}$, with $\expr{w}=\expr{012012}$ and $\tilde{s}'=0$.
We can repeatedly apply \eqref{eq:singletermJW} to show that
\begin{gather*}
\rho_{\tilde{s}}^{\expr{w},\tilde{s}'} \circ (\ident_{\hat{R}_{(\expr{\tilde{s}\sh})^{(p-1)/2}}} \otimes \overline{{\rm dot}_{\tilde{s}}} \otimes \ident_{\hat{R}_{(\expr{\sh\tilde{s}})^{(p-1)/2}}}) = (\ident_{\hat{R}_{\expr{w}}} \otimes \overline{{\rm dot}_{\tilde{s}'}} \otimes \ident_{\hat{R}_{\expr{w}^{-1}}}) \circ \sigma \\
\begin{gathered}
    \begin{tikzpicture}[scale=0.9, font=\footnotesize]
     \node[anchor=south west,inner sep=0] at (0,0) {\scalebox{0.9}[0.9]{\includegraphics{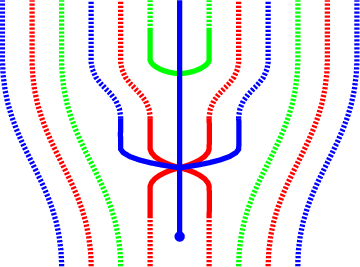}}};
     \node at (79pt,24pt){$\inclscal{red}$};
     \node at (107pt,24pt){$\inclscal{red}$};
     \node at (64pt,72pt){$\projscal{blue}$};
     \node at (79pt,72pt){$\projscal{red}$};
     \node at (107pt,72pt){$\projscal{red}$};
     \node at (121pt,72pt){$\projscal{blue}$};
     \node at (79pt,114pt){$\projscal{green}$};
     \node at (107pt,114pt){$\projscal{green}$};
    \end{tikzpicture}
\end{gathered}
=
\begin{gathered}
    \begin{tikzpicture}[scale=0.9, font=\footnotesize]
     \node[anchor=south west,inner sep=0] at (0,0) {\scalebox{0.9}[0.9]{\includegraphics{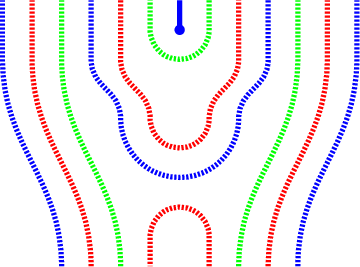}}};
     \node at (107pt,24pt){$\epsscal{red}$};
     \node at (107pt,72pt){$\spescal{red}$};
     \node at (121pt,72pt){$\spescal{blue}$};
     \node at (107pt,114pt){$\spescal{green}$};
    \end{tikzpicture}
\end{gathered}
\end{gather*}
Here $\sigma$ denotes the $\epsilon$-basic standard diagram obtained by deleting strands and braids along the straight-ahead walk and connecting any remaining strands with dashed lines. 
In a similar way, we can show that
\begin{gather*}
(\ident_{\hat{R}_{\expr{w}}} \otimes {\rm dot}_{\tilde{s}'} \otimes \ident_{\hat{R}_{\expr{w}^{-1}}}) \circ \rho_{\tilde{s}}^{\expr{w},\tilde{s}'}= \sigma \circ (\ident_{\hat{R}_{(\expr{\tilde{s}\sh})^{(p-1)/2}}} \otimes {\rm dot}_{\tilde{s}} \otimes \ident_{\hat{R}_{(\expr{\sh\tilde{s}})^{(p-1)/2}}}) \\
\begin{gathered}
    \begin{tikzpicture}[scale=0.9, font=\footnotesize]
     \node[anchor=south west,inner sep=0] at (0,0) {\scalebox{0.9}[0.9]{\includegraphics{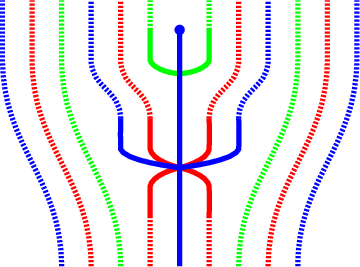}}};
     \node at (79pt,24pt){$\inclscal{red}$};
     \node at (107pt,24pt){$\inclscal{red}$};
     \node at (64pt,72pt){$\projscal{blue}$};
     \node at (79pt,72pt){$\projscal{red}$};
     \node at (107pt,72pt){$\projscal{red}$};
     \node at (121pt,72pt){$\projscal{blue}$};
     \node at (79pt,114pt){$\projscal{green}$};
     \node at (107pt,114pt){$\projscal{green}$};
    \end{tikzpicture}
\end{gathered}
=
\begin{gathered}
    \begin{tikzpicture}[scale=0.9, font=\footnotesize]
     \node[anchor=south west,inner sep=0] at (0,0) {\scalebox{0.9}[0.9]{\includegraphics{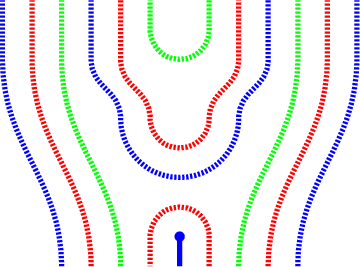}}};
     \node at (107pt,24pt){$\epsscal{red}$};
     \node at (107pt,72pt){$\spescal{red}$};
     \node at (121pt,72pt){$\spescal{blue}$};
     \node at (107pt,114pt){$\spescal{green}$};
    \end{tikzpicture}
\end{gathered}
\end{gather*}
By Proposition~\ref{prop:epscupcapinv} we must have
\begin{equation*}
\sigma \in \{\pm \overline{\epsilon_{\expr{w}}} \circ \epsilon_{(\expr{\tilde{s}\sh})^{(p-1)/2}} \} \text{,}
\end{equation*}
so \eqref{eq:rhomapsdots1} and \eqref{eq:rhomapsdots2} hold with the same sign $\delta_{\tilde{s}}^{\expr{w},s'}$.

Similarly, for any $s_1,s_2 \in \Ss$ we have
\begin{multline}
\overline{{\rm conjbraid}_{s_1,s_2}} \circ (\ident \otimes {\rm fork}_{s_2}) \circ ({\rm conjbraid}_{s_1,s_2} \otimes \ident_{B_{s_2}}) \\
=({\rm fork}_{s_1} \otimes \ident) \circ (\ident_{B_{s_1}} \otimes \overline{{\rm conjbraid}_{s_1,s_2}}) \label{eq:twocolassocconjbraid}
\end{multline}
by two-color associativity \cite[(9.27c)]{emtw}. 
Applying this result (and Proposition~\ref{prop:epscupcapinv}) several times gives \eqref{eq:rhomapsforks} up to sign. 
We leave it to the reader to check that the sign here is also equal to $\delta_{t}^{\expr{w},t'}$. 
\end{proof}

\begin{lem} \label{lem:rhomapsid}
Suppose we have $w \in \Wpcosets$ (with a rex $\expr{w}$) and $t,t' \in \Ss$ as in Definition~\ref{defn:rhomaps}.
Then
\begin{equation*}
\overline{\rho_{t}^{\expr{w},t'}} \circ \rho_{t}^{\expr{w},t'}=(-1)^{\len(w)} \ident_{B^{(F)}_t} \text{.}
\end{equation*}
\end{lem}

In particular, note that unlike $\delta_{t}^{\expr{w},t'}$ the sign here is \emph{not} dependent on $t$.

\begin{proof}
Composing \eqref{eq:twocolassocconjbraid} with $(\ident \otimes \overline{{\rm dot}_{s_2}})$ on the bottom and applying \eqref{eq:singletermJW} gives
\begin{equation}
\overline{{\rm conjbraid}_{s_1,s_2}} \circ {\rm conjbraid}_{s_1,s_2}=\ident \label{eq:conjbraidinv}
\end{equation}
Again, we will assume $t=\tilde{s}$ and draw diagrams for the case where $p=3$ and $\W$ is of type $\widetilde{C_2}$ with $\expr{w}=\expr{012012}$ and $\tilde{s}'=0$. 
By repeatedly applying \eqref{eq:conjbraidinv} with Proposition~\ref{prop:epscupcapinv} we obtain the result up to sign:
\begin{equation*}
\begin{gathered}
    \begin{tikzpicture}[scale=0.9, font=\footnotesize]
     \node[anchor=south west,inner sep=0] at (0,0) {\scalebox{0.9}[0.9]{\includegraphics{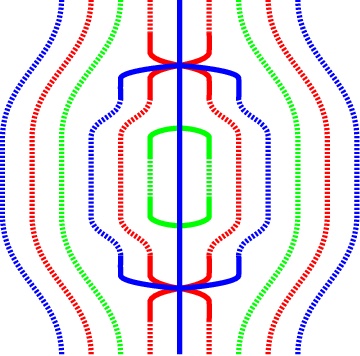}}};
     \node at (79pt,16pt){$\inclscal{red}$};
     \node at (107pt,16pt){$\inclscal{red}$};
     \node at (64pt,48pt){$\projscal{blue}$};
     \node at (79pt,48pt){$\projscal{red}$};
     \node at (107pt,48pt){$\projscal{red}$};
     \node at (121pt,48pt){$\projscal{blue}$};
     \node at (79pt,76pt){$\projscal{green}$};
     \node at (107pt,76pt){$\projscal{green}$};
     \node at (79pt,154pt){$\projscal{red}$};
     \node at (107pt,154pt){$\projscal{red}$};
     \node at (64pt,124pt){$\inclscal{blue}$};
     \node at (79pt,124pt){$\inclscal{red}$};
     \node at (107pt,124pt){$\inclscal{red}$};
     \node at (121pt,124pt){$\inclscal{blue}$};
     \node at (79pt,95pt){$\inclscal{green}$};
     \node at (107pt,95pt){$\inclscal{green}$};
    \end{tikzpicture}
\end{gathered}
=
\begin{gathered}
    \begin{tikzpicture}[scale=0.9, font=\footnotesize]
     \node[anchor=south west,inner sep=0] at (0,0) {\scalebox{0.9}[0.9]{\includegraphics{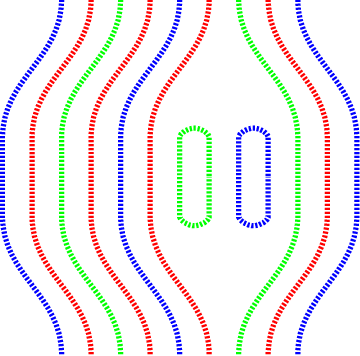}}};
     \node at (107pt,76pt){$\spescal{green}$};
     \node at (107pt,95pt){$\epsscal{green}$};
     \node at (135pt,76pt){$\spescal{blue}$};
     \node at (135pt,95pt){$\epsscal{blue}$};
    \end{tikzpicture}
\end{gathered} \text{.}
\end{equation*}
Here the sign comes from the parity of the number of circles in the final diagram, which is $\len(w)-\len((\tilde{s}\sh)^{(p-1)/2}) \equiv \len(w) \pmod{2}$. 
\end{proof}

\begin{defn}
For each pair $(s,t) \in \Ss \times \Ss$ with $s \neq t$ and $m_{st}<\infty$, we write 
\begin{equation*}
\glsuseri*{braidstpFp} : B^{(F)}_{\underbrace{\scriptstyle\expr{sts \dotsm s}}_{m_{st}}} \rightarrow B^{(F)}_{\underbrace{\scriptstyle\expr{tst \dotsm t}}_{m_{st}}}
\end{equation*}
for the following dual-valued morphism defined in $\hat{R} \otimes \dgrmBSstd$. 
If $s,t \in \Ssf$ then ${\rm braid}^{(F)}_{s,t}={\rm braid}_{s,t}$. 
Otherwise, suppose that $t=\tilde{s}$. 
Let $s',\tilde{s}' \in \Ss$ and $w \in \Wpcosets$ as in Lemma~\ref{lem:conjbraids}, and fix a rex $\expr{w}$ for $w$.
Write
\begin{align*}
\mu^{s',\tilde{s}'}_{\expr{w}}& =\ident_{\hat{R}_{\expr{w}}} \otimes (\ident_{B_{s'}} \otimes \epsilon_{\expr{w}^{-1}} \otimes \ident_{B_{\tilde{s}'}} \otimes \epsilon_{\expr{w}^{-1}} \otimes \ident_{B_{s'}} \otimes \dotsb) \otimes \ident_{\hat{R}_{\expr{w}^{-1}}}\text{,} \\
\mu^{\tilde{s}',s'}_{\expr{w}}& =\ident_{\hat{R}_{\expr{w}}} \otimes (\ident_{B_{\tilde{s}'}} \otimes \epsilon_{\expr{w}^{-1}} \otimes \ident_{B_{s'}} \otimes \epsilon_{\expr{w}^{-1}} \otimes \ident_{B_{\tilde{s}'}} \otimes \dotsb) \otimes \ident_{\hat{R}_{\expr{w}^{-1}}}
\end{align*}
where each morphism contains $(m_{s\tilde{s}}-1)$ copies of $\epsilon_{\expr{w}^{-1}}$ and $m_{s\tilde{s}}$ total copies of $\ident_{B_{s'}}$ and $\ident_{B_{\tilde{s}'}}$.
We define
\begin{equation*}
{\rm braid}^{(F)}_{s,\tilde{s}} = \delta (\overline{\rho_{\tilde{s}}^{\expr{w},\tilde{s}'}} \otimes \overline{\rho_{s}^{\expr{w},s'}} \otimes \overline{\rho_{\tilde{s}}^{\expr{w},\tilde{s}'}} \dotsb) \circ \overline{\mu^{\tilde{s}',s'}_{\expr{w}}} \circ {\rm braid}_{s',\tilde{s}'} \circ \mu^{s',\tilde{s}'}_{\expr{w}} \circ (\rho_{s}^{\expr{w},s'} \otimes \rho_{\tilde{s}}^{\expr{w},\tilde{s}'} \otimes \rho_{s}^{\expr{w},s'} \otimes \dotsb)
\end{equation*}
where
\begin{equation*}
\delta=
\begin{cases}
(-1)^{\len(w)}\delta_{s}^{\expr{w},s'} \delta_{\tilde{s}}^{\expr{w},\tilde{s}'} & \text{if $m_{s\tilde{s}}$ is odd,} \\
(-1)^{\len(w)} & \text{if $m_{s\tilde{s}}$ is even,}
\end{cases}
\end{equation*}
with $\delta_{s}^{\expr{w},s'}$ and $\delta_{\tilde{s}}^{\expr{w},\tilde{s}'}$ the signs from Lemma~\ref{lem:rhomapsdotsforks}.
We similarly set ${\rm braid}^{(F)}_{\tilde{s},s}=\overline{{\rm braid}^{(F)}_{s,\tilde{s}}}$. 
We also define 
\begin{multline*}
{\rm stdbraid}^{(F)}_{s,t}=({\rm bivalent}^{(F)}_t a^{-1}_{F(t)} \otimes {\rm bivalent}^{(F)}_s a^{-1}_{F(s)} \otimes {\rm bivalent}^{(F)}_t a^{-1}_{F(t)} \otimes \dotsb) \\
\circ {\rm braid}^{(F)}_{s,t} \circ (\overline{{\rm bivalent}^{(F)}_s} \otimes \overline{{\rm bivalent}^{(F)}_t} \otimes \overline{{\rm bivalent}^{(F)}_s} \otimes \dotsb) \text{.}
\end{multline*}
%
\end{defn}

\begin{exam} \hfill
\begin{enumerate}
\item Let $p=3$ and $\W$ be of type $\widetilde{A_2}$, following the conventions in Example~\ref{exam:affprho}. 
Figure~\ref{fig:affpbraidtypeA} depicts ${\rm braid}^{(F)}_{1,0}$. 
Here we have $\delta_{1}=+1$, $\delta_{0}=-1$, and $(-1)^{\len(w)}=(-1)^4=+1$, so $\delta=-1$ as $m_{10}=3$ is odd.

\item Let $p=3$ and $\W$ be of type $\widetilde{C_2}$.
As with $\widetilde{A}_2$ we label the generators $\tilde{s}=0$, $1$, and $2$ and color them blue, red, and green respectively, such that $0101=1010$ and $02=20$. 
Figure~\ref{fig:affpbraidtypeB} depicts ${\rm braid}^{(F)}_{1,0}$. 
Here we have $(-1)^{\len(w)}=(-1)^6=+1$, so $\delta=+1$ as $m_{10}=4$ is even.
\end{enumerate}
\end{exam}

\begin{figure}
    \centering
\begin{tikzpicture}[xscale=0.9,yscale=0.76]
\node[anchor=south west,inner sep=0] at (0,0) {\scalebox{0.9}[0.9]{\includegraphics{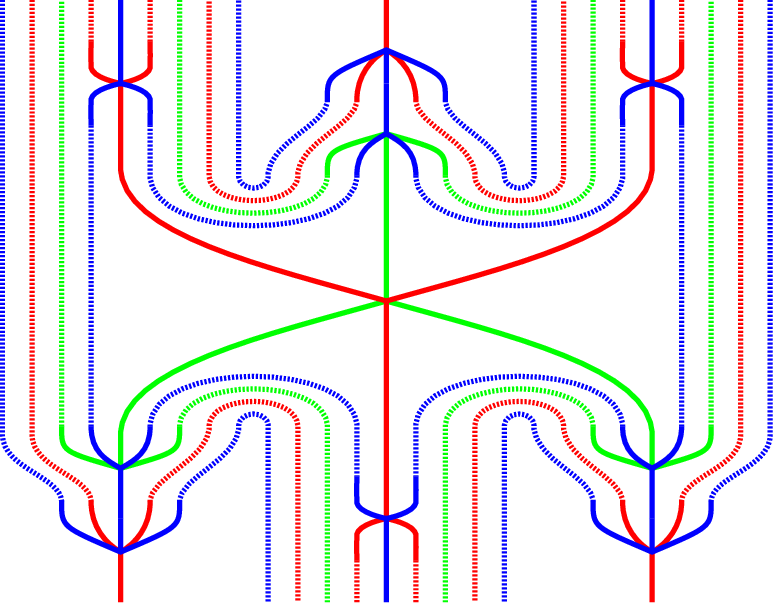}}};
 \node at (-7pt,171pt) {$-$};
 \node[font=\footnotesize] at (35pt,58pt){$\projscal{blue}$};
 \node[font=\footnotesize] at (50pt,58pt){$\projscal{red}$};
 \node[font=\footnotesize] at (35pt,102pt){$\projscal{green}$};
 \node[font=\footnotesize] at (50pt,102pt){$\projscal{blue}$};
 \node[font=\footnotesize] at (79pt,102pt){$\projscal{blue}$};
 \node[font=\footnotesize] at (92pt,102pt){$\projscal{green}$};
 \node[font=\footnotesize] at (79pt,58pt){$\projscal{red}$};
 \node[font=\footnotesize] at (92pt,58pt){$\projscal{blue}$};
 \node[font=\footnotesize] at (177pt,24pt){$\inclscal{red}$};
 \node[font=\footnotesize] at (206pt,24pt){$\inclscal{red}$};
 \node[font=\footnotesize] at (177pt,73pt){$\projscal{blue}$};
 \node[font=\footnotesize] at (206pt,73pt){$\projscal{blue}$};
 \node[font=\footnotesize] at (292pt,58pt){$\projscal{blue}$};
 \node[font=\footnotesize] at (306pt,58pt){$\projscal{red}$};
 \node[font=\footnotesize] at (292pt,101pt){$\projscal{green}$};
 \node[font=\footnotesize] at (306pt,101pt){$\projscal{blue}$};
 \node[font=\footnotesize] at (334pt,101pt){$\projscal{blue}$};
 \node[font=\footnotesize] at (348pt,101pt){$\projscal{green}$};
 \node[font=\footnotesize] at (334pt,58pt){$\projscal{red}$};
 \node[font=\footnotesize] at (348pt,58pt){$\projscal{blue}$};
 \node[font=\footnotesize] at (50pt,319pt){$\projscal{red}$};
 \node[font=\footnotesize] at (79pt,319pt){$\projscal{red}$};
 \node[font=\footnotesize] at (50pt,271pt){$\inclscal{blue}$};
 \node[font=\footnotesize] at (79pt,271pt){$\inclscal{blue}$};
 \node[font=\footnotesize] at (163pt,284pt){$\inclscal{blue}$};
 \node[font=\footnotesize] at (177pt,284pt){$\inclscal{red}$};
 \node[font=\footnotesize] at (163pt,242pt){$\inclscal{green}$};
 \node[font=\footnotesize] at (177pt,242pt){$\inclscal{blue}$};
 \node[font=\footnotesize] at (206pt,242pt){$\inclscal{blue}$};
 \node[font=\footnotesize] at (220pt,242pt){$\inclscal{green}$};
 \node[font=\footnotesize] at (207pt,284pt){$\inclscal{red}$};
 \node[font=\footnotesize] at (220pt,284pt){$\inclscal{blue}$};
 \node[font=\footnotesize] at (305pt,319pt){$\projscal{red}$};
 \node[font=\footnotesize] at (334pt,319pt){$\projscal{red}$};
 \node[font=\footnotesize] at (305pt,271pt){$\inclscal{blue}$};
 \node[font=\footnotesize] at (334pt,271pt){$\inclscal{blue}$};
 \node[font=\footnotesize] at (135pt,103pt){$\epsscal{blue}$};
 \node[font=\footnotesize] at (149pt,103pt){$\epsscal{red}$};
 \node[font=\footnotesize] at (163pt,103pt){$\epsscal{green}$};
 \node[font=\footnotesize] at (177pt,103pt){$\epsscal{blue}$};
 \node[font=\footnotesize] at (263pt,103pt){$\epsscal{blue}$};
 \node[font=\footnotesize] at (276pt,108pt){$\epsscal{red}$};
 \node[font=\footnotesize] at (288pt,113pt){$\epsscal{green}$};
 \node[font=\footnotesize] at (300pt,117pt){$\epsscal{blue}$};
 \node[font=\footnotesize] at (135pt,241pt){$\spescal{blue}$};
 \node[font=\footnotesize] at (148pt,237pt){$\spescal{red}$};
 \node[font=\footnotesize] at (159pt,231pt){$\spescal{green}$};
 \node[font=\footnotesize] at (168pt,223pt){$\spescal{blue}$};
 \node[font=\footnotesize] at (263pt,241pt){$\spescal{blue}$};
 \node[font=\footnotesize] at (277pt,241pt){$\spescal{red}$};
 \node[font=\footnotesize] at (291pt,241pt){$\spescal{green}$};
 \node[font=\footnotesize] at (305pt,241pt){$\spescal{blue}$};
\end{tikzpicture}
    \caption{The morphism ${\rm braid}^{(F)}_{1,0}$ for $p=3$ in type $\widetilde{A_2}$.}
    \label{fig:affpbraidtypeA}
\end{figure}

\begin{sidewaysfigure}
\centering
\begin{minipage}{7.5in}
\centering
    \begin{tikzpicture}[scale=0.65]
    \node[anchor=south west,inner sep=0] at (0,0) {\includegraphics[scale=0.65]{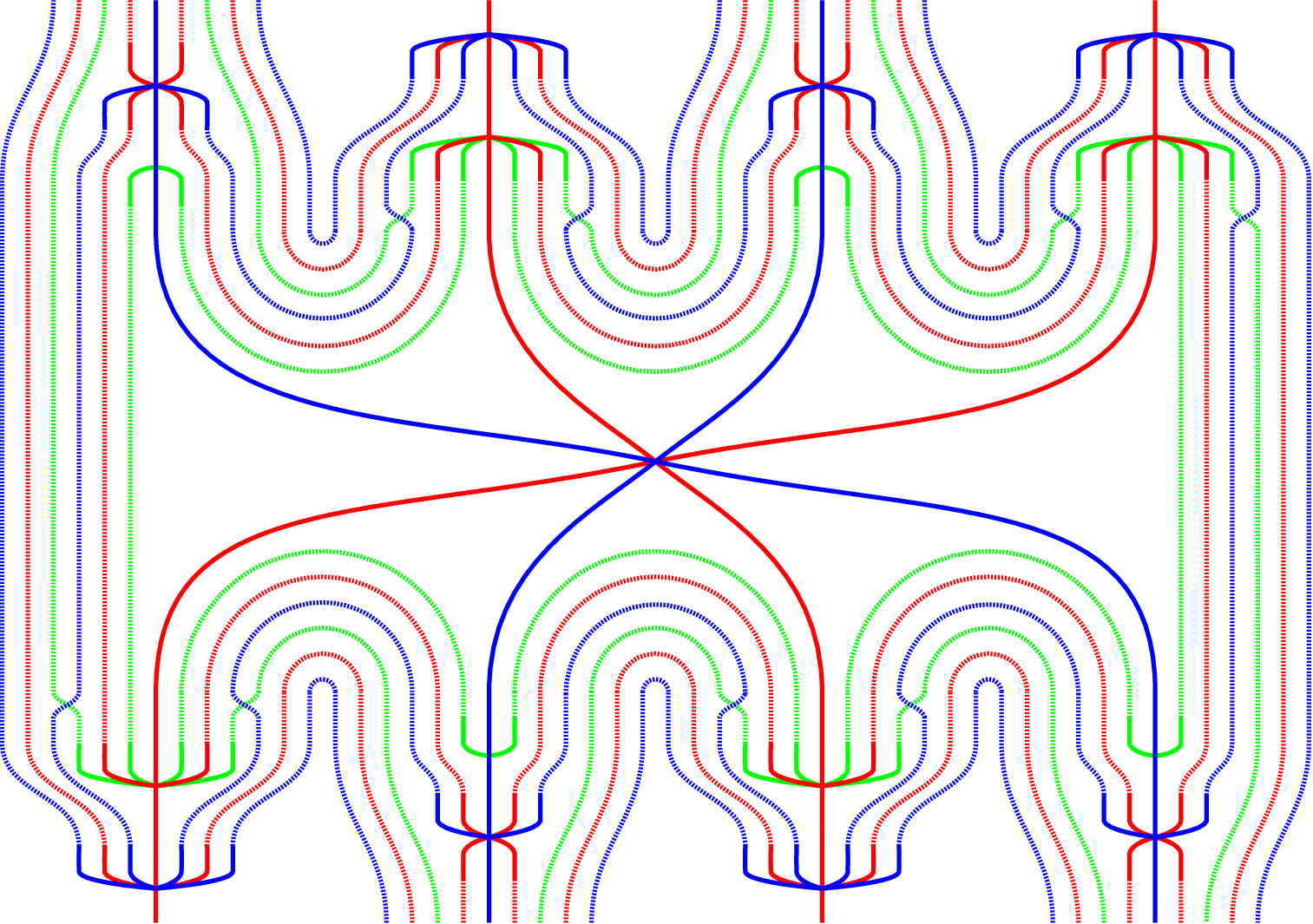}};
     \node[font=\footnotesize] at (50pt,44pt){$\projscal{blue}$};
     \node[font=\footnotesize] at (65pt,44pt){$\projscal{red}$};
     \node[font=\footnotesize] at (79pt,44pt){$\projscal{blue}$};
     \node[font=\footnotesize] at (50pt,100pt){$\projscal{green}$};
     \node[font=\footnotesize] at (65pt,100pt){$\projscal{red}$};
     \node[font=\footnotesize] at (79pt,100pt){$\projscal{green}$};
     \node[font=\footnotesize] at (107pt,100pt){$\projscal{green}$};
     \node[font=\footnotesize] at (122pt,100pt){$\projscal{red}$};
     \node[font=\footnotesize] at (136pt,100pt){$\projscal{green}$};
     \node[font=\footnotesize] at (107pt,44pt){$\projscal{blue}$};
     \node[font=\footnotesize] at (122pt,44pt){$\projscal{red}$};
     \node[font=\footnotesize] at (136pt,44pt){$\projscal{blue}$};
     \node[font=\footnotesize] at (193pt,130pt){$\epsscal{blue}$};
     \node[font=\footnotesize] at (207pt,130pt){$\epsscal{red}$};
     \node[font=\footnotesize] at (221pt,130pt){$\epsscal{green}$};
     \node[font=\footnotesize] at (235pt,130pt){$\epsscal{blue}$};
     \node[font=\footnotesize] at (250pt,130pt){$\epsscal{red}$};
     \node[font=\footnotesize] at (264pt,130pt){$\epsscal{green}$};
     \node[font=\footnotesize] at (264pt,25pt){$\inclscal{red}$};
     \node[font=\footnotesize] at (292pt,25pt){$\inclscal{red}$};
     \node[font=\footnotesize] at (250pt,73pt){$\projscal{blue}$};
     \node[font=\footnotesize] at (264pt,73pt){$\projscal{red}$};
     \node[font=\footnotesize] at (264pt,115pt){$\projscal{green}$};
     \node[font=\footnotesize] at (292pt,115pt){$\projscal{green}$};
     \node[font=\footnotesize] at (292pt,73pt){$\projscal{red}$};
     \node[font=\footnotesize] at (306pt,73pt){$\projscal{blue}$};
     \node[font=\footnotesize] at (378pt,130pt){$\epsscal{blue}$};
     \node[font=\footnotesize] at (392pt,130pt){$\epsscal{red}$};
     \node[font=\footnotesize] at (407pt,130pt){$\epsscal{green}$};
     \node[font=\footnotesize] at (420pt,130pt){$\epsscal{blue}$};
     \node[font=\footnotesize] at (434pt,130pt){$\epsscal{red}$};
     \node[font=\footnotesize] at (449pt,130pt){$\epsscal{green}$};
     \node[font=\footnotesize] at (420pt,44pt){$\projscal{blue}$};
     \node[font=\footnotesize] at (434pt,44pt){$\projscal{red}$};
     \node[font=\footnotesize] at (449pt,44pt){$\projscal{blue}$};
     \node[font=\footnotesize] at (420pt,100pt){$\projscal{green}$};
     \node[font=\footnotesize] at (434pt,100pt){$\projscal{red}$};
     \node[font=\footnotesize] at (449pt,100pt){$\projscal{green}$};
     \node[font=\footnotesize] at (477pt,100pt){$\projscal{green}$};
     \node[font=\footnotesize] at (491pt,100pt){$\projscal{red}$};
     \node[font=\footnotesize] at (506pt,100pt){$\projscal{green}$};
     \node[font=\footnotesize] at (477pt,44pt){$\projscal{blue}$};
     \node[font=\footnotesize] at (491pt,44pt){$\projscal{red}$};
     \node[font=\footnotesize] at (506pt,44pt){$\projscal{blue}$};
     \node[font=\footnotesize] at (563pt,130pt){$\epsscal{blue}$};
     \node[font=\footnotesize] at (577pt,130pt){$\epsscal{red}$};
     \node[font=\footnotesize] at (591pt,130pt){$\epsscal{green}$};
     \node[font=\footnotesize] at (605pt,130pt){$\epsscal{blue}$};
     \node[font=\footnotesize] at (620pt,130pt){$\epsscal{red}$};
     \node[font=\footnotesize] at (634pt,130pt){$\epsscal{green}$};
     \node[font=\footnotesize] at (634pt,25pt){$\inclscal{red}$};
     \node[font=\footnotesize] at (662pt,25pt){$\inclscal{red}$};
     \node[font=\footnotesize] at (620pt,73pt){$\projscal{blue}$};
     \node[font=\footnotesize] at (634pt,73pt){$\projscal{red}$};
     \node[font=\footnotesize] at (634pt,115pt){$\projscal{green}$};
     \node[font=\footnotesize] at (662pt,115pt){$\projscal{green}$};
     \node[font=\footnotesize] at (662pt,73pt){$\projscal{red}$};
     \node[font=\footnotesize] at (676pt,73pt){$\projscal{blue}$};
     \node[font=\footnotesize] at (79pt,488pt){$\projscal{red}$};
     \node[font=\footnotesize] at (107pt,488pt){$\projscal{red}$};
     \node[font=\footnotesize] at (65pt,441pt){$\inclscal{blue}$};
     \node[font=\footnotesize] at (79pt,441pt){$\inclscal{red}$};
     \node[font=\footnotesize] at (79pt,398pt){$\inclscal{green}$};
     \node[font=\footnotesize] at (107pt,398pt){$\inclscal{green}$};
     \node[font=\footnotesize] at (107pt,441pt){$\inclscal{red}$};
     \node[font=\footnotesize] at (122pt,441pt){$\inclscal{blue}$};
     \node[font=\footnotesize] at (193pt,380pt){$\spescal{blue}$};
     \node[font=\footnotesize] at (207pt,380pt){$\spescal{red}$};
     \node[font=\footnotesize] at (221pt,380pt){$\spescal{green}$};
     \node[font=\footnotesize] at (235pt,380pt){$\spescal{blue}$};
     \node[font=\footnotesize] at (250pt,380pt){$\spescal{red}$};
     \node[font=\footnotesize] at (264pt,380pt){$\spescal{green}$};
     \node[font=\footnotesize] at (235pt,469pt){$\inclscal{blue}$};
     \node[font=\footnotesize] at (250pt,469pt){$\inclscal{red}$};
     \node[font=\footnotesize] at (264pt,469pt){$\inclscal{blue}$};
     \node[font=\footnotesize] at (235pt,413pt){$\inclscal{green}$};
     \node[font=\footnotesize] at (250pt,413pt){$\inclscal{red}$};
     \node[font=\footnotesize] at (264pt,413pt){$\inclscal{green}$};
     \node[font=\footnotesize] at (292pt,413pt){$\inclscal{green}$};
     \node[font=\footnotesize] at (306pt,413pt){$\inclscal{red}$};
     \node[font=\footnotesize] at (320pt,413pt){$\inclscal{green}$};
     \node[font=\footnotesize] at (292pt,469pt){$\inclscal{blue}$};
     \node[font=\footnotesize] at (306pt,469pt){$\inclscal{red}$};
     \node[font=\footnotesize] at (320pt,469pt){$\inclscal{blue}$};
     \node[font=\footnotesize] at (378pt,380pt){$\spescal{blue}$};
     \node[font=\footnotesize] at (392pt,380pt){$\spescal{red}$};
     \node[font=\footnotesize] at (406pt,380pt){$\spescal{green}$};
     \node[font=\footnotesize] at (420pt,380pt){$\spescal{blue}$};
     \node[font=\footnotesize] at (434pt,380pt){$\spescal{red}$};
     \node[font=\footnotesize] at (449pt,380pt){$\spescal{green}$};
     \node[font=\footnotesize] at (449pt,488pt){$\projscal{red}$};
     \node[font=\footnotesize] at (477pt,488pt){$\projscal{red}$};
     \node[font=\footnotesize] at (434pt,441pt){$\inclscal{blue}$};
     \node[font=\footnotesize] at (449pt,441pt){$\inclscal{red}$};
     \node[font=\footnotesize] at (449pt,398pt){$\inclscal{green}$};
     \node[font=\footnotesize] at (477pt,398pt){$\inclscal{green}$};
     \node[font=\footnotesize] at (477pt,441pt){$\inclscal{red}$};
     \node[font=\footnotesize] at (491pt,441pt){$\inclscal{blue}$};
     \node[font=\footnotesize] at (563pt,380pt){$\spescal{blue}$};
     \node[font=\footnotesize] at (577pt,380pt){$\spescal{red}$};
     \node[font=\footnotesize] at (591pt,380pt){$\spescal{green}$};
     \node[font=\footnotesize] at (605pt,380pt){$\spescal{blue}$};
     \node[font=\footnotesize] at (620pt,380pt){$\spescal{red}$};
     \node[font=\footnotesize] at (634pt,380pt){$\spescal{green}$};
     \node[font=\footnotesize] at (605pt,469pt){$\inclscal{blue}$};
     \node[font=\footnotesize] at (620pt,469pt){$\inclscal{red}$};
     \node[font=\footnotesize] at (634pt,469pt){$\inclscal{blue}$};
     \node[font=\footnotesize] at (605pt,413pt){$\inclscal{green}$};
     \node[font=\footnotesize] at (620pt,413pt){$\inclscal{red}$};
     \node[font=\footnotesize] at (634pt,413pt){$\inclscal{green}$};
     \node[font=\footnotesize] at (662pt,413pt){$\inclscal{green}$};
     \node[font=\footnotesize] at (676pt,413pt){$\inclscal{red}$};
     \node[font=\footnotesize] at (690pt,413pt){$\inclscal{green}$};
     \node[font=\footnotesize] at (662pt,469pt){$\inclscal{blue}$};
     \node[font=\footnotesize] at (676pt,469pt){$\inclscal{red}$};
     \node[font=\footnotesize] at (690pt,469pt){$\inclscal{blue}$};
    \end{tikzpicture}
\end{minipage}
    \caption{The morphism ${\rm braid}^{(F)}_{1,0}$ for $p=3$ in type $\widetilde{C_2}$.}
    \label{fig:affpbraidtypeB}
\end{sidewaysfigure}

We are now ready for long-awaited construction of the Frobenius functor $F$. 
We will first define $F$ on $\dgrmstdF$ and then extend to $\dgrmBSstdF$.

\begin{thm}
\label{thm:frobfunctorstd}
There is a well-defined faithful $R$-linear monoidal functor 
\begin{equation*}
\gls*{F} : \dgrmstdF \longrightarrow \hat{R} \otimes \dgrmstd
\end{equation*}
defined on generating objects by
\begin{equation*}
R^F_s \longmapsto \hat{R}^{(F)}_s
\end{equation*}
for all $s \in \Ss$, and defined on generating morphisms by
\begin{align*}
{\rm stdcap}^F_s & \longmapsto {\rm stdcap}^{(F)}_s \\
{\rm stdcup}^F_s & \longmapsto {\rm stdcup}^{(F)}_s
\end{align*}
for all $s \in \Ss$, and
\begin{equation*}
{\rm stdbraid}^F_{s,t} \longmapsto {\rm stdbraid}^{(F)}_{s,t}
\end{equation*}
for all $(s,t) \in \Ss \times \Ss$ with $s \neq t$ such that $m_{st}<\infty$. 
\end{thm}

\begin{proof}
We must show that ${\rm stdcap}^{(F)}$, ${\rm stdcup}^{(F)}$, and ${\rm stdbraid}^{(F)}$ satisfy the relations \cite[(4.1)--(4.9)]{ew-soergelcalc} defining a presentation of $\dgrmstdF$. 
Each relation involving only standard caps and cups follows immediately from Proposition~\ref{prop:dotforkFrelns}, so we turn our attention to relations involving standard braids.

The first relation to check is cyclicity of the standard braid. 
First we claim that ${\rm braid}^{(F)}_{s,t}$ is cyclic with respect to ${\rm cap}^{(F)}$ and ${\rm cup}^{(F)}$. 
This follows from Lemma~\ref{lem:epsbiadjbraidcyclic}, because each of the generators in the definition of ${\rm braid}^{(F)}_{s,t}$ (i.e.~${\epsilon}$, ${\rm conjbraid}$, ${\rm braid}$) is cyclic with respect to $\epsilon$-morphisms or cups/caps.
By standardizing using ${\rm bivalent}^{(F)}$ it follows that ${\rm stdbraid}^{(F)}$ is cyclic with respect to ${\rm stdcap}^{(F)}$ and ${\rm stdcup}^{(F)}$.

%
%

The remaining relations \cite[(4.4)--(4.9)]{ew-soergelcalc} only involve standard braids. 
By definition ${\rm stdbraid}^{(F)}$ is an $\epsilon$-basic standard diagram, so these relations must be satisfied up to sign by Proposition~\ref{prop:epscupcapinv}.
It is easy to check that the relations \cite[(4.4), (4.6)--(4.9)]{ew-soergelcalc} are invariant under arbitrary sign changes of braids, because both the left- and right-hand sides of each of these relations have the same number of braids of each color. 
For the remaining relation \cite[(4.5)]{ew-soergelcalc}, it is enough to show that if $m_{s\tilde{s}}=2$ then $\len(w)$ is even. 
In this case, the reflection hyperplanes of $\tilde{s}$ and $\sh$ (which are parallel) are orthogonal to the reflection hyperplane of $s$. 
Thus $\sh s \sh=s$, so $(\tilde{s}\sh)^{(p-1)/2} s (\sh \tilde{s})^{(p-1)/2}=s$ and we may take $w=(\tilde{s} \sh)^{(p-1)/2}$, which has even length. 

Thus $F$ is well defined. 
For faithfulness, we recall that $\Hom$-spaces in $\dgrmstdF$ are either $0$ or are spanned by a basic standard diagram (Theorem~\ref{thm:stdhoms}). 
Since the $F$-image of these basic standard diagrams is obviously non-zero this completes the proof.
\end{proof}

\begin{thm} \label{thm:frobfunctor}
The functor $F$ above extends to a faithful $R$-linear monoidal functor 
\begin{equation*}
F : \dgrmBSstdF \longrightarrow \hat{R} \otimes \dgrmBSstd
\end{equation*}
defined on Bott--Samelson generating objects by
\begin{equation*}
B^F_s \longmapsto B^{(F)}_s
\end{equation*}
for all $s \in \Ss$, and defined on Bott--Samelson generating morphisms by
\begin{align*}
{\rm dot}^F_s & \longmapsto {\rm dot}^{(F)}_s \\
{\rm fork}^F_s & \longmapsto {\rm fork}^{(F)}_s \\
{\rm bivalent}^F_s & \longmapsto {\rm bivalent}^{(F)}_s
\end{align*}
for all $s \in \Ss$, and
\begin{equation*}
{\rm braid}^F_{s,t} \longmapsto {\rm braid}^{(F)}_{s,t}
\end{equation*}
for all distinct $s,t \in \Ss$ such that $m_{st}<\infty$ (and similarly on their duals).
\end{thm}

\begin{proof}
Instead of checking the relations defining $\dgrmBSF$ as in the proofs of Prop~\ref{prop:dotforkFrelns} and Theorem~\ref{thm:frobfunctorstd}, we will use localization. 
This indirect strategy avoids having to check the Zamolodchikov relations by hand. 
The first step is to show that localization in the image of $F$ makes sense.

Using Theorem~\ref{thm:frobfunctorstd} we can extend $F$ to a $Q$-linear monoidal functor
\begin{equation*}
Q \otimes_R \dgrmstdF \longrightarrow Q \otimes_R \hat{R} \otimes_R \dgrmBSstd \longrightarrow \hat{Q} \otimes \dgrmBSstd \text{,}
\end{equation*}
where the first arrow is an extension by scalars to $Q$, and the second arrow uses the embedding $Q \otimes_R \hat{R} \rightarrow \hat{Q}$.
Now consider the putative extension of $F$ to Bott--Samelson morphisms. 
We can localize the image of these morphisms using tensor products of $\pi_{t,0}^{(F)}$ and $\pi_{t,1}^{(F)}$ (and their duals) for $t \in \Ss$. 
Localization in this manner produces matrices of maps in the $F$-image of $Q \otimes \dgrmstdF$. 

By Proposition~\ref{prop:dotforkFrelns}, the dual-valued localization matrices of ${\rm dot}^{(F)}_s$ and ${\rm fork}^{(F)}_s$ coincide with the $F$-images of the dual-valued localization matrices of ${\rm dot}^F_s$ and ${\rm fork}^F_s$ for all $s \in \Ss$. 
We claim it is enough to show that the same holds true for ${\rm braid}^{(F)}_{s,t}$ for all $s,t \in \Ss$ with $m_{st}<\infty$. 
For if so, then the localizations of the remaining relations in $\dgrmBSF$ (those involving braids) hold too.
As localization is faithful, this will complete the proof.

To compute localization matrices it is enough to show that the Jones--Wenzl relation \cite[(10.8i)]{emtw} holds for ${\rm braid}^{(F)}_{s,t}$.
This is immediate if both $s,t \in \Ssf$.
So suppose $s \in \Ssf$ with $m_{s\tilde{s}}<\infty$. 
To check the Jones--Wenzl relation for ${\rm braid}^{(F)}_{s,\tilde{s}}$, we apply Lemma~\ref{lem:rhomapsdotsforks}. 
If we place a ${\rm dot}^{(F)}_{s}$ or ${\rm dot}^{(F)}_{\tilde{s}}$ on top of ${\rm braid}^{(F)}_{s,\tilde{s}}$ the dot ``propagates'' down along the straight-ahead walk until it reaches ${\rm braid}_{s',\tilde{s}'}$, where we can apply the ordinary Jones--Wenzl relation to yield a sum of diagrams. 
The ``leading'' diagram in the Jones--Wenzl relation has no forks and one dot (this is a defining feature of the Jones--Wenzl idempotent). 
We then apply Lemma~\ref{lem:rhomapsdotsforks} again to continue propagating this dot down to the bottom of the diagram. 
See Figure~\ref{fig:affpbraidtypeAspot} for an example.
Finally, we apply Proposition~\ref{prop:epscupcapinv} and Lemma~\ref{lem:rhomapsid} several times to yield $\pm {\rm dot}^{(F)}_t \otimes \ident$ for some $t \in \{s,\tilde{s}\}$. 
A careful count of the sign changes shows that this process changes the sign by $\delta$, so we actually have ${\rm dot}^{(F)}_t \otimes \ident$.
We can similarly simplify other terms in the sum (using \eqref{eq:rhomapsforks} for any forks) to show that the resulting diagram can be generated by ${\rm dot}^{(F)}_t$ and ${\rm fork}^{(F)}_t$ for $t \in \{s,\tilde{s}\}$.
These two facts are enough to imply the Jones--Wenzl relation by Proposition~\ref{prop:JWindirect} (note that we already proved the cyclicity of ${\rm braid}^{(F)}_{s\tilde{s}}$ separately in the proof of Theorem~\ref{thm:frobfunctorstd}).
\end{proof}

\begin{figure}
    \centering
\begin{tikzpicture}[scale=0.9]
\node[anchor=south west,inner sep=0] at (0,0) {\includegraphics[scale=0.9]{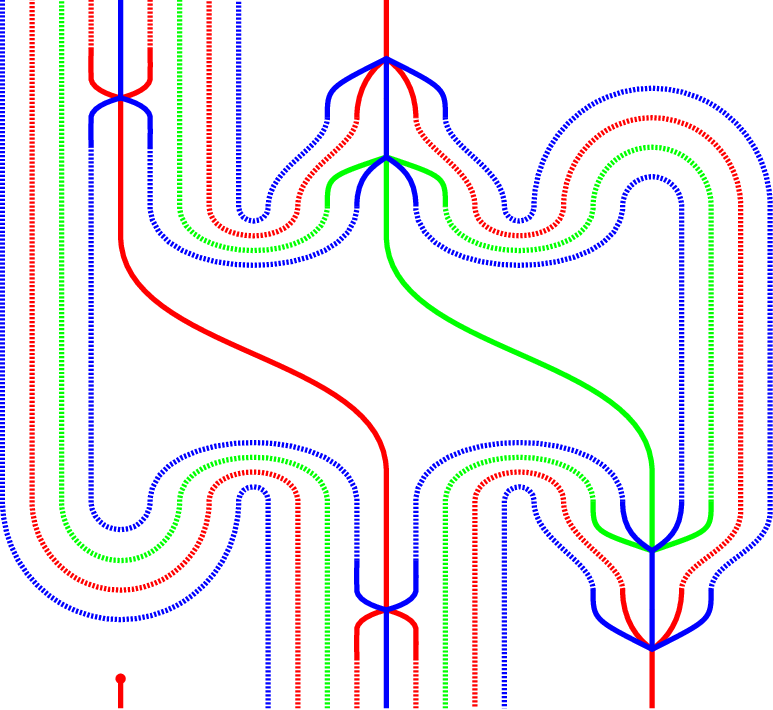}};
 \node at (-7pt,171pt) {$\delta$};
 \node[font=\footnotesize] at (78pt,101pt){$\spescal{blue}$};
 \node[font=\footnotesize] at (93pt,101pt){$\spescal{green}$};
 \node[font=\footnotesize] at (107pt,101pt){$\spescal{red}$};
 \node[font=\footnotesize] at (121pt,101pt){$\spescal{blue}$};
 \node at (60pt,28pt){$\delta_{\color{red} 1}$};
 \node[font=\footnotesize] at (177pt,24pt){$\inclscal{red}$};
 \node[font=\footnotesize] at (206pt,24pt){$\inclscal{red}$};
 \node[font=\footnotesize] at (177pt,73pt){$\projscal{blue}$};
 \node[font=\footnotesize] at (206pt,73pt){$\projscal{blue}$};
 \node[font=\footnotesize] at (292pt,58pt){$\projscal{blue}$};
 \node[font=\footnotesize] at (306pt,58pt){$\projscal{red}$};
 \node[font=\footnotesize] at (292pt,101pt){$\projscal{green}$};
 \node[font=\footnotesize] at (306pt,101pt){$\projscal{blue}$};
 \node[font=\footnotesize] at (334pt,101pt){$\projscal{blue}$};
 \node[font=\footnotesize] at (348pt,101pt){$\projscal{green}$};
 \node[font=\footnotesize] at (334pt,58pt){$\projscal{red}$};
 \node[font=\footnotesize] at (348pt,58pt){$\projscal{blue}$};
 \node[font=\footnotesize] at (50pt,319pt){$\projscal{red}$};
 \node[font=\footnotesize] at (79pt,319pt){$\projscal{red}$};
 \node[font=\footnotesize] at (50pt,271pt){$\inclscal{blue}$};
 \node[font=\footnotesize] at (79pt,271pt){$\inclscal{blue}$};
 \node[font=\footnotesize] at (163pt,284pt){$\inclscal{blue}$};
 \node[font=\footnotesize] at (177pt,284pt){$\inclscal{red}$};
 \node[font=\footnotesize] at (163pt,242pt){$\inclscal{green}$};
 \node[font=\footnotesize] at (177pt,242pt){$\inclscal{blue}$};
 \node[font=\footnotesize] at (206pt,242pt){$\inclscal{blue}$};
 \node[font=\footnotesize] at (220pt,242pt){$\inclscal{green}$};
 \node[font=\footnotesize] at (206pt,284pt){$\inclscal{red}$};
 \node[font=\footnotesize] at (220pt,284pt){$\inclscal{blue}$};
 \node[font=\footnotesize] at (135pt,103pt){$\epsscal{blue}$};
 \node[font=\footnotesize] at (149pt,103pt){$\epsscal{red}$};
 \node[font=\footnotesize] at (163pt,103pt){$\epsscal{green}$};
 \node[font=\footnotesize] at (177pt,103pt){$\epsscal{blue}$};
 \node[font=\footnotesize] at (263pt,103pt){$\epsscal{blue}$};
 \node[font=\footnotesize] at (276pt,108pt){$\epsscal{red}$};
 \node[font=\footnotesize] at (287pt,113pt){$\epsscal{green}$};
 \node[font=\footnotesize] at (299pt,118pt){$\epsscal{blue}$};
 \node[font=\footnotesize] at (135pt,241pt){$\spescal{blue}$};
 \node[font=\footnotesize] at (147pt,235pt){$\spescal{red}$};
 \node[font=\footnotesize] at (158pt,229pt){$\spescal{green}$};
 \node[font=\footnotesize] at (168pt,223pt){$\spescal{blue}$};
 \node[font=\footnotesize] at (263pt,241pt){$\spescal{blue}$};
 \node[font=\footnotesize] at (277pt,241pt){$\spescal{red}$};
 \node[font=\footnotesize] at (291pt,241pt){$\spescal{green}$};
 \node[font=\footnotesize] at (349pt,241pt){$\epsscal{green}$};
 \node[font=\footnotesize] at (363pt,241pt){$\epsscal{red}$};
 \node[font=\footnotesize] at (377pt,241pt){$\epsscal{blue}$};
 \node at (320pt,319pt){$\delta_{\color{blue} 0}$};
 \node[font=\footnotesize] at (306pt,241pt){$\spescal{blue}$};
 \node[font=\footnotesize] at (334pt,241pt){$\epsscal{blue}$};
\end{tikzpicture}
    \caption{The leading term in the Jones--Wenzl relation for ${\rm braid}^{(F)}_{1,0}$ (with $p=3$ in type $\widetilde{A_2}$).}
    \label{fig:affpbraidtypeAspot}
\end{figure}

\section{The category \texorpdfstring{$\dgrmpast$}{D(p|star)}}

By Theorem~\ref{thm:frobfunctor} the Frobenius functor induces a monoidal equivalence of categories between $\dgrmBSF$ and its image. 
In this section we will define a category $\dgrmpast$ extending this image and show that it categorifies the bimodule $\heckepast$.

\subsection{Construction}
\label{sec:Dpastconstr}

Let $\glsuseri*{dgrmBSdeg}$ denote the \defnemph{degrading} of $\dgrmBS$, which is just the category $\dgrmBS$ but without the grading on the $\Hom$-spaces. Similarly let $\dgrmdeg$ denote the additive Karoubi envelope of $\dgrmBSdeg$. 

\begin{defn} \label{defn:menorah}
Suppose $w \in \Wpcosets$ and $t' \in \Ss$ such that $wt'w^{-1}=F(\tilde{s})$ (as in Definition~\ref{defn:rhomaps}, with $t=\tilde{s}$). 
Let $\expr{w}$ be a rex for $w$. 
We define the \defnemph{menorah morphism} 
\begin{align*}
\glsuseri*{menorahstildeexprwtd} & : B^{(F)}_{\tilde{s}} \longrightarrow B_{\expr{w} \expr{t'} \expr{w}^{-1}} \\
\intertext{in $\hat{R} \otimes \dgrmBSstd$ by}
{\rm menorah}_{\tilde{s}}^{\expr{w},t'} & =(\overline{\pi_{\expr{w}}} \otimes \ident_{B_{t'}} \otimes \overline{\pi_{\expr{w}^{-1}}}) \circ \rho_{\tilde{s}}^{\expr{w},t'} \text{,}
\end{align*}
where $\pi_{\expr{w}}$ denotes $\pi_{\seq{e}}$ for the subsequence $\seq{e} \in [\expr{w}]$ whose type is all $1$s.
\end{defn}

\begin{exam} \hfill
\begin{enumerate}
\item Let $p=3$ and $\W$ be of type $\widetilde{A_2}$. 
Following the same diagrammatic convention as Example~\ref{exam:frobdotfork}, we depict two menorah morphisms in Figure~\ref{fig:menorahexpand}.

\item Let $p=3$ and $\W$ be of type $\widetilde{C_2}$. 
Following the same diagrammatic convention as Example~\ref{exam:affprho}, we depict two menorah morphisms in Figure~\ref{fig:menorahexpandtypeB}. 
(Moreover, the diagram in Figure~\ref{subfig:orthomenorah}, when viewed as a type $\widetilde{C_2}$ diagram, also depicts a menorah morphism for this case.)
\end{enumerate}
\end{exam}

\begin{figure}
\centering
\begin{subfigure}{0.4\textwidth}
\centering
\begin{tikzpicture}
\node[anchor=south west,inner sep=0] at (0,0) {\includegraphics{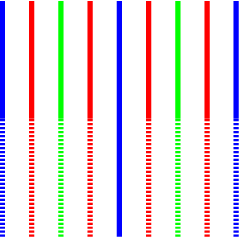}};
 \node[font=\footnotesize] at (7pt,57pt){$\inclscal{blue}$};
 \node[font=\footnotesize] at (21pt,57pt){$\inclscal{red}$};
 \node[font=\footnotesize] at (35pt,57pt){$\inclscal{green}$};
 \node[font=\footnotesize] at (49pt,57pt){$\inclscal{red}$};
 \node[font=\footnotesize] at (77pt,57pt){$\inclscal{red}$};
 \node[font=\footnotesize] at (91pt,57pt){$\inclscal{green}$};
 \node[font=\footnotesize] at (105pt,57pt){$\inclscal{red}$};
 \node[font=\footnotesize] at (119pt,57pt){$\inclscal{blue}$};
\end{tikzpicture}
\caption{$\expr{w}=\expr{0121}$, $t'=0$} \label{subfig:orthomenorah}
\end{subfigure} \qquad
\begin{subfigure}{0.4\textwidth}
\centering
\begin{tikzpicture}
\node[anchor=south west,inner sep=0] at (0,0) {\includegraphics{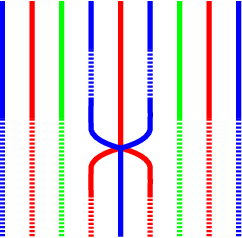}};
 \node[font=\footnotesize] at (7pt,57pt){$\inclscal{blue}$};
 \node[font=\footnotesize] at (21pt,57pt){$\inclscal{red}$};
 \node[font=\footnotesize] at (35pt,57pt){$\inclscal{green}$};
 \node[font=\footnotesize] at (49pt,20pt){$\inclscal{red}$};
 \node[font=\footnotesize] at (77pt,20pt){$\inclscal{red}$};
 \node[font=\footnotesize] at (49pt,67pt){$\projscal{blue}$};
 \node[font=\footnotesize] at (77pt,67pt){$\projscal{blue}$};
 \node[font=\footnotesize] at (49pt,90pt){$\inclscal{blue}$};
 \node[font=\footnotesize] at (77pt,90pt){$\inclscal{blue}$};
 \node[font=\footnotesize] at (92pt,57pt){$\inclscal{green}$};
 \node[font=\footnotesize] at (106pt,57pt){$\inclscal{red}$};
 \node[font=\footnotesize] at (120pt,57pt){$\inclscal{blue}$};
\end{tikzpicture}
\caption{$\expr{w}=\expr{0120}$, $t'=1$} \label{subfig:braidmenorah}
\end{subfigure}
\caption{Two menorah morphisms for $p=3$ in type $\widetilde{A_2}$.}
\label{fig:menorahexpand}
\end{figure}

\begin{figure}
\centering
\begin{subfigure}{0.4\textwidth}
\centering
\begin{tikzpicture}[scale=0.8,font=\footnotesize]
\node[anchor=south west,inner sep=0] at (0,0) {\scalebox{0.8}[0.8]{\includegraphics{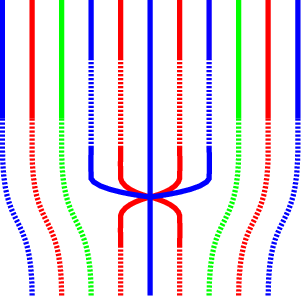}}};
         \node at (64pt,24pt){$\inclscal{red}$};
         \node at (93pt,24pt){$\inclscal{red}$};
         \node at (50pt,72pt){$\projscal{blue}$};
         \node at (64pt,72pt){$\projscal{red}$};
         \node at (93pt,72pt){$\projscal{red}$};
         \node at (107pt,72pt){$\projscal{blue}$};
         \node at (64pt,114pt){$\inclscal{red}$};
         \node at (93pt,114pt){$\inclscal{red}$};
         \node at (50pt,114pt){$\inclscal{blue}$};
         \node at (107pt,114pt){$\inclscal{blue}$};
         \node at (36pt,87pt){$\inclscal{green}$};
         \node at (121pt,87pt){$\inclscal{green}$};
         \node at (22pt,87pt){$\inclscal{red}$};
         \node at (135pt,87pt){$\inclscal{red}$};
         \node at (7pt,87pt){$\inclscal{blue}$};
         \node at (150pt,87pt){$\inclscal{blue}$};
\end{tikzpicture}
\caption{$\expr{w}=\expr{01201}$, $t'=0$}
\end{subfigure} \qquad
\begin{subfigure}{0.4\textwidth}
\centering
\begin{tikzpicture}[scale=0.8,font=\footnotesize]
\node[anchor=south west,inner sep=0] at (0,0) {\scalebox{0.8}[0.8]{\includegraphics{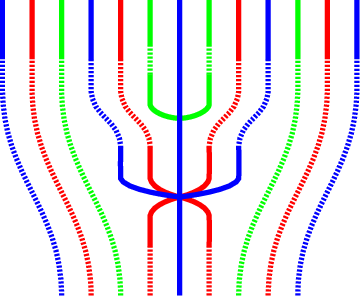}}};
         \node at (79pt,24pt){$\inclscal{red}$};
         \node at (107pt,24pt){$\inclscal{red}$};
         \node at (64pt,72pt){$\projscal{blue}$};
         \node at (79pt,72pt){$\projscal{red}$};
         \node at (107pt,72pt){$\projscal{red}$};
         \node at (121pt,72pt){$\projscal{blue}$};
         \node at (64pt,114pt){$\inclscal{red}$};
         \node at (121pt,114pt){$\inclscal{red}$};
         \node at (50pt,114pt){$\inclscal{blue}$};
         \node at (135pt,114pt){$\inclscal{blue}$};
         \node at (36pt,114pt){$\inclscal{green}$};
         \node at (150pt,114pt){$\inclscal{green}$};
         \node at (22pt,114pt){$\inclscal{red}$};
         \node at (164pt,114pt){$\inclscal{red}$};
         \node at (7pt,114pt){$\inclscal{blue}$};
         \node at (178pt,114pt){$\inclscal{blue}$};
         \node at (79pt,105pt){$\projscal{green}$};
         \node at (107pt,105pt){$\projscal{green}$};
         \node at (79pt,122pt){$\inclscal{green}$};
         \node at (107pt,122pt){$\inclscal{green}$};
\end{tikzpicture}
\caption{$\expr{w}=\expr{012012}$, $t'=0$}
\end{subfigure}
\caption{Two menorah morphisms for $p=3$ in type $\widetilde{C_2}$.}
\label{fig:menorahexpandtypeB}
\end{figure}

\begin{defn}
The category $\widetilde{\dgrm}_{{\rm BS},p|\ast}$ is an $\hat{R}$-linear subcategory of $\hat{R} \otimes \dgrmBSstd$ equipped with the structure of a $(\dgrmBSF,\dgrmBSdeg)$-bimodule defined as follows.
 
\begin{itemize}
\item The objects of $\widetilde{\dgrm}_{{\rm BS},p|\ast}$ are of the form  $\glsuseri*{Bppastpexprxexprw}=B^{(F)}_{F^{-1}(\expr{x})} \otimes B_{\expr{w}}$ for any $\expr{x}|\expr{w} \in \linkexpr$. 
The $(\dgrmBSF,\dgrmBSdeg)$-bimodule structure is defined by
\begin{equation*}
B^F_{\expr{y}} \otimes B^{(p|\ast)}_{\expr{x}|\expr{w}} \otimes B_{\expr{z}}=B^{(p|\ast)}_{F(\expr{y})\expr{x}|\expr{w z}} \text{.}
\end{equation*}

\item The morphisms of $\widetilde{\dgrm}_{{\rm BS},p|\ast}$ are generated (under $\hat{R}$-linear combinations, compositions, and the $(\dgrmBSF,\dgrmBSdeg)$-bimodule structure) by all possible menorah morphisms. 
In particular, $\widetilde{\dgrm}_{{\rm BS},p|\ast}$ contains all morphisms in $\dgrmBSdeg$ and all morphisms in the image of $F$, via the $(\dgrmBSF,\dgrmBSdeg)$-bimodule structure acting on $\ident_{\hat{R}}$.
\end{itemize}
\end{defn}

We call objects in $\widetilde{\dgrm}_{{\rm BS},p|\ast}$ \defnemph{(diagrammatic) $(p|\ast)$-Bott--Samelson bimodules}.
For each $\expr{x} \in \Sspexpr$, we also write $B^{(p)}_{\expr{x}}$ for the object $B^{(F)}_{F^{-1}(\expr{x})}$ in $\widetilde{\dgrm}_{{\rm BS},p|\ast}$.

\begin{notn} \label{notn:lightercol}
We assign a lighter version of the color corresponding to $\tilde{s}$ to the $\Wp$-generator $\tilde{s}_p$ (e.g.~if $\tilde{s}$ is colored blue then $\tilde{s}_p$ is colored cyan). In the diagrams we use this color to abbreviate morphisms which involve $B^{(F)}_{\tilde{s}}$ or $B^F_{\tilde{s}}$, by using solid $\tilde{s}_p$-colored lines. For example, the morphisms corresponding to ${\rm dot}^{(F)}_{\tilde{s}}$, ${\rm fork}^{(F)}_{\tilde{s}}$, and ${\rm braid}^{(F)}_{\tilde{s}s}$ abbreviate to 
\begin{center}
\begin{tabular}{ccc}
\includegraphics{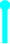} & 
\includegraphics{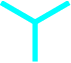} & 
\includegraphics{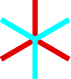} \\
$p$-dot & $p$-fork & $p$-braid
\end{tabular}
\end{center}
so that they look similar to their counterparts in $\dgrmF$. Similarly, the menorah morphism in Figure~\ref{subfig:orthomenorah} abbreviates to
\begin{center}
\begin{tabular}{c}
\includegraphics{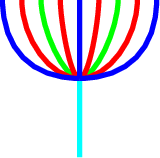} \\
menorah
\end{tabular}
\end{center}
For this reason we will also call these morphisms ``vertices''. 
%
\end{notn}

\begin{rem} \hfill
\begin{enumerate}
\item Only a subset of the menorah morphisms are actually needed to generate $\widetilde{\dgrm}_{{\rm BS},p|\ast}$. 
For example, the menorah morphism in Figure~\ref{subfig:orthomenorah} also generates (using braids, dots and forks) the menorah morphism in Figure~\ref{subfig:braidmenorah}.
We have included such superfluous morphisms in our definitions to simplify the basis construction in \S\ref{sec:pstarlightleaves}.


\item Menorah vertices are, strictly speaking, not cyclic, since some rotations of a menorah vertex do not correspond to a morphism in $\widetilde{\dgrm}_{{\rm BS},p|\ast}$. 
However, menorah vertices are what we call \defnemph{semi-cyclic}. 
In other words, if we twist such a menorah vertex by $180$ degrees clockwise we get its dual, and vice-versa. 

\item More generally, consider diagrams in $\widetilde{\dgrm}_{{\rm BS},p|\ast}$ drawn according to the convention in Notation~\ref{notn:lightercol}, and consider isotopies of such diagrams which never factor through a non-$(p|\ast)$-Bott--Samelson bimodule. 
Then any such isotopy class defines a unique morphism (not just up to sign).
\end{enumerate}
\end{rem}

The grading on $\widetilde{\dgrm}_{{\rm BS},p|\ast}$ inherited from $\dgrmBSstd$ is not a very useful invariant because $\hat{R}$ is no longer meaningfully graded. However we can define a valuation on morphisms in the sense of Definition~\ref{defn:valmod}. 

\begin{defn}
Let $L$ be a (possibly dual-valued) morphism in $\widetilde{\dgrm}_{{\rm BS},p|\ast}$. 
Localize $\Ss$-colored strings as in \S\ref{sec:localization} and the $\Ssp$-colored strings as in the proof of Theorem~\ref{thm:frobfunctor} to obtain a dual-valued matrix of standard morphisms. 
We can simplify the entries of this matrix so that they each consist of a single basic standard diagram with a dual-valued coefficient on the left. 
The \defnemph{valuation} of $L$, which we write as $\glsuseri*{val} L$, is defined to be the minimal valuation of all these coefficients.
\end{defn}

Note that this valuation is a special case of the induced valuation from Example~\ref{exam:valmodind}.

\begin{exam} \hfill
\begin{enumerate}
\item In Example~\ref{exam:localization} we showed that for each $s \in \Ss$ the localization matrix of ${\rm dot}_s$ is 
\begin{equation*}
\begin{bmatrix}
0 & \binom{a_s}{1}
\end{bmatrix} \text{.}
\end{equation*}
Recall that $\val a_s=2$ if $s \in \Ssf$ but $\val a_{\tilde{s}}=0$. Thus
\begin{equation*}
\val {\rm dot}_s=\begin{cases}
1 & \text{if $s \in \Ssf$,} \\
0 & \text{if $s=\tilde{s}$.}
\end{cases}
\end{equation*}

\item By Proposition~\ref{prop:dotforkFrelns}, the localization matrix of ${\rm fork}^{(F)}_{\tilde{s}}$ coincides with that of ${\rm fork}^F_{\tilde{s}}$:
\begin{equation*}
\begin{bmatrix}
0 & \binom{1}{-(a_{F(\tilde{s})})^{-1}} & \binom{1}{(a_{F(\tilde{s})})^{-1}} & 0 \\
\binom{1}{-(a_{F(\tilde{s})})^{-1}} & 0 & 0 & \binom{1}{(a_{F(\tilde{s})})^{-1}}
\end{bmatrix} \text{.}
\end{equation*}
Since $\val a_{F(\tilde{s})}=2$, this means that $\val({\rm fork}^{(F)}_{\tilde{s}})=-1$.
\end{enumerate}
\end{exam}

The previous example generalizes as follows.

\begin{lem} \hfill \label{lem:valF}
\begin{enumerate}[label={\rm (\roman*)}]
\item Let $\LLdbl$ be a double leaves map in $\dgrmBSF$. Then $\val F(\LLdbl) =\deg \LLdbl$. \label{enum:valLLdbl}

\item Let $L$ be a homogeneous morphism in $\dgrmBSF$ and $P$ a morphism in $\widetilde{\dgrm}_{{\rm BS},p|\ast}$. Then $\val(F(L) \otimes P)=\val F(L)+\val P$.
\end{enumerate}
\end{lem}

\begin{proof} \hfill
\begin{enumerate}[label={(\roman*)}]
\item Every double leaves map is the composition of some combination of tensor products of dots, forks, and braids. 
In particular, there are no non-trivial scalars in any region of the $\Ss$-graph for $\LLdbl$. 
The localization matrices of all dots, forks, and braids have entries in $\field(a_{F(s)} : s \in \Ss) \subseteq Q$. 
But for all $s \in \Ss$ we see that $\val a_{F(s)}=2=\deg a_{F(s)}$, and the localization matrix of $F(\LLdbl)$ is identical to that of $\LLdbl$ by the proof of Theorem~\ref{thm:frobfunctor}. 
This means that every left coefficient in the localization matrix for $F(\LLdbl)$ has valuation equal to the degree of $\LLdbl$.

\item The left coefficients in the localization matrix of $F(L) \otimes P$ are all of the form $fw(g)$, where $f$ is the left coefficient of some basic standard diagram with domain $\hat{Q}_w$ in the localization matrix of $F(L)$, while $g$ is a left coefficient from the localization matrix of $P$. 
From \ref{enum:valLLdbl} the only generating morphism of $\dgrmBSF$ whose valuation is \emph{not} equal to its degree is $a_{\tilde{s}} \ident_{R^F}$, which has degree $2$ but valuation $0$. 
Thus all left coefficients in the localization matrix of $F(L)$ have the same valuation. 
Moreover, each entry in the localization matrix of $F(L)$ has domain isomorphic to $\hat{Q}_{F(x)}$ for some $x \in \W$, i.e.~$w=F(x) \in \Wp$. 
For any $s \in \Ss$ and any $x \in \W$
\begin{equation*}
F(x)(a_s) \in a_s+\bigoplus_{t \in \Ss} \field a_{F(t)}=a_s+\bigoplus_{t \in \Ssf} \field a_t \text{,}
\end{equation*}
so $\val a_s=\val(F(x)(a_s))$. This means that $\val w(g)=\val g$, so $\val(fw(g))=\val f+\val g$ has minimal value $\val F(L)+\val P$. \qedhere
\end{enumerate}
\end{proof}

\begin{defn}
Let $\gls*{dgrmBSpast}$ be the following extension of $\widetilde{\dgrm}_{{\rm BS},p|\ast}$. 
The objects of $\dgrmBSpast$ are the formal symbols $B\langle m\rangle$, for each $m \in \ZZ$ and each object $B$ in $\widetilde{\dgrm}_{{\rm BS},p|\ast}$. 
The morphisms in $\dgrmBSpast$ are given by
\begin{equation*}
\Hom_{\dgrmBSpast}^{\blacklozenge}(B\langle m\rangle,B'\langle n\rangle)=(\Hom_{\widetilde{\dgrm}_{{\rm BS},p|\ast}}^{\blacklozenge}(B,B'))\langle n-m\rangle
\end{equation*}
for all objects $B,B'$ in $\widetilde{\dgrm}_{{\rm BS},p|\ast}$ and all $m,n \in \ZZ$. 
As with $\dgrmBS$, we will not always distinguish between morphisms in $\widetilde{\dgrm}_{{\rm BS},p|\ast}$ and morphisms in $\dgrmBSpast$.

We define the \defnemph{diagrammatic $(p|\ast)$-Hecke category} $\gls*{dgrmpast}$ to be the additive Karoubi envelope of $\dgrmBSpast$. In other words, $\dgrmpast$ is the closure of $\dgrmBSpast$ with respect to all finite direct sums, direct summands, and valuation shifts. 
\end{defn}


\begin{rem} \hfill
\begin{enumerate}
\item Even with the valuation structure on $\dgrmBSpast$, the right $\dgrmBSdeg$-module structure on $\dgrmBSpast$ does \emph{not} extend to a graded $\dgrmBS$-module structure. The primary reason for this is a mismatch between the degree and valuation of morphisms in $\dgrmBS$ (e.g.~$\deg({\rm dot}_{\tilde{s}})=1 \neq 0=\val({\rm dot}_{\tilde{s}})$). Unlike degree, valuation does not behave well algebraically in $\dgrmBS$; for example there are cases where $\val(L \otimes L') \neq \val L+\val L'$ for $L,L'$ two morphisms in $\dgrmBS$. We will return to this problem in \S\ref{sec:decat}.

\item Unlike with $\Ss$-graphs, the valuation of a diagram in $\dgrmBSpast$ is not necessarily the sum of the valuations of all vertices and scalars in the diagram. This is because for $L,L'$ two composable morphisms in $\dgrmBSpast$, we only have an inequality $\val(L \circ L') \geq \val L+\val L'$ rather than an equality.
\end{enumerate}
\end{rem}

\subsection{\texorpdfstring{$(p|\ast)$}{(p|star)}-light leaves}
\label{sec:pstarlightleaves}

We will construct a basis for the $\Hom$-spaces in $\dgrmBSpast$ analogous to the light leaves basis for $\dgrmBS$. Before we describe the construction it is necessary to generalize the notion of a rex move in $\dgrmBS$ to $\dgrmBSpast$. 

\begin{defn}
A \defnemph{braid-like menorah} is a morphism in $\dgrmBSpast$ of the form
\begin{equation*}
B^{(F)}_{\tilde{s}} \otimes B_{\expr{w}} \xrightarrow{{\rm menorah}_{\tilde{s}}^{\expr{w},t'} \otimes \ident} B_{\expr{w}} \otimes B_{t'} \otimes B_{\expr{w}^{-1}} \otimes B_{\expr{w}} \xrightarrow{ \ident \otimes {\rm cap}_{\expr{w}^{-1}}} B_{\expr{w}} \otimes B_{t'}
\end{equation*}
or its dual. 
We call a morphism in $\dgrmBSpast$ an \defnemph{mrex move} if it can be generated (using composition and the tensor product) from identity morphisms, braid morphisms, and braid-like menorahs.
\end{defn}

In other words, mrex moves correspond to morphisms in $\dgrmBSpast$ which do not factor through $(p|\ast)$-Bott--Samelson bimodules of shorter length than the domain or codomain.

\begin{exam}
Let $p=3$ and $\W$ be of type $\widetilde{A_2}$. 
Following the same diagrammatic convention as Example~\ref{exam:frobdotfork}, Figure~\ref{fig:mrexshamash} depicts two braid-like menorahs.
\end{exam}

\begin{figure}
    \begin{subfigure}{0.4\textwidth}
        \centering
        \includegraphics[scale=0.85]{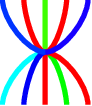}
    \end{subfigure}
    ~
    \begin{subfigure}{0.4\textwidth}
        \centering
        \includegraphics[scale=0.85]{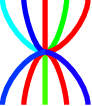}
    \end{subfigure}
    \caption{Two braid-like menorahs, for $p=3$ in type $\widetilde{A_2}$.}
    \label{fig:mrexshamash}
\end{figure}

\begin{constr} 
Let $\expr{x} \in \linkexpr$, $\patt{r} \in [\expr{x}]_{p|\ast}$ and $\seq{e} \in [\patt{r}]$. 
Suppose $\expr{w} \in \Sspexpr$ and $\expr{z} \in \expr{\Ss}$ are rexes for $\hat{e}\hat{r}^{-1}$ and $\hat{r}$. 
The \defnemph{$(p|\ast)$-light leaves morphism} $\glsuseri*{pastLLseqeexprwexprz} :B^{(p|\ast)}_{\expr{x}} \rightarrow B^{(p)}_{\expr{w}} \otimes \hat{R}_{\expr{z}}$ is a morphism in $\dgrmpast$ constructed inductively as follows. 
For each non-negative integer $i \leq \len(\expr{x})$ let $\expr{x}_{\leq i}$, $\patt{r}_{\leq i}$, and $\seq{e}_{\leq i}$ consist of the first $i$ terms of $\expr{x}$, $\patt{r}$, and $\seq{e}$ respectively. 
Fix rexes $\expr{w}_{\leq i} \in \Sspexpr$ and $\expr{z}_{\leq i} \in \expr{\Ss}$ for $\widehat{e_{\leq i}}(\widehat{r_{\leq i}})^{-1}$ and $\widehat{r_{\leq i}}$ respectively such that $\expr{w}=\expr{w}_{\leq \len(\expr{x})}$ and $\expr{z}=\expr{z}_{\leq \len(\expr{x})}$. 
We set $\pastLL_i=\pastLL_{\seq{e}_{\leq i},\expr{w}_{\leq i},\expr{z}_{\leq i}}$ and define $\pastLL_i=\phi_i \circ (\pastLL_{i-1} \otimes \ident_{B^{(p|\ast)}_{s_i}})$, where $\phi_i$ depends on the decorated type of $\seq{e}_i$. 
There are six possibilities for $\phi_i$, which are illustrated in Figure~\ref{fig:pastLL}. 
In these pictures, boxes labeled ``mrex'' are mrex moves, and trapezoids labeled ``std'' are basic standard diagrams with codomain corresponding to a rex. 
\end{constr}

\begin{figure}
    \begin{subfigure}[b]{0.3\textwidth}
        \centering
        \begin{tikzpicture}
        \node[anchor=south west,inner sep=0] at (0,0) {\includegraphics{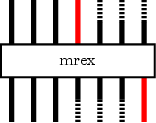}};
         \node[font=\footnotesize] at (43pt,10pt){$\inclscal{black}$};
         \node[font=\footnotesize] at (53pt,10pt){$\inclscal{black}$};
         \node[font=\footnotesize] at (64pt,10pt){$\inclscal{black}$};
         \node[font=\footnotesize] at (53pt,50pt){$\projscal{black}$};
         \node[font=\footnotesize] at (64pt,50pt){$\projscal{black}$};
         \node[font=\footnotesize] at (75pt,50pt){$\projscal{black}$};
        \end{tikzpicture}
        \caption{$\upseq 1$}
    \end{subfigure}
    ~
    \begin{subfigure}[b]{0.3\textwidth}
        \centering
        \begin{tikzpicture}
        \node[anchor=south west,inner sep=0] at (0,0) {\includegraphics{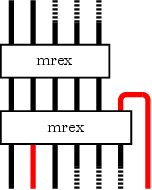}};
         \node[font=\footnotesize] at (42pt,10pt){$\inclscal{black}$};
         \node[font=\footnotesize] at (53pt,10pt){$\inclscal{black}$};
         \node[font=\footnotesize] at (64pt,10pt){$\inclscal{black}$};
         \node[font=\footnotesize] at (32pt,81pt){$\projscal{black}$};
         \node[font=\footnotesize] at (42pt,81pt){$\projscal{black}$};
         \node[font=\footnotesize] at (53pt,81pt){$\projscal{black}$};
        \end{tikzpicture}
        \caption{$\downseq 1$}
    \end{subfigure}
    ~
    \begin{subfigure}[b]{0.3\textwidth}
        \centering
        \begin{tikzpicture}
        \node[anchor=south west,inner sep=0] at (0,0) {\includegraphics{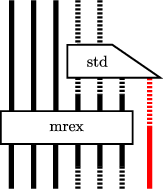}};
         \node[font=\footnotesize] at (43pt,10pt){$\inclscal{black}$};
         \node[font=\footnotesize] at (53pt,10pt){$\inclscal{black}$};
         \node[font=\footnotesize] at (64pt,10pt){$\inclscal{black}$};
         \node[font=\footnotesize] at (43pt,46pt){$\projscal{black}$};
         \node[font=\footnotesize] at (53pt,46pt){$\projscal{black}$};
         \node[font=\footnotesize] at (64pt,46pt){$\projscal{black}$};
         \node[font=\footnotesize] at (79pt,30pt){$\projscal{red}$};
        \end{tikzpicture}
        \caption{$\emptyset 1$}
    \end{subfigure}
    \begin{subfigure}[b]{0.3\textwidth}
        \centering
        \begin{tikzpicture}
        \node[anchor=south west,inner sep=0] at (0,0) {\includegraphics{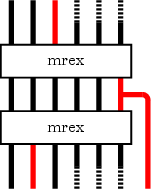}};
         \node[font=\footnotesize] at (42pt,10pt){$\inclscal{black}$};
         \node[font=\footnotesize] at (53pt,10pt){$\inclscal{black}$};
         \node[font=\footnotesize] at (64pt,10pt){$\inclscal{black}$};
         \node[font=\footnotesize] at (42pt,81pt){$\projscal{black}$};
         \node[font=\footnotesize] at (53pt,81pt){$\projscal{black}$};
         \node[font=\footnotesize] at (64pt,81pt){$\projscal{black}$};
        \end{tikzpicture}
        \caption{$\downseq 0$}
    \end{subfigure}
    ~
    \begin{subfigure}[b]{0.3\textwidth}
        \centering
        \begin{tikzpicture}
        \node[anchor=south west,inner sep=0] at (0,0) {\includegraphics{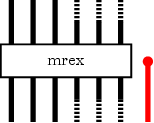}};
         \node[font=\footnotesize] at (43pt,10pt){$\inclscal{black}$};
         \node[font=\footnotesize] at (53pt,10pt){$\inclscal{black}$};
         \node[font=\footnotesize] at (64pt,10pt){$\inclscal{black}$};
         \node[font=\footnotesize] at (43pt,50pt){$\projscal{black}$};
         \node[font=\footnotesize] at (53pt,50pt){$\projscal{black}$};
         \node[font=\footnotesize] at (64pt,50pt){$\projscal{black}$};
        \end{tikzpicture}
        \caption{$\upseq 0$}
    \end{subfigure}
    ~
    \begin{subfigure}[b]{0.3\textwidth}
        \centering
        \begin{tikzpicture}
        \node[anchor=south west,inner sep=0] at (0,0) {\includegraphics{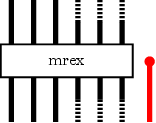}};
         \node[font=\footnotesize] at (43pt,10pt){$\inclscal{black}$};
         \node[font=\footnotesize] at (53pt,10pt){$\inclscal{black}$};
         \node[font=\footnotesize] at (64pt,10pt){$\inclscal{black}$};
         \node[font=\footnotesize] at (43pt,50pt){$\projscal{black}$};
         \node[font=\footnotesize] at (53pt,50pt){$\projscal{black}$};
         \node[font=\footnotesize] at (64pt,50pt){$\projscal{black}$};
         \node[font=\footnotesize] at (79pt,30pt){$\projscal{red}$};
        \end{tikzpicture}
        \caption{$\emptyset 0$}
    \end{subfigure}
    \caption{Six maps for constructing $(p|\ast)$-light leaves.}
    \label{fig:pastLL}
\end{figure}

As with the light leaves morphisms, the construction of $\pastLL_{\seq{e},\expr{w},\expr{z}}$ depends on several non-canonical choices and as such is not uniquely defined. 

\begin{exam}
Let $p=3$ and $\W$ be of type $\widetilde{A_2}$. Following the notational convention in Example~\ref{exam:w-twisted-p-defect} set $\expr{x}=\expr{0}_p|\expr{101202122} \in \linkexpr$. 
We depict a light leaves morphism for the subsequence of type $1111110110$ for the pattern $\ast \ast 1111 \ast \ast \ast \ast \in [\expr{x}]_{p|\ast}$ in Figure~\ref{fig:LLexample}, following the same diagrammatic conventions as Example~\ref{exam:frobdotfork}.
\begin{figure}
    \centering
    \begin{tikzpicture}
    \node[anchor=south west,inner sep=0] at (0,0) {\includegraphics{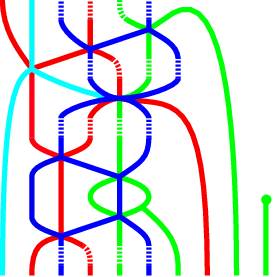}};
     \node[font=\footnotesize] at (35pt,5pt){$\projscal{blue}$};
     \node[font=\footnotesize] at (35pt,13pt){$\inclscal{blue}$};
     \node[font=\footnotesize] at (49pt,6pt){$\projscal{red}$};
     \node[font=\footnotesize] at (49pt,15pt){$\inclscal{red}$};
     \node[font=\footnotesize] at (63pt,6pt){$\projscal{green}$};
     \node[font=\footnotesize] at (63pt,15pt){$\inclscal{green}$};
     \node[font=\footnotesize] at (77pt,6pt){$\projscal{blue}$};
     \node[font=\footnotesize] at (77pt,15pt){$\inclscal{blue}$};
     \node[font=\footnotesize] at (35pt,68pt){$\projscal{blue}$};
     \node[font=\footnotesize] at (35pt,76pt){$\inclscal{blue}$};
     \node[font=\footnotesize] at (49pt,68pt){$\projscal{red}$};
     \node[font=\footnotesize] at (49pt,76pt){$\inclscal{red}$};
     \node[font=\footnotesize] at (63pt,68pt){$\projscal{green}$};
     \node[font=\footnotesize] at (63pt,76pt){$\inclscal{green}$};
     \node[font=\footnotesize] at (77pt,68pt){$\projscal{blue}$};
     \node[font=\footnotesize] at (77pt,76pt){$\inclscal{blue}$};
     \node[font=\footnotesize] at (49pt,95pt){$\projscal{blue}$};
     \node[font=\footnotesize] at (49pt,102pt){$\inclscal{blue}$};
     \node[font=\footnotesize] at (63pt,96pt){$\projscal{red}$};
     \node[font=\footnotesize] at (63pt,105pt){$\inclscal{red}$};
     \node[font=\footnotesize] at (77pt,96pt){$\projscal{green}$};
     \node[font=\footnotesize] at (77pt,105pt){$\inclscal{green}$};
     \node[font=\footnotesize] at (92pt,96pt){$\projscal{blue}$};
     \node[font=\footnotesize] at (92pt,105pt){$\inclscal{blue}$};
     \node[font=\footnotesize] at (35pt,125pt){$\projscal{blue}$};
     \node[font=\footnotesize] at (49pt,125pt){$\projscal{red}$};
     \node[font=\footnotesize] at (63pt,129pt){$\projscal{green}$};
     \node[font=\footnotesize] at (77pt,129pt){$\projscal{blue}$};
    \end{tikzpicture}
    \caption{A $(p|\ast)$-light leaves morphism.}
    \label{fig:LLexample}
\end{figure}
\end{exam}

\begin{lem} \label{lem:valpastLL}
Let $\expr{x} \in \linkexpr$, $\patt{r} \in [\expr{x}]_{p|\ast}$, and $\seq{e} \in [\patt{r}]$, and suppose $\expr{w} \in \Sspexpr$ and $\expr{z} \in \expr{\Ss}$ are rexes for $\hat{e}\hat{r}^{-1}$ and $\hat{r}$ respectively. 
Then $\val \pastLL_{\seq{e},\expr{w},\expr{z}}=\deg \pastLL_{\seq{e},\expr{w},\expr{z}}=\dfct_p(\seq{e})$.
\end{lem}

\begin{proof}
The localization matrices of the morphisms in Figure~\ref{fig:pastLL} have left coefficients lying in the subfield of $\hat{Q}$ generated by $\{a_{F(s)} : s \in \Ss\}$, which is also generated by $\{a_s : s \in \Ssf\}$. 
This means that the valuation of these morphisms is equal to the degree. 
It is easy to check that the $\upseq 0$ morphism has degree $+1$, the $\downseq 0$ morphism has degree $-1$, and all other morphisms have degree $0$. 
The total degree of $\pastLL_{\seq{e},\expr{w},\expr{z}}$ is the sum of these degrees and therefore equals $\dfct_p(\seq{e})$.
\end{proof}

Let $\expr{w} \in \Sspexpr$ and $\expr{z} \in \expr{\Ss}$ be rexes. 
We write $\pastLL_{[\expr{x}],\expr{w},\expr{z}}$ to denote a complete selection of $(p|\ast)$-light leaves morphisms $\pastLL_{\seq{e},\expr{w},\expr{z}}$ over all patterns $\patt{r} \in [\expr{x}]_{p|\ast}$ and all $\seq{e} \in [\patt{r}]$ such that $\hat{e}\hat{r}^{-1}=w$ and $\hat{r}=z$. 
As with ordinary light leaves, for expressions $\expr{x},\expr{y} \in \linkexpr$ and patterns $\patt{q} \in [\expr{x}]_{p|\ast}$ and $\patt{r} \in [\expr{y}]_{p|\ast}$, if we have $\seq{e} \in [\patt{q}]$ and $\seq{f} \in [\patt{r}]$ such that $\hat{e}=\hat{f}$, then we can construct the \defnemph{$(p|\ast)$-double leaves morphism} $\glsuseri*{pastLLdblseqeseqf}= \overline{\pastLL_{\seq{f},\expr{w},\expr{z}}} \circ \pastLL_{\seq{e},\expr{w},\expr{z}} : B^{(p|\ast)}_{\expr{x}} \rightarrow B^{(p|\ast)}_{\expr{y}}$ in $\dgrmBSpast$, where $\expr{w}$ and $\expr{z}$ are rexes for $\hat{e}\hat{q}^{-1}=\hat{f}\hat{r}^{-1}$ and $\hat{q}=\hat{r}$ respectively. 
Now suppose for each $w \in \Wp$ and $z \in \Wpcosets$ we have fixed corresponding rexes $\expr{w} \in \Sspexpr$ and $\expr{z} \in \expr{\Ss}$. 
We write $\pastLLdbl_{[\expr{x}]}^{[\expr{y}]}$ to denote a complete selection of $(p|\ast)$-double leaves morphisms $B^{(p|\ast)}_{\expr{x}} \rightarrow B^{(p|\ast)}_{\expr{y}}$.

For $\expr{x} \in \linkexpr$, write $\expr{x}_1 \in \expr{\Ss}$ for the expanded $\Ss$-generator form of $\expr{x}$, where $(\expr{\tilde{s}\sh})^{(p-1)/2}\expr{\tilde{s}}(\expr{\sh\tilde{s}})^{(p-1)/2}$ is substituted for each $\tilde{s}_p$. 
For $\seq{e} \in [\expr{x}]$ there is a corresponding $\Ss$-expansion $\seq{e}_1 \in [\expr{x}_1]$, given as follows.
Each $\tilde{s}_p$-term of type $t \in \{0,1\}$ in $\seq{e}$ corresponds to a subsequence of terms $(\expr{\tilde{s}\sh})^{(p-1)/2}\expr{\tilde{s}}(\expr{\sh\tilde{s}})^{(p-1)/2}$ in $\seq{e}_1$, with the middle $\tilde{s}$-term of type $t$ and the other terms of type $1$.

\begin{thm} \label{thm:pastLLbasis}
Let $\expr{x} \in \linkexpr$, and let $\expr{w} \in \Sspexpr$ and $\expr{z} \in \expr{\Ss}$ be rexes. 
Suppose we have fixed a complete selection $\pastLL_{[\expr{x}],\expr{w},\expr{z}}$ of $(p|\ast)$-light leaves morphisms. 
Fix a rex $\expr{y} \in \expr{\Ss}$ for $wz$. 
Then there exists a complete set of partially standardized ordinary light leaves morphisms $\LL'_{[\expr{x}_1],\expr{y}}$, each of the form
\begin{gather*}
\LL'_{\seq{e}_1,\expr{y}}:B^{(p|\ast)}_{\expr{x}} \xrightarrow{\text{inclusion}} B_{\expr{x}_1} \xrightarrow{\LL} B_{\expr{y}} \xrightarrow{\text{projection}} \hat{R}_{\expr{y}} \\ 
\begin{gathered}
\begin{tikzpicture}
\node[anchor=south west,inner sep=0] at (0,0) {\includegraphics{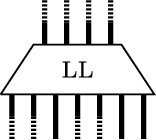}};
 \node[font=\footnotesize] at (11pt,10pt){$\inclscal{black}$};
 \node[font=\footnotesize] at (21pt,10pt){$\inclscal{black}$};
 \node[font=\footnotesize] at (43pt,10pt){$\inclscal{black}$};
 \node[font=\footnotesize] at (53pt,10pt){$\inclscal{black}$};
 \node[font=\footnotesize] at (26pt,58pt){$\projscal{black}$};
 \node[font=\footnotesize] at (37pt,58pt){$\projscal{black}$};
 \node[font=\footnotesize] at (48pt,58pt){$\projscal{black}$};
 \node[font=\footnotesize] at (59pt,58pt){$\projscal{black}$};
\end{tikzpicture}
\end{gathered} \text{,}
\end{gather*}
which is spanned by partially standardized $(p|\ast)$-light leaves morphisms $\pastLL'_{[\expr{x}],\expr{w},\expr{z}}$, each of the form
\begin{gather*}
\pastLL'_{\seq{e},\expr{w},\expr{z}} : B^{(p|\ast)}_{\expr{x}} \xrightarrow{\pastLL} B^{(p)}_{\expr{w}} \otimes \hat{R}_{\expr{z}} \xrightarrow{\text{projection}} \hat{R}_{\expr{w}_1 \expr{z}} \xrightarrow{\text{standard}} \hat{R}_{\expr{y}} \\
\begin{gathered}
\begin{tikzpicture}
\node[anchor=south west,inner sep=0] at (0,0) {\includegraphics{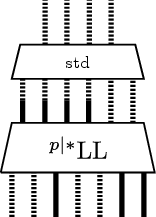}};
 \node[font=\footnotesize] at (16pt,57pt){$\projscal{black}$};
 \node[font=\footnotesize] at (27pt,57pt){$\projscal{black}$};
 \node[font=\footnotesize] at (37pt,57pt){$\projscal{black}$};
 \node[font=\footnotesize] at (48pt,57pt){$\projscal{black}$};
\end{tikzpicture}
\end{gathered}  \text{.}
\end{gather*}
\end{thm}

\begin{proof}
First we determine the effect of partially standardizing an mrex move. 
Projectors placed on top of a braid ``propagate'' through the braid (cf.~\cite[(5.28)]{ew-soergelcalc}):
\input{fig/braidpropagate.tex} \unskip
If we place projectors on a braid-like menorah then the resulting morphism is just the composition of several projectors and an $\epsilon$-basic standard morphism, e.g.~
\begin{align*}
\begin{gathered}
\begin{tikzpicture}
\node[anchor=south west,inner sep=0] at (0,0) {\includegraphics{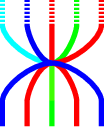}};
 \node[font=\footnotesize] at (6pt, 50pt){$\projscal{realcyan}$};
 \node[font=\footnotesize] at (18pt,50pt){$\projscal{blue}$};
 \node[font=\footnotesize] at (30pt,50pt){$\projscal{red}$};
 \node[font=\footnotesize] at (42pt,50pt){$\projscal{green}$};
 \node[font=\footnotesize] at (54pt,50pt){$\projscal{red}$};
\end{tikzpicture}
\end{gathered}
& =
\begin{gathered}
\begin{tikzpicture}
\node[anchor=south west,inner sep=0] at (0,0) {\includegraphics{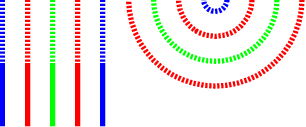}};
 \node[font=\footnotesize] at (6pt, 30pt){$\projscal{blue}$};
 \node[font=\footnotesize] at (19pt,30pt){$\projscal{red}$};
 \node[font=\footnotesize] at (31pt,30pt){$\projscal{green}$};
 \node[font=\footnotesize] at (43pt,30pt){$\projscal{red}$};
 \node[font=\footnotesize] at (55pt,30pt){$\projscal{blue}$};
 \node[font=\footnotesize] at (114pt,56pt){$\spescal{blue}$};
 \node[font=\footnotesize] at (126pt,56pt){$\spescal{red}$};
 \node[font=\footnotesize] at (139pt,56pt){$\spescal{green}$};
 \node[font=\footnotesize] at (151pt,56pt){$\spescal{red}$};
\end{tikzpicture}
\end{gathered}
 \text{,} \\
\begin{gathered}
\begin{tikzpicture}
\node[anchor=south west,inner sep=0] at (0,0) {\includegraphics{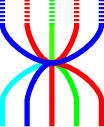}};
 \node[font=\footnotesize] at (6pt, 50pt){$\projscal{blue}$};
 \node[font=\footnotesize] at (18pt,50pt){$\projscal{red}$};
 \node[font=\footnotesize] at (30pt,50pt){$\projscal{green}$};
 \node[font=\footnotesize] at (42pt,50pt){$\projscal{red}$};
 \node[font=\footnotesize] at (54pt,50pt){$\projscal{blue}$};
\end{tikzpicture}
\end{gathered}
& =
\begin{gathered}
\begin{tikzpicture}
\node[anchor=south west,inner sep=0] at (0,0) {\includegraphics{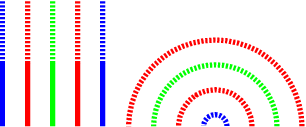}};
 \node[font=\footnotesize] at (6pt, 30pt){$\projscal{blue}$};
 \node[font=\footnotesize] at (19pt,30pt){$\projscal{red}$};
 \node[font=\footnotesize] at (31pt,30pt){$\projscal{green}$};
 \node[font=\footnotesize] at (43pt,30pt){$\projscal{red}$};
 \node[font=\footnotesize] at (55pt,30pt){$\projscal{blue}$};
 \node[font=\footnotesize] at (114pt,6pt){$\epsscal{blue}$};
 \node[font=\footnotesize] at (126pt,6pt){$\epsscal{red}$};
 \node[font=\footnotesize] at (139pt,6pt){$\epsscal{green}$};
 \node[font=\footnotesize] at (151pt,6pt){$\epsscal{red}$};
\end{tikzpicture}
\end{gathered}
 \text{.}
\end{align*}

Similarly, using the Jones--Wenzl relations we can ``pull'' a dot placed on the top of an ordinary braid or a $p$-affine braid through the braid to get a rex move on a smaller expression, plus a dot on the bottom. 
The same is true for dots on a braid-like menorah. 

Next we observe the effects of partially standardizing the maps $\phi_i$ above. 
Fix some $\seq{e} \in [\expr{x}]$. 
As in Figure~\ref{fig:pastLL}, boxes labeled ``rex'' are rex moves between two ordinary reduced expressions, boxes labeled  ``mrex'' are mrex moves between two reduced $(p|1)$-expressions, and boxes labeled ``std'' are basic standard diagrams to a standard bimodule corresponding to some reduced expression. 
We will show that in each case, we get either a partially standardized verson of one of the four maps used for defining ordinary light leaves \cite[Figure~2]{ew-soergelcalc}, or a linear combination of such maps. 

When $i$ is an indeterminate index with decoration $\upseq$, we can easily show that the partially standardized version of $\phi_i$ is nearly the same as that in the ordinary case. For example, when $i$ is of decorated type $\upseq 1$ we have
\input{fig/LLU1propagate.tex}
The calculation for $\upseq 0$ is similar.

When $i$ has decoration $\downseq$ we have to split the diagram into a sum. For example, when $i$ is of decorated type $\downseq 0$ we have
\input{fig/LLD0propagate.tex} \unskip
Again, the calculation for $\downseq 1$ is similar. 

If $i$ is a fixed index, then the resulting map depends on the decorated type for the corresponding index in the $\Ss$-expansion $\seq{e}_1$. 
If the corresponding ordinary decorated type is not $\downseq 0$, then we get a single term as in the calculation for $\upseq 0$ above. 
Otherwise, if the corresponding ordinary decorated type is $\downseq 0$, then we have
\input{fig/LLempty0fix.tex}


Now let $\seq{e}_1 \in [\expr{x}_1]$ be a subsequence expressing $wz$, and suppose the partially standardized morphism $\LL'_{\seq{e}_1,\expr{y}}$ is non-zero.
If any of the standardized indices in the domain (i.e.~any generator except the middle one in $(\expr{\tilde{s}\sh})^{(p-1)/2}\expr{\tilde{s}}(\expr{\sh\tilde{s}})^{(p-1)/2}$, the $\Ss$-expansion of $\tilde{s}_p$) has type $0$, then by pulling bivalent vertices through braid moves, $\LL'_{\seq{e}_1,\expr{y}}$ vanishes. 
So without loss of generality all of these indices must have type $1$, and there is a unique $\seq{e} \in [\patt{r}]$ for some $\patt{r} \in [\expr{x}]_{p|\ast}$ whose $\Ss$-expansion is $\seq{e}_1$.

Suppose we have already shown that $\LL'_{\seq{f}_1,\expr{y}}$ lies in the span of partially standardized $(p|\ast)$-light leaves for all $\seq{f}_1<\seq{e}_1$, where the subsequences are ordered using the path dominance order from \cite[{\S 2.4}]{ew-soergelcalc}. 
We will show that $\LL'_{\seq{e}_1,\expr{y}}$ lies in the span of the partially standardized $(p|\ast)$-light leaves. 
To do this, we use the above calculations to pull the bivalent inclusions/projections (and any dots introduced by $\downseq$-decorated indices) in $\pastLL'_{\seq{e},\expr{w},\expr{z}}$ through the $\phi_i$ down to the bottom of the diagram. 
The goal is to get the resulting map to look like an ordinary light leaves morphism. 
The first step might look like
\input{fig/bigLLpropagate.tex} \unskip
and continue downwards to the bottom of the diagram. 
After pulling through all the $\phi_i$ and getting to the bottom we have shown that $\pastLL'_{\seq{e},\expr{w},\expr{z}}$ is equal to a scalar multiple of $\LL'_{\seq{e}_1,\expr{y}}$, plus a linear combination of partially standardized ordinary light leaves morphisms $\LL'_{\seq{f}_1,\expr{y}}$ for $\seq{f}_1<\seq{e}_1$.
By induction each $\LL'_{\seq{f}_1,\expr{y}}$ lies in the span of the partially standardized $(p|\ast)$-light leaves morphisms, so we are done.
\end{proof}

Let $w \in \W$, and suppose for all $\expr{x},\expr{y} \in \linkexpr$ we have fixed a complete selection of double leaves morphisms $\pastLLdbl_{[\expr{x}]}^{[\expr{y}]}$. 
Recall that $I_{\ngeq w}=\{z \in \W : z \ngeq w\}$ is an ideal in the Bruhat order and let $\pastLLdbl_{\ngeq w}$ denote the span of the $\pastLLdbl$ maps 
\begin{equation*}
\{\pastLLdbl_{\seq{e}}^{\seq{f}} : \expr{x},\expr{y} \in \linkexpr,\ \seq{e} \in [\expr{x}],\ \seq{f} \in [\expr{y}],\ \hat{e}=\hat{f} \in I_{\ngeq w}\} \text{.}
\end{equation*}
In other words, $\pastLLdbl_{\ngeq w}$ is spanned by double leaves morphisms which factor through $I_{\ngeq w}$. 
It follows from Theorem~\ref{thm:pastLLbasis} and \cite[Claim~6.19]{ew-soergelcalc} that $\pastLLdbl_{\ngeq w}$ is a $2$-sided ideal of morphisms in $\dgrmBSpast$, namely the $\hat{R}$-extension of the ideal $\LLdbl_{\ngeq w}$ in $\dgrmBS$. 
Let $\dgrm_{{\rm BS},p|\ast}^{\geq w}=\dgrmBSpast/\pastLLdbl_{\ngeq w}$, and let $\dgrm_{p|\ast}^{\geq w}$ be the corresponding quotient of $\dgrmpast$. 
Using Theorem~\ref{thm:pastLLbasis} and Lemma~\ref{lem:valpastLL} in conjunction with Theorem~\ref{thm:LLbasis} we obtain a light leaves basis result.

\begin{cor} \label{cor:pastLLbasis}
Let $\expr{x} \in \linkexpr$, and let $\expr{w} \in \Sspexpr$ and $\expr{z} \in \expr{\Ss}$ be rexes. 
The light leaves morphisms $\pastLL_{[\expr{x}],\expr{w},\expr{z}}$ form a valuation basis for $\Hom_{\dgrm_{p|\ast}^{\geq w}}^{\blacklozenge}(B^{(p|\ast)}_{\expr{x}},B^{(p)}_{\expr{w}} \otimes \hat{R}_{\expr{z}})$.
\end{cor}

The analogous double leaves basis result follows by the same reasoning as in \cite[{\S 7.3}]{ew-soergelcalc}.

\begin{cor} \label{cor:pastLLdblbasis}
Let $\expr{x},\expr{y} \in \linkexpr$. 
The double leaves morphisms $\pastLLdbl_{[\expr{x}]}^{[\expr{y}]}$ form a valuation basis for $\Hom_{\dgrmBSpast}^{\blacklozenge}(B^{(p|\ast)}_{\expr{x}},B^{(p|\ast)}_{\expr{y}})$.
\end{cor}

In particular, since any valuation basis is an $\hat{R}$-module basis, these results give ungraded/non-valuation bases for $\Hom$-spaces even when we forget the valuation structure entirely.

\subsection{Indecomposable objects}
\label{sec:linkindecomp}

As with Theorem~\ref{thm:LLdblbasis}, Corollary~\ref{cor:pastLLdblbasis} is strong enough to give a complete classification of the indecomposable objects in $\dgrmpast$. We closely follow the strategy of \cite[\S 6.6]{ew-soergelcalc}.

\begin{lem}
The category $\dgrmpast$ is Krull--Schmidt.
\end{lem}

\begin{proof}
It is sufficient to show that $E=\End_{\dgrmpast}^{\blacklozenge}(B)$ is local for any indecomposable object $B$ in $\dgrmpast$. 
By Corollary \ref{cor:pastLLdblbasis} $E$ is a finite $\hat{R}$-algebra. 
Any finite algebra over a complete local ring is either local or contains an idempotent, but $E$ cannot contain an idempotent as $\dgrmpast$ is a Karoubi envelope. 
Thus $E$ is local and $\dgrmpast$ is Krull--Schmidt.
\end{proof}

The Frobenius functor extends naturally to a functor $F:\dgrmF \rightarrow \hat{R} \otimes \dgrm$. 
We will similarly extend our notation by writing $\glsuseri*{BxpFp}$ and $\glsuseri*{Bpppy}$, where $x \in \W$ and $y \in \Wp$, for the images under $F$ of the indecomposable Soergel bimodules $B^F_x$ and $B^F_{F^{-1}(y)}$ in $\dgrmF$ respectively.

\begin{lem} \label{lem:blockdecomp}
Let $\expr{x},\expr{y} \in \expr{\Ss}$ and $w \in \Wpcosets$. The map
\begin{align*}
F \otimes \ident_{\hat{R}_{w}} : \Hom_{\hat{R} \otimes \dgrmF}^{\blacklozenge}(B^{F}_{\expr{x}}, B^{F}_{\expr{y}}) & \longrightarrow \Hom_{\dgrmpast}^{\blacklozenge}(B^{(F)}_{\expr{x}} \otimes \hat{R}_w, B^{(F)}_{\expr{y}} \otimes \hat{R}_w) \\
f & \longmapsto F(f) \otimes \ident_{\hat{R}_w}
\end{align*}
is a valuation isomorphism.
\end{lem}

\begin{proof}
Let $\expr{x}'=F(\expr{x})$ and $\expr{y}'=F(\expr{y})$. 
Fix a rex $\expr{w} \in \expr{\Ss}$ for $w$, and let $M$ and $M'$ denote $\Hom_{\dgrmpast}^{\blacklozenge}(B^{(F)}_{\expr{x}} \otimes \hat{R}_{\expr{w}}, B^{(F)}_{\expr{y}} \otimes \hat{R}_{\expr{w}})$ and $\Hom_{\dgrmpast}^{\blacklozenge}(B^{(p|\ast)}_{\expr{x}'|\expr{w}}, B^{(p|\ast)}_{\expr{y}'|\expr{w}})$ respectively. 
By Corollary~\ref{cor:pastLLdblbasis} the $(p|\ast)$-double leaves morphisms $\pastLLdbl_{[\expr{x}'|\expr{w}]}^{[\expr{y}'|\expr{w}]}$ form a basis for $M'$. 
Since $B^{(F)}_{\expr{x}} \otimes \hat{R}_{\expr{w}}$ is a summand of $B^{(p|\ast)}_{\expr{x}'|\expr{w}}$ (and similarly for $\expr{y}$) this means $M$ is spanned by $p_{\expr{y}} \circ \pastLLdbl_{[\expr{x}'|\expr{w}]}^{[\expr{y}'|\expr{w}]}\circ i_{\expr{x}}$, where $i_{\expr{x}}$ and $p_{\expr{y}}$ are the inclusion and projection maps for these summands.

Now note that $p_{\expr{y}} \circ \pastLLdbl_{\seq{e}}^{\seq{f}} \circ i_{\expr{x}}$ is non-zero only when the subsequences $\seq{e}$ and $\seq{f}$ come from the pattern $\ast \dotsm \ast | 1 \dotsm 1 \in [\expr{x}'|\expr{w}]_{p|\ast}$ and $\ast \dotsm \ast | 1 \dotsm 1 \in [\expr{y}'|\expr{w}]_{p|\ast}$ respectively. 
This shows that $M$ has basis $F(\LLdbl_{[\expr{x}]}^{[\expr{y}]}) \otimes \ident_{\hat{R}_{\expr{w}}}$, which proves the result. 
\end{proof}

\begin{thm} \label{thm:Frobindecomp}
Let $x \in \Wp$ and $w \in \Wpcosets$. The object $B^{(p)}_{x} \otimes \hat{R}_{w}$ is indecomposable in $\dgrmpast$.
\end{thm}

\begin{proof}
Let $\expr{x} \in \Sspexpr$ be a rex for $x$, and set $\expr{y}=F^{-1}(\expr{x}) \in \expr{\Ss}$ and $y=F^{-1}(x)$. 
Since $B^F_{y}$ is an indecomposable summand of $B^F_{\expr{y}}$, it follows that $B^{(p)}_x \otimes \hat{R}_{w}$ is a direct summand of $B^{(p)}_{\expr{x}} \otimes \hat{R}_{w}$. 
Let $E$ denote the endomorphism algebra of $B^{(p)}_x \otimes \hat{R}_w$.
By Lemma~\ref{lem:blockdecomp} $\End_{\dgrmpast}^{\blacklozenge}(B^{(p)}_{\expr{x}} \otimes \hat{R}_w) \iso \hat{R} \otimes \End_{\dgrmF}^{\bullet}(B^{F}_{\expr{y}})$, and this isomorphism restricts to an isomorphism $E \iso \hat{R} \otimes \End_{\dgrmF}^{\bullet}(B^F_y)$.
To simplify our notation we will assume that $w=1$ for the rest of this proof. 

Let $E_0 \iso \hat{R} \otimes_R \End_{\dgrmF}^0(B^F_{y})$ be the $\hat{R}$-subalgebra of $E$ generated by the ($F$-images of) degree $0$ morphisms in the ordinary diagrammatic category $\dgrmF$. 
If $L$ is the $F$-image of an endomorphism of $B^F_y$ of non-positive degree $-2n$ (note that all endomorphisms in $\dgrmF$ have even degree) then
\begin{equation*}
L=a_{\tilde{s}}^{-n}(a_{\tilde{s}}^n L) \in \hat{R}\End_{\dgrmF}^0(B^F_y)=E_0 \text{.}
\end{equation*}
This shows that $E_0$ also contains the $F$-images of all non-positive degree morphisms.
In addition, Lemma~\ref{lem:valF} implies that valuation equals degree for double leaves in $\dgrmF$, so the ideal $\End_{\dgrmpast}^{>0}(B^{(p)}_x)$ of all positive valuation morphisms in $E$ contains the $F$-images of all positive degree double leaves in $\dgrmF$.
Since $E$ is generated by the $F$-images of double leaves in $\dgrmF$, we have
\begin{equation}
E=E_0+\End_{\dgrmpast}^{>0}(B^{(p)}_x) \text{.} \label{eq:endomdecomp}
\end{equation}

The category $\dgrmF$ is Krull--Schmidt \cite[Lemma~6.25]{ew-soergelcalc}, so the ring $\End^0_{\dgrmF}(B^F_y)$ is local. 
Let $\mathfrak{m}$ be the maximal ideal of this ring, and let
\begin{equation*}
I=\hat{R}F(\mathfrak{m})+(a_s : s \in \Ssf) E_0+\End_{\dgrmpast}^{>0}(B^{(p)}_x) \subseteq E \text{,}
\end{equation*}
where $(a_s : s \in \Ssf)$ is the maximal ideal of $\hat{R}$. 
The first two terms are ideals in $E_0$, so the decomposition \eqref{eq:endomdecomp} shows that $I$ is an ideal in $E$. 
Clearly $E=E_0+I$ follows from \eqref{eq:endomdecomp} as well.

We will show that all morphisms in $E \setminus I$ are invertible, and thus that $E$ is local with maximal ideal $I$ and that $B^{(p)}_x$ is indecomposable in $\dgrmpast$. 
Suppose $g \in E \setminus I$. 
We write
\begin{equation*}
g=r_0 g_0+r_{\mathfrak{m}} g_{\mathfrak{m}}+r_{\rm f} g_{\rm f}+r_{>0} g_{>0}
\end{equation*}
where $r_0,r_{\mathfrak{m}},r_{>0} \in \hat{R}$, $r_{\rm f} \in (a_s : s \in \Ssf)$, $g_0,g_{\rm f} \in F(\End_{\dgrmF}^0(B^F_y))$, $g_{\mathfrak{m}} \in F(\mathfrak{m})$, and $g_{>0} \in \End_{\dgrmpast}^{>0}(B^{(p)}_x)$. 

Clearly $r_0 \notin (a_s : s \in \Ssf)$ and $g_0 \notin F(\mathfrak{m})$ as $g \notin I$. 
Thus we can write
\begin{equation*}
r_0 g_0+r_{\mathfrak{m}}g_{\mathfrak{m}}=r_0 g_0 \left(1+\frac{r_{\mathfrak{m}}}{r_0}g_0^{-1} g_{\mathfrak{m}}\right) \text{.}
\end{equation*}
Since $\End_{\dgrmF}^0(B^F_y)$ is finite-dimensional the maximal ideal $\mathfrak{m}$ is nilpotent. 
But $g_0^{-1} g_{\mathfrak{m}}$ is contained in $F(\mathfrak{m})$, so the above morphism is invertible using the formula $(1+x)^{-1}=1+x+x^2+\dotsb$ and the completeness of $\End^0_{\dgrmF}(B^F_y)$. 

The remaining two terms in the sum for $g$ are contained in an ideal $J=(a_s : s \in \Ssf)E_0+\End_{\dgrmpast}^{>0}(B^{(p)}_x)$. 
From Theorem~\ref{thm:LLdblbasis} and Corollary~\ref{cor:pastLLdblbasis} $J$ is generated as an $\hat{R}$-module by morphisms $a_s F(\LLdbl_{\seq{e}}^{\seq{f}})$ (for $\dfct(\seq{e})+\dfct(\seq{f})=0$ and any $s \in \Ssf$) and $\pastLLdbl_{\seq{e}}^{\seq{f}}$ (for $\dfct_p(\match{e})+\dfct_p(\match{f})>0$). 
This basis is finite, so for sufficiently large $n$ we have $J^n \leq (a_s : s \in \Ssf)J$. 
Yet $\hat{R}$ is complete with respect to its maximal ideal $(a_s : s \in \Ssf)$ so $g \in (r_0 g_0+r_{\mathfrak{m}} g_{\mathfrak{m}})+J$ is invertible using the same trick as above.
\end{proof}

For $x \in \Wp$ and $w \in \Wpcosets$, we will now write $\glsuseri*{Bppastpxbarw}$ for the indecomposable object $B^{(p)}_x \otimes \hat{R}_w$ from the previous theorem.

\begin{thm} \label{thm:Dpastindecomp}
Every indecomposable object in $\dgrmpast$ is valuation isomorphic to an object of the form $B^{(p|\ast)}_{y|w}\langle m\rangle$ for some $y \in \Wp$, $w \in \Wpcosets$, and $m \in \ZZ$.
\end{thm}

\begin{proof}
Suppose $B$ is a direct summand of $B^{(p|\ast)}_{\expr{x}}$ for some $\expr{x} \in \linkexpr$, and that $\eta \in \End_{\dgrmBSpast}^{\blacklozenge}(B^{(p|\ast)}_{\expr{x}})$ is an idempotent corresponding to this summand. 
We can write
\begin{equation*}
\eta=\sum \lambda_{\seq{e},z,\seq{f}} \pastLLdbl_{\seq{e}}^{\seq{f}} \text{,}
\end{equation*}
where $\lambda_{\seq{e},z,\seq{f}} \in \hat{R}$ and the sum is over all $z \in \W$ and all subsequences $\seq{e},\seq{f} \in [\expr{x}]$ such that $\hat{e}=\hat{f}=z$. 
Pick $z' \in \W$ maximal in the Bruhat order such that there is a non-zero coefficient $\lambda_{\seq{e}',z',\seq{f}'} \neq 0$. 
In $\dgrmpast^{\geq z'}$ we get
\begin{equation*}
\eta=\sum_{\hat{e}=\hat{f}=z'} \lambda_{\seq{e},z',\seq{f}} (\overline{\pastLL_{\seq{f},\expr{y},\expr{w}}} \circ \pastLL_{\seq{e},\expr{y},\expr{w}}) \text{,}
\end{equation*}
where the sum is over all subsequences $\seq{e},\seq{f} \in [\expr{x}]$ such that $\hat{e}=\hat{f}=z'$, for some rexes $\expr{y} \in \Sspexpr$ and $\expr{w} \in \expr{\Ss}$ with $w \in \Wpcosets$. 
Now assume that for all subsequences in the above sum we have
\begin{equation*}
\pastLL_{\seq{e},\expr{y},\expr{w}} \circ \eta \circ \overline{\pastLL_{\seq{f},\expr{y},\expr{w}}} \in (a_s : s \in \Ssf) \leq \hat{R}=\End_{\dgrmpast^{\geq z'}}^{\blacklozenge}(B^{(p)}_{\expr{y}} \otimes \hat{R}_{\expr{w}})
\end{equation*}
Then by expanding out $\eta^3=\eta$ we get $\lambda_{\seq{e},z',\seq{f}} \in (a_s : s \in \Ssf)_{s \in \Ss}$. 
But this implies that
\begin{equation*}
\eta \in (a_s : s \in \Ssf) \End_{\dgrmBSpast}^{\blacklozenge}(B^{(p|\ast)}_{\expr{x}}) \leq J\left(\End_{\dgrmBSpast}^{\blacklozenge}(B^{(p|\ast)}_{\expr{x}})\right)
\end{equation*}
where $J(\cdot)$ denotes the Jacobson radical. 
Since $\eta$ is idempotent, we obtain a contradiction. 
Hence there must be some subsequences $\seq{e},\seq{f}$ for which the following composition
\begin{equation*}
B^{(p)}_{\expr{y}} \otimes \hat{R}_w \xrightarrow{\overline{\pastLL_{\seq{f},\expr{y},\expr{w}}}} B \xrightarrow{\pastLL_{\seq{e},\expr{y},\expr{w}}} B^{(p)}_{\expr{y}} \otimes \hat{R}_w
\end{equation*}
is invertible in $\dgrmpast^{\geq z'}$. 
This induces an invertible morphism
\begin{equation*}
B^{(p)}_{y} \otimes \hat{R}_w \xrightarrow{i} B \xrightarrow{p} B^{(p)}_{y} \otimes \hat{R}_w
\end{equation*}
which proves the result.
\end{proof}

Let $M \iso \bigoplus_{i} \hat{R}\langle d_i\rangle$ be a free $\hat{R}$-valuation module. We write 
\begin{equation*}
\valrk M=\sum_{i} v^{d_i} \in \laur \text{,}
\end{equation*}
which we call the \defnemph{valuation rank} of $M$.

\begin{defn}
The \defnemph{$(p|\ast)$-diagrammatic character} is the $\laur$-linear map
\begin{align*}
\gls*{charctrpast} : [\dgrmpast] & \longrightarrow \heckepast \\
[B] & \longmapsto \sum_{x \in \W} (\valrk \Hom_{\dgrm_{p|\ast}^{\geq x}}^{\blacklozenge}(B,B^{(p|\ast)}_{x})) h^{(p|\ast)}_x\text{.}
\end{align*}
\end{defn}

Since $\dgrmpast$ has a left $\dgrmF$-action, the Grothendieck group $[\dgrmpast]$ has the structure of a left $[\dgrmF]$-module. 
By Corollary~\ref{cor:charctrheckeiso} we have $[\dgrmF] \iso \hecke \stackrel{F}{\iso} \heckep$, so $[\dgrmpast]$ has the structure of a left $\heckep$-module via $F$. 

Suppose $\expr{y} \in \Sspexpr$ and $\expr{w} \in \expr{\Ss}$ are rexes with $w \in \Wpcosets$. 
In the category $\dgrm_{p|\ast}^{\geq yw}$, the objects $B^{(p|\ast)}_{y|w}$ and $B^{(p|\ast)}_{\expr{y}|\expr{w}}$ are isomorphic. 
Thus using Corollary~\ref{cor:pastLLbasis} and Corollary~\ref{cor:ddfheckepastuntwisted} we observe that
\begin{equation}
\charctr_{p|\ast} {[B^{(p|\ast)}_{\expr{x}}]}=b^{(p|\ast)}_{\expr{x}} \label{eq:Bpastcharctr}
\end{equation}
for all $\expr{x} \in \linkexpr$.

\begin{thm}
The map $\charctr_{p|\ast} : [\dgrmpast] \rightarrow \heckepast$ is an isomorphism of left $\heckep$-modules. 
\end{thm}

\begin{proof}
The category $\dgrmpast$ is the closure of $\dgrmBSpast$ with respect to direct sums, direct summands, and valuation shifts. This means that the $\laur$-module $[\dgrmpast]$ is spanned by the isomorphism classes $\{[B^{(p|\ast)}_{\expr{x}}] : \expr{x} \in \linkexpr\}$. The calculation \eqref{eq:Bpastcharctr} along with Proposition~\ref{prop:ddfheckepasttwisted} then immediately implies that $\charctr_{p|\ast}$ is a left $\heckep$-module homomorphism. Moreover, the isomorphism classes $\{[B^{(p|\ast)}_x] : x \in \W\}$ form a $\laur$-basis for $[\dgrmpast]$ by Theorem~\ref{thm:Dpastindecomp}. For any $x \in \W$ we have
\begin{equation*}
\charctr_{p|\ast} {[B^{(p|\ast)}_x]} \in h^{(p|\ast)}_x+\sum_{y<x} \laur h^{(p|\ast)}_y
\end{equation*}
using the construction of $B^{(p|\ast)}_x$ and the definition of $\charctr_{p|\ast}$. This means that $\{\charctr_{p|\ast} {[B^{(p|\ast)_x}]} : x \in \W\}$ is a basis for $\heckepast$ as well, which shows that $\charctr_{p|\ast}$ is an isomorphism.
\end{proof}

\section{The functor \texorpdfstring{$\pr$}{Psi}}

The category $\dgrmpast$ has a $(\dgrmF,\dgrmdeg)$-bimodule structure. 
So far we have focused our attention on the left $\dgrmF$-module structure of this category. 
For example, we can reinterpret Theorem~\ref{thm:Dpastindecomp} as the following categorical decomposition
\begin{equation}
\dgrmpast \iso \bigoplus_{w \in \Wpcosets} \dgrmF \otimes \hat{R}_w \label{eq:blockdecomp}
\end{equation}
of left $\dgrmF$-modules. 
This decomposition is also a \defnemph{block decomposition}, because
\begin{equation*}
\Hom_{\dgrmpast}(\dgrmF \otimes \hat{R}_w,\dgrmF \otimes \hat{R}_{z})=0
\end{equation*}
for any distinct $w,z \in \Wpcosets$. 
By considering the right $\dgrmdeg$-action on $\dgrmpast$ in terms of this decomposition, we will construct a functor $\pr$ defined on $\dgrmdeg$.

\subsection{Construction}

\label{sec:prconstruct}

In this section we will mostly refer to standard diagrams by their isomorphism class, i.e.~we will write $\hat{R}_w$ instead of $\hat{R}_{\expr{w}}$. To deal with potential ambiguity we introduce several notions of equivalence between morphisms.

\begin{defn}
We say that two morphisms $f,g$ in $\dgrmstd$ are \defnemph{basic standard equivalent} if there exist basic standard diagrams $\sigma,\sigma'$ such that $f=\sigma \circ g \circ \sigma'$.  More generally, we say that two morphisms $h,k$ in $\dgrmBSstd$ are basic standard equivalent if there exist matrices $\Theta,\Theta'$ whose entries are basic standard diagrams such that for $H,K$ the localization matrices of $h,k$ respectively, we have $H=\Theta \circ K \circ \Theta'$. 
\end{defn}


The first step in our construction of $\pr$ is to use patterns to localize a subset of the boundary strings of an $\Ss$-graph. 

\begin{defn}
Let $\expr{x} \in \expr{\Ss}$. 
Suppose $\patt{r}=(\patt{r}_1,\patt{r}_2,\dotsc,\patt{r}_m)$ is a pattern for $\expr{x}$. 
We write 
\begin{equation*}
\gls*{Bpattr}= \bigotimes_{i=1}^m \begin{cases}
B_{s_i} & \text{if $\patt{r}_i$ is indeterminate with generator $s_i$,} \\
\hat{R}_{\patt{r}_i} & \text{if $\patt{r}_i$ is fixed}
\end{cases}
\end{equation*}
in $\dgrmpast$. 
Similarly we write 
\begin{equation*}
\glsuseri*{pipattr} = 
\bigotimes_{i=1}^m
\begin{cases}
\ident_{B_{s_i}} & \text{if $\patt{r}_i$ is indeterminate with generator $s_i$,} \\
\pi_{\patt{r}_i} & \text{if $\patt{r}_i$ is fixed}
\end{cases}
\end{equation*}
a morphism $B_{\expr{x}} \longrightarrow B_{\patt{r}}$ in $\dgrmpast$.
\end{defn}

\begin{lem} \label{lem:locpattdecomp}
Let $\expr{x} \in \expr{\Ss}$ and $w \in \Wpcosets$. Then
\begin{equation*}
\hat{R}_w \otimes B_{\expr{x}} \xrightarrow{\bigoplus_{\patt{r}} \pi_{\patt{r}}} \bigoplus_{\patt{r} \in [\expr{x}]_{\ast}(w)} B_{\patt{r}} 
\end{equation*}
is an isomorphism in $\dgrmpast$.
\end{lem}

\begin{proof}
Let $\patt{r}=(\patt{r}_1,\patt{r}_2,\dotsc,\patt{r}_m) \in [\expr{x}]_{\ast}(w)$. 
For each fixed term $\patt{r}_i$ with generator $s_i$ we have $w\hat{r}_{\leq i-1} s_i (w\hat{r}_{\leq i-1})^{-1} \notin \Ssp$. 
This means that $(w\hat{r}_{\leq i-1}) a_{s_i}$ is invertible in $\hat{R}$ by Lemma~\ref{lem:rlzcoefs}. 
Thus we have a splitting
\begin{equation*}
\hat{R}_w \otimes B_{\expr{x}}=\hat{R}_w \otimes B_{\expr{x}_{\leq i-1}} \otimes B_{s_i} \otimes B_{\expr{x}_{\geq i+1}} \iso \hat{R}_w \otimes B_{\expr{x}_{\leq i-1}} \otimes \left(\hat{R} \oplus \hat{R}_{s_i}\right) \otimes B_{\expr{x}_{\geq i+1}}
\end{equation*}
in $\dgrmpast$. 
Applying this splitting to each fixed term of a pattern, over all patterns in $[\expr{x}]_{\ast}(w)$, gives the decomposition.
\end{proof}

By inserting standard caps where necessary we obtain an isomorphism $B_{\patt{r}} \rightarrow B^{(p)}_{w\patt{r}(w\hat{r})^{-1}} \otimes \hat{R}_{w\hat{r}}$ which is basic standard equivalent to the identity. This gives a reinterpretation of Lemma~\ref{lem:locpattdecomp}.

\begin{cor} \label{cor:locpattp}
Let $\expr{x} \in \expr{\Ss}$ and $w \in \Wpcosets$. There is an isomorphism 
\begin{equation*}
\hat{R}_w \otimes B_{\expr{x}} \xrightarrow{\sim} \bigoplus_{z \in \Wpcosets} \left(\bigoplus_{\substack{\patt{r} \in [\expr{x}]_{\ast}(w)\\ w\hat{r}=z}} B^{(p)}_{w\patt{r}z^{-1}}\right) \otimes \hat{R}_{z} \text{,}
\end{equation*}
in $\dgrmpast$, where each $\patt{r}$-component is basic standard equivalent to $\pi_{\patt{r}}$.
\end{cor}

Note that by Theorem~\ref{thm:stdhoms} the components of the decomposition above are uniquely determined by the expressions $\expr{w}$ and $\expr{z}$ used to define the standard bimodules in the domain and codomain. 
The two decompositions in Lemma~\ref{lem:locpattdecomp} and Corollary~\ref{cor:locpattp} give rise to two different ways to localize an $\Ss$-graph with respect to patterns.

\begin{defn}
Let $f:B_{\expr{x}} \rightarrow B_{\expr{y}}$ be a morphism in $\dgrmBSdeg$. 
Suppose $w,z \in \Wpcosets$ with $\patt{q} \in [\expr{x}]_{\ast}(w)$ and $\patt{r} \in [\expr{y}]_{\ast}(w)$ such that $w\hat{q}=w\hat{r}=z$. 
\begin{enumerate}
\item The \defnemph{partial localization} of $\ident_{\hat{R}_w} \otimes f$ with respect to $\patt{q}$ and $\patt{r}$ is the composition
\begin{equation*}
B_{\patt{q}} \xrightarrow{\overline{\pi_{\patt{q}}}} \hat{R}_w \otimes B_{\expr{x}} \xrightarrow{\ident \otimes f} \hat{R}_w \otimes B_{\expr{y}} \xrightarrow{\pi_{\patt{r}}} B_{\patt{r}} \text{.}
\end{equation*}

\item The \defnemph{basic standard equivalent partial localization} of $\ident_{\hat{R}_w} \otimes f \otimes \ident_{\hat{R}_{z^{-1}}}$ with respect to $\patt{q}$ and $\patt{r}$ is the composition
\begin{equation*}
\xymatrix@C-6pt{
B^{(p)}_{w\patt{q}(w\hat{q})^{-1}} \ar[r]^{\sim} & B_{\patt{q}} \otimes \hat{R}_{z^{-1}} \ar[r]^-{\overline{\pi_{\patt{q}}} \otimes \ident} & \hat{R}_w \otimes B_{\expr{x}} \otimes \hat{R}_{z^{-1}} \ar[d]^{\ident \otimes f \otimes \ident} & & \\
& & \hat{R}_w \otimes B_{\expr{y}} \otimes \hat{R}_{z^{-1}} \ar[r]^-{\pi_{\patt{r}} \otimes \ident} & B_{\patt{r}} \otimes \hat{R}_{z^{-1}} \ar[r]^{\sim} & B^{(p)}_{w\patt{r}(w\hat{r})^{-1}} 
}
\end{equation*}
where the isomorphisms are basic standard equivalent to the identity.
\end{enumerate}
\end{defn}

From these two localization methods, we obtain two related analogues of the localization matrix. 
Let $\gls*{hatRdgrmBSFdev}$ be the \defnemph{devaluation} of $\hat{R} \otimes \dgrmBSF$, analogous to the degrading of $\dgrmBS$. 
We also write $\gls*{dgrmBSdegoplus}$ and $\gls*{hatRdgrmBSFdevoplus}$ for the additive envelopes of $\dgrmBSdeg$ and $\hat{R} \otimes \dgrm_{\rm BS}^{F,{\rm dev}}$ respectively.



\begin{defn} \hfill
\begin{enumerate}
\item Let $\expr{x} \in \expr{\Ss}$. 
We define two dual-valued $\Wpcosets \times \Wpcosets$ matrices $\glsuseri*{prprime}(B_{\expr{x}})$ and $\pr(B_{\expr{x}})$ of objects in $\dgrmpast$ and $\hat{R} \otimes \dgrm_{\rm BS}^{F,{\rm dev},\oplus}$ respectively as follows. 
\begin{enumerate}[label=(\roman*)]
\item The $(w,z)$-entry of $\pr'(B_{\expr{x}})$ is 
\begin{equation*}
\bigoplus_{\substack{\patt{q} \in [\expr{x}]_{\ast}(w) \\ w\hat{q}=z}} B_{\patt{q}} \text{.}
\end{equation*}

\item The $(w,z)$-entry of $\pr(B_{\expr{x}})$ is 
\begin{equation*}
\bigoplus_{\substack{\patt{q} \in [\expr{x}]_{\ast}(w) \\ w\hat{q}=z}} B^F_{F^{-1}(w\patt{q}z^{-1})}
\end{equation*}
\end{enumerate}

\item Let $f:B_{\expr{x}} \rightarrow B_{\expr{y}}$ be a morphism in $\dgrmBSdeg$. 
We define two dual-valued $\Wpcosets \times \Wpcosets$ matrices $\pr'(f)$ and $\pr(f)$ of morphisms in $\dgrmpast$ and $\hat{R} \otimes \dgrm_{\rm BS}^{F,{\rm dev},\oplus}$ as follows. 
\begin{enumerate}[label=(\roman*)]
\item The $(w,z)$-entry of $\pr'(f)$ is the direct sum of all partial localizations of $\ident_{\hat{R}_w} \otimes f$ with respect to patterns $\patt{q} \in [\expr{x}]_{\ast}(w)$ and $\patt{r} \in [\expr{y}]_{\ast}(w)$ for which $w\hat{q}=w\hat{r}=z$. 

\item The $(w,z)$-entry of $\pr(f)$ is the ($F$-preimage of the) direct sum of all basic standard equivalent partial localizations of $\ident_{\hat{R}_w} \otimes f \otimes \ident_{\hat{R}_{z^{-1}}}$ with respect to patterns $\patt{q} \in [\expr{x}]_{\ast}(w)$ and $\patt{r} \in [\expr{y}]_{\ast}(w)$ for which $w\hat{q}=w\hat{r}=z$. 
\end{enumerate}
\end{enumerate}
These constructions extend to both $\dgrmBSdegsum$ and to $\dgrmdeg$ in the natural way.
\end{defn}


For $f : B_{\expr{x}} \rightarrow B_{\expr{y}}$ note that the matrix of the domains (resp.~codomains) of the entries of $\pr(f)$ is given by $\pr(B_{\expr{x}})$ (resp.~$\pr(B_{\expr{y}})$) and similarly for $\pr'$. 
We will show in the next section that $\pr$ is a functor and that it describes a connection between Soergel bimodules in $\dgrmdeg$ and ``smaller'' Soergel bimodules in $\dgrmF$ as claimed in the introduction. 
Strictly speaking we do not need to define $\pr'$ in order to define $\pr$, but we have included this construction because it is usually much easier to calculate $\pr(f)$ from $\pr'(f)$, as the following example shows.

\begin{exam} \label{exam:prdash}
Let $p=3$ and $\W$ be of type $\widetilde{A_1}$. 
Label the unique finite generator $1$ (colored red) and the affine generator $0$ (colored blue). Here is an example of $\pr'(f)$ for a morphism $f:B_{\expr{0}} \rightarrow B_{\expr{010}}$ in $\dgrmBSdeg$:
\input{fig/prdashexample.tex} \unskip
Here is $\pr(f)$ for the same morphism:
\input{fig/prexample.tex} \unskip
\end{exam}

\subsection{Categorical results}
\label{sec:mainresults}


\begin{defn}
We define the \defnemph{matrix category} $\glsuseri*{MatWpcosetshatRdgrmBSFdevoplus}$ to be the following $\hat{R}$-linear monoidal category. 
\begin{itemize}
\item The objects are $\Wpcosets \times \Wpcosets$ matrices of objects in $\hat{R} \otimes \dgrm_{\rm BS}^{F,{\rm dev},\oplus}$. 
The tensor product of two matrices of objects $A$ and $A'$ is the categorified matrix product, with entries
\begin{equation}
(AA')_{z,w}=\bigoplus_{y \in \Wpcosets} A_{z,y} \otimes A'_{y,w} \text{.} \label{eq:prodmatrix}
\end{equation}

\item The morphisms $A \rightarrow A'$ are $\Wpcosets \times \Wpcosets$ matrices $L$, with each entry 
\begin{equation*}
L_{z,w} : A_{z,w} \longrightarrow A'_{z,w} 
\end{equation*}
a morphism in $\hat{R} \otimes \dgrm_{\rm BS}^{F,{\rm dev},\oplus}$. 
The composition of two composable matrices of morphisms $L$ and $L'$ is the categorified Hadamard product or entrywise product, with entries
\begin{equation}
(L \gls*{circWpcosets} L')_{z,w}=L_{z,w} \circ L'_{z,w} \text{.} \label{eq:hdmrdmatrix}
\end{equation}
The tensor product of two arbitrary matrices of morphisms $(L_{z,w})$ and $(L'_{z,w})$ is the categorified product matrix
\begin{equation}
(LL')_{z,w}= \bigoplus_{y \in \Wpcosets} L_{z,y} \otimes L'_{y,w} \text{.} \label{eq:prodmatrixmorphisms}
\end{equation}
\end{itemize}
\end{defn}

\begin{rem} \label{rem:matcatmotiv}
The matrix category can be motivated as follows. 
Let
\begin{equation*}
(\hat{R} \otimes \dgrm_{\rm BS}^{F,{\rm dev},\oplus})^{\Wpcosets}=\bigoplus_{w \in \Wpcosets} (\hat{R} \otimes \dgrmBS^{F,{\rm dev},\oplus})_w
\end{equation*}
denote the (external) direct sum category of $|\Wpcosets|$ copies of $(\hat{R} \otimes \dgrm_{\rm BS}^{F,{\rm dev},\oplus})$. 
The monoidal structure on $\dgrm_{\rm BS}^{F,{\rm dev},\oplus}$ induces a left $(\hat{R} \otimes \dgrm_{\rm BS}^{F,{\rm dev},\oplus})$-module structure on $(\hat{R} \otimes \dgrm_{\rm BS}^{F,{\rm dev},\oplus})^{\Wpcosets}$. 
The category of endofunctors of $(\hat{R} \otimes \dgrm_{\rm BS}^{F,{\rm dev},\oplus})^{\Wpcosets}$ has the structure of a monoidal category, with endofunctor composition as the tensor product. 
%
%
The matrix category $\Mat_{\Wpcosets}(\hat{R} \otimes \dgrm_{\rm BS}^{F,{\rm dev},\oplus})$ naturally embeds into this endofunctor category via the categorified ``row vector times matrix'' action
\begin{align*}
A : \bigoplus_{w \in \Wpcosets} (\hat{R} \otimes \dgrmBS^{F,{\rm dev}\oplus})_w & \longrightarrow \bigoplus_{w \in \Wpcosets} (\hat{R} \otimes \dgrmBS^{F,{\rm dev},\oplus})_w \\
(C_w)_{w \in \Wpcosets} & \longmapsto \left(\bigoplus_{z \in \Wpcosets} C_z \otimes A_{z,w}\right)_{w \in \Wpcosets} \text{.}
\end{align*}
From this embedding we may recover the formulas \eqref{eq:prodmatrix}--\eqref{eq:prodmatrixmorphisms}.
\end{rem}

As $\pr$ maps objects and morphisms to $\Mat_{\Wpcosets}(\hat{R} \otimes \dgrm_{\rm BS}^{F,{\rm dev},\oplus})$, we are now prepared to state our main theorem.

\begin{thm} \label{thm:linkagefunctor}
The mapping $\pr$ is a monoidal functor
\begin{equation*}
\pr: \dgrmBSdeg \longrightarrow \Mat_{\Wpcosets}(\hat{R} \otimes \dgrmBS^{F,{\rm dev},\oplus}) \text{,}
\end{equation*}
which extends to a monoidal functor on the additive Karoubi envelopes
\begin{equation*}
\pr: \dgrmdeg \longrightarrow \Mat_{\Wpcosets}(\hat{R} \otimes \dgrm^{F,{\rm dev}}) \text{.}
\end{equation*}
\end{thm}

\begin{proof}
Let $\expr{x},\expr{y} \in \expr{\Ss}$ and $w \in \Wpcosets$. 
From the definition of $[\patt{x y}]_{\ast}$ there is a bijection
\begin{align*}
\{(\patt{q'},\patt{r'}) : \patt{q'} \in [\expr{x}]_{\ast}(w),\; \patt{r'} \in [\expr{y}]_{\ast}(w\hat{q'})\} & \longrightarrow [\expr{x y}]_{\ast}(w) \\
(\patt{q'},\patt{r'}) & \longmapsto \patt{q' r'} \text{.}
\end{align*}
For any $z \in \Wpcosets$ this bijection restricts to
\begin{equation*}
\{(\patt{q'},\patt{r'}) : \patt{q'} \in [\expr{x}]_{\ast}(w),\; \patt{r'} \in [\expr{y}]_{\ast}(w\hat{q'}),\; w\hat{q'}\hat{r'}=z\} \longrightarrow \{\patt{q} \in [\expr{x y}]_{\ast}(w) : w\hat{q}=z\} \text{.}
\end{equation*}
The set $P_{w,z}$ on the left-hand side is partitioned by the sets
\begin{equation*}
P_{w,z}^{z'}=\{(\patt{q'},\patt{r'}) : \patt{q'} \in [\expr{x}]_{\ast}(w),\; \patt{r'} \in [\expr{y}]_{\ast}(w\hat{q'}),\; w\hat{q'}=z',\; z'\hat{r'}=z\}
\end{equation*}
indexed over $z' \in \Wpcosets$. Now we observe that the $(w,z)$-entry of $\pr(B_{\expr{x y}})$ is
\begin{align*}
\bigoplus_{\substack{\patt{q} \in [\expr{x}]_{\ast}(w) \\ w\hat{q}=z}} B^F_{F^{-1}(w\patt{q}z^{-1})}& =\bigoplus_{(\patt{q'},\patt{r'}) \in P_{w,z}} B^F_{F^{-1}(w\patt{q' r'}z^{-1})} \\
& =\bigoplus_{(\patt{q'},\patt{r'}) \in P_{w,z}} B^F_{F^{-1}(w\patt{q'}(w\hat{q'})^{-1})} \otimes B^F_{F^{-1}((w\hat{q'})\patt{r'}z^{-1})} \\
& =\bigoplus_{z' \in \Wpcosets} \left(\bigoplus_{(\patt{q'},\patt{r'}) \in P^{z'}_{w,z}} B^F_{F^{-1}(w\patt{q'}(z')^{-1})} \otimes B^F_{F^{-1}(z'\patt{r'}z^{-1})}\right) \\
& =\bigoplus_{z' \in \Wpcosets} (\pr(B_{\expr{x}}))_{w,z'} \otimes (\pr(B_{\expr{y}}))_{z',z} \text{.}
\end{align*}
This shows that $\pr$ is monoidal on objects, and a similar argument shows that it is monoidal on morphisms as well.



Now let $g:B_{\expr{x}} \rightarrow B_{\expr{y}}$ be a morphism in $\dgrmBSdeg$. 
For $w \in \Wpcosets$, the localization matrix of $\ident_{\hat{R}_w} \otimes g$ can be viewed as a block diagonal matrix, with a separate block $G_{w,w'}$ for each $w' \in \Wpcosets$ corresponding to entries which are standard morphisms on $\hat{Q}_{x'w'}$ for some $x' \in \Wp$. 
By construction the $(w,w')$-entry of $\pr'(g)$ has localization matrix equal to $G_{w,w'}$. 
If $h : B_{\expr{y}} \rightarrow B_{\expr{z}}$ is another morphism in $\dgrmBSdeg$ with blocks $H_{w,w'}$, then the localization matrix of $h \circ g$ has blocks $H_{w,w'} \circ G_{w,w'}$. 
This implies that the entries of $\pr'(h \circ g)$ and $\pr'(h) \circ_{\Wpcosets} \pr'(g)$ have identical localization matrices. 
As localization is faithful, it follows that $\pr'$ is a functor. 
We conclude that $\pr$ is also a functor, because $F(\pr(g)_{w,w'}) \otimes \ident_{\hat{R}_{w'}}$ is basic standard equivalent to $\pr'(g)_{w,w'}$ by definition. 
Finally the matrix category is equivalent as an additive category to a direct sum of copies of $\hat{R} \otimes \dgrm_{\rm BS}^{F,{\rm dev},\oplus}$, so its additive Karoubi envelope is just $\Mat(\hat{R} \otimes \dgrm^{F,{\rm dev}})$ and the functor extends to $\dgrmdeg$ as claimed.
%
\end{proof}


\begin{exam}
As before suppose $p=3$ and $\W$ is of type $\widetilde{A_1}$ with labeling as in Example~\ref{exam:prdash}. 
Here is an example of the Hadamard product on an idempotent (up to scaling):
\input{fig/hdmrdexample.tex}
\end{exam}


The next result will show that the monoidal action of $\dgrmdeg$ on $(\hat{R} \otimes \dgrm^{F,{\rm dev}})^{\Wpcosets}$ defined by $\pr$ (see Remark~\ref{rem:matcatmotiv}) is essentially the same as the right monoidal action on $\dgrmpast$. 

\begin{prop} \label{prop:prequalstensorproduct}
The left $(\hat{R} \otimes \dgrm^{F,{\rm dev}})$-module equivalence
\begin{align*}
(\hat{R} \otimes \dgrm^{F,{\rm dev}})^{\Wpcosets} & \longrightarrow \bigoplus_{w \in \Wpcosets} (\dgrm^{F,{\rm dev}} \otimes \hat{R}_w) \iso \dgrmpast \\
(C_w)_{w \in \Wpcosets} & \longmapsto (C_w \otimes \hat{R}_w)_{w \in \Wpcosets}
\end{align*}
is $\dgrmdeg$-equivariant up to natural isomorphism.
\end{prop}

\begin{proof}
It's enough to check equivariance on all Bott--Samelson bimodules; namely, for all $\expr{x},\expr{y} \in \expr{\Ss}$ and $w \in \Wpcosets$, we must show that
\begin{equation*}
B^{F}_{\expr{x}} \otimes \hat{R}_w \otimes B_{\expr{y}} \iso B^{F}_{\expr{x}} \otimes \bigoplus_{z \in \Wpcosets} \pr(B_{\expr{y}})_{w,z} \text{.}
\end{equation*}
By the definition of $\pr$, this isomorphism follows by tensoring $B^{F}_{\expr{x}}$ with the decomposition of $\hat{R}_w \otimes B_{\expr{y}}$ from Corollary~\ref{cor:locpattp}. This decomposition is clearly natural on Bott--Samelson bimodules, so the equivariance is natural as well.
\end{proof}



%

%


\subsection{Decategorified results}
\label{sec:decat}

In this section we aim to decategorify the functor $\pr$. 
Although $\pr$ is inherently ungraded, we will present these results using certain categories which have valuation structure. 
We hope that future work will explain precisely how the grading on $\dgrm$ interacts with this valuation structure.

Taking the valuation on $\hat{R} \otimes \dgrmF$ into account, consider the category $\Mat_{\Wpcosets}(\hat{R} \otimes \dgrmF)$. 
This matrix category has a valuation on morphisms in the sense of Definition~\ref{defn:valmod}; namely, for $L$ a matrix of morphisms we set $\val L$ be the minimal valuation of all the entries. 
For each $\expr{x} \in \expr{\Ss}$ we write $\glsuseri*{Bexprxpastp}$ for the object in $\Mat_{\Wpcosets}(\hat{R} \otimes \dgrmF)$ corresponding to $\pr(B_{\expr{x}})$ with each entry having valuation shift $0$. 
In this matrix category it is also useful to define a selective version of a valuation shift.

\begin{defn}
For each $w \in \Wpcosets$ and $m \in \ZZ$ we define $\langle m \rangle_w \in \Mat_{\Wpcosets}(\hat{R} \otimes \dgrmBS^{F,\oplus})$ by
\begin{equation*}
(\gls*{langlemranglew})_{z,z'}=\begin{cases}
\hat{R} \langle m \rangle & \text{if $z=z'=w$,} \\
\hat{R} & \text{if $z=z' \neq w$,} \\
0 & \text{otherwise.}
\end{cases}
\end{equation*}
\end{defn}

\begin{rem}
We can rewrite the valuation shift functor on $\Mat_{\Wpcosets}(\hat{R} \otimes \dgrmF)$ using the language of partial valuation shifts as $\langle 1 \rangle=\prod_{w \in \Wpcosets} \langle 1\rangle_{w}$. 
In general partial valuation shifts do \emph{not} commute with the Bott--Samelson objects $B_{\expr{x}}^{(\ast)}$ for $\expr{x} \in \expr{\Ss}$.
\end{rem}

\begin{defn}
The category $\gls*{dgrmBSpastpoplus}$ is the full additive $\hat{R}$-linear monoidal subcategory of $\Mat_{\Wpcosets}(\hat{R} \otimes \dgrmF)$ generated by $B^{(\ast)}_{\expr{x}}$ and $\langle m \rangle_w$ for all $\expr{x} \in \expr{\Ss}$, $m \in \ZZ$, and $w \in \Wpcosets$.
%
%
\end{defn}

The category $\dgrm_{\rm BS}^{(\ast),\oplus}$ inherits a monoidal action on $\bigoplus_{w \in \Wpcosets} (\dgrmF \otimes \hat{R}_w)$, defined by
\begin{align*}
D : \bigoplus_{w \in \Wpcosets} (\dgrmF \otimes \hat{R}_w) & \longrightarrow \bigoplus_{w \in \Wpcosets} (\dgrmF \otimes \hat{R}_w) \\
(C_w \otimes \hat{R}_w)_{w \in \Wpcosets} & \longmapsto \left(\left(\bigoplus_{z \in \Wpcosets} C_z \otimes D_{z,w} \right) \otimes \hat{R}_w \right)_{w \in \Wpcosets} \text{,}
\end{align*}
which comes from a valuation module version of the equivalence in Proposition~\ref{prop:prequalstensorproduct}. 
This defines a monoidal action of $\dgrm_{\rm BS}^{(\ast),\oplus}$ on $\dgrmpast$ which respects the valuation structure.

\begin{thm} \label{thm:chast}
There is an isomorphism of $\laur$-algebras mapping
\begin{align*}
\heckeast & \longrightarrow [\dgrm_{\rm BS}^{(\ast),\oplus}] \\
b^{(\ast)}_s & \longmapsto [B^{(\ast)}_s] \\
u_w & \longmapsto [\langle 1 \rangle_w]
\end{align*}
for all $s \in \Ss$ and $w \in \Wpcosets$. 
Moreover, $\charctr_{p|\ast} : [\dgrmpast] \rightarrow \heckepast$ is an isomorphism of right $[\dgrm_{\rm BS}^{(\ast),\oplus}] \iso \heckeast$-modules.
\end{thm}

\begin{proof}
For each $\expr{x} \in \expr{\Ss}$ and $w \in \W$, we calculate
\begin{align*}
\charctr_{p|\ast} [\hat{R}_w \cdot B^{(\ast)}_{\expr{x}}] & =\charctr_{p|\ast} \left[\bigoplus_{\patt{q} \in [\expr{x}]_{\ast}(w)} B^F_{F^{-1}(w\patt{q}(w\hat{q})^{-1})} \otimes \hat{R}_{w\hat{q}}\right] \\
& =\sum_{\patt{q} \in [\expr{x}]_{\ast}(w)} b^{(p)}_{w\patt{q}(w\hat{q})^{-1}} h^{(p|\ast)}_{w\hat{q}} \\
& =h^{(p|\ast)}_{w} \cdot b^{(\ast)}_{\expr{x}}
\end{align*}
by Proposition~\ref{prop:ddfheckepasttwisted}. 
Similarly $[\langle 1\rangle_w]$ acts on $[\dgrmpast]$ as $u_w$. 
As the monoidal action commutes with the left $(\hat{R} \otimes \dgrmF)$-module action and the objects $\hat{R}_w$ span $\dgrmpast$ with respect to this action, the isomorphism classes $[B^{(\ast)}_{\expr{x}}]$ and $[\langle 1 \rangle]_w$ have the correct action on all of $[\dgrmpast]$. 
Thus some quotient of $[\dgrm_{\rm BS}^{(\ast),\oplus}]$ is isomorphic to $\heckeast$. 
But the monoidal action of $\dgrm_{\rm BS}^{(\ast),\oplus}$ on $\dgrmpast$ is faithful because it is the restriction of the faithful action of $\Mat_{\Wpcosets}(\hat{R} \otimes \dgrmF)$ on $\dgrmpast$, so we are done.
\end{proof}

Recall the matrix recursion representation $\xi$ on $\W$ from \S \ref{sec:heckeast-heckepast}.
The ungraded embedding of $\dgrm_{\rm BS}^{{\rm deg},\oplus}$ into $\dgrm_{\rm BS}^{(\ast),\oplus}$, along with Theorem~\ref{thm:chast} immediately implies the following.

\begin{cor}
The functor $\pr : \dgrmdeg \rightarrow \Mat_{\Wpcosets}(\hat{R} \otimes \dgrm^{F,{\rm dev}})$ is a categorification of the matrix recursion representation $\xi$.
\end{cor}

We can now prove our lower bound on the $p$-canonical basis as discussed in the introduction. 
Recall that the \defnemph{$p$-canonical basis} is defined by $\glsuseri*{pbx}=\charctr [B_x]$ for all $x \in \W$.

\begin{cor} \label{cor:sbim-highordlink}
Let $x \in \W$. Then
\begin{equation*}
\prescript{p}{v=1}{b}_x \in \sum_{\substack{y \in \W \\ w \in \Wpcosets\\ F(y)w \leq x}} \Nneg F(\prescript{p}{v=1}{b}_y) w \text{.}
\end{equation*}
\end{cor}

\begin{proof}
Consider $1 \cdot \prescript{p}{v=1}{b}_x$ in $\prescript{}{v=1}{\hecke}_{p|\ast}$. 
By the previous corollary this corresponds to sum of the ungraded characters in the first row of $\pr([B_x])$. 
The result follows immediately.
%
\end{proof}

\begin{exam} \label{exam:sbim-highordlink}
Let $p=3$ and $\W$ be of type $\widetilde{A_1}$.
We have $\prescript{3}{}{b}_{010}b_{1}=b_{0101}+b_{01}$. 
By \cite[Proposition 4.2(6)]{jensen-williamson} this is a sum of $p$-canonical basis elements, so we can apply Corollary~\ref{cor:sbim-highordlink}. 
Setting $v=1$ we get
\begin{equation*}
\prescript{3}{v=1}{b}_{{\textcolor{blue} 0}{\textcolor{red} 1}{\textcolor{blue} 0}}b_{\textcolor{red} 1}=2(\prescript{3}{v=1}{b}_{\textcolor{red} 1})\ident+(\prescript{3}{v=1}{b}_{\textcolor{red} 1}+\prescript{3}{v=1}{b}_{{\textcolor{blue} 0}_3}){\textcolor{blue} 0}+(\prescript{3}{v=1}{b}_{\textcolor{red} 1}+\prescript{3}{v=1}{b}_{{\textcolor{blue} 0}_3}){{\textcolor{blue} 0}{\textcolor{red} 1}} \text{.}
\end{equation*}
We depict this using \defnemph{weight diagrams} in Figure~\ref{fig:weightdiagram}, where the alcove corresponding to $y \in \W$ is marked with a number of dots equal to the coefficient of $\prescript{}{v=1}{h}_y=y$. 
One can visualize the two decompositions above by coloring the dots (i.e.~standard subquotients) according to which canonical basis element (i.e.~indecomposable summand) they lie in. 
Since the decompositions lead to different colorings, we draw a complete colored weight diagram for each decomposition. 
Corollary~\ref{cor:sbim-highordlink} implies that the $p$-canonical summands partition the colors in the $\Wp$-weight diagram. 
In particular, it is easy to see that $\prescript{3}{}{b}_{0101} \neq b_{0101}$. 
Otherwise the green dots and the black dots in the $\W$-weight diagram correspond to different $p$-canonical basis elements, but this cannot be the case because it is impossible to partition the colors in the $\Wp$-weight diagram below in the same manner. 
Weight diagrams are very similar to the diagrams in \cite{jg-jensen} depicting tilting characters, and the processes of applying Corollary~\ref{cor:sbim-highordlink} or Andersen's tilting character lower bounds to a potential diagram are essentially identical.

\begin{figure}
\begin{center}
\begin{tikzpicture}[line cap=rect]
  \node at (-4,0) {\textover[c]{$\W$}{$\Wp$}};
  \draw[thick] (-3.5, 0) -- (-3,0) 
         -- 
         node[below] {$\textcolor{red}{1}\textcolor{blue}{0}\textcolor{red}{1}$} (-2,0) 
         -- 
         node[below] {$\textcolor{red}{1}\textcolor{blue}{0}$} (-1,0)
         -- 
         node[below] {$\textcolor{red}{1}$} 
         (0,0) -- 
         node[below] {$\ident$} (1,0) 
         -- 
         node[below] {$\textcolor{blue}{0}$} (2,0) 
         -- 
         node[below] {$\textcolor{blue}{0}\textcolor{red}{1}$} (3,0) 
         -- 
         node[below] {$\textcolor{blue}{0}\textcolor{red}{1}\textcolor{blue}{0}$} (4,0) 
         -- 
         node[below] {$\textcolor{blue}{0}\textcolor{red}{1}\textcolor{blue}{0}\textcolor{red}{1}$} (5,0) 
         -- node[below] {$\textcolor{blue}{0}\textcolor{red}{1}\textcolor{blue}{0}\textcolor{red}{1}\textcolor{blue}{0}$} (6,0) 
         -- (6.5,0);
  \foreach \x in {-3,-1,...,5}
  \draw[thick, blue] (\x,-0.125) edge (\x,0.125);
  \foreach \x in {-2,0,...,6}
  \draw[thick,red] (\x,-0.125) edge (\x,0.125);
%
  \foreach \x in {-3,3}
  \draw[thick,blue] (\x,-0.25) edge (\x,0.25);
  \foreach \x in {0,6}
  \draw[thick, red] (\x,-0.25) edge (\x,0.25);
  \node[dot,fill=black,above=2pt] at (-2.5,0) {};
  \node[dot,fill=black,above=2pt] at (-1.5,0) {};
  \node[dot,fill=black,above=2pt] at (-0.5,0) {};
  \node[dot,fill=black,above=2pt] at (-0.5,0) {};
  \node[dot,fill=green,above=8pt] at (-0.5,0) {};
  \node[dot,fill=black,above=2pt] at (0.5,0) {};
  \node[dot,fill=black,above=2pt] at (0.5,0) {};
  \node[dot,fill=green,above=8pt] at (0.5,0) {};
  \node[dot,fill=black,above=2pt] at (1.5,0) {};
  \node[dot,fill=green,above=8pt] at (1.5,0) {};
  \node[dot,fill=black,above=2pt] at (2.5,0) {};
  \node[dot,fill=green,above=8pt] at (2.5,0) {};
  \node[dot,fill=black,above=2pt] at (3.5,0) {};
  \node[dot,fill=black,above=2pt] at (4.5,0) {};
\end{tikzpicture}

\begin{tikzpicture}[line cap=rect]
  \node at (-4,0) {$\Wp$};
  \draw[thick] (-3.5, 0) -- (-3,0) 
         -- 
         node[below] {$\textcolor{red}{1}\textcolor{blue}{0}\textcolor{red}{1}$} (-2,0) 
         -- 
         node[below] {$\textcolor{red}{1}\textcolor{blue}{0}$} (-1,0)
         -- 
         node[below] {$\textcolor{red}{1}$} 
         (0,0) -- 
         node[below] {$\ident$} (1,0) 
         -- 
         node[below] {$\textcolor{blue}{0}$} (2,0) 
         -- 
         node[below] {$\textcolor{blue}{0}\textcolor{red}{1}$} (3,0) 
         -- 
         node[below] {$\textcolor{blue}{0}\textcolor{red}{1}\textcolor{blue}{0}$} (4,0) 
         -- 
         node[below] {$\textcolor{blue}{0}\textcolor{red}{1}\textcolor{blue}{0}\textcolor{red}{1}$} (5,0) 
         -- node[below] {$\textcolor{blue}{0}\textcolor{red}{1}\textcolor{blue}{0}\textcolor{red}{1}\textcolor{blue}{0}$} (6,0) 
         -- (6.5,0);
  \foreach \x in {-3,-1,...,5}
  \draw[thick, blue] (\x,-0.125) edge (\x,0.125);
  \foreach \x in {-2,0,...,6}
  \draw[thick,red] (\x,-0.125) edge (\x,0.125);
%
  \foreach \x in {-3,3}
  \draw[thick,blue] (\x,-0.25) edge (\x,0.25);
  \foreach \x in {0,6}
  \draw[thick, red] (\x,-0.25) edge (\x,0.25);
  \node[dot,fill=magenta,above=2pt] at (-2.5,0) {};
  \node[dot,fill=yellow,above=2pt] at (-1.5,0) {};
  \node[dot,fill=orange,above=2pt] at (-0.5,0) {};
  \node[dot,fill=gray,above=8pt] at (-0.5,0) {};
  \node[dot,fill=orange,above=2pt] at (0.5,0) {};
  \node[dot,fill=gray,above=8pt] at (0.5,0) {};
  \node[dot,fill=violet,above=2pt] at (1.5,0) {};
  \node[dot,fill=yellow,above=8pt] at (1.5,0) {};
  \node[dot,fill=teal,above=2pt] at (2.5,0) {};
  \node[dot,fill=magenta,above=8pt] at (2.5,0) {};
  \node[dot,fill=teal,above=2pt] at (3.5,0) {};
  \node[dot,fill=violet,above=2pt] at (4.5,0) {};
\end{tikzpicture}
\end{center}
\caption{Weight diagrams for $\prescript{3}{}{b}_{010}b_{1}$.}
\label{fig:weightdiagram}
\end{figure}
\end{exam}

%
\begin{rem} 
We broadly conjecture that there is a Kazhdan--Lusztig-type construction of a self-dual $\toralalg$-basis $\{b^{(\ast)}_x : x \in \W\}$ of $\heckeast$, and that there is a new category related to $\dgrm_{\rm BS}^{(\ast),\oplus}$ with indecomposable objects $\{B^{(\ast)}_x : x \in \W\}$. 
Unfortunately, precisely characterizing such constructions is tricky. 
On the algebraic side, it is unclear what should take the place of the degree condition on coefficients of the standard basis elements $b^{(\ast)}_x$. 
On the categorical side, it is unclear which morphisms in $\dgrm_{\rm BS}^{(\ast),\oplus}$ should be considered, as extra morphisms will produce too many indecomposable objects. 
Once the correct definition is found, the next natural step would be to prove a Soergel's conjecture-like result (see \cite[Conjecture~3.18]{ew-soergelcalc}), equating the isomorphism class $[B^{(\ast)}_x]$ with $b^{(\ast)}_x$ under the isomorphism in Theorem~\ref{thm:chast} for fixed $x \in \W$ when $p$ is sufficiently large. 
Their images in $\ZZ\W$ should then match the images of a ``second generation'' Kazhdan--Lusztig basis in $\hecke$ analogous to second generation tilting characters \cite{lw-tiltinggens,lw-billiards}.
\end{rem}

\subsection{Tilting characters}
\label{sec:tiltingchars}

We conclude with the application of Corollary \ref{cor:sbim-highordlink} to characters of indecomposable tilting modules, as described in the introduction. 
Suppose $G$ is a semisimple, simply connected algebraic group with root system $\Phif$ over an algebraically closed field of characteristic $p\geq h$, where $h$ is the Coxeter number of $\Phif$. 
We consider the affine Cartan matrix of type $\widetilde{\Phi_{\rm f}^\vee}^\vee$ (see Remark \ref{rem:affrootsystems}) and construct $\W$, $V$, and $\dgrm$ accordingly.

Recall that the \defnemph{Kazhdan--Lusztig basis} $\{b_x : x \in \W\}$ is the unique self-dual $\laur$-basis of $\hecke$ of the form
\begin{equation}
b_x=\sum_{y \in \W} h_{y,x} h_y
\end{equation}
such that for all $x,y \in \W$
\begin{itemize}
\item $h_{x,x}=1$;

\item $h_{y,x} \in v \ZZ[v]$ whenever $y \neq x$.
\end{itemize}
The polynomials $\{h_{y,x}: x,y \in \W\}$ are the \defnemph{Kazhdan--Lusztig polynomials} for $\W$. 
Similarly, if we write the $p$-canonical basis elements $\{\prescript{p}{}{b}_x : x \in \W\}$ in terms of the standard basis
\begin{equation}
\prescript{p}{}{b}_x=\sum_{y \in \W} \prescript{p}{}{h}_{y,x} h_y
\end{equation}
we call the Laurent polynomials $\{\prescript{p}{}{h}_{y,x} : x,y \in \W\}$ the \defnemph{$p$-Kazhdan--Lusztig polynomials} for $\W$.

Let $\Wfcosets$ denote the set of minimal length representatives for the right cosets $\Wf \backslash \W$. 
(This is also the set of dominant elements, as defined in the introduction.)
The \defnemph{antispherical Kazhdan--Lusztig polynomials} $\{n_{y,x} : x,y \in \Wfcosets\}$ can be defined by the formula
\begin{equation}
n_{y,x} = \sum_{z \in \Wf} (-v)^{\len(z)} h_{zy,x}
\end{equation}
for all $x,y \in \Wfcosets$.
We may similarly define (see e.g.~\cite[\S 1.4]{riche-williamson}) the \defnemph{antispherical $p$-Kazhdan--Lusztig polynomials}
\begin{equation}
\prescript{p}{}{n}_{y,x} = \sum_{z \in \Wf} (-v)^{\len(z)} \prescript{p}{}{h}_{zy,x} \text{.}
\end{equation}

We recall some algebraic groups notation from the introduction. 
To summarize, we write $\cdot_p$ for the $p$-dilated dot action, $T(\lambda)$ (and $T_l(\lambda)$) for the indecomposable (quantum) tilting module of highest weight $\lambda$, $\charctr M$ for the formal character of $M$, and $\chi(\lambda)$ for the Weyl character of highest weight $\lambda$. 
As in the introduction (but unlike in the rest of the paper) we will use $C_l$ here to denote the set of dominant weights in the interior of the fundamental $\rho$-shifted $l$-alcove.
For $w \in \Wpcosets$ we also recall the Frobenius scaling map $F^w$ on characters, defined by setting 
\begin{equation*}
F^w(\chi(x \cdot_p \lambda))=\chi(F(x)w \cdot_p \lambda)
\end{equation*}
for all $\lambda \in C_p$ and $x \in \Wfcosets$ and extending to all characters linearly.

Soergel's character formula for quantum tilting modules takes the following form.

\begin{thm}[{\cite[Conjecture~7.1]{soergel-KL}}]
Let $\lambda \in C_l$ and $x \in \Wfcosets$. 
Then
\begin{equation*}
\charctr T_l(x \cdot_l \lambda)=\sum_{y \in \Wfcosets} n_{y,x}(1) \chi(y \cdot_l \lambda) \text{.}
\end{equation*}
\end{thm}

The $p$-canonical tilting character formula is very similar.

\begin{thm}[{\cite[Theorem~7.6]{amrw-koszultilting}}]
Let $\lambda \in C_p$ and $x \in \Wfcosets$.
Then
\begin{equation*}
\charctr T(x \cdot_p \lambda)=\sum_{y \in \Wfcosets} \prescript{p}{}{n}_{y,x}(1) \chi(y \cdot_p \lambda) \text{.}
\end{equation*}
\end{thm}

Now we prove the recursive tilting character formula lower bound (Theorem \ref{thm:tiltcharbound}) from the introduction.
For clarity, we first show that the special case follows from Andersen's observation and Andersen's conjecture. 
Then we prove the general case using Corollary \ref{cor:sbim-highordlink}.

\begin{proof}[Proof of the recursive tilting character lower bound, special case]
First we note that for $\lambda \in C_p$, if $w \in \Wpcosets$ then $w \cdot_p \lambda \in C_{p^2}$, whereas if $x \in \W$ then $x \cdot_{p^2} \lambda=F(x) \cdot_p \lambda$.
Now suppose $x \in \Wfcosets$ with $x=F(y)w$ for some $y \in \Wfcosets$ and $w \in \Wpcosets$. 
We compute
\begin{align*}
\charctr T_{p^2}(x \cdot_p \lambda)& =\charctr T_{p^2}(F(y)w \cdot_p \lambda) \\
& =\charctr T_{p^2}(F(y) \cdot_p (w \cdot_p \lambda)) \\
& =\charctr T_{p^2}(y \cdot_{p^2} (w \cdot_p \lambda)) \\
& =\sum_z n_{z,y}(1) \chi(z \cdot_{p^2} (w \cdot_p \lambda)) \\
& =\sum_z n_{z,y}(1) \chi(F(z)w \cdot_p \lambda)) \\
& =F^w(\charctr T_p(y \cdot_p \lambda)) \text{.}
\end{align*}
In other words, quantum tilting characters ``scale up'' via the Frobenius map.

Now let $x,x' \in \Wpcosets$.
If we assume Andersen's conjecture, then Andersen's observation with $l=p^2$ takes the form
\begin{align*}
\charctr T(F(x')x \cdot_p \lambda) & \in \sum_{z \leq F(x')x} \Nneg \charctr T_{p^2}(z \cdot_p \lambda) \\
& =\sum_{\substack{y \in \Wpcosets \\ w \in \Wpcosets \\ F(y)w \leq F(x')x}} \Nneg F^w(\charctr T_{p}(y \cdot_p \lambda)) \\
& =\sum_{\substack{y \in \Wpcosets \\ w \in \Wpcosets \\ F(y)w \leq F(x')x}} \Nneg F^w(\charctr T(y \cdot_p \lambda))
\end{align*}
as claimed.
\end{proof} 

%
%

\begin{proof}[Proof of the recursive tilting character lower bound, general case]
Corollary \ref{cor:sbim-highordlink} is equivalent to
\begin{equation}
\prescript{p}{}{h}_{z,x}(1) \in \sum_{\substack{y \in \W\\ F(y)w \leq x}} \Nneg \prescript{p}{}{h}_{z',y}(1) \label{eq:pKLpolybound}
\end{equation}
for all $x,z,z' \in \W$ and $w \in \Wpcosets$ such that $z=F(z')w$. 
Taking a signed sum of \eqref{eq:pKLpolybound} over a right coset in $\Wf \backslash \W$ we obtain
\begin{equation}
\prescript{p}{}{n}_{z,x}(1) \in \sum_{\substack{y \in \Wfcosets\\ F(y)w \leq x}} \Nneg \prescript{p}{}{n}_{z',y}(1) \label{eq:pKLantispherpolybound}
\end{equation}
for $x,z,z' \in \Wfcosets$ and $w \in \Wpcosets$ such that $z=F(z')w$.
Now multiply \eqref{eq:pKLantispherpolybound} by $\chi(z \cdot_p \lambda)$ and sum over all $z \in \Wfcosets$ to obtain
\begin{align*}
\charctr T(x \cdot_p \lambda)& =\sum_{z \in\Wfcosets} \prescript{p}{}{n}_{z,x}(1) \chi(z \cdot_p \lambda) \\
& \in \sum_{\substack{y,z,z' \in \Wfcosets\\ w \in \Wpcosets\\ z=F(z')w \\ F(y)w \leq x}} \Nneg \prescript{p}{}{n}_{z',y}(1) \chi(z \cdot_p \lambda) \\
& =\sum_{\substack{y \in \Wfcosets \\ w \in \Wpcosets\\ F(y)w \leq x}} \Nneg F^w\left(\sum_{z' \in \Wfcosets} \prescript{p}{}{n}_{z',y}(1) \chi(z' \cdot_p \lambda)\right) \\
& =\sum_{\substack{y \in \Wfcosets \\ w \in \Wpcosets\\ F(y)w \leq x}} \Nneg F^w(\charctr T(y \cdot_p \lambda)) \qedhere
\end{align*}
\end{proof}

\appendix

\section{Affine reflection groups}
\label{app:affrefl}

We work using the notation and assumptions of \ref{sec:affWnotn}.
Let $\glsuseri*{Vaprime}=V_{A}/\ZZ c$. 
For each $s \in \Ss$ write $\alpha_s$ for the image of $a_s$ in $V'_{A}$. 
For any $t \in \Ss$ for which $c_{t}=1$, the set $\{\alpha_s : s \in \Ss \setminus \{t\}\}$ is a free basis for $V'_{A}$. 
Since $V'_{A}$ is free, its dual $(V'_{A})^{\ast}$ is naturally isomorphic to the annihilator
\begin{equation}
\gls*{Z}=(\ZZ c)^\circ=\{b \in V_{A}^\ast : \langle b, c\rangle=0 \} \subset V_{A}^\ast \text{.}
\end{equation}

\begin{lem} \label{lem:gradrlz}
For all $s \in \Ss$, we have $a_s^\vee \in Z$.
The subsets 
\begin{gather*}
\{\alpha_s : s \in \Ss\} \subset V'_{A} \text{,} \\
\{a_s^\vee : s \in \Ss\} \subset Z \iso (V'_{A})^\ast 
\end{gather*}
give $V'_{A}$ the structure of a realization with Cartan matrix $A$.
The quotient map $V_{A} \rightarrow V'_{A}$ is a morphism of realizations.
\end{lem}

We call $V'_{A}$ the \defnemph{gradient realization} of $(\W,\Ss)$ with respect to $A$. 
For consistency we will also write $\alpha_s^\vee=a_s^\vee \in Z$ for each $s \in \Ss$. 

\begin{proof}
For all $s \in \Ss$ we have
\begin{equation*}
\langle a_s^\vee, c \rangle=\sum_{t \in \Ss} \langle a_s^\vee, a_t \rangle c_t=\sum_{t \in \Ss} a_{st}c_t=0 \text{.}
\end{equation*}
This simultaneously shows that $a_s^\vee \in Z$ and that $c$ is fixed by $\W$. 
Hence $V'_{A}$ has the structure of a realization, and it is immediate that the quotient map $V_{A} \rightarrow V'_{A}$ is a morphism of realizations.
By Remark~\ref{rem:realizations}~\ref{item:morphism-realizations} it follows that the Cartan matrix of $V'_{A}$ is $A$.
\end{proof}

Let $V'_{A,\RR}=\RR \otimes V'_{A}$ and let 
\begin{equation}
\glsuseri*{Phif}=\{w(\alpha_s) : w \in \W,\ s \in \Ss\} \subset V'_{A,\RR} \text{.}
\end{equation}
Define a symmetric bilinear form on $V'_{A,\RR}$ by setting
\begin{equation}
(\alpha_s,\alpha_t)=n_s a_{st} \qquad \text{for all $s,t \in \Ss$.} \label{eq:gradinnerproduct}
\end{equation}
This choice ensures that
\begin{equation*}
\frac{2(\alpha_s,\alpha_t)}{(\alpha_s,\alpha_s)}=a_{st} \qquad \text{for all $s,t \in \Ss$.}
\end{equation*}
The form $(-,-)$ is well defined because
\begin{equation*}
(\alpha_s,c)=\sum_{t \in \Ss} n_s a_{st}c_t=0 
\end{equation*} 
for all $s \in \Ss$.
In fact $(-,-)$ is an inner product, because $V'_{A,\RR}$ has a basis which is a proper subset of $\{\alpha_s : s \in \Ss\}$ and any proper submatrix of $(n_s a_{st})_{s,t \in \Ss}$ is positive definite.

\begin{prop} \label{prop:finrootsystem}
The subset $\Phif \subset V'_{A,\RR}$ forms an irreducible, not necessarily reduced finite root system with respect to the inner product $(-,-)$. 
The action of $\W$ on $V'_{A,\RR}$ factors through $\glsuseri*{Wf}$, the Weyl group of $\Phif$.
Moreover, there exists some $\glsuseri*{tildes} \in \Ss$ such that $\{\alpha_s : s \in \Ss \setminus \{\tilde{s}\}\}$ is a basis of $\Phif$. 
\end{prop}

\begin{proof}
Each generator $s \in \Ss$ acts on $V'_{A,\RR}$ as a reflection with respect to the inner product. 
Also $\W$ preserves the lattice $\sum_{s \in \Ss} \ZZ \alpha_s \subset V'_{A,\RR}$, which spans $V'_{A,\RR}$.
The only non-affine reflection groups which preserve a lattice arise from finite (but not necessarily reduced) root systems. 
This immediately implies that $\Phif$ is such a root system and that $\W$ acts on $V'_{A,\RR}$ via the standard action of $\Wf$. 
Irreducibility of $\Phif$ follows immediately from the indecomposability of $A$.

For the last claim, observe that the relation
\begin{equation*}
\sum_{s \in \Ss} c_s \alpha_s=0
\end{equation*}
implies that we can write every root in $\Phif$ as a linear combination of $\{\alpha_s : s \in \Ss\}$ with coefficients all of the same sign (in fact with all coefficients non-negative).
This means that some subset of $\{\alpha_s : s \in \Ss\}$ is a basis of $\Phif$. 
As $\Phif$ has rank $|\Ss|-1$ the result follows.
\end{proof}

The finite root system $\Phif$ is called the \defnemph{gradient root system} associated to $A$ (see \cite[\S 6]{macdonald-paper}).
As in \S \ref{sec:affWnotn} we fix a choice of such an $\tilde{s}$ for the rest of this section, and set $\glsuseri*{Ssf}=\Ss \setminus \{\tilde{s}\}$. 
We call the generators in $\Ssf$ the \defnemph{finite generators} and $\tilde{s}$ the \defnemph{affine generator} of $\W$.
An immediate consequence of Proposition~\ref{prop:finrootsystem} is the following.

\begin{cor} \label{cor:finweylsemidirectproduct}
The Cartan matrix of $\Phif$ is $(a_{st})_{s,t \in \Ssf}$.
The subgroup of $\W$ generated by $\Ssf$ is isomorphic to $\Wf$.
\end{cor}

The root $\alpha_{\tilde{s}} \in \Phif$ is negative, because $\alpha_{\tilde{s}}=-\sum_{s \in \Ssf} c_s \alpha_s$ and each $c_s$ is non-negative.
Let $\gls*{alphah}=-\alpha_{\tilde{s}}$ be the corresponding positive root, and let $\gls*{sh}$ be the image of $\tilde{s}$ in the quotient $\W \rightarrow \Wf$.
It is clear that $\sh$ acts on $V'_{A}$ via reflection in $\alphah$.
We also set
\begin{align}
a_{\rm h}& =\sum_{s \in \Ssf} c_s a_s \in V_A \text{,} &
a_{\rm h}^\vee & =\sum_{s \in \Ssf} c_s^\vee a_s^\vee \in V_A^\ast
\end{align}
where the coefficients $(c_s^\vee)_{s\in \Ssf}$ are given by $\alpha_{\rm h}^\vee=\sum_{s \in \Ssf} c_s^\vee \alpha_s^\vee$.
(Note that a consequence of this second definition is that $a_{\rm h}^\vee=-a_{\tilde{s}}^\vee$.)
The next lemma is useful for calculations involving the action of $\sh$ on $V_{A}$.

\begin{lem} \label{lem:shcalc} \hfill
\begin{enumerate}[label={\rm (\roman*)}]
\item We have
\begin{equation*}
\sh(b)=b-\langle a_{\rm h}^\vee, b \rangle a_{\rm h}
\end{equation*}
for all $b \in V_{A}$.

\item The subgroup $\W_{\rm h}=\langle \sh, \tilde{s}\rangle \leq \W$ is isomorphic (as a Coxeter group) to the infinite dihederal group.
The submodule $V_{A,{\rm h}}=\ZZ a_{\rm h} \oplus \ZZ a_{\tilde{s}}$ along with the subsets
\begin{gather*}
\{a_{\rm h}, a_{\tilde{s}}\} \subset V_{A,{\rm h}} \\
\{a_{\rm h}^\vee+(V_{A,{\rm h}})^\circ, a_{\tilde{s}}^\vee+(V_{A,{\rm h}})^\circ\} \subset V_{A}^\ast/(V_{A, {\rm h}})^\circ \iso V_{A,{\rm h}}^\ast
\end{gather*}
is the universal realization of $\W_{\rm h}$ with respect to the Cartan matrix
\begin{equation*}
\begin{bmatrix}
2 & -2 \\
-2 & 2
\end{bmatrix} \text{.}
\end{equation*}
\end{enumerate}
\end{lem}

\begin{proof}
For the first part, suppose $\alpha_{\rm h}=w(\alpha_s)$ for some $w \in \Wf$ and $s \in \Ssf$. 
This implies that $\sh=wsw^{-1}$.
By \eqref{eq:nonsimplereflections}, the result follows if we can show that $a_{\rm h}=w(a_s)$ and $a_{\rm h}^\vee=w(a_s^\vee)$.
In terms of the simple roots, we already know that $w(\alpha_t)=\sum_{t \in \Ssf} c_t \alpha_t$.
We could have also calculated the coefficients $(c_t)_{t \in \Ssf}$ using $w$, $s$, and the Cartan matrix of $\Phif$.
But this shows that $w(a_s)=\sum_{t \in \Ssf} c_t a_t=a_{\rm h}$, as the coefficients here are calculated in exactly the same way using $(a_{st})_{s,t \in \Ssf}$, which is the Cartan matrix of $\Phif$ by the previous corollary. 
A similar argument works for $a_{\rm h}^\vee$.

To prove the second part, we calculate
\begin{equation*}
\langle a_{\rm h}^\vee, a_{\tilde{s}} \rangle=\sum_{s \in \Ssf} c_s^\vee \langle a_s^\vee, a_{\tilde{s}}+\ZZ c \rangle=\sum_{s \in \Ssf} c_s^\vee \langle \alpha_s^\vee, \alpha_{\tilde{s}} \rangle=-\langle \alpha_{\rm h}^\vee, \alphah \rangle=-2 \text{,}
\end{equation*}
and similarly
\begin{equation*}
\langle a_{\tilde{s}}^\vee, a_{\rm h} \rangle=\sum_{s \in \Ssf} c_s \langle a_{\tilde{s}}^\vee, a_s+\ZZ c \rangle=\sum_{s \in \Ssf} c_s \langle \alpha_{\tilde{s}}^\vee, \alpha_s \rangle=-\langle \alpha_{\rm h}^\vee, \alphah \rangle=-2 \text{.}
\end{equation*}
Also $\langle a_{\tilde{s}}^\vee, a_s\rangle=2$ and
\begin{equation*}
\langle a_{\rm h}^\vee, a_{\rm h} \rangle=\sum_{s,t \in \Ssf} c_s^\vee c_t\langle a_s^\vee, a_t+\ZZ{c} \rangle=\sum_{s,t \in \Ssf} c_s^\vee c_t \langle \alpha_s^\vee,\alpha_t \rangle=\langle \alpha_{\rm h}^\vee, \alphah \rangle=2 \text{.}
\end{equation*}
This shows that the representation of $\W_{\rm h}$ on $V_{A,{\rm h}}$ factors through the universal realization of the infinite dihedral group  with the above Cartan matrix. 
The computation \ref{eq:dihedralcalc1} follows, which shows that $\tilde{s}s_{\rm h}$ has infinite order, so $\W_{\rm h}$ is infinite dihedral.
\end{proof}

We now describe how the universal realization gives rise to an affine reflection representation of $\W$, following \cite[Chapter~6.5]{humphreys} and \cite[Chapter~1.2]{macdonald}.
Let $V_{A,\RR}$ and $Z_{\RR} \iso (V'_{A,\RR})^\ast$ denote the $\RR$-scalar extensions of $V_{A}$ and $Z$ respectively.
The inner product on $V'_{A,\RR}$ transfers easily to $Z_{\RR}$, with
\begin{equation*}
(\alpha_s^\vee,\alpha_t^\vee)=\frac{a_{st}}{n_t}
\end{equation*}
for all $s,t \in \Ss$.
Now let $E$ be the affine hyperplane
\begin{equation}
\gls*{E}=\left\{b \in V_{A,\RR}^\ast : \langle b, c \rangle =1\right\}
\end{equation}
in $V_{A,\RR}^\ast$. 
Clearly $E$ is just a translation of $Z_{\RR}$ in $V_{A,\RR}$, so it has the structure of an affine inner product space, i.e.~a Euclidean space. 

Let $\mathrm{Aff}(E)$ denote the space of affine real-valued functions on $E$.
The derivative of an affine function $f \in \mathrm{Aff}(E)$ is an element of $Z_{\RR}^\ast$, which we will denote by $Df$.
The following lemma explains the use of the word ``gradient'' in ``gradient realization''.

\begin{lem} \label{lem:derivatives}
There is a commutative square of linear maps
\begin{align*}
\xymatrix{
V_{A,\RR} \ar[r]^{\sim} \ar[d] & \mathrm{Aff}(E) \ar[d]^{D} \\
V'_{A,\RR} \ar[r]^{\sim} & Z_{\RR}^{\ast}
}
\end{align*}
where the first row is restriction to $E$. 
\end{lem}

\begin{proof}
Any linear function restricted to an affine hyperplane is an affine function.
For any $s \in \Ss$ the derivative of $\langle - , a_s\rangle|_E$ is $\alpha_s \in V'_{A,\RR} \iso Z_{\RR}^\ast$, since
\begin{equation*}
\langle b+\lambda, a_s \rangle=\langle b,a_s \rangle+\langle \lambda, a_s \rangle=\langle b, a_s \rangle+\langle \lambda, \alpha_s\rangle
\end{equation*}
for any $b \in E$ and $\lambda \in Z_{\RR}$.
As $\{a_s: s\in \Ss\}$ is a basis of $V_{A,\RR}$, this shows that the diagram commutes. 
It follows that the first row is an isomorphism, because $\dim V_{A,\RR}=|\Ss|=\dim \mathrm{Aff}(E)$.
\end{proof}

\begin{prop}
\label{prop:rlz-refl}
The affine Weyl group $\W$ acts on $E$ by Euclidean isometries, with each $s \in \Ss$ corresponding to reflection in the affine hyperplane
\begin{equation*}
E \cap (\RR a_s)^\circ=\{b \in E : \langle b, a_s \rangle=0\} \text{.}
\end{equation*}
\end{prop}

\begin{proof}
For any $w \in \W$ and $b \in E$, we have
\begin{equation*}
\langle w(b), c \rangle=\langle b, w^{-1}(c) \rangle=\langle b, c\rangle=1 \text{,}
\end{equation*}
so $w(b) \in E$ and thus $E$ is stable under the action of $\W$.

Now let $s \in \Ss$. 
To show that $s$ acts on $E$ as an affine reflection, it is enough to show that its derivative is a linear reflection in $Z_{\RR}$.
For any $s \in \Ss$, the derivative of the action of $s$ on $E$ is the linear map
\begin{align*}
Ds : Z_{\RR} & \longrightarrow Z_{\RR} \\
\lambda & \longmapsto \lambda-\langle \lambda, \alpha_s \rangle \alpha_s^\vee
\end{align*}
by Lemma~\ref{lem:derivatives}. 
This coincides with the action on $Z_{\RR}$ arising from the gradient realization $V'_{A,\RR}$.
Thus $Ds$ is a linear reflection by Lemma~\ref{prop:finrootsystem} and we are done.
\end{proof}

We can now think of $\W$ as an affine reflection group. 
By Lemma~\ref{lem:derivatives} the kernel of the quotient map $\W \rightarrow \Wf$ is the lattice $L \leq \W$ of translations.
As this quotient map splits by Corollary~\ref{cor:finweylsemidirectproduct}, we have $\W \iso \Wf \ltimes L$ as a reflection group.
For $\lambda \in Z_{\RR}$, write $t(\lambda):E \rightarrow E$ for the corresponding translation map, and let
\begin{equation}
\Lambda=\{\lambda \in Z_{\RR} : t(\lambda) \in L\} \leq \ZZ\Phi_{\rm f}^\vee \text{.}
\end{equation}
The conjugation action of $\Wf$ on $L$ is as follows.
If $w \in \Wf$ and $\lambda \in \Lambda$ then $wt(\lambda)w^{-1}=t(w(\lambda))$.

For any $b \in E$, we compute
\begin{equation} \label{eq:alphahveetranslation}
\begin{aligned}
\tilde{s}\sh(b)& =\tilde{s}(b-\langle b, a_{\rm h} \rangle \alpha_{\rm h}^\vee) \\
& =b-\langle b, a_{\tilde{s}}\rangle \alpha_{\tilde{s}}^\vee+\langle b, a_{\rm h} \rangle \alpha_{\rm h}^\vee \\
& =b+\langle b, a_{\tilde{s}}+a_{\rm h} \rangle \alpha_{\rm h}^\vee \\
& =b+\alpha_{\rm h}^\vee
\end{aligned}
\end{equation}
using Lemma~\ref{lem:shcalc}.
This means that $\tilde{s}\sh$ is the translation $t(\alpha_{\rm h}^\vee) \in L$, and thus $t(w(\alpha_{\rm h}^\vee)) \in L$ for all $w \in \Wf$.
As a result we have shown $\Lambda \geq \ZZ\Wf(\alpha_{\rm h}^\vee)$, and in fact we have equality.
An easy way to see this is as follows. 
Note that $\Lambda$ is in bijection with $\W/\Wf$ and consider a coset 
\begin{equation*}
w\Wf=w_1 \tilde{s} w_{2} \tilde{s} w_3 \tilde{s} \dotsm  w_{m-1} \tilde{s} w_m \tilde{s} \Wf \in \W/\Wf \text{,} 
\end{equation*}
where $w_i \in \Wf$ for each $i$.
Working inductively, we can rewrite this coset as
\begin{equation*}
[w_1 (\tilde{s} \sh) w_1^{-1}][(w_1 \sh^{-1} w_2) (\tilde{s} \sh) (w_1 \sh^{-1} w_2)^{-1}] \dotsm \Wf \text{,}
\end{equation*}
which corresponds to an element of $\ZZ\Wf(\alpha_{\rm h}^\vee)$.

%

\begin{rem} \label{rem:affrootsystems} \hfill
\begin{enumerate}[label=(\arabic*)]
\item Proposition~\ref{prop:rlz-refl} shows that the set $\gls*{Phi}=\{w(a_s): w \in \W,\ s \in \Ss\} \subset V_{A,\RR}$ is an affine root system in the sense of Macdonald \cite{macdonald}. 
This corresponds to the set of real roots in the root system of the Kac--Moody algebra $\mathfrak{g}(A)$ in the terminology of \cite{kac}.
Both \cite{macdonald} and \cite{kac} include a classification of affine root systems by their Dynkin diagrams, which we will frequently use in the examples.
We warn the reader that the labeling of roots in \cite{kac} does not always agree with our convention (which more closely matches \cite{macdonald}). 
Namely, the root labeled $\alpha_0$ in $A^{(2)}_{2l}$ from \cite[Table~Aff~2]{kac} is \emph{not} the root we call $a_{\tilde{s}}$.

\item The root vector $\alphah$ is the highest root or the highest short root of $\Phif$.
The first case occurs when $\Phi=\widetilde{\Phif}$, or equivalently, $A$ is the Cartan matrix of the extended Dynkin diagram associated to $\Phif$ (i.e.~a non-twisted affine Dynkin diagram or $A^{(2)}_{2l}$ in the terminology of \cite{kac}).
When this happens, $\Lambda$ is equal to the whole coroot lattice $\ZZ\Phi_{\rm f}^\vee$.
The second case occurs only if $\Phi$ is the dual or twisted affine root system $\widetilde{\Phi_{\rm f}^\vee}^\vee$.
This means that $A$ is the transpose of the Cartan matrix of the extended Dynkin diagram associated to $\Phif^\vee$.
When this happens, $\Lambda=\ZZ\Wf(\alpha_{\rm h}^\vee)$ is a sublattice of the coroot lattice which is isomorphic (as a $\Wf$-lattice) to $\ZZ\Phif$.
This latter case is particularly important for applications to modular representation theory as described in the introduction.
By direct calculation or a classification of root subsystems (e.g.~\cite[Tables~3.1--3.6]{balnojan-hertling}) the index of the lattice $\Lambda$ inside $Z$ is 
\begin{equation*}
[Z: \Lambda]=\begin{cases}
1 & \text{if $\Phif$ is simply laced,} \\
2^{n-1} & \text{if $\Phif$ is of type $B_n$,} \\
2 & \text{if $\Phif$ is of type $C_n$,} \\
4 & \text{if $\Phif$ is of type $F_4$,} \\ 
3 & \text{if $\Phif$ is of type $G_2$.}
\end{cases}
\end{equation*}
\end{enumerate}
\end{rem}

\printnoidxglossary[type=main,sort=letter,title={List of symbols},style=mcolindex]


\bibliographystyle{habbrv2}
\bibliography{matrixrecursionrefs}
\end{document}

%% file: matrixrecursiontikzstyles.tex
\newlength{\superthick}
\newlength{\cornerradius}
\tikzstyle{corner}=[rounded corners=\cornerradius]
\tikzstyle{dot}=[circle, inner sep=0pt, minimum size=4.8pt]
\tikzstyle{string}=[line width=\superthick]
\tikzstyle{std}=[string,dash pattern=on 0.9pt off 0.9pt]
\definecolor{realcyan}{rgb}{0,1,1}

\newcommand{\rot}[1]{%
    \text{\rotatebox[origin=c]{-180}{${#1}$}}%
  }
\newcommand{\inclscal}[1]{{\color{#1} \spadesuit}}
\newcommand{\projscal}[1]{{\color{#1} \rot{\spadesuit}}}
\newcommand{\epsscal}[1]{{\color{#1} \clubsuit}}
\newcommand{\spescal}[1]{{\color{#1} \rot{\clubsuit}}}

\DeclareMathOperator{\pr}{\Psi}

%% file: matrixrecursion-symbolmacros.tex
\newcommand{\textover}[3][l]{%
 \makebox[\widthof{#3}][#1]{#2}%
 }

\newcommand{\iso}{\cong}
\newcommand{\NN}{\mathbb{N}}
\newcommand{\Nneg}{\mathbb{Z}_{\geq 0}}
\newcommand{\ZZ}{\mathbb{Z}}

\newcommand{\RR}{\mathbb{R}}

\newcommand{\field}{\Bbbk}

\newcommand{\SL}{{\rm SL}}
\newcommand{\ident}{{\rm id}}

\newcommand{\defnemph}[1]{\emph{#1}}

\newcommand{\Phif}{\Phi_{\rm f}}
\newcommand{\Sigmaf}{{\Sigma_{\rm f}}}
\newcommand{\alphah}{\alpha_{\rm h}}
\newcommand{\expr}[1]{\underline{#1}}
\newcommand{\emptexpr}{\varnothing}
\newcommand{\seq}[1]{\mathbf{#1}}
\newcommand{\patt}[1]{\expr{#1}}
\newcommand{\match}[1]{\seq{#1}}
\newcommand{\Wf}{W_{\rm f}}
\newcommand{\W}{W}
\newcommand{\Wl}{\W_l}
\newcommand{\Wp}{\W_p}
\newcommand{\Ssf}{S_{{\rm f}}}
\newcommand{\Ss}{S}
\newcommand{\Ssl}{\Ss_l}
\newcommand{\Ssp}{\Ss_p}
\newcommand{\Sslexpr}{\expr{\Ss}_l}
\newcommand{\Sspexpr}{\expr{\Ss}_p}

\newcommand{\sh}{s_{\rm h}}
\newcommand{\linklexpr}{\expr{\Ss}_{l|1}}
\newcommand{\linkexpr}{\expr{\Ss}_{p|1}}
\newcommand{\Wlcosets}{\prescript{l}{}{\W}}
\newcommand{\Wpcosets}{\prescript{p}{}{\W}}
\newcommand{\toralalg}{\prescript{p}{}{\mathbb{U}}}
\newcommand{\toralsubsets}{\prescript{p}{}{\mathcal{U}}}
\newcommand{\Wfcosets}{\prescript{{\rm f}}{}{\W}}
\newcommand{\alcs}{\mathcal{C}}

\newcommand{\alcsl}{\alcs_l}
\newcommand{\alcsp}{\alcs_p}
\newcommand{\alcslfun}{\prescript{l}{}{\alcs}}
\newcommand{\alcspfun}{\prescript{p}{}{\alcs}}
\newcommand{\len}{\ell}
\newcommand{\upseq}{\mathrm{U}}
\newcommand{\downseq}{\mathrm{D}}
\newcommand{\stayseq}{\mathrm{N}}
\newcommand{\laur}{\ZZ[v^{\pm 1}]}

\newcommand{\hecke}{\mathbb{H}}
\newcommand{\heckel}{\hecke_l}
\newcommand{\heckep}{\hecke_p}

\newcommand{\heckelast}{\hecke_{l|\ast}}
\newcommand{\heckepast}{\hecke_{p|\ast}}
\newcommand{\heckeast}{\hecke_{\ast}}

\newcommand{\dgrm}{\mathcal{D}}
\newcommand{\dgrmdeg}{\mathcal{D}^{\rm deg}}

\newcommand{\dgrmBS}{\mathcal{D}_{\rm BS}}
\newcommand{\dgrmBSdeg}{\mathcal{D}_{\rm BS}^{\rm deg}}
\newcommand{\dgrmBSdegsum}{\mathcal{D}_{\rm BS}^{{\rm deg},\oplus}}
\newcommand{\dgrmBSF}{\mathcal{D}_{\rm BS}^F}
\newcommand{\dgrmstdF}{\mathcal{D}_{\rm std}^F}
\newcommand{\dgrmF}{\mathcal{D}^F}

\newcommand{\dgrmstd}{\mathcal{D}_{\rm std}}
\newcommand{\dgrmBSstd}{\mathcal{D}_{{\rm BS},{\rm std}}}
\newcommand{\dgrmBSstdF}{\mathcal{D}_{{\rm BS},{\rm std}}^F}
\newcommand{\dgrmBSpast}{\mathcal{D}_{{\rm BS},p|\ast}}

\newcommand{\dgrmpast}{\mathcal{D}_{p|\ast}}
\newcommand{\LL}{\mathrm{LL}}
\newcommand{\LLdbl}{\mathbb{LL}}
\newcommand{\pastLL}{\prescript{p|\ast}{}{\LL}}
\newcommand{\pastLLdbl}{\prescript{p|\ast}{}{\LLdbl}}
\newcommand{\Mat}{\mathcal{M}}

\DeclareMathOperator{\Hom}{Hom}
\DeclareMathOperator{\End}{End}

\DeclareMathOperator{\charctr}{ch}
\DeclareMathOperator{\val}{val}
\DeclareMathOperator{\grk}{grk}
\DeclareMathOperator{\valrk}{valrk}

\DeclareMathOperator{\dfct}{def}

%% file: matrixrecursion-symbollist.tex
\newglossaryentry{xi}
{
    name={\ensuremath{\xi}},
    sort={xzzzi},
    description={\nopostdesc}
}

\newglossaryentry{pie}
{
    name={\ensuremath{\pi_{\seq{e}}}},
    sort={pie},
    description={\nopostdesc}
}

\newglossaryentry{rhotexprwtd}
{
    name={\ensuremath{\rho_{t}^{\expr{w},t'}}},
    sort={rhotwtprime},
    description={\nopostdesc}
}

\newglossaryentry{pisubs}
{
    name={\ensuremath{\pi_{s,0}}, \ensuremath{\pi_{s,1}}},
    user1={\ensuremath{\pi_{{\color{red} s},0}}},
    sort={pis},
    description={\nopostdesc}
}
\newglossaryentry{A}
{
    name={\ensuremath{A}},
    sort={A},
    description={\nopostdesc}
}
\newglossaryentry{C}
{
    name={\ensuremath{C}, \ensuremath{C_l}, \ensuremath{C_p}},
    user1={\ensuremath{C}},
    sort={C},
    description={\nopostdesc}
}
\newglossaryentry{len}
{
    name={\ensuremath{\len(x)}, \ensuremath{\len(\expr{x})}},
    user1={\ensuremath{\len}},
    sort={lx},
    description={\nopostdesc}
}
\newglossaryentry{Phif}
{
    type=fake,
    name={\ensuremath{\Phif}},
    user1={\ensuremath{\Phif}},
    sort={Pzhif},
    description={\nopostdesc}
}
\newglossaryentry{Phi}
{
    name={\ensuremath{\Phi}},
    sort={Pzhi},
    description={\nopostdesc}
}
\newglossaryentry{Sigma}
{
    name={\ensuremath{\Sigmaf}, \ensuremath{\Sigma}},
    user1={\ensuremath{\Sigma}},
    sort={Szzigmaf},
    description={\nopostdesc}
}
\newglossaryentry{alphah}
{
    type=fake,
    name={\ensuremath{\alphah}},
    sort={azh},
    description={\nopostdesc}
}
\newglossaryentry{tildes}
{
    name={\ensuremath{\tilde{s}}},
    user1={\ensuremath{\tilde{s}}},
    sort={satilde},
    description={\nopostdesc}
}
\newglossaryentry{sh}
{
    name={\ensuremath{\sh}},
    sort={sh},
    description={\nopostdesc}
}
\newglossaryentry{exprsh}
{
    name={\ensuremath{\expr{\sh}}},
    user1={\ensuremath{\expr{\sh}}},
    sort={shexpr},
    description={\nopostdesc}
}
\newglossaryentry{mst}
{
    name={\ensuremath{m_{st}}},
    sort={mst},
    description={\nopostdesc}
}
\newglossaryentry{tildesl}
{
    name={\ensuremath{\tilde{s}_l}, \ensuremath{\tilde{s}_p}},
    user1={\ensuremath{\tilde{s}_l}},
    sort={satildel},
    description={\nopostdesc}
}
\newglossaryentry{atildesl}
{
    name={\ensuremath{a_{\tilde{s}_l}}, \ensuremath{a_{\tilde{s}_p}}},
    user1={\ensuremath{a_{\tilde{s}_l}}},
    sort={asltildel},
    description={\nopostdesc}
}
\newglossaryentry{Ss}
{
    name={\ensuremath{\Ss}},
    user1={\ensuremath{\Ss}},
    sort={Sz},
    description={\nopostdesc}
}
\newglossaryentry{Wf}
{
    name={\ensuremath{\Wf}},
    user1={\ensuremath{\Wf}},
    sort={Wf},
    description={\nopostdesc}
}
\newglossaryentry{W}
{
    name={\ensuremath{\W}},
    user1={\ensuremath{\W}},
    sort={W},
    description={\nopostdesc}
}
\newglossaryentry{as}
{
    name={\ensuremath{a_s}, \ensuremath{a_s^\vee}},
    user1={\ensuremath{a_s}},
    sort={as},
    description={\nopostdesc}
}
\newglossaryentry{asF}
{
    name={\ensuremath{a_s^F}, \ensuremath{(a_s^F)^\vee}},
    user1={\ensuremath{a_s^F}},
    sort={asF},
    description={\nopostdesc}
}
\newglossaryentry{Va}
{
    type=fake,
    name={\ensuremath{V_{A}}},
    user1={V_{A}},
    description={\nopostdesc}
}
\newglossaryentry{ah}
{
    type=fake,
    name={\ensuremath{a_{\rm h}}, \ensuremath{a_{\rm h}^\vee}},
    user1={\ensuremath{a_{\rm h}}},
    description={\nopostdesc}
}
\newglossaryentry{Vaprime}
{
    type=fake,
    name={\ensuremath{(V'_{A},\{\alpha_s\},\{\alpha_s^\vee\})}},
    user1={\ensuremath{V'_{A}}},
    description={\nopostdesc}
}
\newglossaryentry{c}
{
    name={\ensuremath{c}},
    sort={c},
    description={\nopostdesc}
}
\newglossaryentry{Z}
{
    type=fake,
    name={\ensuremath{Z}},
    description={\nopostdesc}
}
\newglossaryentry{btildes}
{
    type=fake,
    name={\ensuremath{b_{\tilde{s}}}},
    description={\nopostdesc}
}
\newglossaryentry{E}
{
    type=fake,
    name={\ensuremath{E}},
    description={\nopostdesc}
}
\newglossaryentry{F}
{
    name={\ensuremath{F}},
    sort={F},
    description={\nopostdesc}
}
\newglossaryentry{Wl}
{
    name={\ensuremath{\Wl}, \ensuremath{\Wp}},
    user1={\ensuremath{\Wl}},
    sort={WlWp},
    description={\nopostdesc}
}
\newglossaryentry{Ssl}
{
    name={\ensuremath{\Ssl}, \ensuremath{\Ssp}},
    user1={\ensuremath{\Ssl}},
    sort={SzlSzp},
    description={\nopostdesc}
}
\newglossaryentry{Ssf}
{
    name={\ensuremath{\Ssf}},
    user1={\ensuremath{\Ssf}},
    sort={Szf},
    description={\nopostdesc}
}
\newglossaryentry{exprSs}
{
    name={\ensuremath{\expr{\Ss}}, \ensuremath{\Sspexpr}},
    user1={\ensuremath{\expr{\Ss}}},
    sort={Szexpr},
    description={\nopostdesc}
}
\newglossaryentry{linklexpr}
{
    name={\ensuremath{\linklexpr}, \ensuremath{\linkexpr}},
    user1={\ensuremath{\linklexpr}},
    sort={Szponeexpr},
    description={\nopostdesc}
}
\newglossaryentry{Wlcosets}
{
    name={\ensuremath{\Wlcosets}, \ensuremath{\Wpcosets}},
    user1={\ensuremath{\Wlcosets}},
    sort={Wlcosets},
    description={\nopostdesc}
}
\newglossaryentry{VAF}
{
    type=fake,
    name={\ensuremath{(V_{A}^F,a_{F(s)},(a_{F(s)})^\vee)}},
    user1={\ensuremath{V_{A}^F}},
    description={\nopostdesc}
}
\newglossaryentry{exprx}
{
    name={\ensuremath{\expr{x}}},
    user1={\ensuremath{\expr{x}}},
    sort={xexpr},
    description={\nopostdesc}
}
\newglossaryentry{pattr}
{
    name={\ensuremath{\patt{r}}},
    user1={\ensuremath{\patt{r}}},
    sort={rpatt},
    description={\nopostdesc}
}
\newglossaryentry{toralalg}
{
    name={\ensuremath{\toralalg}},
    sort={Up},
    description={\nopostdesc}
}
\newglossaryentry{Wpcosetssast}
{
    name={\ensuremath{\Wpcosets(s,\ast)}},
    sort={Wpcosetssast},
    description={\nopostdesc}
}
\newglossaryentry{toralsubsets}
{
    name={\ensuremath{\toralsubsets}},
    sort={Upsubsets},
    description={\nopostdesc}
}
\newglossaryentry{hasts}
{
    name={\ensuremath{h^{(\ast)}_s}},
    user1={\ensuremath{h^{(\ast)}_s}},
    sort={hsast},
    description={\nopostdesc}
}
\newglossaryentry{hastx}
{
    name={\ensuremath{h^{(\ast)}_{x}}},
    user1={\ensuremath{h^{(\ast)}}},
    sort={hxast},
    description={\nopostdesc}
}
\newglossaryentry{hlast}
{
    name={\ensuremath{h^{(l|\ast)}_{x}}, \ensuremath{h^{(p|\ast)}_{x}}},
    user1={\ensuremath{h^{(l|\ast)}}},
    sort={hxpast},
    description={\nopostdesc}
}
\newglossaryentry{alcs}
{
    type=fake,
    name={\ensuremath{A_0 \in \alcs}},
    user1={\ensuremath{\alcs}},
    description={\nopostdesc}
}
\newglossaryentry{alcsl}
{
    type=fake,
    name={\ensuremath{A_{0,p} \in \alcsp}},
    user1={\ensuremath{\alcsl}},
    description={\nopostdesc}
}
\newglossaryentry{alcslfun}
{
    type=fake,
    name={\ensuremath{\alcslfun}, \ensuremath{\alcspfun}},
    user1={\ensuremath{\alcslfun}},
    description={\nopostdesc}
}
\newglossaryentry{hecke}
{
    name={\ensuremath{\hecke}, \ensuremath{\heckel}, \ensuremath{\heckep}},
    user1={\ensuremath{\hecke}},
    sort={Hz},
    description={\nopostdesc}
}
\newglossaryentry{hx}
{
    name={\ensuremath{h_x}, \ensuremath{h^{(p)}_x}},
    user1={\ensuremath{h_x}},
    sort={hx},
    description={\nopostdesc}
}
\newglossaryentry{bs}
{
    name={\ensuremath{b_s}, \ensuremath{b_s^{(p)}}},
    user1={\ensuremath{b_s}},
    sort={bs},
    description={\nopostdesc}
}
\newglossaryentry{bexprx}
{
    name={\ensuremath{b_{\expr{x}}}, \ensuremath{b_{\expr{x}}^{(p)}}},
    user1={\ensuremath{b_{\expr{x}}}},
    sort={bxaexpr},
    description={\nopostdesc}
}
\newglossaryentry{basts}
{
    name={\ensuremath{b^{(\ast)}_s}},
    user1={\ensuremath{b^{(\ast)}_s}},
    sort={bsast},
    description={\nopostdesc}
}
\newglossaryentry{bexprxast}
{
    name={\ensuremath{b^{(\ast)}_{\expr{x}}}},
    user1={\ensuremath{b^{(\ast)}_{\expr{x}}}},
    description={\nopostdesc},
    sort={bxaexprast}
}
\newglossaryentry{seqe}
{
    name={\ensuremath{\seq{e}}},
    plural={\ensuremath{\seq{c}}},
    user1={\ensuremath{\seq{e}}},
    sort={e},
    description={\nopostdesc}
}
\newglossaryentry{bexprxb}
{
    name={\ensuremath{[\expr{x}]}},
    user1={\ensuremath{[\expr{x}]}},
    sort={xexprbrack},
    description={\nopostdesc}
}
\newglossaryentry{bpattrb}
{
    name={\ensuremath{[\patt{r}]}},
    user1={\ensuremath{\patt{r}}},
    sort={rpattbrack},
    description={\nopostdesc}
}
\newglossaryentry{bexprxbast}
{
    name={\ensuremath{[\expr{x}]_{\ast}(w)}},
    user1={\ensuremath{[\expr{x}]_{\ast}}},
    sort={xexprbrackast},
    description={\nopostdesc}
}
\newglossaryentry{bexprxzlast}
{
    name={\ensuremath{[\expr{x}|\expr{w}]_{l|\ast}}, \ensuremath{[\expr{x}|\expr{w}]_{p|\ast}}},
    user1={\ensuremath{[\expr{x}|\expr{z}]_{l|\ast}}},
    sort={xwexprbracklast},
    description={\nopostdesc}
}
\newglossaryentry{dfct}
{
    name={\ensuremath{\dfct(\seq{e})}},
    user1={\ensuremath{\dfct}},
    sort={def},
    description={\nopostdesc}
}
\newglossaryentry{dfctlw}
{
    name={\ensuremath{\dfct_l^w(\seq{e})}, \ensuremath{\dfct_p^w(\seq{e})}},
    user1={\ensuremath{\dfct}},
    sort={deflw},
    description={\nopostdesc}
}
\newglossaryentry{dfctl}
{
    name={\ensuremath{\dfct_l(\seq{e})}, \ensuremath{\dfct_p(\seq{e})}},
    user1={\ensuremath{\dfct}},
    sort={defl},
    description={\nopostdesc}
}
\newglossaryentry{uz}
{
    name={\ensuremath{u_w}},
    user1={\ensuremath{u_z}},
    sort={uaw},
    description={\nopostdesc}
}
\newglossaryentry{heckelast}
{
    name={\ensuremath{\heckelast}, \ensuremath{\heckepast}},
    user1={\ensuremath{\heckelast}},
    sort={Hzlast},
    description={\nopostdesc}
}
\newglossaryentry{xbarw}
{
    name={\ensuremath{x|w}},
    plural={\ensuremath{x|w}},
    sort={xbarw},
    description={\nopostdesc}
}
\newglossaryentry{exprxbarexprw}
{
    name={\ensuremath{\expr{x}|\expr{w}}},
    user1={\ensuremath{\expr{x}|\expr{w}}},
    sort={xexprbarwwexpr},
    description={\nopostdesc}
}
\newglossaryentry{blastexprxbarexprw}
{
    name={\ensuremath{b^{(l|\ast)}_{\expr{x}|\expr{w}}}, \ensuremath{b^{(p|\ast)}_{\expr{x}|\expr{w}}}},
    plural={\ensuremath{b}},
    user1={\ensuremath{b^{(p|\ast)}_{\expr{x}|\expr{w}}}},
    sort={bxbarwpast},
    description={\nopostdesc}
}
\newglossaryentry{heckeast}
{
    name={\ensuremath{\heckeast}},
    sort={Hzast},
    description={\nopostdesc}
}
\newglossaryentry{V}
{
    name={\ensuremath{V}},
    user1={\ensuremath{V}},
    sort={Vz},
    description={\nopostdesc}
}
\newglossaryentry{VF}
{
    name={\ensuremath{V^F}},
    sort={VzF},
    description={\nopostdesc}
}
\newglossaryentry{R}
{
    name={\ensuremath{R}},
    sort={Rz},
    description={\nopostdesc},
}
\newglossaryentry{deg}
{
    name={\ensuremath{\deg f}},
    user1={\ensuremath{\deg}},
    sort={deg},
    description={\nopostdesc}
}
\newglossaryentry{val}
{
    name={\ensuremath{\val f}},
    user1={\ensuremath{\val}},
    sort={val},
    description={\nopostdesc}
}
\newglossaryentry{hatR}
{
    name={\ensuremath{\hat{R}}},
    sort={Rzhat},
    description={\nopostdesc}
}
\newglossaryentry{dgrm}
{
    name={\ensuremath{\dgrm}},
    sort={Dz},
    description={\nopostdesc}
}
\newglossaryentry{dgrmBSdeg}
{
    name={\ensuremath{\dgrmBSdeg}, \ensuremath{\dgrmdeg}},
    user1={\ensuremath{\dgrmBSdeg}},
    sort={DzBSdeg},
    description={\nopostdesc}
}
\newglossaryentry{dgrmBSdegoplus}
{
    name={\ensuremath{\dgrm_{\rm BS}^{{\rm deg},\oplus}}},
    user1={\ensuremath{\dgrm_{\rm BS}^{{\rm deg},\oplus}}},
    sort={DzBSdegoplus},
    description={\nopostdesc}
}
\newglossaryentry{dgrmBS}
{
    name={\ensuremath{\dgrmBS}},
    user1={\ensuremath{\dgrmBS}},
    sort={DzBS},
    description={\nopostdesc}
}
\newglossaryentry{Bexprx}
{
    name={\ensuremath{B_{\expr{x}}}, \ensuremath{B_{\expr{x}}^F}},
    user1={\ensuremath{B_{\expr{x}}}},
    sort={Bzxaexpr},
    description={\nopostdesc}
}
\newglossaryentry{Bpattr}
{
    name={\ensuremath{B_{\patt{r}}}},
    user1={\ensuremath{B_{\patt{x}}}},
    sort={Bzrpatt},
    description={\nopostdesc}
}
\newglossaryentry{Bexprxpastp}
{
    name={\ensuremath{B_{\expr{x}}^{(\ast)}}},
    user1={\ensuremath{B_{\expr{x}}^{(\ast)}}},
    sort={Bzxaexprast},
    description={\nopostdesc}
}
\newglossaryentry{pipattr}
{
    name={\ensuremath{\pi_{\patt{r}}}},
    user1={\ensuremath{\pi_{\patt{r}}}},
    sort={pirpatt},
    description={\nopostdesc}
}
\newglossaryentry{BexprxpFp}
{
    name={\ensuremath{B_{\expr{x}}^{(F)}}},
    user1={\ensuremath{B_{\expr{x}}^{(F)}}},
    sort={BzxFaexpr},
    description={\nopostdesc}
}
\newglossaryentry{BxpFp}
{
    name={\ensuremath{B_{x}^{(F)}}},
    user1={\ensuremath{B_{x}^{(F)}}},
    sort={BzxF},
    description={\nopostdesc}
}
\newglossaryentry{Bw}
{
    name={\ensuremath{B_{x}}, \ensuremath{B_{x}^F}},
    user1={\ensuremath{B_{w}}},
    sort={Bzw},
    description={\nopostdesc}
}
\newglossaryentry{dgrmBSF}
{
    name={\ensuremath{\dgrmBSF}, \ensuremath{\dgrmF}, \ensuremath{\dgrm_{\rm std}^F}},
    user1={\ensuremath{\dgrmBSF}},
    sort={DzBSF},
    description={\nopostdesc}
}
\newglossaryentry{hatRdgrmBSFdev}
{
    name={\ensuremath{\hat{R} \otimes \dgrm_{\rm BS}^{F,{\rm dev}}}},
    sort={RzhatDBSFdev},
    description={\nopostdesc}
}
\newglossaryentry{hatRdgrmBSFdevoplus}
{
    name={\ensuremath{\hat{R} \otimes \dgrm_{\rm BS}^{F,{\rm dev},\oplus}}},
    sort={RzhatDBSFdevoplus},
    description={\nopostdesc}
}
\newglossaryentry{dots}
{
    name={\ensuremath{{\rm dot}_s}, \ensuremath{{\rm dot}_s^F}},
    user1={\ensuremath{{\rm dot}_s}},
    sort={dots},
    description={\nopostdesc}
}
\newglossaryentry{dotspFp}
{
    name={\ensuremath{{\rm dot}_s^{(F)}}},
    user1={\ensuremath{{\rm dot}_s^{(F)}}},
    sort={dotsF},
    description={\nopostdesc}
}
\newglossaryentry{forks}
{
    name={\ensuremath{{\rm fork}_s}, \ensuremath{{\rm fork}_s^F}},
    user1={\ensuremath{{\rm fork}_s}},
    sort={forks},
    description={\nopostdesc}
}
\newglossaryentry{forkspFp}
{
    name={\ensuremath{{\rm fork}_s^{(F)}}},
    user1={\ensuremath{{\rm fork}_s^{(F)}}},
    sort={forksF},
    description={\nopostdesc}
}
\newglossaryentry{braidst}
{
    name={\ensuremath{{\rm braid}_{s,t}}, \ensuremath{{\rm braid}_{s,t}^F}},
    user1={\ensuremath{{\rm braid}_{s,t}}},
    sort={braidst},
    description={\nopostdesc}
}
\newglossaryentry{braidstpFp}
{
    name={\ensuremath{{\rm braid}_{s,t}^{(F)}}},
    user1={\ensuremath{{\rm braid}_{s,t}^{(F)}}},
    sort={braidstF},
    description={\nopostdesc}
}
\newglossaryentry{bivalents}
{
    name={\ensuremath{{\rm bivalent}_{s}}},
    user1={\ensuremath{{\rm bivalent}_{s}}},
    sort={bivalents},
    description={\nopostdesc}
}
\newglossaryentry{menorahstildeexprwtd}
{
    name={\ensuremath{{\rm menorah}_{\tilde{s}}^{\expr{w},t'}}},
    user1={\ensuremath{{\rm menorah}_{\tilde{s}}^{\expr{w},t'}}},
    sort={menorahswtprime},
    description={\nopostdesc}
}
\newglossaryentry{stdcups}
{
    name={\ensuremath{{\rm stdcup}_{s}}, \ensuremath{{\rm stdcup}_s^F}},
    user1={\ensuremath{{\rm stdcup}_{s}}},
    sort={stdcups},
    description={\nopostdesc}
}
\newglossaryentry{epsilons}
{
    name={\ensuremath{\epsilon_{s}}, \ensuremath{\epsilon_{\expr{x}}}},
    user1={\ensuremath{\epsilon_{s}}},
    sort={epsilonwexpr},
    description={\nopostdesc}
}
\newglossaryentry{stdcapspFp}
{
    name={\parbox{3.5cm}{\ensuremath{{\rm stdcap}_s^{(F)}}, \ensuremath{{\rm stdcup}_s^{(F)}}, \ensuremath{{\rm stdbraid}_{s,t}^{(F)}}}},
    user1={\ensuremath{{\rm stdcap}_{s}^{(F)}}},
    sort={stdcapspFp},
    description={\nopostdesc}
}
\newglossaryentry{stdcaps}
{
    name={\ensuremath{{\rm stdcap}_s}, \ensuremath{{\rm stdcap}_s^F}},
    user1={\ensuremath{{\rm stdcap}_s}},
    sort={stdcaps},
    description={\nopostdesc}
}
\newglossaryentry{stdbraidst}
{
    name={\ensuremath{{\rm stdbraid}_{s,t}}, \ensuremath{{\rm stdcup}_{s,t}^F}},
    user1={\ensuremath{{\rm stdbraid}_{s,t}}},
    sort={stdbraidst},
    description={\nopostdesc}
}
\newglossaryentry{HomdgrmBSbullet}
{
    name={\ensuremath{\Hom_{\dgrmBS}^{\bullet}}, \ensuremath{\Hom_{\dgrmBSpast}^{\bullet}}},
    plural={\ensuremath{\Hom}},
    user1={\ensuremath{\Hom_{\dgrmBS}^{\bullet}}},
    sort={HomDBS},
    description={\nopostdesc}
}
\newglossaryentry{pip}
{
    name={\ensuremath{(i)}; \ensuremath{\langle i \rangle}},
    plural={\ensuremath{\langle i \rangle}},
    user1={\ensuremath{(i)}},
    sort={i},
    description={\nopostdesc}
}
\newglossaryentry{langlemranglew}
{
    name={\ensuremath{\langle m \rangle_w}},
    sort={mw},
    description={\nopostdesc}
}
\newglossaryentry{LLseqeexprw}
{
    name={\ensuremath{\LL_{\seq{e},\expr{w}}}; \ensuremath{\LL_{[\expr{x}],\expr{w}}}},
    user1={\ensuremath{\LL_{\seq{e},\expr{w}}}},
    sort={LL},
    description={\nopostdesc}
}
\newglossaryentry{LLbexprxbexprw}
{
    type=fake,
    name={\ensuremath{\LL_{[\expr{x}],\expr{w}}}},
    user1={\ensuremath{\LL_{[\expr{x}],\expr{w}}}},
    description={\nopostdesc}
}
\newglossaryentry{bBb}
{
    name={\ensuremath{[B]}},
    sort={Bz},
    description={\nopostdesc}
}
\newglossaryentry{bdgrmb}
{
    name={\ensuremath{[\dgrm]}, \ensuremath{[\dgrmpast]}},
    user1={\ensuremath{[\dgrm]}},
    sort={Dzbrack},
    description={\nopostdesc}
}
\newglossaryentry{charctr}
{
    name={\ensuremath{\charctr}},
    sort={ch},
    description={\nopostdesc}
}
\newglossaryentry{charctrpast}
{
    name={\ensuremath{\charctr_{p|\ast}}},
    sort={chpast},
    description={\nopostdesc}
}
\newglossaryentry{Q}
{
    name={\ensuremath{Q}},
    sort={Q},
    description={\nopostdesc}
}
\newglossaryentry{hatQ}
{
    name={\ensuremath{\hat{Q}}},
    sort={Qahat},
    description={\nopostdesc}
}
\newglossaryentry{Qexprx}
{
    name={\ensuremath{Q_{\expr{x}}}, \ensuremath{Q_{\expr{x}}^F}},
    user1={\ensuremath{Q_{\expr{x}}}},
    sort={Qxexpr},
    description={\nopostdesc}
}
\newglossaryentry{QotimesdgrmtildeBSstd}
{
    name={\ensuremath{Q \otimes \dgrmBSstd}, \ensuremath{Q \otimes \dgrmstd}},
    user1={\ensuremath{Q \otimes \widetilde{\dgrm}_{\rm BS,std}}},
    sort={QDBSstd},
    description={\nopostdesc}
}
\newglossaryentry{Rx}
{
    name={\ensuremath{R_{x}}, \ensuremath{R_{x}^F}},
    user1={\ensuremath{R_x}},
    sort={Rzx},
    description={\nopostdesc}
}
\newglossaryentry{Rexprx}
{
    name={\ensuremath{R_{\expr{x}}}, \ensuremath{\hat{R}_{\expr{x}}}},
    user1={\ensuremath{R_{\expr{x}}}},
    sort={Rzxexpr},
    description={\nopostdesc}
}
\newglossaryentry{RexprxpFp}
{
    name={\ensuremath{R_{\expr{x}}^{(F)}}},
    user1={\ensuremath{R_{\expr{x}}^{(F)}}},
    sort={RxexprpFp},
    description={\nopostdesc}
}
\newglossaryentry{dgrmF}
{
    name={\ensuremath{\dgrmF}},
    sort={DzF},
    description={\nopostdesc}
}
\newglossaryentry{dgrmstd}
{
    name={\ensuremath{\dgrmstd}},
    user1={\ensuremath{\dgrmstd}},
    sort={Dzstd},
    description={\nopostdesc}
}
\newglossaryentry{binomfg}
{
    name={\ensuremath{\binom{f}{g}}},
    sort={fg},
    description={\nopostdesc}
}
\newglossaryentry{pcdotp}
{
    name={\ensuremath{(\cdot)}},
    sort={aa},
    description={\nopostdesc}
}
\newglossaryentry{dgrmBSstd}
{
    name={\ensuremath{\dgrmBSstd}},
    sort={DzBSstd},
    description={\nopostdesc}
}
\newglossaryentry{RdgrmBSstd}
{
    name={\ensuremath{\prescript{}{R}{\dgrm}_{\rm BS,std}}},
    user1={\ensuremath{\prescript{}{R}{\dgrm}_{\rm BS,std}}},
    sort={DzBSstdR},
    description={\nopostdesc}
}
\newglossaryentry{dgrmBSpast}
{
    name={\ensuremath{\dgrmBSpast}},
    sort={DzBSpast},
    description={\nopostdesc}
}
\newglossaryentry{Bppastpexprxexprw}
{
    name={\ensuremath{B^{(p|\ast)}_{\expr{x}|\expr{w}}}, \ensuremath{B^{(p)}_{\expr{x}}}},
    user1={\ensuremath{B^{(p|\ast)}_{\expr{x}|\expr{w}}}},
    sort={Bzxzexprwexprpast},
    description={\nopostdesc}
}
\newglossaryentry{Bpppy}
{
    name={\ensuremath{B^{(p)}_{x}}},
    user1={\ensuremath{B^{(p)}_{y}}},
    sort={Bzxp},
    description={\nopostdesc}
}
\newglossaryentry{Bppastpxbarw}
{
    name={\ensuremath{B^{(p|\ast)}_{x|w}}},
    user1={\ensuremath{B^{(p|\ast)}_{x|w}}},
    sort={Bzxzabarwpast},
    description={\nopostdesc}
}
\newglossaryentry{dgrmpast}
{
    name={\ensuremath{\dgrmpast}},
    sort={Dzpast},
    description={\nopostdesc}
}
\newglossaryentry{dgrmBSpastpoplus}
{
    name={\ensuremath{\dgrm_{\rm BS}^{(\ast),\oplus}}},
    sort={DzBSastoplus},
    description={\nopostdesc}
}
\newglossaryentry{LLdblseqeseqf}
{
    name={\ensuremath{\LLdbl_{\seq{e}}^{\seq{f}}}; \ensuremath{\LLdbl_{[\expr{x}]}^{[\expr{y}]}}},
    user1={\ensuremath{\LLdbl_{\seq{e}}^{\seq{f}}}},
    sort={LLdbl},
    description={\nopostdesc}
}
\newglossaryentry{LLdblbexprxbbexpryb}
{
    type=fake,
    name={\ensuremath{\LLdbl_{[\expr{x}]}^{[\expr{y}]}}},
    user1={\ensuremath{\LLdbl_{[\expr{x}]}^{[\expr{y}]}}},
    description={\nopostdesc}
}
\newglossaryentry{pastLLseqeexprwexprz}
{
    name={\ensuremath{\pastLL_{\seq{e},\expr{w},\expr{z}}}; \ensuremath{\pastLL_{[\expr{x}],\expr{w},\expr{z}}}},
    user1={\ensuremath{\pastLL_{\seq{e},\expr{w},\expr{z}}}},
    sort={LLpast},
    description={\nopostdesc}
}
\newglossaryentry{pastLLdblseqeseqf}
{
    name={\ensuremath{\pastLLdbl_{\seq{e}}^{\seq{f}}}; \ensuremath{\pastLLdbl_{[\expr{x}]}^{[\expr{y}]}}},
    user1={\ensuremath{\pastLLdbl_{\seq{e}}^{\seq{f}}}},
    sort={LLpastdbl},
    description={\nopostdesc}
}
\newglossaryentry{prprime}
{
    name={\ensuremath{\pr'}; \ensuremath{\pr}},
    user1={\ensuremath{\pr'}},
    sort={Pzsi},
    description={\nopostdesc}
}
\newglossaryentry{MatWpcosetshatRdgrmBSFdevoplus}
{
    name={\ensuremath{\Mat_{\Wpcosets}}(-)},
    user1={\ensuremath{\Mat_{\Wpcosets}(\hat{R} \otimes \dgrm_{\rm BS}^{F,{\rm dev},\oplus})}},
    sort={MWpDBSF},
    description={\nopostdesc}
}
\newglossaryentry{circWpcosets}
{
    name={\ensuremath{\circ_{\Wpcosets}}},
    sort={Wpcosetscirc},
    description={\nopostdesc}
}
\newglossaryentry{psi}
{
    name={\ensuremath{\psi}},
    sort={psi},
    description={\nopostdesc}
}
\newglossaryentry{voneheckeast}
{
    name={\ensuremath{\prescript{}{u,v=1}{\hecke}_{\ast}}, \ensuremath{\prescript{}{v=1}{\hecke}_{p|\ast}}},
    user1={\ensuremath{\prescript{}{u,v=1}{\hecke}_{\ast}}},
    sort={Hzzastvone},
    description={\nopostdesc}
}
\newglossaryentry{vonebpastps}
{
    name={\ensuremath{\prescript{}{u,v=1}{b}^{(\ast)}_s}, \ensuremath{\prescript{}{v=1}{h}^{(p|\ast)}_{x|w}}},
    user1={\ensuremath{\prescript{}{u,v=1}{b}^{(\ast)}_s}},
    sort={bxzastvone},
    type={fake},
    description={\nopostdesc}
}
\newglossaryentry{pbx}
{
    name={\ensuremath{\prescript{p}{}{b}_x}},
    user1={\ensuremath{\prescript{p}{}{b}_x}},
    sort={bxap},
    description={\nopostdesc}
}

%% file: fig/mixedrelations.tex
\begin{equation}
\label{eq:stdBSstd}
\begin{gathered}
    \begin{tikzpicture}[baseline=(origin)]
        \coordinate (origin) at (0,0);
        \node[circle, minimum size=0.5cm] (spacer) at (origin) {};
        \path (origin) ++(0,0.875) coordinate (cod);
        \path (origin) ++(0,-0.875) coordinate (dom);
        \path (origin) ++(0,0.375) coordinate (id);
        \path (origin) ++(0,-0.375) coordinate (di);
        \draw[red, std] (dom) to (di)
                        (id) to (cod);
        \draw[red,string] (id) to (di);
    \end{tikzpicture}
\end{gathered}=\begin{gathered}
    \begin{tikzpicture}[baseline=(origin)]
        \coordinate (origin) at (0,0);
        \node[circle, minimum size=0.8cm] (spacer) at (0.15,0) {};
        \path (origin) ++(0,0.875) coordinate (cod);
        \path (origin) ++(0,-0.875) coordinate (dom);
        \path (origin) ++(0,0.375) coordinate (id);
        \path (origin) ++(0,-0.375) coordinate (di);
        \draw[red, std] (dom) to (cod);
        \node[dot,red,fill=red] (dot1) at (0.3,0.3) {};
        \node[dot,red,fill=red] (dot2) at (0.3,-0.3) {};
        \draw[red,string] (dot1.center) to (dot2.center);
    \end{tikzpicture}
\end{gathered} \text{,}
\end{equation}
\begin{equation}
\label{eq:BSstdBS}
\begin{gathered}
    \begin{tikzpicture}[baseline=(origin)]
        \coordinate (origin) at (0,0);
        \node[circle, minimum size=0.5cm] (spacer) at (flex) {};
        \path (origin) ++(0,0.875) coordinate (cod);
        \path (origin) ++(0,-0.875) coordinate (dom);
        \path (origin) ++(0,0.375) coordinate (id);
        \path (origin) ++(0,-0.375) coordinate (di);
        \draw[red, string] (dom) to (di)
                        (id) to (cod);
        \draw[red,std] (id) to (di);
    \end{tikzpicture}
\end{gathered}=\begin{gathered}
    \begin{tikzpicture}[baseline=(origin)]
        \coordinate (origin) at (0,0);
        \node[circle, minimum size=0.8cm] (spacer) at (0.15,0) {};
        \path (origin) ++(0,0.875) coordinate (cod);
        \path (origin) ++(0,-0.875) coordinate (dom);
        \path (origin) ++(0,0.375) coordinate (id);
        \path (origin) ++(0,-0.375) coordinate (di);
        \draw[red, string] (dom) to (cod);
        \node[dot,red,fill=red] (dot1) at (0.3,0.3) {};
        \node[dot,red,fill=red] (dot2) at (0.3,-0.3) {};
        \draw[red,string] (dot1.center) to (dot2.center);
    \end{tikzpicture}
\end{gathered}- \begin{gathered}
    \begin{tikzpicture}[baseline=(spacer)]
        \node[circle, minimum size=0.5cm] (spacer) at (0,0) {};
        \coordinate (above) at (0,0.875);
        \coordinate (below) at (0,-0.875);
        \node[dot, fill=red] (updot) at (0,0.375) {};
        \node[dot, fill=red] (downdot) at (0,-0.375) {};
%
        \draw[red, string] (above) to (updot.center)
                           (below) to (downdot.center);
    \end{tikzpicture}
\end{gathered}\text{.}
\end{equation}

%% file: fig/braidpropagate.tex
\begin{equation*}
\begin{gathered}
\begin{tikzpicture}[xscale=0.5,yscale=0.66,baseline=(origin.center),font=\footnotesize]
\coordinate (origin) at (0,0) {};
\path (origin) ++(-2,-1) coordinate (dom1) {} ++(0,-0.25) coordinate (below1) {} ++(0,-0.42) coordinate (stddom1) {};
\path (origin) ++(-1.25,-1) coordinate (dom2) {} ++(0,-0.25) coordinate (below2) {} ++(0,-0.42) coordinate (stddom2) {};
\path (origin) ++(-0.5,-1) coordinate (dom3) {} ++(0,-0.25) coordinate (below3) {} ++(0,-0.42) coordinate (stddom3) {};
\path (origin) ++(2,-1) coordinate (domlast) {} ++(0,-0.25) coordinate (belowlast) {} ++(0,-0.42) coordinate (stddomlast) {};
\path (origin) ++(-2,1) coordinate (cod1) {} ++(0,0.25) coordinate (above1) {} node[right=-2pt] {$\projscal{blue}$} ++(0,0.42) coordinate (stdcod1) {};
\path (origin) ++(-1.25,1) coordinate (cod2) {} ++(0,0.25) coordinate (above2) {} node[right=-2pt] {$\projscal{red}$} ++(0,0.42) coordinate (stdcod2) {};
\path (origin) ++(-0.5,1) coordinate (cod3) {} ++(0,0.25) coordinate (above3) {} node[right=-2pt] {$\projscal{blue}$} ++(0,0.42) coordinate (stdcod3) {};
\path (origin) ++(2,1) coordinate (codlast) {} ++(0,0.25) coordinate (abovelast) {} node[right=-2pt] {$\projscal{blue}$} ++(0,0.42) coordinate (stdcodlast) {};
\draw[red, corner, line width=\superthick] (origin) to (dom1) to (below1)
                                   (origin) to (dom3) to (below3)
                                   (origin) to (domlast) to (belowlast)
                                   (origin) to (cod2) to (above2);
\draw[red, line width=\superthick] (below1) to (stddom1)
                                   (below3) to (stddom3)
                                   (belowlast) to (stddomlast);
\draw[red, line width=\superthick, std]  (above2) to (stdcod2);
\draw[blue, corner, line width=\superthick] (origin) to (dom2) to (below2)
                                            (origin) to (cod1) to (above1)
                                            (origin) to (cod3) to (above3)
                                            (origin) to (codlast) to (abovelast);
\draw[blue, line width=\superthick] (below2) to (stddom2);
\draw[blue, line width=\superthick,std]  (above1) to (stdcod1)
                                         (above3) to (stdcod3)
                                         (abovelast) to (stdcodlast);
\path (origin) ++(0.75,-1.2) node[font=\large] (domdots) {$\dotsm$};
\path (origin) ++(0.75,1.2) node[font=\large] (coddots) {$\dotsm$};
\end{tikzpicture}
\end{gathered}
=
\begin{gathered}
\begin{tikzpicture}[xscale=0.5,yscale=0.66,baseline=(origin.center),font=\footnotesize]
\coordinate (origin) at (0,0) {};
\path (origin) ++(-2,-1) coordinate (dom1) {} ++(0,-0.25) coordinate (below1) {} node[right=-2pt] {$\projscal{red}$} ++(0,-0.42) coordinate (stddom1) {};
\path (origin) ++(-1.25,-1) coordinate (dom2) {} ++(0,-0.25) coordinate (below2) {} node[right=-2pt] {$\projscal{blue}$} ++(0,-0.42) coordinate (stddom2) {};
\path (origin) ++(-0.5,-1) coordinate (dom3) {} ++(0,-0.25) coordinate (below3) {} node[right=-2pt] {$\projscal{red}$}++(0,-0.42) coordinate (stddom3) {};
\path (origin) ++(2,-1) coordinate (domlast) {} ++(0,-0.25) coordinate (belowlast) {} node[right=-2pt] {$\projscal{red}$} ++(0,-0.42) coordinate (stddomlast) {};
\path (origin) ++(-2,1) coordinate (cod1) {} ++(0,0.25) coordinate (above1) {} ++(0,0.42) coordinate (stdcod1) {};
\path (origin) ++(-1.25,1) coordinate (cod2) {} ++(0,0.25) coordinate (above2) {} ++(0,0.42) coordinate (stdcod2) {};
\path (origin) ++(-0.5,1) coordinate (cod3) {} ++(0,0.25) coordinate (above3) {} ++(0,0.42) coordinate (stdcod3) {};
\path (origin) ++(2,1) coordinate (codlast) {} ++(0,0.25) coordinate (abovelast) {} ++(0,0.42) coordinate (stdcodlast) {};
\draw[red, corner, line width=\superthick, std] (origin) to (dom1) to (below1)
                                   (origin) to (dom3) to (below3)
                                   (origin) to (domlast) to (belowlast)
                                   (origin) to (cod2) to (above2);
\draw[red, line width=\superthick] (below1) to (stddom1)
                                   (below3) to (stddom3)
                                   (belowlast) to (stddomlast);
\draw[red, line width=\superthick, std]  (above2) to (stdcod2);
\draw[blue, corner, line width=\superthick,std] (origin) to (dom2) to (below2)
                                            (origin) to (cod1) to (above1)
                                            (origin) to (cod3) to (above3)
                                            (origin) to (codlast) to (abovelast);
\draw[blue, line width=\superthick] (below2) to (stddom2);
\draw[blue, line width=\superthick,std]  (above1) to (stdcod1)
                                         (above3) to (stdcod3)
                                         (abovelast) to (stdcodlast);
\path (origin) ++(0.75,-1.2) node[font=\large] (domdots) {$\dotsm$};
\path (origin) ++(0.75,1.2) node[font=\large] (coddots) {$\dotsm$};
\end{tikzpicture}
\end{gathered} \text{,} \qquad \qquad 
\begin{gathered}
\begin{tikzpicture}[xscale=0.5,yscale=0.66,baseline=(origin.center),font=\footnotesize]
\coordinate (origin) at (0,0) {};
\path (origin) ++(-2,-1) coordinate (dom1) {} ++(0,-0.25) coordinate (below1) {} ++(0,-0.42) coordinate (stddom1) {};
\path (origin) ++(-1.25,-1) coordinate (dom2) {} ++(0,-0.25) coordinate (below2) {} ++(0,-0.42) coordinate (stddom2) {};
\path (origin) ++(-0.5,-1) coordinate (dom3) {} ++(0,-0.25) coordinate (below3) {} ++(0,-0.42) coordinate (stddom3) {};
\path (origin) ++(2,-1) coordinate (domlast) {} ++(0,-0.25) coordinate (belowlast) {} ++(0,-0.42) coordinate (stddomlast) {};
\path (origin) ++(-2,1) coordinate (cod1) {} ++(0,0.25) coordinate (above1) {} node[right=-2pt] {$\projscal{realcyan}$} ++(0,0.42) coordinate (stdcod1) {};
\path (origin) ++(-1.25,1) coordinate (cod2) {} ++(0,0.25) coordinate (above2) {} node[right=-2pt] {$\projscal{red}$} ++(0,0.42) coordinate (stdcod2) {};
\path (origin) ++(-0.5,1) coordinate (cod3) {} ++(0,0.25) coordinate (above3) {} node[right=-2pt] {$\projscal{realcyan}$} ++(0,0.42) coordinate (stdcod3) {};
\path (origin) ++(2,1) coordinate (codlast) {} ++(0,0.25) coordinate (abovelast) {} node[right=-2pt] {$\projscal{realcyan}$} ++(0,0.42) coordinate (stdcodlast) {};
\draw[red, corner, line width=\superthick] (origin) to (dom1) to (below1)
                                   (origin) to (dom3) to (below3)
                                   (origin) to (domlast) to (belowlast)
                                   (origin) to (cod2) to (above2);
\draw[red, line width=\superthick] (below1) to (stddom1)
                                   (below3) to (stddom3)
                                   (belowlast) to (stddomlast);
\draw[red, line width=\superthick, std]  (above2) to (stdcod2);
\draw[realcyan, corner, line width=\superthick] (origin) to (dom2) to (below2)
                                            (origin) to (cod1) to (above1)
                                            (origin) to (cod3) to (above3)
                                            (origin) to (codlast) to (abovelast);
\draw[realcyan, line width=\superthick] (below2) to (stddom2);
\draw[realcyan, line width=\superthick,std]  (above1) to (stdcod1)
                                         (above3) to (stdcod3)
                                         (abovelast) to (stdcodlast);
\path (origin) ++(0.75,-1.2) node[font=\large] (domdots) {$\dotsm$};
\path (origin) ++(0.75,1.2) node[font=\large] (coddots) {$\dotsm$};
\end{tikzpicture}
\end{gathered}
=
\begin{gathered}
\begin{tikzpicture}[xscale=0.5,yscale=0.66,baseline=(origin.center),font=\footnotesize]
\coordinate (origin) at (0,0) {};
\path (origin) ++(-2,-1) coordinate (dom1) {} ++(0,-0.25) coordinate (below1) {} node[right=-2pt] {$\projscal{red}$} ++(0,-0.42) coordinate (stddom1) {};
\path (origin) ++(-1.25,-1) coordinate (dom2) {} ++(0,-0.25) coordinate (below2) {} node[right=-2pt] {$\projscal{realcyan}$} ++(0,-0.42) coordinate (stddom2) {};
\path (origin) ++(-0.5,-1) coordinate (dom3) {} ++(0,-0.25) coordinate (below3) {} node[right=-2pt] {$\projscal{red}$}++(0,-0.42) coordinate (stddom3) {};
\path (origin) ++(2,-1) coordinate (domlast) {} ++(0,-0.25) coordinate (belowlast) {} node[right=-2pt] {$\projscal{red}$} ++(0,-0.42) coordinate (stddomlast) {};
\path (origin) ++(-2,1) coordinate (cod1) {} ++(0,0.25) coordinate (above1) {} ++(0,0.42) coordinate (stdcod1) {};
\path (origin) ++(-1.25,1) coordinate (cod2) {} ++(0,0.25) coordinate (above2) {} ++(0,0.42) coordinate (stdcod2) {};
\path (origin) ++(-0.5,1) coordinate (cod3) {} ++(0,0.25) coordinate (above3) {} ++(0,0.42) coordinate (stdcod3) {};
\path (origin) ++(2,1) coordinate (codlast) {} ++(0,0.25) coordinate (abovelast) {} ++(0,0.42) coordinate (stdcodlast) {};
\draw[red, corner, line width=\superthick, std] (origin) to (dom1) to (below1)
                                   (origin) to (dom3) to (below3)
                                   (origin) to (domlast) to (belowlast)
                                   (origin) to (cod2) to (above2);
\draw[red, line width=\superthick] (below1) to (stddom1)
                                   (below3) to (stddom3)
                                   (belowlast) to (stddomlast);
\draw[red, line width=\superthick, std]  (above2) to (stdcod2);
\draw[realcyan, corner, line width=\superthick,std] (origin) to (dom2) to (below2)
                                            (origin) to (cod1) to (above1)
                                            (origin) to (cod3) to (above3)
                                            (origin) to (codlast) to (abovelast);
\draw[realcyan, line width=\superthick] (below2) to (stddom2);
\draw[realcyan, line width=\superthick,std]  (above1) to (stdcod1)
                                         (above3) to (stdcod3)
                                         (abovelast) to (stdcodlast);
\path (origin) ++(0.75,-1.2) node[font=\large] (domdots) {$\dotsm$};
\path (origin) ++(0.75,1.2) node[font=\large] (coddots) {$\dotsm$};
\end{tikzpicture}
\end{gathered} \text{.}
\end{equation*}

%% file: fig/LLU1propagate.tex
\begin{equation*}
\begin{tikzpicture}[scale=0.4,baseline=(btwnc2.center)] 
\def\strings{6}
\pgfmathsetmacro{\stringgap}{1.0/\strings}
\node (mrex) at (0,0) [draw, rectangle, minimum width=2.4cm, minimum height=0.6cm] {mrex};
\path (mrex) ++(0,4.5cm) node (std) [draw, rectangle, minimum width=2.4cm, minimum height=0.6cm] {std};
\foreach \i in {1,...,\strings}
{
    \pgfmathsetmacro{\xoffset}{(\i-0.5)*\stringgap}
    \path ($(std.north west)!\xoffset!(std.north east)$) node (stdb\i) {};
    \path (stdb\i) ++(0,1cm) node (stda\i) {};
    \path ($(std.south west)!\xoffset!(std.south east)$) node (btwna\i) {};
    \path ($(mrex.north west)!\xoffset!(mrex.north east)$) node (btwnd\i) {};
    \path ($(btwna\i.center)!0.33!(btwnd\i.center)$) node (btwnb\i) {};
    \path ($(btwna\i.center)!0.67!(btwnd\i.center)$) node (btwnc\i) {};
    \path ($(mrex.south west)!\xoffset!(mrex.south east)$) node (mrexa\i) {};
    \path (mrexa\i) ++(0,-1cm) node (mrexb\i) {}
                    ++(0,-1cm) node (mrexc\i) {};
}
\foreach \i in {1,...,3}
{
    \pgfmathsetmacro{\xoffset}{(\i-0.5)*\stringgap}
    \draw (mrexa\i.center) edge[string] (mrexc\i.center);
}
\foreach \i in {4,...,5}
{
    \pgfmathsetmacro{\xoffset}{(\i-0.5)*\stringgap}
    \draw (mrexa\i.center) edge[string] (mrexb\i.center) 
          (mrexb\i.center) edge[std] (mrexc\i.center);
    \draw (mrexb\i.center) node[right=-2pt] {\footnotesize $\inclscal{black}$};
}
\draw (mrexa6.center) edge[string,red] (mrexc6.center);
\foreach \i in {1,3,4}
{
    \pgfmathsetmacro{\xoffset}{(\i-0.5)*\stringgap}
    \draw (btwna\i.center) edge[std] (btwnb\i.center)
          (btwnb\i.center) edge[string] (btwnd\i.center);
    \draw (btwnb\i.center) node[right=-2pt] {\footnotesize $\projscal{black}$};
}
\draw (btwna2.center) edge[std,red] (btwnb2.center)
      (btwnb2.center) edge[string,red] (btwnd2.center);
\draw (btwnb2.center) node[right=-2pt] {\footnotesize $\projscal{red}$};
\foreach \i in {5,6}
{
    \pgfmathsetmacro{\xoffset}{(\i-0.5)*\stringgap}
    \draw (btwna\i.center) edge[std] (btwnc\i.center)
          (btwnc\i.center) edge[string] (btwnd\i.center);
    \draw (btwnc\i.center) node[right=-2pt] {\footnotesize $\projscal{black}$};
}
\foreach \i in {3,...,5}
{
    \pgfmathsetmacro{\xoffset}{(\i-0.5)*\stringgap}
    \draw (stda\i.center) edge[std] (stdb\i.center);
}
\draw (stda2.center) edge[std,red] (stdb2.center);
\end{tikzpicture}\sim 
\begin{tikzpicture}[scale=0.4,baseline=(std)] 
\def\strings{6}
\pgfmathsetmacro{\stringgap}{1.0/\strings}
\node (std) at (0,0) [draw, rectangle, minimum width=2.4cm, minimum height=0.6cm] {std};
\foreach \i in {1,...,\strings}
{
    \pgfmathsetmacro{\xoffset}{(\i-0.5)*\stringgap}
    \path ($(std.north west)!\xoffset!(std.north east)$) node (topb\i) {};
    \path (topb\i) ++(0,1cm) node (topa\i) {};
    \path ($(std.south west)!\xoffset!(std.south east)$) node (bota\i) {};
    \path (bota\i) ++(0,-1cm) node (botb\i) {}
                   ++(0,-1cm) node (botc\i) {};
}
\foreach \i in {1,...,3}
{
    \pgfmathsetmacro{\xoffset}{(\i-0.5)*\stringgap}
    \draw (bota\i.center) edge[std] (botb\i.center)
          (botb\i.center) edge[string] (botc\i.center);
    \draw (botb\i.center) node[right=-2pt] {\footnotesize $\projscal{black}$};
}
\foreach \i in {4,...,5}
{
    \pgfmathsetmacro{\xoffset}{(\i-0.5)*\stringgap}
    \draw (bota\i.center) edge[std] (botc\i.center);
}
\draw (bota6.center) edge[std,red] (botb6.center)
      (botb6.center) edge[string,red] (botc6.center);
\draw (botb6.center) node[right=-2pt] {\footnotesize $\projscal{red}$};
\foreach \i in {3,...,5}
{
    \pgfmathsetmacro{\xoffset}{(\i-0.5)*\stringgap}
    \draw (topa\i.center) edge[std] (topb\i.center);
}
\draw (topa2.center) edge[std,red] (topb2.center);
\end{tikzpicture}
\sim
\begin{tikzpicture}[scale=0.4,baseline=(btwnb2.center)] 
\node (std) at (0,0) [draw, rectangle, minimum width=2.0cm, minimum height=0.6cm] {std};
\path (std) ++(1.5cm,3.5cm) node (rex)  [draw, rectangle, minimum width=1.6cm, minimum height=0.6cm] {rex};
\pgfmathsetmacro{\stringgap}{1.0/4.0}
\foreach \i in {1,...,4}
{
    \pgfmathsetmacro{\xoffset}{(\i-0.5)*\stringgap}
    \path ($(rex.north west)!\xoffset!(rex.north east)$) node (rexc\i) {};
    \path (rexc\i) ++(0,1cm) node (rexb\i) {}
                   ++(0,1cm) node (rexa\i) {};
    \path ($(rex.south west)!\xoffset!(rex.south east)$) node (btwna\i) {};
    \pgfmathsetmacro{\xoffset}{0.4+(\i-0.5)*\stringgap*0.8}
    \path ($(std.north west)!\xoffset!(std.north east)$) node (btwnc\i) {};
    \path ($(btwna\i.center)!0.5!(btwnc\i.center)$) node (btwnb\i) {};
}
\path (btwnc4) ++(0,-3.5cm) node (btwnd4) {};
\pgfmathsetmacro{\stringgap}{1.0/5.0}
\foreach \i in {1,...,5}
{
    \pgfmathsetmacro{\xoffset}{(\i-0.5)*\stringgap}
    \path ($(std.south west)!\xoffset!(std.south east)$) node (stda\i) {};
    \path (stda\i) ++(0,-1cm) node (stdb\i) {}
                   ++(0,-1cm) node (stdc\i) {};
}
\foreach \i in {1,...,3}
{
    \pgfmathsetmacro{\xoffset}{(\i-0.5)*\stringgap}
    \draw (stda\i.center) edge[std] (stdb\i.center) 
          (stdb\i.center) edge[string] (stdc\i.center);
    \draw (stdb\i.center) node[right=-2pt] {\footnotesize $\projscal{black}$};
}
\foreach \i in {4,...,5}
{
    \pgfmathsetmacro{\xoffset}{(\i-0.5)*\stringgap}
    \draw (stda\i.center) edge[std] (stdc\i.center);
}
\pgfmathsetmacro{\stringgap}{1.0/4.0}
\foreach \i in {1,...,3}
{
    \pgfmathsetmacro{\xoffset}{(\i-0.5)*\stringgap}
    \draw (btwna\i.center) edge[string] (btwnb\i.center)
          (btwnb\i.center) edge[std] (btwnc\i.center);
    \draw (btwnb\i.center) node[right=-2pt] {\footnotesize $\inclscal{black}$};
}
\draw (btwna4.center) edge[string,red] (btwnd4.center);
\foreach \i in {2,...,4}
{
    \pgfmathsetmacro{\xoffset}{(\i-0.5)*\stringgap}
    \draw (rexa\i.center) edge[std] (rexb\i.center)
          (rexb\i.center) edge[string] (rexc\i.center);
    \draw (rexb\i.center) node[right=-2pt] {\footnotesize $\projscal{black}$};
}
\draw (rexa1.center) edge[std,red] (rexb1.center)
      (rexb1.center) edge[string,red] (rexc1.center);
\draw (rexb1.center) node[right=-2pt] {\footnotesize $\projscal{red}$};
\end{tikzpicture} \text{,}
\end{equation*}
where $\sim$ denotes equality up left multiplication by an invertible scalar in $\hat{R}$.

%% file: fig/LLD0propagate.tex
\begingroup
\allowdisplaybreaks
\begin{align*}
\begin{tikzpicture}[scale=0.39,baseline=(mrextop)] 
\def\strings{6}
\pgfmathsetmacro{\stringgap}{1.0/\strings}
\node (mrexbot) at (0,0) [draw, rectangle, minimum width=2.34cm, minimum height=0.585cm] {mrex};
\node (mrextop) at (0,2.5cm) [draw, rectangle, minimum width=2.34cm, minimum height=0.585cm] {mrex};
\node (std) at (0,7cm) [draw, rectangle, minimum width=2.34cm, minimum height=0.585cm] {std};
\foreach \i in {1,...,\strings}
{
    \pgfmathsetmacro{\xoffset}{(\i-0.5)*\stringgap}
    \path ($(std.north west)!\xoffset!(std.north east)$) node (stdb\i) {};
    \path (stdb\i) ++(0,1cm) node (stda\i) {};
    \path ($(std.south west)!\xoffset!(std.south east)$) node (btwntopa\i) {};
    \path (btwntopa\i) ++(0,-1cm) node (btwntopb\i) {}
                       ++(0,-1cm) node (btwntopc\i) {}
                       ++(0,-1cm) node (btwntopd\i) {};
    \path ($(mrextop.south west)!\xoffset!(mrextop.south east)$) node (btwnbota\i) {};
    \path (btwnbota\i) ++(0,-1cm) node (btwnbotb\i) {};
    \path ($(mrexbot.south west)!\xoffset!(mrexbot.south east)$) node (mrexbota\i) {};
    \path (mrexbota\i) ++(0,-1cm) node (mrexbotb\i) {}
                       ++(0,-1cm) node (mrexbotc\i) {};
}
\path (mrexbotc6) ++(1cm,0) coordinate (mrexbotc7);
\foreach \i in {1,3}
{
    \pgfmathsetmacro{\xoffset}{(\i-0.5)*\stringgap}
    \draw (mrexbota\i.center) edge[string] (mrexbotc\i.center);
}
\draw (mrexbota2.center) edge[string,red] (mrexbotc2.center);
\foreach \i in {4,...,6}
{
    \pgfmathsetmacro{\xoffset}{(\i-0.5)*\stringgap}
    \draw (mrexbota\i.center) edge[string] (mrexbotb\i.center) 
          (mrexbotb\i.center) edge[std] (mrexbotc\i.center);
    \draw (mrexbotb\i.center) node[right=-2pt] {\footnotesize $\inclscal{black}$};
}
\foreach \i in {1,...,5}
{
    \pgfmathsetmacro{\xoffset}{(\i-0.5)*\stringgap}
    \draw (btwnbota\i.center) edge[string] (btwnbotb\i.center);
}
\draw (btwnbota6.center) edge[string,red] (btwnbotb6.center);
\draw[string, red, corner] (mrexbotc7) |- ($(btwnbota6.center)!0.5!(btwnbotb6.center)$);
\foreach \i in {1,3}
{
    \pgfmathsetmacro{\xoffset}{(\i-0.5)*\stringgap}
    \draw (btwntopa\i.center) edge[std] (btwntopb\i.center)
          (btwntopb\i.center) edge[string] (btwntopd\i.center);
    \draw (btwntopb\i.center) node[right=-2pt] {\footnotesize $\projscal{black}$};
}
\draw (btwntopa2.center) edge[std,red] (btwntopb2.center)
      (btwntopb2.center) edge[string,red] (btwntopd2.center);
      \draw (btwntopb2.center) node[right=-2pt] {\footnotesize $\projscal{red}$};
\foreach \i in {4,...,6}
{
    \pgfmathsetmacro{\xoffset}{(\i-0.5)*\stringgap}
    \draw (btwntopa\i.center) edge[std] (btwntopc\i.center)
          (btwntopc\i.center) edge[string] (btwntopd\i.center);
    \draw (btwntopc\i.center) node[right=-2pt] {\footnotesize $\projscal{black}$};
}
\foreach \i in {3,...,5}
{
    \pgfmathsetmacro{\xoffset}{(\i-0.5)*\stringgap}
    \draw (stda\i.center) edge[std] (stdb\i.center);
}
\draw (stda2.center) edge[std,red] (stdb2.center);
\end{tikzpicture}
\sim
\begin{tikzpicture}[scale=0.39,baseline=(mrex)] 
\def\strings{6}
\pgfmathsetmacro{\stringgap}{1.0/\strings}
\node (mrex) at (0,0) [draw, rectangle, minimum width=2.34cm, minimum height=0.585cm] {mrex};
\node (std) at (0,3.5cm) [draw, rectangle, minimum width=2.34cm, minimum height=0.585cm] {std};
\foreach \i in {1,...,\strings}
{
    \pgfmathsetmacro{\xoffset}{(\i-0.5)*\stringgap}
    \path ($(std.north west)!\xoffset!(std.north east)$) node (stdb\i) {};
    \path (stdb\i) ++(0,1cm) node (stda\i) {};
    \path ($(std.south west)!\xoffset!(std.south east)$) node (btwna\i) {};
    \path (btwna\i) ++(0,-1cm) node (btwnb\i) {}
                    ++(0,-1cm) node (btwnc\i) {};
    \path ($(mrex.south west)!\xoffset!(mrex.south east)$) node (mrexa\i) {};
    \path (mrexa\i) ++(0,-1cm) node (mrexb\i) {}
                    ++(0,-1cm) node (mrexc\i) {};
}
\path (mrexc6.center) ++(1cm,0) coordinate (mrexc7);
\foreach \i in {1,3}
{
    \pgfmathsetmacro{\xoffset}{(\i-0.5)*\stringgap}
    \draw (mrexa\i.center) edge[string] (mrexc\i.center);
}
\draw (mrexa2.center) edge[string,red] (mrexc2.center);
\foreach \i in {4,...,6}
{
    \pgfmathsetmacro{\xoffset}{(\i-0.5)*\stringgap}
    \draw (mrexa\i.center) edge[string] (mrexb\i.center) 
          (mrexb\i.center) edge[std] (mrexc\i.center);
    \draw (mrexb\i.center) node[right=-2pt] {\footnotesize $\inclscal{black}$};
}
\foreach \i in {1,...,5}
{
    \pgfmathsetmacro{\xoffset}{(\i-0.5)*\stringgap}
    \draw (btwna\i.center) edge[std] (btwnb\i.center)
          (btwnb\i.center) edge[string] (btwnc\i.center);
    \draw (btwnb\i.center) node[right=-2pt] {\footnotesize $\projscal{black}$};
}
\draw (btwna6.center) edge[std,red] (btwnb6.center)
      (btwnb6.center) edge[string,red] (btwnc6.center);
\draw (btwnb6.center) node[right=-2pt] {\footnotesize $\projscal{red}$};
\draw[string, red, corner] (mrexc7) |- ($(btwnb6.center)!0.5!(btwnc6.center)$);
\foreach \i in {3,...,5}
{
    \pgfmathsetmacro{\xoffset}{(\i-0.5)*\stringgap}
    \draw (stda\i.center) edge[std] (stdb\i.center);
}
\draw (stda2.center) edge[std,red] (stdb2.center);
\end{tikzpicture}
& =
\begin{tikzpicture}[scale=0.39,baseline=(mrex)] 
\def\strings{6}
\pgfmathsetmacro{\stringgap}{1.0/\strings}
\node (mrex) at (0,0) [draw, rectangle, minimum width=2.34cm, minimum height=0.585cm] {mrex};
\node (std) at (0,3.5cm) [draw, rectangle, minimum width=2.34cm, minimum height=0.585cm] {std};
\foreach \i in {1,...,\strings}
{
    \pgfmathsetmacro{\xoffset}{(\i-0.5)*\stringgap}
    \path ($(std.north west)!\xoffset!(std.north east)$) node (stdb\i) {};
    \path (stdb\i) ++(0,1cm) node (stda\i) {};
    \path ($(std.south west)!\xoffset!(std.south east)$) node (btwna\i) {};
    \path (btwna\i) ++(0,-1cm) node (btwnb\i) {}
                    ++(0,-1cm) node (btwnc\i) {};
    \path ($(mrex.south west)!\xoffset!(mrex.south east)$) node (mrexa\i) {};
    \path (mrexa\i) ++(0,-1cm) node (mrexb\i) {}
                    ++(0,-1cm) node (mrexc\i) {};
}
\path (mrexc6.center) ++(2cm,0) coordinate (mrexc7);
\foreach \i in {1,3}
{
    \pgfmathsetmacro{\xoffset}{(\i-0.5)*\stringgap}
    \draw (mrexa\i.center) edge[string] (mrexc\i.center);
}
\draw (mrexa2.center) edge[string,red] (mrexc2.center);
\foreach \i in {4,...,6}
{
    \pgfmathsetmacro{\xoffset}{(\i-0.5)*\stringgap}
    \draw (mrexa\i.center) edge[string] (mrexb\i.center) 
          (mrexb\i.center) edge[std] (mrexc\i.center);
    \draw (mrexb\i.center) node[right=-2pt] {\footnotesize $\inclscal{black}$};
}
\foreach \i in {1,...,5}
{
    \pgfmathsetmacro{\xoffset}{(\i-0.5)*\stringgap}
    \draw (btwna\i.center) edge[std] (btwnb\i.center)
          (btwnb\i.center) edge[string] (btwnc\i.center);
    \draw (btwnb\i.center) node[right=-2pt] {\footnotesize $\projscal{black}$};
}
\draw (btwna6.center) edge[std,red] (btwnb6.center)
      (btwnb6.center) edge[string,red] (btwnc6.center);
\draw (btwnb6.center) node[right=-2pt] {\footnotesize $\projscal{red}$};
\path (mrexc7) ++(0,2.75cm) node[dot,fill=red,name=reddot] {};
\draw[string,red] (mrexc7) edge (reddot.center);
\draw (reddot) node[above left=-2pt] {\footnotesize $a_{\textcolor{red} s}^{-1}$};
\foreach \i in {3,...,5}
{
    \pgfmathsetmacro{\xoffset}{(\i-0.5)*\stringgap}
    \draw (stda\i.center) edge[std] (stdb\i.center);
}
\draw (stda2.center) edge[std,red] (stdb2.center);
\end{tikzpicture}+
\begin{tikzpicture}[scale=0.39,baseline=(mrex)] 
\def\strings{6}
\pgfmathsetmacro{\stringgap}{1.0/\strings}
\node (mrex) at (0,0) [draw, rectangle, minimum width=2.34cm, minimum height=0.585cm] {mrex};
\node (std) at (0,3.5cm) [draw, rectangle, minimum width=2.34cm, minimum height=0.585cm] {std};
\foreach \i in {1,...,\strings}
{
    \pgfmathsetmacro{\xoffset}{(\i-0.5)*\stringgap}
    \path ($(std.north west)!\xoffset!(std.north east)$) node (stdb\i) {};
    \path (stdb\i) ++(0,1cm) node (stda\i) {};
    \path ($(std.south west)!\xoffset!(std.south east)$) node (btwna\i) {};
    \path (btwna\i) ++(0,-1cm) node (btwnb\i) {}
                    ++(0,-1cm) node (btwnc\i) {};
    \path ($(mrex.south west)!\xoffset!(mrex.south east)$) node (mrexa\i) {};
    \path (mrexa\i) ++(0,-1cm) node (mrexb\i) {}
                    ++(0,-1cm) node (mrexc\i) {};
}
\path (mrexc6.center) ++(2cm,0) coordinate (mrexc7);
\foreach \i in {1,3}
{
    \pgfmathsetmacro{\xoffset}{(\i-0.5)*\stringgap}
    \draw (mrexa\i.center) edge[string] (mrexc\i.center);
}
\draw (mrexa2.center) edge[string,red] (mrexc2.center);
\foreach \i in {4,...,6}
{
    \pgfmathsetmacro{\xoffset}{(\i-0.5)*\stringgap}
    \draw (mrexa\i.center) edge[string] (mrexb\i.center) 
          (mrexb\i.center) edge[std] (mrexc\i.center);
    \draw (mrexb\i.center) node[right=-2pt] {\footnotesize $\inclscal{black}$};
}
\foreach \i in {1,...,5}
{
    \pgfmathsetmacro{\xoffset}{(\i-0.5)*\stringgap}
    \draw (btwna\i.center) edge[std] (btwnb\i.center)
          (btwnb\i.center) edge[string] (btwnc\i.center);
    \draw (btwnb\i.center) node[right=-2pt] {\footnotesize $\projscal{black}$};
}
\path (btwna6.center) ++(1cm,-0.5cm) node (turn) {};
\draw[std,red,corner] (btwna6.center) |- (turn.center);
\draw[string,red,corner] (mrexc7.center) |- (turn.center);
\draw (turn.center) node[above=0pt] {\footnotesize $\projscal{red}$};
\path (btwnc6.center) ++(0,0.5cm) node[dot,fill=red,name=reddot] {};
\draw (btwnc6.center) edge[string,red] (reddot.center);
\draw (reddot) node[right=1pt] {\footnotesize $a_{\textcolor{red} s}^{-1}$};
\foreach \i in {3,...,5}
{
    \pgfmathsetmacro{\xoffset}{(\i-0.5)*\stringgap}
    \draw (stda\i.center) edge[std] (stdb\i.center);
}
\draw (stda2.center) edge[std,red] (stdb2.center);
\end{tikzpicture} \\
& \sim\begin{tikzpicture}[scale=0.39,baseline=(rexbot)] 
\node (std) at (0,0) [draw, rectangle, minimum width=2.34cm, minimum height=0.585cm] {std};
\node (rexbot) at (1cm,3.5cm) [draw, rectangle, minimum width=1.56cm, minimum height=0.585cm] {rex};
\node (rextop) at (1cm,7cm) [draw, rectangle, minimum width=1.56cm, minimum height=0.585cm] {rex};
\pgfmathsetmacro{\stringgap}{1.0/4.0}
\foreach \i in {1,...,4}
{
    \pgfmathsetmacro{\xoffset}{(\i-0.5)*\stringgap}
    \path ($(rextop.north west)!\xoffset!(rextop.north east)$) node (rextopc\i) {};
    \path (rextopc\i) ++(0,1cm) node (rextopb\i) {}
                      ++(0,1cm) node (rextopa\i) {};
    \path ($(rextop.south west)!\xoffset!(rextop.south east)$) node (btwntopa\i) {};
    \path (btwntopa\i) ++(0,-2cm) node (btwntopb\i) {};
    \path ($(rexbot.south west)!\xoffset!(rexbot.south east)$) node (btwnbota\i) {};
    \path (btwnbota\i) ++(0,-1cm) node (btwnbotb\i) {}
                       ++(0,-1cm) node (btwnbotc\i) {};
}
\pgfmathsetmacro{\stringgap}{1.0/6.0}
\foreach \i in {1,...,6}
{
    \pgfmathsetmacro{\xoffset}{(\i-0.5)*\stringgap}
    \path ($(std.south west)!\xoffset!(std.south east)$) node (stda\i) {};
    \path (stda\i) ++(0,-1cm) node (stdb\i) {}
                   ++(0,-1cm) node (stdc\i) {};
}
\path (stdc6.center) ++(2cm,0) node (stdc7) {};
\foreach \i in {1,3}
{
    \pgfmathsetmacro{\xoffset}{(\i-0.5)*\stringgap}
    \draw (stda\i.center) edge[std] (stdb\i.center) 
          (stdb\i.center) edge[string] (stdc\i.center);
    \draw (stdb\i.center) node[right=-2pt] {\footnotesize $\projscal{black}$};
}
\draw (stda2.center) edge[std,red] (stdb2.center) 
      (stdb2.center) edge[string,red] (stdc2.center);
\draw (stdb2.center) node[right=-2pt] {\footnotesize $\projscal{red}$};
\foreach \i in {4,...,6}
{
    \pgfmathsetmacro{\xoffset}{(\i-0.5)*\stringgap}
    \draw (stda\i.center) edge[std] (stdc\i.center);
}
\draw[string, red, corner] (stdc7.center) |- ($(btwntopa6.center)!0.5!(btwntopb6.center)$);
\pgfmathsetmacro{\stringgap}{1.0/4.0}
\foreach \i in {1,3,4}
{
    \pgfmathsetmacro{\xoffset}{(\i-0.5)*\stringgap}
    \draw (btwnbota\i.center) edge[string] (btwnbotb\i.center)
          (btwnbotb\i.center) edge[std] (btwnbotc\i.center);
    \draw (btwnbotb\i.center) node[right=-2pt] {\footnotesize $\inclscal{black}$};
}
\draw (btwnbota2.center) edge[string,red] (btwnbotb2.center)
      (btwnbotb2.center) edge[std,red] (btwnbotc2.center);
\draw (btwnbotb2.center) node[right=-2pt] {\footnotesize $\inclscal{red}$};
\foreach \i in {1,...,3}
{
    \pgfmathsetmacro{\xoffset}{(\i-0.5)*\stringgap}
    \draw (btwntopa\i.center) edge[string] (btwntopb\i.center);
}
\draw (btwntopa4.center) edge[string,red] (btwntopb4.center);
\foreach \i in {1,3,4}
{
    \pgfmathsetmacro{\xoffset}{(\i-0.5)*\stringgap}
    \draw (rextopa\i.center) edge[std] (rextopb\i.center)
          (rextopb\i.center) edge[string] (rextopc\i.center);
    \draw (rextopb\i.center) node[right=-2pt] {\footnotesize $\projscal{black}$};
}
\draw (rextopa2.center) edge[std,red] (rextopb2.center)
      (rextopb2.center) edge[string,red] (rextopc2.center);
\draw (rextopb2.center) node[right=-2pt] {\footnotesize $\projscal{red}$};
\end{tikzpicture}+
\begin{tikzpicture}[scale=0.39,baseline=(rexbot)] 
\node (std) at (0,0) [draw, rectangle, minimum width=2.34cm, minimum height=0.585cm] {std};
\node (rexbot) at (1cm,3.5cm) [draw, rectangle, minimum width=1.56cm, minimum height=0.585cm] {rex};
\node (rextop) at (1cm,7cm) [draw, rectangle, minimum width=1.56cm, minimum height=0.585cm] {rex};
\pgfmathsetmacro{\stringgap}{1.0/4.0}
\foreach \i in {1,...,4}
{
    \pgfmathsetmacro{\xoffset}{(\i-0.5)*\stringgap}
    \path ($(rextop.north west)!\xoffset!(rextop.north east)$) node (rextopc\i) {};
    \path (rextopc\i) ++(0,1cm) node (rextopb\i) {}
                      ++(0,1cm) node (rextopa\i) {};
    \path ($(rextop.south west)!\xoffset!(rextop.south east)$) node (btwntopa\i) {};
    \path (btwntopa\i) ++(0,-2cm) node (btwntopb\i) {};
    \path ($(rexbot.south west)!\xoffset!(rexbot.south east)$) node (btwnbota\i) {};
    \path (btwnbota\i) ++(0,-1cm) node (btwnbotb\i) {}
                       ++(0,-1cm) node (btwnbotc\i) {};
}
\pgfmathsetmacro{\stringgap}{1.0/6.0}
\foreach \i in {1,...,6}
{
    \pgfmathsetmacro{\xoffset}{(\i-0.5)*\stringgap}
    \path ($(std.south west)!\xoffset!(std.south east)$) node (stda\i) {};
    \path (stda\i) ++(0,-1cm) node (stdb\i) {}
                   ++(0,-1cm) node (stdc\i) {};
}
\path (stdc6.center) ++(2cm,0) node (stdc7) {};
\foreach \i in {1,3}
{
    \pgfmathsetmacro{\xoffset}{(\i-0.5)*\stringgap}
    \draw (stda\i.center) edge[std] (stdb\i.center) 
          (stdb\i.center) edge[string] (stdc\i.center);
    \draw (stdb\i.center) node[right=-2pt] {\footnotesize $\projscal{black}$};
}
\draw (stdb2.center) edge[string,red] (stdc2.center);
\node[dot,fill=red] at (stdb2.center) {};
\foreach \i in {4,...,6}
{
    \pgfmathsetmacro{\xoffset}{(\i-0.5)*\stringgap}
    \draw (stda\i.center) edge[std] (stdc\i.center);
}
\path (stdc7.center) ++(-1cm,8.5cm) node (turn) {};
\draw[string,red,corner] (stdc7.center) |- (turn.center)
                         (btwntopa4.center) |- (turn.center);
\draw (turn) node[name=label,below=-1pt] {\footnotesize $a_{\textcolor{red} s}^{-1}$};
\pgfmathsetmacro{\stringgap}{1.0/4.0}
\foreach \i in {1,3,4}
{
    \pgfmathsetmacro{\xoffset}{(\i-0.5)*\stringgap}
    \draw (btwnbota\i.center) edge[string] (btwnbotb\i.center)
          (btwnbotb\i.center) edge[std] (btwnbotc\i.center);
    \draw (btwnbotb\i.center) node[right=-2pt] {\footnotesize $\inclscal{black}$};
}
\draw (btwnbota2.center) edge[string,red] (btwnbotb2.center);
\node[dot,fill=red] at (btwnbotb2.center)  {};
\draw (btwnbotb2.center) node[below=-2pt] {\footnotesize$a_{\textcolor{red} s}^{-1}$};
\foreach \i in {1,...,3}
{
    \pgfmathsetmacro{\xoffset}{(\i-0.5)*\stringgap}
    \draw (btwntopa\i.center) edge[string] (btwntopb\i.center);
}
\path ($(btwntopa4.center)!0.75!(btwntopb4.center)$) node[name=reddot,dot,fill=red] {};
\draw (reddot.center) edge[string,red] (btwntopb4.center);
\foreach \i in {1,3,4}
{
    \pgfmathsetmacro{\xoffset}{(\i-0.5)*\stringgap}
    \draw (rextopa\i.center) edge[std] (rextopb\i.center)
          (rextopb\i.center) edge[string] (rextopc\i.center);
    \draw (rextopb\i.center) node[right=-2pt] {\footnotesize $\projscal{black}$};
}
\draw (rextopa2.center) edge[std,red] (rextopb2.center)
      (rextopb2.center) edge[string,red] (rextopc2.center);
\draw (rextopb2.center) node[right=-2pt] {\footnotesize $\projscal{red}$};
\end{tikzpicture} \text{.}
\end{align*}
\endgroup

%% file: fig/LLempty0fix.tex
\begin{equation*}
\begin{tikzpicture}[scale=0.4,baseline=(mrex)] 
\def\strings{6}
\pgfmathsetmacro{\stringgap}{1.0/\strings}
\node (mrex) at (0,0) [draw, rectangle, minimum width=2.4cm, minimum height=0.6cm] {mrex};
\node (std) at (0,3.5cm) [draw, rectangle, minimum width=2.4cm, minimum height=0.6cm] {std};
\foreach \i in {1,...,\strings}
{
    \pgfmathsetmacro{\xoffset}{(\i-0.5)*\stringgap}
    \path ($(std.north west)!\xoffset!(std.north east)$) node (stdb\i) {};
    \path (stdb\i) ++(0,1cm) node (stda\i) {};
    \path ($(std.south west)!\xoffset!(std.south east)$) node (btwna\i) {};
    \path (btwna\i) ++(0,-1cm) node (btwnb\i) {}
                    ++(0,-1cm) node (btwnc\i) {};
    \path ($(mrex.south west)!\xoffset!(mrex.south east)$) node (mrexa\i) {};
    \path (mrexa\i) ++(0,-1cm) node (mrexb\i) {}
                    ++(0,-1cm) node (mrexc\i) {};
}
\path (mrexc6.center) ++(2cm,0) coordinate (mrexc7);
\foreach \i in {1,...,3}
{
    \pgfmathsetmacro{\xoffset}{(\i-0.5)*\stringgap}
    \draw (mrexa\i.center) edge[string] (mrexc\i.center);
}
\foreach \i in {5,6}
{
    \pgfmathsetmacro{\xoffset}{(\i-0.5)*\stringgap}
    \draw (mrexa\i.center) edge[string] (mrexb\i.center) 
          (mrexb\i.center) edge[std] (mrexc\i.center);
    \node at (mrexb\i.center) [right=-2pt] {\footnotesize $\inclscal{black}$};
}
    \draw (mrexa4.center) edge[red,string] (mrexb4.center) 
          (mrexb4.center) edge[red,std] (mrexc4.center);
    \node at (mrexb4.center) [right=-2pt] {\footnotesize $\inclscal{red}$};
\foreach \i in {1,...,3,5,6}
{
    \pgfmathsetmacro{\xoffset}{(\i-0.5)*\stringgap}
    \draw (btwna\i.center) edge[std] (btwnb\i.center)
          (btwnb\i.center) edge[string] (btwnc\i.center);
    \node at (btwnb\i.center) [right=-2pt] {\footnotesize $\projscal{black}$};
}
\draw (btwna4.center) edge[std,red] (btwnb4.center)
      (btwnb4.center) edge[string,red] (btwnc4.center);
\node at (btwnb4.center) [right=-2pt] {\footnotesize $\projscal{red}$};
\path (mrexc7) ++(0,2.75cm) node[dot,fill=red,name=reddot] {};
\draw[string,red] (mrexc7) edge (reddot.center);
\node at (reddot) [above right=-2pt] {\footnotesize $\projscal{red}$};
\foreach \i in {3,...,5}
{
    \pgfmathsetmacro{\xoffset}{(\i-0.5)*\stringgap}
    \draw (stda\i.center) edge[std] (stdb\i.center);
}
\draw (stda2.center) edge[std,red] (stdb2.center);
\end{tikzpicture}
\sim
\begin{tikzpicture}[scale=0.4,baseline=(mrextop)] 
\def\strings{6}
\pgfmathsetmacro{\stringgap}{1.0/\strings}
\node (mrexbot) at (0,0) [draw, rectangle, minimum width=2.4cm, minimum height=0.6cm] {mrex};
\node (mrextop) at (0,3.5cm) [draw, rectangle, minimum width=2.4cm, minimum height=0.6cm] {mrex};
\node (std) at (0,8cm) [draw, rectangle, minimum width=2.4cm, minimum height=0.6cm] {std};
\foreach \i in {1,...,\strings}
{
    \pgfmathsetmacro{\xoffset}{(\i-0.5)*\stringgap}
    \path ($(std.north west)!\xoffset!(std.north east)$) node (stdb\i) {};
    \path (stdb\i) ++(0,1cm) node (stda\i) {};
    \path ($(std.south west)!\xoffset!(std.south east)$) node (btwntopa\i) {};
    \path (btwntopa\i) ++(0,-1cm) node (btwntopb\i) {}
                       ++(0,-1cm) node (btwntopc\i) {}
                       ++(0,-1cm) node (btwntopd\i) {};
    \path ($(mrextop.south west)!\xoffset!(mrextop.south east)$) node (btwnbota\i) {};
    \path (btwnbota\i) ++(0,-2cm) node (btwnbotb\i) {};
    \path ($(mrexbot.south west)!\xoffset!(mrexbot.south east)$) node (mrexbota\i) {};
    \path (mrexbota\i) ++(0,-1cm) node (mrexbotb\i) {}
                       ++(0,-1cm) node (mrexbotc\i) {};
}
\path (mrexbotc6) ++(1cm,0) coordinate (mrexbotc7);
\foreach \i in {1,...,3}
{
    \pgfmathsetmacro{\xoffset}{(\i-0.5)*\stringgap}
    \draw (mrexbota\i.center) edge[string] (mrexbotc\i.center);
}
\foreach \i in {5,6}
{
    \pgfmathsetmacro{\xoffset}{(\i-0.5)*\stringgap}
    \draw (mrexbota\i.center) edge[string] (mrexbotb\i.center) 
          (mrexbotb\i.center) edge[std] (mrexbotc\i.center);
    \node at (mrexbotb\i.center) [right=-2pt] {\footnotesize $\inclscal{black}$};
}
    \draw (mrexbota4.center) edge[red,string] (mrexbotb4.center) 
          (mrexbotb4.center) edge[red,std] (mrexbotc4.center);
    \node at (mrexbotb4.center) [right=-2pt] {\footnotesize $\inclscal{red}$};
\foreach \i in {1,...,5}
{
    \pgfmathsetmacro{\xoffset}{(\i-0.5)*\stringgap}
    \draw (btwnbota\i.center) edge[string] (btwnbotb\i.center);
}
\draw (btwnbota6.center) edge[string,red] (btwnbotb6.center);
\draw[string, red, corner] (mrexbotc7) |- node[above right=-2pt] {\footnotesize $\projscal{red}$} ($(btwnbota6.center)!0.5!(btwnbotb6.center)$);
\foreach \i in {1,...,3}
{
    \pgfmathsetmacro{\xoffset}{(\i-0.5)*\stringgap}
    \draw (btwntopa\i.center) edge[std] (btwntopb\i.center)
          (btwntopb\i.center) edge[string] (btwntopd\i.center);
    \node at (btwntopb\i.center) [right=-2pt] {\footnotesize $\projscal{black}$};
}
\foreach \i in {5,6}
{
    \pgfmathsetmacro{\xoffset}{(\i-0.5)*\stringgap}
    \draw (btwntopa\i.center) edge[std] (btwntopc\i.center)
          (btwntopc\i.center) edge[string] (btwntopd\i.center);
    \node at (btwntopc\i.center) [right=-2pt] {\footnotesize $\projscal{black}$};
}
    \draw (btwntopa4.center) edge[red,std] (btwntopc4.center)
          (btwntopc4.center) edge[red,string] (btwntopd4.center);
    \node at (btwntopc4.center) [right=-2pt] {\footnotesize $\projscal{red}$};
\foreach \i in {3,...,5}
{
    \pgfmathsetmacro{\xoffset}{(\i-0.5)*\stringgap}
    \draw (stda\i.center) edge[std] (stdb\i.center);
}
\draw (stda2.center) edge[std,red] (stdb2.center);
\end{tikzpicture}-\begin{tikzpicture}[scale=0.4,baseline=(mrextop)] 
\def\strings{6}
\pgfmathsetmacro{\stringgap}{1.0/\strings}
\node (mrexbot) at (0,0) [draw, rectangle, minimum width=2.4cm, minimum height=0.6cm] {mrex};
\node (mrextop) at (0,3.5cm) [draw, rectangle, minimum width=2.4cm, minimum height=0.6cm] {mrex};
\node (std) at (0,8cm) [draw, rectangle, minimum width=2.4cm, minimum height=0.6cm] {std};
\foreach \i in {1,...,\strings}
{
    \pgfmathsetmacro{\xoffset}{(\i-0.5)*\stringgap}
    \path ($(std.north west)!\xoffset!(std.north east)$) node (stdb\i) {};
    \path (stdb\i) ++(0,1cm) node (stda\i) {};
    \path ($(std.south west)!\xoffset!(std.south east)$) node (btwntopa\i) {};
    \path (btwntopa\i) ++(0,-1cm) node (btwntopb\i) {}
                       ++(0,-1cm) node (btwntopc\i) {}
                       ++(0,-1cm) node (btwntopd\i) {};
    \path ($(mrextop.south west)!\xoffset!(mrextop.south east)$) node (btwnbota\i) {};
    \path (btwnbota\i) ++(0,-2cm) node (btwnbotb\i) {};
    \path ($(mrexbot.south west)!\xoffset!(mrexbot.south east)$) node (mrexbota\i) {};
    \path (mrexbota\i) ++(0,-1cm) node (mrexbotb\i) {}
                       ++(0,-1cm) node (mrexbotc\i) {};
}
\path (mrexbotc6) ++(1cm,0) coordinate (mrexbotc7);
\foreach \i in {1,3}
{
    \pgfmathsetmacro{\xoffset}{(\i-0.5)*\stringgap}
    \draw (mrexbota\i.center) edge[string] (mrexbotc\i.center);
}
\path ($(mrexbota2.center)!0.5!(mrexbotb2.center)$) node[dot,fill=red] (reddot1) {};
\path ($(mrexbotb2.center)!0.5!(mrexbotc2.center)$) node[dot,fill=red] (reddot2) {};
\draw (mrexbota2.center) edge[string,red] (reddot1.center);
\draw (mrexbotc2.center) edge[string,red] (reddot2.center);
\node (label) at (mrexbotb2) [yshift=1pt, xshift=4pt] {\footnotesize $a_{\textcolor{red} s}^{-1}$};
\foreach \i in {4,...,6}
{
    \pgfmathsetmacro{\xoffset}{(\i-0.5)*\stringgap}
    \draw (mrexbota\i.center) edge[string] (mrexbotb\i.center) 
          (mrexbotb\i.center) edge[std] (mrexbotc\i.center);
    \node at (mrexbotb\i.center) [right=-2pt] {\footnotesize $\inclscal{black}$};
}
\foreach \i in {1,...,5}
{
    \pgfmathsetmacro{\xoffset}{(\i-0.5)*\stringgap}
    \draw (btwnbota\i.center) edge[string] (btwnbotb\i.center);
}
\path ($(btwnbota6.center)!0.25!(btwnbotb6.center)$) node (turn) {};
\path ($(btwnbota6.center)!0.75!(btwnbotb6.center)$) node[dot,fill=red] (reddot3) {};
\node at (reddot3) [above right=-2pt] {\footnotesize $\projscal{red}$};
\draw (btwnbotb6.center) edge[string,red] (reddot3.center);
\draw[string, red, corner] (mrexbotc7) |- (turn.center) -- (btwnbota6.center);
\foreach \i in {1,3}
{
    \pgfmathsetmacro{\xoffset}{(\i-0.5)*\stringgap}
    \draw (btwntopa\i.center) edge[std] (btwntopb\i.center)
          (btwntopb\i.center) edge[string] (btwntopd\i.center);
    \node at (btwntopb\i.center) [right=-2pt] {\footnotesize $\projscal{black}$};
}
\draw (btwntopa2.center) edge[std,red] (btwntopb2.center)
      (btwntopb2.center) edge[string,red] (btwntopd2.center);
      \node at (btwntopb2.center) [right=-2pt] {\footnotesize $\projscal{red}$};
\foreach \i in {4,...,6}
{
    \pgfmathsetmacro{\xoffset}{(\i-0.5)*\stringgap}
    \draw (btwntopa\i.center) edge[std] (btwntopc\i.center)
          (btwntopc\i.center) edge[string] (btwntopd\i.center);
    \node at (btwntopc\i.center) [right=-2pt] {\footnotesize $\projscal{black}$};
}
\foreach \i in {3,...,5}
{
    \pgfmathsetmacro{\xoffset}{(\i-0.5)*\stringgap}
    \draw (stda\i.center) edge[std] (stdb\i.center);
}
\draw (stda2.center) edge[std,red] (stdb2.center);
\end{tikzpicture} \text{.}
\end{equation*}

%% file: fig/bigLLpropagate.tex
\begin{equation*}
\begin{tikzpicture}[scale=0.4,baseline=(mrextop)] 
\node (mrexbot) at (0,0) [draw, rectangle, minimum width=2.0cm, minimum height=0.6cm] {mrex};
\node (mrextop) at (0.5cm,3.5cm) [draw, rectangle, minimum width=2.4cm, minimum height=0.6cm] {mrex};
\node (std) at (0.5cm,8cm) [draw, rectangle, minimum width=2.4cm, minimum height=0.6cm] {std};
\foreach \i in {1,...,6}
{
    \pgfmathsetmacro{\stringgap}{1.0/6}
    \pgfmathsetmacro{\xoffset}{(\i-0.5)*\stringgap}
    \path ($(std.north west)!\xoffset!(std.north east)$) node (stdb\i) {};
    \path (stdb\i) ++(0,1cm) node (stda\i) {};
    \path ($(std.south west)!\xoffset!(std.south east)$) node (btwntopa\i) {};
    \path (btwntopa\i) ++(0,-1cm) node (btwntopb\i) {}
                       ++(0,-1cm) node (btwntopc\i) {}
                       ++(0,-1cm) node (btwntopd\i) {};
    \path ($(mrextop.south west)!\xoffset!(mrextop.south east)$) node (btwnbota\i) {};
    \path (btwnbota\i) ++(0,-1cm) node (btwnbotb\i) {}
                       ++(0,-1cm) node (btwnbotc\i) {};
    \pgfmathsetmacro{\stringgap}{1.0/5.0}
    \pgfmathsetmacro{\xoffset}{(\i-0.5)*\stringgap}                      
    \path ($(mrexbot.south west)!\xoffset!(mrexbot.south east)$) node (mrexbota\i) {};
    \path (mrexbota\i) ++(0,-1cm) node (mrexbotb\i) {}
                       ++(0,-1cm) node (mrexbotc\i) {};
}
\path (mrexbotc6) ++(2.4cm,0) coordinate (mrexbotc7);
\path (mrexbotc3) ++(0,-0.5cm) node (ellipses) {$\vdots$};
\pgfmathsetmacro{\stringgap}{1.0/5.0}
\foreach \i in {1,...,3}
{
    \pgfmathsetmacro{\xoffset}{(\i-0.5)*\stringgap}
    \draw (mrexbota\i.center) edge[string] (mrexbotc\i.center);
}
\foreach \i in {4,5}
{
    \pgfmathsetmacro{\xoffset}{(\i-0.5)*\stringgap}
    \draw (mrexbota\i.center) edge[string] (mrexbotb\i.center) 
          (mrexbotb\i.center) edge[std] (mrexbotc\i.center);
    \draw (mrexbotb\i.center) node[right=-2pt] {\footnotesize $\inclscal{black}$};
}
\node[dot,fill=blue] (bluedot) at (mrexbotb6) {};
\draw (mrexbotb6.center) edge[string,blue] (mrexbotc6.center);
\path (mrexbotc6) coordinate[right=2pt] (mrexbotbraceb) {};
\coordinate (mrexbotbracea) at (mrexbotbraceb |- btwnbotb5.center) {};
\draw [decoration={brace,amplitude=0.5em},decorate]
        (mrexbotbracea.north east) -- node[text=black,right=2pt] {\tiny $\phi_{n-1}$} (mrexbotbraceb.south east);
\pgfmathsetmacro{\stringgap}{1.0/6}
\foreach \i in {1,...,3}
{
    \pgfmathsetmacro{\xoffset}{(\i-0.5)*\stringgap}
    \draw (btwnbota\i.center) edge[string] (btwnbotc\i.center);
}
\foreach \i in {4,5}
{
    \pgfmathsetmacro{\xoffset}{(\i-0.5)*\stringgap}
    \draw (btwnbota\i.center) edge[string] ($(btwnbota\i.center)!0.5!(btwnbotb\i.center)$)
          ($(btwnbota\i.center)!0.5!(btwnbotb\i.center)$) edge [std]
          ($(btwnbotb\i.center)!0.5!(btwnbotc\i.center)$)
          ($(btwnbotb\i.center)!0.5!(btwnbotc\i.center)$) edge[string] (btwnbotc\i.center);
    \draw ($(btwnbota\i.center)!0.5!(btwnbotb\i.center)$) node[right=-2pt] {\footnotesize $\inclscal{black}$};
    \draw ($(btwnbotb\i.center)!0.5!(btwnbotc\i.center)$) node[right=-2pt] {\footnotesize $\projscal{black}$};
}
\draw[string, red, corner] (mrexbotc7) |- ($(btwnbota6.center)!0.5!(btwnbotb6.center)$) -- (btwnbota6.center);
\foreach \i in {1,3}
{
    \pgfmathsetmacro{\xoffset}{(\i-0.5)*\stringgap}
    \draw (btwntopa\i.center) edge[std] (btwntopb\i.center)
          (btwntopb\i.center) edge[string] (btwntopd\i.center);
    \draw (btwntopb\i.center) node[right=-2pt] {\footnotesize $\projscal{black}$};
}
\draw (btwntopa2.center) edge[std,red] (btwntopb2.center)
      (btwntopb2.center) edge[string,red] (btwntopd2.center);
      \draw (btwntopb2.center) node[right=-2pt] {\footnotesize $\projscal{red}$};
\foreach \i in {4,...,6}
{
    \pgfmathsetmacro{\xoffset}{(\i-0.5)*\stringgap}
    \draw (btwntopa\i.center) edge[std] (btwntopc\i.center)
          (btwntopc\i.center) edge[string] (btwntopd\i.center);
    \draw (btwntopc\i.center) node[right=-2pt] {\footnotesize $\projscal{black}$};
}
\path (mrexbotc7 |- btwnbotb5) coordinate[right=2pt] (mrextopbraceb) {};
\coordinate (mrextopbracea) at (mrextopbraceb |- btwntopb6.center) {};
\draw [decoration={brace,amplitude=0.5em},decorate]
        (mrextopbracea.south east) -- node[text=black,right=2pt] {\tiny $\phi_n$} (mrextopbraceb.south east);
\foreach \i in {3,...,5}
{
    \pgfmathsetmacro{\xoffset}{(\i-0.5)*\stringgap}
    \draw (stda\i.center) edge[std] (stdb\i.center);
}
\draw (stda2.center) edge[std,red] (stdb2.center);
\path (ellipses.south west -| mrexbot.west) coordinate[left=2pt] (LLbraceb) {};
\coordinate (LLbracea) at (LLbraceb |-mrextopbracea) {};
\draw [decoration={brace,amplitude=0.5em},decorate]
        (LLbraceb.west) -- node[text=black,left=2pt] {\tiny $\prescript{p|\ast}{}{\mathrm{LL}}$} (LLbracea.south west);
\end{tikzpicture}
\sim
\begin{tikzpicture}[scale=0.4,baseline=(std)] 
\node (mrex) at (0,0) [draw, rectangle, minimum width=2.0cm, minimum height=0.6cm] {mrex};
\node (std) at (0,3.5cm) [draw, rectangle, minimum width=2.0cm, minimum height=0.6cm] {std};
\node (rex) at (1.5cm,7cm) [draw, rectangle, minimum width=1.6cm, minimum height=0.6cm] {rex};
\pgfmathsetmacro{\stringgap}{1.0/4.0}
\foreach \i in {1,...,4}
{
    \pgfmathsetmacro{\xoffset}{(\i-0.5)*\stringgap}
    \path ($(rex.north west)!\xoffset!(rex.north east)$) node (rexc\i) {};
    \path (rexc\i) ++(0,1cm) node (rexb\i) {}
                   ++(0,1cm) node (rexa\i) {};
    \path ($(rex.south west)!\xoffset!(rex.south east)$) node (btwntopa\i) {};
    \path (btwntopa\i) ++(0,-1cm) node (btwntopb\i) {}
                       ++(0,-1cm) node (btwntopc\i) {};
}
\pgfmathsetmacro{\stringgap}{1.0/5.0}
\foreach \i in {1,...,6}
{
    \pgfmathsetmacro{\xoffset}{(\i-0.5)*\stringgap}
    \path ($(std.south west)!\xoffset!(std.south east)$) node (btwnbota\i) {};
    \path ($(mrex.north west)!\xoffset!(mrex.north east)$) node (btwnbotd\i) {};
    \path ($(btwnbota\i.center)!0.33!(btwnbotd\i.center)$) node (btwnbotb\i) {};
    \path ($(btwnbota\i.center)!0.67!(btwnbotd\i.center)$) node (btwnbotc\i) {};
    \path ($(mrex.south west)!\xoffset!(mrex.south east)$) node (mrexa\i) {};
    \path (mrexa\i) ++(0,-1cm) node (mrexb\i) {}
                    ++(0,-1cm) node (mrexc\i) {};
}
\path (mrexc6) ++(1.0cm,0) coordinate (mrexc7);
\path (mrexc3) ++(0,-0.5cm) node (ellipses) {$\vdots$};
\pgfmathsetmacro{\stringgap}{1.0/5.0}
\foreach \i in {1,...,3}
{
    \pgfmathsetmacro{\xoffset}{(\i-0.5)*\stringgap}
    \draw (mrexa\i.center) edge[string] (mrexc\i.center);
}
\foreach \i in {4,5}
{
    \pgfmathsetmacro{\xoffset}{(\i-0.5)*\stringgap}
    \draw (mrexa\i.center) edge[string] (mrexb\i.center) 
          (mrexb\i.center) edge[std] (mrexc\i.center);
    \draw (mrexb\i.center) node[right=-2pt] {\footnotesize $\inclscal{black}$};
}
\node[dot,fill=blue] (bluedot) at (mrexb6) {};
\draw (bluedot.center) edge[string,blue] (mrexc6.center);
\foreach \i in {1,...,3}
{
    \pgfmathsetmacro{\xoffset}{(\i-0.5)*\stringgap}
    \draw (btwnbota\i.center) edge[std] (btwnbotb\i.center)
          (btwnbotb\i.center) edge[string] (btwnbotd\i.center);
    \path (btwnbotb\i.center) node [right=-2pt] {\footnotesize $\projscal{black}$};
}
\foreach \i in {4,5}
{
    \pgfmathsetmacro{\xoffset}{(\i-0.5)*\stringgap}
    \draw (btwnbota\i.center) edge[std] (btwnbotc\i.center)
          (btwnbotc\i.center) edge[string] (btwnbotd\i.center);
    \draw (btwnbotc\i.center) node[right=-2pt] {\footnotesize $\projscal{black}$};
}
\pgfmathsetmacro{\stringgap}{1.0/4.0}
\foreach \i in {1,...,3}
{
    \pgfmathsetmacro{\xoffset}{(\i-0.5)*\stringgap}
    \draw (btwntopa\i.center) edge[string] (btwntopb\i.center)
          (btwntopb\i.center) edge[std] (btwntopc\i.center);
    \draw (btwntopb\i.center) node[right=-2pt] {\footnotesize $\inclscal{black}$};
}
\draw[string, red, corner] (mrexc7) |- ($(btwntopa4.center)!0.5!(btwntopb4.center)$) -- (btwntopa4.center);
\foreach \i in {2,...,4}
{
    \pgfmathsetmacro{\xoffset}{(\i-0.5)*\stringgap}
    \draw (rexa\i.center) edge[std] (rexb\i.center)
          (rexb\i.center) edge[string] (rexc\i.center);
    \draw (rexb\i.center) node[right=-2pt] {\footnotesize $\projscal{black}$};
}
\draw (rexa1.center) edge[std,red] (rexb1.center)
      (rexb1.center) edge[string,red] (rexc1.center);
\draw (rexb1.center) node[right=-2pt] {\footnotesize $\projscal{red}$};
\end{tikzpicture}
\sim
\begin{tikzpicture}[scale=0.4,baseline=(rexbot)] 
\node (std) at (0,0) [draw, rectangle, minimum width=2.0cm, minimum height=0.6cm] {std};
\node (rexbot) at (1cm,3.5cm) [draw, rectangle, minimum width=1.2cm, minimum height=0.6cm] {rex};
\node (rextop) at (1.5cm,7cm) [draw, rectangle, minimum width=1.6cm, minimum height=0.6cm] {rex};
\foreach \i in {1,...,4}
{
    \pgfmathsetmacro{\stringgap}{1.0/4.0}
    \pgfmathsetmacro{\xoffset}{(\i-0.5)*\stringgap}
    \path ($(rextop.north west)!\xoffset!(rextop.north east)$) node (rexc\i) {};
    \path (rexc\i) ++(0,1cm) node (rexb\i) {}
                   ++(0,1cm) node (rexa\i) {};
    \path ($(rextop.south west)!\xoffset!(rextop.south east)$) node (btwntopa\i) {};
    \path (btwntopa\i) ++(0,-0.5cm) node (btwntopb\i) {}
                       ++(0,-1cm) node (btwntopc\i) {}
                       ++(0,-0.5cm) node (btwntopd\i) {};
    \pgfmathsetmacro{\stringgap}{1.0/3.0}
    \pgfmathsetmacro{\xoffset}{(\i-0.5)*\stringgap}
    \path ($(rexbot.south west)!\xoffset!(rexbot.south east)$) node (btwnbota\i) {};
    \path (btwnbota\i) ++(0,-1cm) node (btwnbotb\i) {}
                       ++(0,-1cm) node (btwnbotc\i) {};
}
\pgfmathsetmacro{\stringgap}{1.0/5.0}
\foreach \i in {1,...,7}
{
    \pgfmathsetmacro{\xoffset}{(\i-0.5)*\stringgap}
    \path ($(std.south west)!\xoffset!(std.south east)$) node (stda\i) {};
    \path (stda\i) ++(0,-1cm) node (stdb\i) {}
                   ++(0,-1cm) node (stdc\i) {};
}
\path (stdc3) ++(0,-0.5cm) node (ellipses) {$\vdots$};
\foreach \i in {1,...,3}
{
    \draw (stda\i.center) edge[std] (stdb\i.center)
          (stdb\i.center) edge[string] (stdc\i.center);
    \draw (stdb\i.center) node[right=-2pt] {\footnotesize $\projscal{black}$};
}
\foreach \i in {4,5}
{
    \draw (stda\i.center) edge[std] (stdc\i.center);
}
\node[dot,fill=blue] (bluedot) at (stdb6) {};
\draw (bluedot.center) edge[string,blue] (stdc6.center);
\foreach \i in {1,...,3}
{
    \draw (btwnbota\i.center) edge[string] (btwnbotb\i.center)
          (btwnbotb\i.center) edge[std] (btwnbotc\i.center);
    \path (btwnbotb\i.center) node [right=-2pt] {\footnotesize $\inclscal{black}$};
}
\foreach \i in {1,...,3}
{
    \draw (btwntopa\i.center) edge[string] (btwntopb\i.center)
          (btwntopb\i.center) edge[std] (btwntopc\i.center)
          (btwntopc\i.center) edge[string] (btwntopd\i.center);
    \draw (btwntopb\i.center) node[right=-2pt] {\footnotesize $\inclscal{black}$};
    \draw (btwntopc\i.center) node[right=-2pt] {\footnotesize $\projscal{black}$};
}
\draw[string, red, corner] (stdc7.center) |- (btwntopb4.center) -- (btwntopa4.center);
\foreach \i in {2,...,4}
{
    \pgfmathsetmacro{\xoffset}{(\i-0.5)*\stringgap}
    \draw (rexa\i.center) edge[std] (rexb\i.center)
          (rexb\i.center) edge[string] (rexc\i.center);
    \draw (rexb\i.center) node[right=-2pt] {\footnotesize $\projscal{black}$};
}
\draw (rexa1.center) edge[std,red] (rexb1.center)
      (rexb1.center) edge[string,red] (rexc1.center);
\draw (rexb1.center) node[right=-2pt] {\footnotesize $\projscal{red}$};
\end{tikzpicture}
\end{equation*}

%% file: fig/prdashexample.tex
\begin{equation*}
\pr'\left(\begin{tikzpicture}[yscale=-1,scale=0.3,baseline=(fork),font=\tiny]
    \node (fork) at (0,0) {};
    \path (fork.center) ++(-1,1)  coordinate (forktl);
    \path (fork.center) ++(0,1)   coordinate (forktop);
    \path (fork.center) ++(1,1)   coordinate (forktr);
    \path (fork.center) ++(-1,-1) coordinate (forkbl);
    \path (fork.center) ++(0,-1)  coordinate (forkbot);
    \path (fork.center) ++(1,-1)  coordinate (forkbr);
    \coordinate (dom1) at (-1,-2) {};
    \coordinate (dom2) at (0,-2)  {};
    \coordinate (dom3) at (1,-2)  {};
    \coordinate (cod1) at (-1,2) {};
    \coordinate (cod2) at (0,2)  {};
    \coordinate (cod3) at (1,2)  {};
    \draw[string, red, corner] 
                                        (dom2) -- (forkbot);
    \node[dot,fill=red] (reddot3) at (forkbot) {};
    \draw[string, blue, corner] (fork.center) -- (cod2)
                                         (fork.center) -| (dom1)
                                         (fork.center) -| (dom3);
\end{tikzpicture}\right)=
\begin{pmatrix}
    \begin{pmatrix}
        \begin{tikzpicture}[scale=0.6,baseline=(reddot),font=\tiny,yscale=-1]
            \node[dot,fill=red] (reddot) at (0,0) {};
            \node (label) at (reddot) [anchor=east] {$\binom{a_{\textcolor{blue} 0}^{-1}}{1}$};
            \draw[string,red] (reddot.center) -- ++(0,-1cm);
        \end{tikzpicture} \\ 
        \begin{tikzpicture}[scale=0.6,baseline=(turn),font=\tiny,yscale=-1]
            \node (turn) at (0,0) {};
            \node (dom1) at (-0.5cm,-1cm) {};
            \node (dom2) at (0.5cm,-1cm) {};
            \draw[std,blue,corner] (dom1.center) |- (turn.center) -| (dom2.center);
            \node (label) at (dom1 |- turn) [anchor=east] {$\binom{-a_{\textcolor{blue} 0}^{-1}}{a_{\textcolor{red} 1}+2a_{\textcolor{blue}0}}$};
        \end{tikzpicture}
    \end{pmatrix} & 
    \begin{pmatrix}
        \begin{tikzpicture}[scale=0.6,baseline=(reddot),font=\tiny,yscale=-1]
            \coordinate (dom1) at (-0.5,0);
            \coordinate (dom2) at (0,0);
            \node[dot,fill=red] (reddot) at (-0.5,1) {};
            \draw[string,red] (reddot.center) -- node [black,anchor=north east] {$\binom{a_{\textcolor{blue} 0}^{-1}}{1}$} (dom1);
            \draw[std,blue] (dom2) -- ++(0,2cm);
        \end{tikzpicture} \\ 
        \begin{tikzpicture}[scale=0.6,baseline=(origin),font=\tiny,yscale=-1]
            \coordinate (origin) at (0,0) {};
            \coordinate (dom1) at (-0.5cm,-1cm) {};
            \coordinate (dom2) at (0.5cm,-1cm) {};
            \draw[std,blue,corner] (dom1) -- ++(0,2cm);
            \node (label) at (dom1 |- origin) [anchor=east] {$\binom{-a_{\textcolor{blue} 0}^{-1}}{a_{\textcolor{red} 1}+2a_{\textcolor{blue}0}}$};
        \end{tikzpicture}
    \end{pmatrix} & 0 \\
    \begin{pmatrix}
        \begin{tikzpicture}[scale=0.6,baseline=(reddot),font=\tiny,yscale=-1]
            \coordinate (dom1) at (-0.5,0);
            \coordinate (dom2) at (0,0);
            \coordinate (dom3) at (0.5,0);
            \node[dot,fill=red] (reddot) at (0.5,1) {};
            \draw (reddot.center) edge[string,red] (dom3);
            \draw[std,blue] (dom1) -- node [black,anchor=east] {$\binom{a_{\textcolor{blue} 0}^{-1}}{1}$} ++(0,2cm);
            \draw[std,blue] (dom2) -- ++(0,2cm);
        \end{tikzpicture} \\ 
        \begin{tikzpicture}[scale=0.6,baseline=(origin),font=\tiny,yscale=-1]
            \coordinate (origin) at (0,0) {};
            \coordinate (dom1) at (-0.25cm,-1cm) {};
            \coordinate (dom2) at (0.25cm,-1cm) {};
            \draw[std,blue,corner] (dom1) -- ++(0,2cm)
                                   (dom2) -- ++(0,2cm);
            
            \node (label) at (dom1 |- origin) [anchor=east] {$\binom{-a_{\textcolor{blue} 0}^{-1}}{a_{\textcolor{red} 1}+2a_{\textcolor{blue}0}}$};
        \end{tikzpicture}
    \end{pmatrix} &
    \begin{pmatrix}
        \begin{tikzpicture}[scale=0.6,baseline=(reddot),font=\tiny,yscale=-1]
            \coordinate (dom1) at (-0.5,0);
            \coordinate (dom2) at (0,0);
            \coordinate (dom3) at (0.5,0);
            \coordinate (dom4) at (1,0);
            \coordinate (turn) at (0.5,1);
            \node[dot,fill=red] (reddot) at (0.5,0.5) {};
            \draw (reddot.center) edge[string,red] (dom3);
            \draw[std,blue] (dom1) -- node [black,anchor=east] {$\binom{a_{\textcolor{blue} 0}^{-1}}{1}$} ++(0,2cm);
            \draw[std,blue,corner] (dom2) |- (turn) -| (dom4);
        \end{tikzpicture} \\ 
        \begin{tikzpicture}[scale=0.6,baseline=(origin),font=\tiny,yscale=-1]
            \coordinate (origin) at (0,0) {};
            \coordinate (dom1) at (-0.5cm,-1cm) {};
            \coordinate (dom2) at (0.5cm,-1cm) {};
            \draw[std,blue,corner] (dom1) -- ++(0,2cm);
            \node (label) at (dom1 |- origin) [anchor=east] {$\binom{-a_{\textcolor{blue} 0}^{-1}}{a_{\textcolor{red} 1}+2a_{\textcolor{blue}0}}$};
        \end{tikzpicture}
    \end{pmatrix} & 0 \\
    0 & 0 & a'
\end{pmatrix}
\end{equation*}
where
\begin{equation*}
a'=
\begin{tikzpicture}[scale=0.6,baseline=(fork),font=\tiny,yscale=-1]
        \coordinate (dom1) at (-0.5,0);
        \coordinate (dom2) at (0,0);
        \coordinate (dom3) at (0.5,0);
        \coordinate (dom4) at (2.0,0);
        \coordinate (fork) at (1.25,1);
        \node [black,anchor=south] (label) at (fork) {$\binom{1}{a_{\textcolor{red} 1}}$};
        \draw[std,blue] (dom1) -- ++(0,2cm);
        \draw[std,red] (dom2) -- ++(0,2cm);
        \draw[string,blue,corner] (dom3) |- (fork) -| (dom4)
                                  (fork) -- ++(0,1cm);
\end{tikzpicture}
\text{.}
\end{equation*}

%% file: fig/prexample.tex
\begin{equation*}
\pr\left(\begin{tikzpicture}[yscale=-1,scale=0.3,baseline=(fork),font=\tiny]
    \node (fork) at (0,0) {};
    \path (fork.center) ++(-1,1)  coordinate (forktl);
    \path (fork.center) ++(0,1)   coordinate (forktop);
    \path (fork.center) ++(1,1)   coordinate (forktr);
    \path (fork.center) ++(-1,-1) coordinate (forkbl);
    \path (fork.center) ++(0,-1)  coordinate (forkbot);
    \path (fork.center) ++(1,-1)  coordinate (forkbr);
    \coordinate (dom1) at (-1,-2) {};
    \coordinate (dom2) at (0,-2)  {};
    \coordinate (dom3) at (1,-2)  {};
    \coordinate (cod1) at (-1,2) {};
    \coordinate (cod2) at (0,2)  {};
    \coordinate (cod3) at (1,2)  {};
    \draw[string, red, corner] 
                                        (dom2) -- (forkbot);
    \node[dot,fill=red] (reddot3) at (forkbot) {};
    \draw[string, blue, corner] (fork.center) -- (cod2)
                                         (fork.center) -| (dom1)
                                         (fork.center) -| (dom3);
\end{tikzpicture}\right)=
\begin{pmatrix}
    \begin{pmatrix}
        \begin{tikzpicture}[scale=0.6,baseline=(reddot),font=\tiny,yscale=-1]
            \node[dot,fill=red] (reddot) at (0,0) {};
            \node (label) at (reddot) [anchor=east] {$\binom{a_{\textcolor{blue} 0}^{-1}}{1}$};
            \draw[string,red] (reddot.center) -- ++(0,-1cm);
        \end{tikzpicture} \\ 
        \begin{tikzpicture}[scale=0.6,baseline=(turn),font=\tiny,yscale=-1]
            \node (turn) at (0,0) {};
            \node (label) at (turn) {$\binom{-a_{\textcolor{blue} 0}^{-1}}{a_{\textcolor{red} 1}+2a_{\textcolor{blue}0}}$};
        \end{tikzpicture}
    \end{pmatrix} & 
    \begin{pmatrix}
        \begin{tikzpicture}[scale=0.6,baseline=(reddot),font=\tiny,yscale=-1]
            \coordinate (dom1) at (-0.5,0);
            \node[dot,fill=red] (reddot) at (-0.5,1) {};
            \draw[string,red] (reddot.center) -- node [black,anchor=north east] {$\binom{a_{\textcolor{blue} 0}^{-1}}{1}$} (dom1);
        \end{tikzpicture} \\ 
        \begin{tikzpicture}[scale=0.6,baseline=(origin),font=\tiny,yscale=-1]
            \coordinate (origin) at (0,0) {};
            \node (label) at (origin) {$\binom{-a_{\textcolor{blue} 0}^{-1}}{a_{\textcolor{red} 1}+2a_{\textcolor{blue}0}}$};
        \end{tikzpicture}
    \end{pmatrix} & 0 \\
    \begin{pmatrix}
        \begin{tikzpicture}[scale=0.6,baseline=(reddot),font=\tiny,yscale=-1]
            \coordinate (dom3) at (0.5,0);
            \node[dot,fill=red] (reddot) at (0.5,1) {};
            \node (label) at (reddot) [anchor=east] {$\binom{a_{\textcolor{blue} 0}^{-1}}{1}$};
            \draw (reddot.center) edge[string,red] (dom3);
        \end{tikzpicture} \\ 
        \begin{tikzpicture}[scale=0.6,baseline=(origin),font=\tiny,yscale=-1]
            \coordinate (origin) at (0,0) {};
%
            \node (label) at (origin) {$\binom{-a_{\textcolor{blue} 0}^{-1}}{a_{\textcolor{red} 1}+2a_{\textcolor{blue}0}}$};
        \end{tikzpicture}
    \end{pmatrix} &
    \begin{pmatrix}
        \begin{tikzpicture}[scale=0.6,baseline=(reddot),font=\tiny,yscale=-1]
            \coordinate (dom3) at (0.5,0);
            \node[dot,fill=red] (reddot) at (0.5,1) {};
            \node (label) at (reddot) [anchor=east] {$\binom{a_{\textcolor{blue} 0}^{-1}}{1}$};
            \draw (reddot.center) edge[string,red] (dom3);
        \end{tikzpicture} \\ 
        \begin{tikzpicture}[scale=0.6,baseline=(origin),font=\tiny,yscale=-1]
            \coordinate (origin) at (0,0) {};
            \node (label) at (origin) {$\binom{-a_{\textcolor{blue} 0}^{-1}}{a_{\textcolor{red} 1}+2a_{\textcolor{blue}0}}$};
        \end{tikzpicture}
    \end{pmatrix} & 0 \\
    0 & 0 & a
\end{pmatrix}
\end{equation*}
where
\begin{equation*}
a=
\begin{tikzpicture}[scale=0.6,baseline=(fork),font=\tiny,yscale=-1]
        \coordinate (dom1) at (-1.5,0);
        \coordinate (dom2) at (1.5,0);
        \coordinate (fork) at (0,1);
        \node [anchor=south] (label) at (fork) {$\binom{a_{\textcolor{blue} 0}(a_{\textcolor{red} 1}+2a_{\textcolor{blue} 0})}{-a_{\textcolor{blue} 0}^{-1}}$};
        \path (dom1) -- ++(0,2cm);
        \draw[string,realcyan,corner] (dom1) |- (fork) -| (dom2)
                                  (fork) -- ++(0,1cm);
\end{tikzpicture}
\text{.}
\end{equation*}

%% file: fig/hdmrdexample.tex
\begin{align*}
\pr\left(\begin{tikzpicture}[scale=0.3,baseline=(fork),font=\tiny]
    \node (fork) at (0,0) {};
    \path (fork.center) ++(-1,1)  coordinate (forktl);
    \path (fork.center) ++(0,1)   coordinate (forktop);
    \path (fork.center) ++(1,1)   coordinate (forktr);
    \path (fork.center) ++(-1,-1) coordinate (forkbl);
    \path (fork.center) ++(0,-1)  coordinate (forkbot);
    \path (fork.center) ++(1,-1)  coordinate (forkbr);
    \coordinate (dom1) at (-1,-2) {};
    \coordinate (dom2) at (0,-2)  {};
    \coordinate (dom3) at (1,-2)  {};
    \coordinate (cod1) at (-1,2) {};
    \coordinate (cod2) at (0,2)  {};
    \coordinate (cod3) at (1,2)  {};
    \draw[string, red, corner] (cod2) -- (forktop)
                               (dom2) -- (forkbot);
    \node[dot,fill=red] (reddot1) at (forktop) {};
    \node[dot,fill=red] (reddot2) at (forkbot) {};
    \draw[string, blue] (fork.center) -| (cod1)
                        (fork.center) -| (cod3)
                        (fork.center) -| (dom1)
                        (fork.center) -| (dom3);
\end{tikzpicture}^{\circ 2}\right)
& =
\begin{pmatrix}
A & A & 0 \\
A & A & 0 \\
0 & 0 & 
\begin{tikzpicture}[yscale=0.3,xscale=0.9,baseline=(fork),font=\tiny]
    \node (fork) at (0,0) {};
    \path (fork.center) ++(-1,1)  coordinate (forktl);
    \path (fork.center) ++(0,1)   coordinate (forktop);
    \path (fork.center) ++(1,1)   coordinate (forktr);
    \path (fork.center) ++(-1,-1) coordinate (forkbl);
    \path (fork.center) ++(0,-1)  coordinate (forkbot);
    \path (fork.center) ++(1,-1)  coordinate (forkbr);
    \coordinate (dom1) at (-1,-2) {};
    \coordinate (dom2) at (0,-2)  {};
    \coordinate (dom3) at (1,-2)  {};
    \coordinate (cod1) at (-1,2) {};
    \coordinate (cod2) at (0,2)  {};
    \coordinate (cod3) at (1,2)  {};
%
%
    \draw[string, realcyan] (fork.center) -| (cod1)
                        (fork.center) -| (cod3)
                        (fork.center) -| (dom1)
                        (fork.center) -| (dom3);
    \node (botlabel) at (forkbot) {$\binom{-a_{\textcolor{blue} 0}^{-1}}{a_{\textcolor{blue} 0}(a_{\textcolor{red} 1}+2a_{\textcolor{blue} 0})}$};
    \node (toplabel) at (forktop) {$\binom{a_{\textcolor{blue} 0}(a_{\textcolor{red} 1}+2a_{\textcolor{blue} 0})}{-a_{\textcolor{blue} 0}^{-1}}$};
\end{tikzpicture}
\end{pmatrix}^{\circ_{\prescript{p}{}{W}} 2} \\
& =
\begin{pmatrix}
A^2 & A^2 & 0 \\
A^2 & A^2 & 0 \\
0 & 0 & 
\begin{tikzpicture}[yscale=0.3,xscale=0.9,baseline=(fork),font=\tiny]
    \node (origin) at (0,0) {$-a_{\textcolor{red} 1}-2a_{\textcolor{blue} 0}$};
    \path (origin.center) ++(0,-1) coordinate (barbot);
    \path (origin.center) ++(0,1) coordinate (bartop);
    \path (bartop) ++(0,1)   coordinate (top);
    \path (barbot) ++(0,-1)  coordinate (bot);
    \path (barbot) ++(-1,-2) coordinate (dom1);
    \path (barbot) ++(1,-2) coordinate (dom2);
    \path (bartop) ++(-1,2) coordinate (cod1);
    \path (bartop) ++(1,2) coordinate (cod2);
%
%
    \draw[string, realcyan] (bartop) -| (cod1)
                        (bartop) -| (cod2)
                        (barbot) -| (dom1)
                        (barbot) -| (dom2)
                        (dom1) -- (cod1)
                        (dom2) -- (cod2);
    \node (botlabel) at (bot) {$\binom{-a_{\textcolor{blue} 0}^{-1}}{a_{\textcolor{blue} 0}(a_{\textcolor{red} 1}+2a_{\textcolor{blue} 0})}$};
    \node (toplabel) at (top) {$\binom{a_{\textcolor{blue} 0}(a_{\textcolor{red} 1}+2a_{\textcolor{blue} 0})}{-a_{\textcolor{blue} 0}^{-1}}$};
\end{tikzpicture}
\end{pmatrix} \\
& =-2\pr\left(\begin{tikzpicture}[scale=0.3,baseline=(fork),font=\tiny]
    \node (fork) at (0,0) {};
    \path (fork.center) ++(-1,1)  coordinate (forktl);
    \path (fork.center) ++(0,1)   coordinate (forktop);
    \path (fork.center) ++(1,1)   coordinate (forktr);
    \path (fork.center) ++(-1,-1) coordinate (forkbl);
    \path (fork.center) ++(0,-1)  coordinate (forkbot);
    \path (fork.center) ++(1,-1)  coordinate (forkbr);
    \coordinate (dom1) at (-1,-2) {};
    \coordinate (dom2) at (0,-2)  {};
    \coordinate (dom3) at (1,-2)  {};
    \coordinate (cod1) at (-1,2) {};
    \coordinate (cod2) at (0,2)  {};
    \coordinate (cod3) at (1,2)  {};
    \draw[string, red, corner] (cod2) -- (forktop)
                               (dom2) -- (forkbot);
    \node[dot,fill=red] (reddot1) at (forktop) {};
    \node[dot,fill=red] (reddot2) at (forkbot) {};
    \draw[string, blue] (fork.center) -| (cod1)
                        (fork.center) -| (cod3)
                        (fork.center) -| (dom1)
                        (fork.center) -| (dom3);
\end{tikzpicture}\right)
\end{align*}
where
\begin{equation*}
A=\begin{pmatrix}
    \begin{tikzpicture}[scale=0.3,baseline=(origin),font=\tiny]
        \node (origin) at (0,0) {};
        \path (origin.center) ++(0,1)   coordinate (top);
        \path (origin.center) ++(0,-1)  coordinate (bot);
        \coordinate (dom) at (0,-2)  {};
        \coordinate (cod) at (0,2)  {};
        \draw[string, red, corner] (cod) -- (top)
                                   (dom) -- (bot);
        \node[dot,fill=red] (reddot1) at (top) {};
        \node[dot,fill=red] (reddot2) at (bot) {};
        \node[anchor=east] (label) at (origin) {$a_{\textcolor{blue} 0}^{-1}$};
    \end{tikzpicture} & 
    \begin{tikzpicture}[scale=0.3,baseline=(origin),font=\tiny]
        \node (origin) at (0,0) {};
        \path (origin.center) ++(0,1)   coordinate (top);
        \path (origin.center) ++(0,-1)  coordinate (bot);
        \coordinate (dom) at (0,-2)  {};
        \coordinate (cod) at (0,2)  {};
        \draw[string, red, corner] (cod) -- (top);
        \node[dot,fill=red] (reddot1) at (top) {};
    %
        \node[anchor=east] (label) at (origin) {$\binom{a_{\textcolor{blue} 0}^{-1}(a_{\textcolor{red} 1}+2a_{\textcolor{blue} 0})}{-a_{\textcolor{blue} 0}^{-1}}$};
    \end{tikzpicture} \\
    \begin{tikzpicture}[yscale=-1,scale=0.3,baseline=(origin),font=\tiny]
        \node (origin) at (0,0) {};
        \path (origin.center) ++(0,1)   coordinate (top);
        \path (origin.center) ++(0,-1)  coordinate (bot);
        \coordinate (dom) at (0,-2)  {};
        \coordinate (cod) at (0,2)  {};
        \draw[string, red, corner] (cod) -- (top);
        \node[dot,fill=red] (reddot1) at (top) {};
    %
        \node[anchor=east] (label) at (origin) {$\binom{-a_{\textcolor{blue} 0}^{-1}}{a_{\textcolor{blue} 0}^{-1}(a_{\textcolor{red} 1}+2a_{\textcolor{blue} 0})}$};
    \end{tikzpicture} & -\frac{a_{\textcolor{red} 1}+2a_{\textcolor{blue} 0}}{a_{\textcolor{blue} 0}}
\end{pmatrix}
\end{equation*}
and
\begin{align*}
A^2& =
\begin{pmatrix}
    \begin{tikzpicture}[scale=0.3,baseline=(origin),font=\tiny]
            \node (origin) at (0,0) {};
            \path (origin.center) ++(0,1.33)   coordinate (top);
            \path (origin.center) ++(0,-1.33)  coordinate (bot);
            \coordinate (dom) at (0,-2)  {};
            \coordinate (cod) at (0,2)  {};
            \draw[string, red, corner] (cod) -- (top)
                                       (dom) -- (bot);
            \node[dot,fill=red] (reddot1) at (top) {};
            \node[dot,fill=red] (reddot2) at (bot) {};
            \node[dot,fill=red] at (0,-0.67) {};
            \node[dot,fill=red] at (0,0.67) {};
            \draw[string,red] (0,0.67) -- (0,-0.67);
            \node[anchor=east] (label) at (origin) {$a_{\textcolor{blue} 0}^{-2}$};
        \end{tikzpicture}-\begin{tikzpicture}[scale=0.3,baseline=(origin),font=\tiny]
        \node (origin) at (0,0) {};
        \path (origin.center) ++(0,1)   coordinate (top);
        \path (origin.center) ++(0,-1)  coordinate (bot);
        \coordinate (dom) at (0,-2)  {};
        \coordinate (cod) at (0,2)  {};
        \draw[string, red, corner] (cod) -- (top)
                                   (dom) -- (bot);
        \node[dot,fill=red] (reddot1) at (top) {};
        \node[dot,fill=red] (reddot2) at (bot) {};
        \node[anchor=east] (label) at (origin) {$\frac{a_{\textcolor{red} 1}+2a_{\textcolor{blue} 0}}{a_{\textcolor{blue} 0}^2}$};
    \end{tikzpicture} & 
    \begin{tikzpicture}[scale=0.3,baseline=(origin),font=\tiny]
        \node (origin) at (0,0) {};
        \path (origin.center) ++(0,1.33)   coordinate (top);
        \path (origin.center) ++(0,-1.33)  coordinate (bot);
        \coordinate (dom) at (0,-2)  {};
        \coordinate (cod) at (0,2)  {};
        \draw[string, red, corner] (cod) -- (top);
        \node[dot,fill=red] (reddot1) at (top) {};
        \node[dot,fill=red] at (0,-0.67) {};
        \node[dot,fill=red] at (0,0.67) {};
        \draw[string,red] (0,0.67) -- (0,-0.67);
        \node[anchor=east] (label) at (origin) {$\binom{a_{\textcolor{blue} 0}^{-2}(a_{\textcolor{red} 1}+2a_{\textcolor{blue} 0})}{-a_{\textcolor{blue} 0}^{-2}}$};
    \end{tikzpicture}+
    \begin{tikzpicture}[scale=0.3,baseline=(origin),font=\tiny]
        \node (origin) at (0,0) {};
        \path (origin.center) ++(0,1)   coordinate (top);
        \path (origin.center) ++(0,-1)  coordinate (bot);
        \coordinate (dom) at (0,-2)  {};
        \coordinate (cod) at (0,2)  {};
        \draw[string, red, corner] (cod) -- (top);
        \node[dot,fill=red] (reddot1) at (top) {};
    %
        \node[anchor=east] (label) at (origin) {$\binom{-a_{\textcolor{blue} 0}^{-2}(a_{\textcolor{red} 1}+2a_{\textcolor{blue} 0})^2}{a_{\textcolor{blue} 0}^{-2}(a_{\textcolor{red} 1}+2a_{\textcolor{blue} 0})}$};
    \end{tikzpicture} \\
    \begin{tikzpicture}[yscale=-1,scale=0.3,baseline=(origin),font=\tiny]
        \node (origin) at (0,0) {};
        \path (origin.center) ++(0,1.33)   coordinate (top);
        \path (origin.center) ++(0,-1.33)  coordinate (bot);
        \coordinate (dom) at (0,-2)  {};
        \coordinate (cod) at (0,2)  {};
        \draw[string, red, corner] (cod) -- (top);
        \node[dot,fill=red] (reddot1) at (top) {};
        \node[dot,fill=red] at (0,-0.67) {};
        \node[dot,fill=red] at (0,0.67) {};
        \draw[string,red] (0,0.67) -- (0,-0.67);
        \node[anchor=east] (label) at (origin) {$\binom{-a_{\textcolor{blue} 0}^{-2}}{a_{\textcolor{blue} 0}^2(a_{\textcolor{red} 1}+2a_{\textcolor{blue} 0})}$};
    \end{tikzpicture}+
    \begin{tikzpicture}[yscale=-1,scale=0.3,baseline=(origin),font=\tiny]
        \node (origin) at (0,0) {};
        \path (origin.center) ++(0,1)   coordinate (top);
        \path (origin.center) ++(0,-1)  coordinate (bot);
        \coordinate (dom) at (0,-2)  {};
        \coordinate (cod) at (0,2)  {};
        \draw[string, red, corner] (cod) -- (top);
        \node[dot,fill=red] (reddot1) at (top) {};
    %
        \node[anchor=east] (label) at (origin) {$\binom{a_{\textcolor{blue} 0}^{-2}(a_{\textcolor{red} 1}+2a_{\textcolor{blue} 0})}{-a_{\textcolor{blue} 0}^{-2}(a_{\textcolor{red} 1}+2a_{\textcolor{blue} 0})^2}$};
    \end{tikzpicture} 
    & 
    \begin{tikzpicture}[yscale=-1,scale=0.3,baseline=(origin),font=\tiny]
        \node (origin) at (0,0) {};
        \node[dot,fill=red] at (0,-0.67) {};
        \node[dot,fill=red] at (0,0.67) {};
        \draw[string,red] (0,0.67) -- (0,-0.67);
        \node[anchor=east] (label) at (origin) {$-\frac{a_{\textcolor{red} 1}+2a_{\textcolor{blue} 0}}{a_{\textcolor{blue} 0}^2}$};
    \end{tikzpicture}+\frac{(a_{\textcolor{red} 1}+2a_{\textcolor{blue} 0})^2}{a_{\textcolor{blue} 0}^2}
\end{pmatrix} \\
& =-2A \text{.}
\end{align*}